\documentclass{amsart}
\usepackage{amsmath,amssymb,amsthm,bm,bbm}
\usepackage{amscd}
\usepackage{graphicx,enumerate}
\usepackage{multirow}
\usepackage{layout}
\usepackage[OT2,T1]{fontenc}
\DeclareSymbolFont{cyrletters}{OT2}{wncyr}{m}{n}
\DeclareMathSymbol{\Sha}{\mathalpha}{cyrletters}{"58}
\usepackage[OT2,OT1]{fontenc}
\newcommand\cyr{\renewcommand\rmdefault{wncyr}
\renewcommand\sfdefault{wncyss}
\renewcommand\encodingdefault{OT2}
\normalfont\selectfont}
\DeclareTextFontCommand{\textcyr}{\cyr}

\theoremstyle{plain}
\newtheorem{theorem}{Theorem}[section]
\newtheorem*{theorem-nn}{Theorem}

\newtheorem{proposition}[theorem]{Proposition}
\newtheorem*{proposition-nn}{Proposition}

\theoremstyle{definition}
\newtheorem{definition}[theorem]{Definition}
\newtheorem{example}[theorem]{Example}
\newtheorem{remark}[theorem]{Remark}
\newtheorem{details}[theorem]{Details}
\newtheorem*{acknowledgments}{Acknowledgments}

\theoremstyle{remark}

\newcommand{\bZ}{\mathbbm{Z}}\newcommand{\bQ}{\mathbbm{Q}}
\newcommand{\bC}{\mathbbm{C}}
\newcommand{\bG}{\mathbbm{G}}\newcommand{\bF}{\mathbbm{F}}
\newcommand{\bA}{\mathbbm{A}}\newcommand{\bP}{\mathbbm{P}}
\newcommand{\cC}{\mathcal{C}}\newcommand{\cD}{\mathcal{D}}
\newcommand{\cH}{\mathcal{H}}\newcommand{\cS}{\mathcal{S}}

\newcommand{\GL}{{\rm GL}}\newcommand{\SL}{{\rm SL}}
\newcommand{\PSL}{{\rm PSL}}

\newcommand{\Ch}{\hspace*{-5.2mm}\text{\LARGE\cyr ch}}

\title{Norm one tori and Hasse norm principle}

\author[A. Hoshi]{Akinari Hoshi}
\address{Department of Mathematics, Niigata University, Niigata 950-2181, Japan}
\email{hoshi@math.sc.niigata-u.ac.jp}

\author[K. Kanai]{Kazuki Kanai}
\address{Graduate School of Science and Technology, Niigata University, Niigata 950-2181, Japan}
\email{kanai@m.sc.niigata-u.ac.jp}

\author[A. Yamasaki]{Aiichi Yamasaki}
\address{Department of Mathematics, Kyoto University, Kyoto 606-8502, Japan}
\email{aiichi.yamasaki@gmail.com}

\thanks{
This work was partially supported by JSPS KAKENHI Grant Numbers 
16K05059, 19K03418, 20K03511. 
Parts of the work were finished when the
first-named author and the third-named author 
were visiting the National Center for Theoretic Sciences (Taipei),
whose support is gratefully acknowledged.
}


\begin{document}
\maketitle
\begin{abstract}
Let $k$ be a field and $T$ be an algebraic $k$-torus. 
In 1969, 
over a global field $k$, Voskresenskii proved that there exists 
an exact sequence $0\to A(T)\to H^1(k,{\rm Pic}\,\overline{X})^\vee\to \Sha(T)\to 0$ where $A(T)$ is the kernel of the weak approximation of $T$, 
$\Sha(T)$ is the Shafarevich-Tate group of $T$, 
$X$ is a smooth $k$-compactification of $T$, 
$\overline{X}=X\times_k\overline{k}$, 
${\rm Pic}\,\overline{X}$ is the Picard group of $\overline{X}$ and 
$\vee$ stands for the Pontryagin dual. 
On the other hand, in 1963, Ono proved that 
for the norm one torus $T=R^{(1)}_{K/k}(\bG_m)$ of $K/k$, 
$\Sha(T)=0$ if and only if the Hasse norm principle holds for $K/k$. 
First, we determine $H^1(k,{\rm Pic}\, \overline{X})$ 
for algebraic $k$-tori $T$ up to dimension $5$. 
Second, we determine $H^1(k,{\rm Pic}\, \overline{X})$ 
for norm one tori $T=R^{(1)}_{K/k}(\bG_m)$ with 
$[K:k]=n\leq 15$ and $n\neq 12$. 
We also show that 
$H^1(k,{\rm Pic}\, \overline{X})=0$ 
for $T=R^{(1)}_{K/k}(\bG_m)$ 
when the Galois group of the Galois closure of $K/k$ 
is the Mathieu group $M_n\leq S_n$ 
with $n=11,12,22,23,24$. 
Third, we give 
a necessary and sufficient condition for the Hasse norm principle 
for $K/k$ with $[K:k]=n\leq 15$ and $n\neq 12$. 
As applications of the results, 
we get the group $T(k)/R$ of $R$-equivalence classes 
over a local field $k$ via Colliot-Th\'{e}l\`{e}ne and Sansuc's formula  
and the Tamagawa number $\tau(T)$ over a number field $k$ 
via Ono's formula $\tau(T)=|H^1(k,\widehat{T})|/|\Sha(T)|$. 
\end{abstract}
\tableofcontents
%
\section{Introduction}\label{S1}

Let $k$ be a field, 
$\overline{k}$ be a fixed separable closure of $k$ and 
$\mathcal{G}={\rm Gal}(\overline{k}/k)$ be the absolute Galois group of $k$. 
Let $T$ be an algebraic $k$-torus, 
i.e. a group $k$-scheme with fiber product (base change) 
$T\times_k \overline{k}=
T\times_{{\rm Spec}\, k}\,{\rm Spec}\, \overline{k}
\simeq (\bG_{m,\overline{k}})^n$; 
$k$-form of the split torus $(\bG_m)^n$. 
Then there exists the minimal (canonical) finite Galois extension $K/k$ 
with Galois group $G={\rm Gal}(K/k)$ such that 
$T$ splits over $K$: $T\times_k K\simeq (\bG_{m,K})^n$. 
It is also well-known that 
there is the duality between the category of $G$-lattices, 
i.e. finitely generated $\bZ[G]$-modules which are $\bZ$-free 
as abelian groups, 
and the category of algebraic $k$-tori which split over $K$ 
(see Ono \cite[Section 1.2]{Ono61}, 
Voskresenskii \cite[page 27, Example 6]{Vos98} and 
Knus, Merkurjev, Rost and Tignol \cite[page 333, Proposition 20.17]{KMRT98}). 
Indeed, if $T$ is an algebraic $k$-torus, then the character 
module $\widehat{T}={\rm Hom}(T,\bG_m)$ of $T$ 
may be regarded as a $G$-lattice. 
Let $X$ be a smooth $k$-compactification of $T$, 
i.e. smooth projective $k$-variety $X$ 
containing $T$ as a dense open subvariety, 
and $\overline{X}=X\times_k\overline{k}$. 
There exists such a smooth $k$-compactification of an algebraic $k$-torus $T$ 
over any field $k$ (due to Hironaka \cite{Hir64} for ${\rm char}\, k=0$, 
see Colliot-Th\'{e}l\`{e}ne, Harari and Skorobogatov 
\cite[Corollaire 1]{CTHS05} for any field $k$). 
A $\mathcal{G}$-lattice $P$ is said to be {\it permutation} if 
$P$ has a $\bZ$-basis permuted by $\mathcal{G}$ 
and 
a $\mathcal{G}$-lattice $F$ is said to be {\it flabby} 
(resp. {\it coflabby}) 
if $\widehat H^{-1}(\mathcal{H},F)=0$ 
(resp. $H^1(\mathcal{H},F)=0$) 
for any closed subgroup $\mathcal{H}\leq \mathcal{G}$ 
where $\widehat H$ is the Tate cohomology. 
\begin{theorem}[{Voskresenskii \cite[Section 4, page 1213]{Vos69}, \cite[Section 3, page 7]{Vos70}, see also \cite[Section 4.6]{Vos98}, \cite[Theorem 1.9]{Kun07}, \cite{Vos74} and \cite[Theorem 5.1, page 19]{CT07} for any field $k$}]\label{thVos69}
Let $k$ be a field 
and $\mathcal{G}={\rm Gal}(\overline{k}/k)$. 
Let $T$ be an algebraic $k$-torus, 
$X$ be a smooth $k$-compactification of $T$ 
and $\overline{X}=X\times_k\overline{k}$. 
Then there exists an exact sequence of $\mathcal{G}$-lattices 
\begin{align*}
0\to \widehat{T}\to \widehat{Q}\to {\rm Pic}\,\overline{X}\to 0
\end{align*}
where $\widehat{Q}$ is permutation 
and ${\rm Pic}\ \overline{X}$ is flabby. 
\end{theorem}
We have 
$H^1(k,{\rm Pic}\,\overline{X})\simeq H^1(G,{\rm Pic}\, X_K)$ 
where $K$ is the splitting field of $T$, $G={\rm Gal}(K/k)$ and 
$X_K=X\times_k K$. 
Hence Theorem \ref{thVos69} says that 
for $G$-lattices $M=\widehat{T}$ and $P=\widehat{Q}$, 
the exact sequence 
$0\to M\to P\to {\rm Pic}\, X_K\to 0$ 
gives a flabby resolution of 
$M$ and the flabby class of $M$ is 
$[M]^{fl}=[{\rm Pic}\ X_K]$ as $G$-lattices 
(see Section \ref{S3}, cf. Endo and Miyata's theorem 
\cite[Lemma 1.1]{EM75} (= Theorem \ref{thEM} in the present paper)). 

Let $k$ be a global field, 
i.e. a number field (a finite extension of $\bQ$) 
or a function field of an algebraic curve over 
$\bF_q$ (a finite extension of $\bF_q(t))$. 
Let $T$ be an algebraic $k$-torus 
and $T(k)$ be the group of $k$-rational points of $T$. 
Then $T(k)$ 
embeds into $\prod_{v\in V_k} T(k_v)$ by the diagonal map 
where 
$V_k$ is the set of all places of $k$ and 
$k_v$ is the completion of $k$ at $v$. 
Let $\overline{T(k)}$ be the closure of $T(k)$  
in the product $\prod_{v\in V_k} T(k_v)$. 
The group 
\begin{align*}
A(T)=\left(\prod_{v\in V_k} T(k_v)\right)/\overline{T(k)}
\end{align*}
is called {\it the kernel of the weak approximation} of $T$. 
We say that {\it $T$ has the weak approximation property} if $A(T)=0$. 

Let $E$ be a principal homogeneous space (= torsor) under $T$.  
{\it Hasse principle holds for $E$} means that 
if $E$ has a $k_v$-rational point for all $k_v$, 
then $E$ has a $k$-rational point. 
The set $H^1(k,T)$ classifies all such torsors $E$ up 
to (non-unique) isomorphism. 
We define {\it the Shafarevich-Tate group}
\begin{align*}
\Sha(T)={\rm Ker}\left\{H^1(k,T)\xrightarrow{\rm res} \bigoplus_{v\in V_k} 
H^1(k_v,T)\right\}.
\end{align*}
Then 
Hasse principle holds for all torsors $E$ under $T$ 
if and only if $\Sha(T)=0$. 
\begin{theorem}[{Voskresenskii \cite[Theorem 5, page 1213]{Vos69}, 
\cite[Theorem 6, page 9]{Vos70}, see also \cite[Section 11.6, Theorem, page 120]{Vos98}}]\label{thV}
Let $k$ be a global field, 
$T$ be an algebraic $k$-torus and $X$ be a smooth $k$-compactification of $T$. 
Then there exists an exact sequence
\begin{align*}
0\to A(T)\to H^1(k,{\rm Pic}\,\overline{X})^{\vee}\to \Sha(T)\to 0
\end{align*}
where $M^{\vee}={\rm Hom}(M,\bQ/\bZ)$ is the Pontryagin dual of $M$. 
Moreover, if $L$ is the splitting field of $T$ and $L/k$ 
is an unramified extension, then $A(T)=0$ and 
$H^1(k,{\rm Pic}\,\overline{X})^{\vee}\simeq \Sha(T)$. 
\end{theorem}
For the last assertion, see \cite[Theorem, page 120]{Vos98}.
It follows that 
$H^1(k,{\rm Pic}\,\overline{X})=0$ if and only if $A(T)=0$ and $\Sha(T)=0$, 
i.e. $T$ has the weak approximation property and 
Hasse principle holds for all torsors $E$ under $T$. 
Theorem \ref{thV} was generalized 
to the case of linear algebraic groups by Sansuc \cite{San81}.\\


{\it The norm one torus 
$R^{(1)}_{K/k}(\bG_m)$ of $K/k$} 
is the kernel of the norm map 
$R_{K/k}(\bG_m)\rightarrow \bG_m$ 
where 
$R_{K/k}$ is the Weil restriction 
(see \cite[page 37, Section 3.12]{Vos98}).
Such a torus $R^{(1)}_{K/k}(\bG_m)$ 
is biregularly isomorphic to the norm hypersurface 
$f(x_1,\ldots,x_n)=1$ where 
$f\in k[x_1,\ldots,x_n]$ is the polynomial of total 
degree $n$ defined by the norm map $N_{K/k}:K^\times\to k^\times$.
When $K/k$ is a finite Galois extension, 
we have that: 

\begin{theorem}[{Voskresenskii \cite[Theorem 7]{Vos70}, Colliot-Th\'{e}l\`{e}ne and Sansuc \cite[Proposition 1]{CTS77}}]
Let $k$ be a field and 
$K/k$ be a finite Galois extension with Galois group $G={\rm Gal}(K/k)$. 
Let $T=R^{(1)}_{K/k}(\bG_m)$ be the norm one torus of $K/k$ 
and $X$ be a smooth $k$-compactification of $T$. 
Then 
$H^1(H,{\rm Pic}\, X_K)\simeq H^3(H,\bZ)$ for any subgroup $H$ of $G$. 
In particular, 
$H^1(k,{\rm Pic}\, \overline{X})\simeq
H^1(G,{\rm Pic}\, X_K)\simeq H^3(G,\bZ)$ which is isomorphic to 
the Schur multiplier $M(G)$ of $G$.
\end{theorem}
In other words, for $G$-lattice $J_G=\widehat{T}$, 
$H^1(H,[J_G]^{fl})\simeq H^3(H,\bZ)$ for any subgroup $H$ of $G$ 
and $H^1(G,[J_G]^{fl})\simeq H^3(G,\bZ)\simeq H^2(G,\bQ/\bZ)$; 
the Schur multiplier of $G$. 
By the exact sequence $0\to\bZ\to\bZ[G]\to J_G\to 0$, 
we also have $\delta:H^1(G,J_G)\simeq H^2(G,\bZ)\simeq G^{ab}\simeq 
G/[G,G]$ where $\delta$ is the connecting homomorphism and 
$G^{ab}$ is the abelianization of $G$ (for details, see Section \ref{S2}). \\


Let $K$ 
be a finitely generated field extension of a field $k$. 
A field $K$ is called {\it rational over $k$} 
(or {\it $k$-rational} for short) 
if $K$ is purely transcendental over $k$, 
i.e. $K$ is isomorphic to $k(x_1,\ldots,x_n)$, 
the rational function field over $k$ with $n$ variables $x_1,\ldots,x_n$ 
for some integer $n$. 
$K$ is called {\it stably $k$-rational} 
if $K(y_1,\ldots,y_m)$ is $k$-rational for some algebraically 
independent elements $y_1,\ldots,y_m$ over $K$. 
Two fields 
$K$ and $K^\prime$ are called {\it stably $k$-isomorphic} if 
$K(y_1,\ldots,y_m)\simeq K^\prime(z_1,\ldots,z_n)$ over $k$ 
for some algebraically independent elements $y_1,\ldots,y_m$ over $K$ 
and $z_1,\ldots,z_n$ over $K^\prime$. 
When $k$ is an infinite field, 
$K$ is called {\it retract $k$-rational} 
if there is a $k$-algebra $R$ contained in $K$ such that 
(i) $K$ is the quotient field of $R$, and (ii) 
the identity map $1_R : R\rightarrow R$ factors through a localized 
polynomial ring over $k$, i.e. there is an element $f\in k[x_1,\ldots,x_n]$, 
which is the polynomial ring over $k$, and there are $k$-algebra 
homomorphisms $\varphi : R\rightarrow k[x_1,\ldots,x_n][1/f]$ 
and $\psi : k[x_1,\ldots,x_n][1/f]\rightarrow R$ satisfying 
$\psi\circ\varphi=1_R$ (cf. \cite{Sal84}). 
$K$ is called {\it $k$-unirational} 
if $k\subset K\subset k(x_1,\ldots,x_n)$ for some integer $n$. 
It is not difficult to see that 
``$k$-rational'' $\Rightarrow$ ``stably $k$-rational'' $\Rightarrow$ 
``retract $k$-rational'' $\Rightarrow$ ``$k$-unirational''. 

An algebraic $k$-torus $T$ is said to be {\it $k$-rational} 
(resp. {\it stably $k$-rational}, {\it retract $k$-rational}) 
if the function field $k(T)$ of $T$ is $k$-rational 
(resp. stably $k$-rational, retract $k$-rational). 

Note that an algebraic $k$-torus $T$ is always $k$-unirational 
(see \cite[page 40, Example 21]{Vos98}).
Tori of dimension $n$ over $k$ correspond bijectively 
to the elements of the set $H^1(\mathcal{G},\GL_n(\bZ))$ 
where $\mathcal{G}={\rm Gal}(k_{\rm s}/k)$ since 
${\rm Aut}((\bG_m)^n)=\GL_n(\bZ)$. 
The algebraic $k$-torus $T$ of dimension $n$ is determined uniquely 
by the integral representation $h : \mathcal{G}\rightarrow \GL_n(\bZ)$ 
up to conjugacy, and the group $h(\mathcal{G})$ is a finite subgroup of 
$\GL_n(\bZ)$ (see \cite[page 57, Section 4.9]{Vos98})).

There are $2$ (resp. $13$, $73$, $710$, $6079$) $\bZ$-classes forming 
$2$ (resp. $10$, $32$, $227$, $955$) $\bQ$-classes 
in $\GL_1(\bZ)$ (resp. $\GL_2(\bZ)$, $\GL_3(\bZ)$, $\GL_4(\bZ)$, $\GL_5(\bZ)$).
It is easy to see that all the $1$-dimensional algebraic $k$-tori $T$, 
i.e. the trivial torus $\bG_m$ and the norm one torus 
$R^{(1)}_{K/k}(\bG_m)$ of $K/k$ with $[K:k]=2$, are $k$-rational. 
Voskresenskii \cite{Vos67} proved that 
all the $13$ cases of $2$-dimensional algebraic $k$-tori, 
which correspond to $13$ $\bZ$-conjugacy classes of 
finite subgroups of ${\rm GL}_2(\bZ)$, 
are $k$-rational. 
Note that whether all the $13$ cases indeed occur or not 
depends on a base field $k$. 
The same applies for the numbers $15$, $216$ and $3003$ 
in Theorems \ref{thKun1}, \ref{thmain1-4} and \ref{thmain1-5} below. 
We also note that 
$T$ is retract $k$-rational 
$\Rightarrow$ $H^1(k,{\rm Pic}\,\overline{X})=0$ 
and over global field $k$, 
$H^1(k,{\rm Pic}\,\overline{X})=0\Rightarrow A(T)\simeq\Sha(T)=0$ 
(see Section \ref{S3} and also Manin \cite[\S 30]{Man74}). 

Kunyavskii \cite{Kun90} solved 
the rationality problem for $3$-dimensional algebraic $k$-tori. 
In the classification, there exist $73$ cases of 
$3$-dimensional algebraic $k$-tori 
which correspond to $73$ $\bZ$-conjugacy classes of 
finite subgroups of ${\rm GL}_3(\bZ)$, and 
$15$ cases of them are not $k$-rational 
(resp. not stably $k$-rational, not retract $k$-rational). 
Using the classification, 
Kunyavskii \cite{Kun84} showed that 
only $2$ cases of algebraic $k$-tori of dimension $3$ satisfy 
the non-vanishing $H^1(k,{\rm Pic}\,\overline{X})\neq 0$ 
among the $15$ cases of non-rational $k$-tori. 
These two $k$-tori are norm one tori $T=R^{(1)}_{K/k}(\bG_m)$ 
with $[K:k]=4$: 

\begin{theorem}[{Kunyavskii \cite[Proposition 1]{Kun84}}]\label{thKun1}
Let $k$ be a 
field, $T$ be an algebraic $k$-torus of dimension $3$ 
and $X$ be a smooth $k$-compactification of $T$. 
Then, among the $($at most$)$ 
$15$ cases of non-rational algebraic $k$-tori $T$, 
\begin{align*}
H^1(k,{\rm Pic}\, \overline{X})=
\begin{cases}
\bZ/2\bZ&{\rm if}\ T=R^{(1)}_{K_1/k}(\bG_m)\ {\rm or}\ R^{(1)}_{K_2/k}(\bG_m)\\
0&{\rm otherwise}
\end{cases}
\end{align*}
where $K_1/k$ $($resp. $K_2/k$$)$ is a field extension of degree $4$ 
whose Galois closure $L_1/k$ $($resp. $L_2/k$$)$ satisfies 
${\rm Gal}(L_1/k)\simeq V_4$; the Klein four group 
$($resp. ${\rm Gal}(L_2/k)\simeq A_4$; the alternating group of degree $4$$)$. 
In particular, if $k$ is a global field, then 
$A(T)\simeq\Sha(T)=0$ except for 
$T=R^{(1)}_{K_1/k}(\bG_m)$ and $T=R^{(1)}_{K_2/k}(\bG_m)$. 
\end{theorem}

Hoshi and Yamasaki \cite{HY17} classified stably/retract 
$k$-rational algebraic $k$-tori of dimensions $4$ and $5$. 
In the classification, there exist $710$ (resp. $6079$) 
cases of $4$-dimensional (resp. $5$-dimensional) algebraic $k$-tori 
which correspond to $710$ (resp. $6079$) $\bZ$-conjugacy classes 
of finite subgroups of ${\rm GL}_4(\bZ)$ (resp. ${\rm GL}_5(\bZ)$), 
and $216$ (resp. $3003$) cases of them 
are not retract $k$-rational. 

The first main result (Theorem \ref{thmain1-4} and Theorem \ref{thmain1-5}) 
of this paper is to classify the algebraic $k$-tori $T$ with non-vanishing 
$H^1(k,{\rm Pic}\, \overline{X})\neq 0$ in dimensions 
$4$ and $5$: 
\begin{theorem}[see Theorem \ref{thmain1-4p} for the detailed statement]\label{thmain1-4}
Let $k$ be a field, $T$ be an algebraic $k$-torus of dimension $4$ 
and $X$ be a smooth $k$-compactification of $T$. 
Among the $($at most$)$ 
$216$ cases of not retract rational algebraic $k$-tori $T$,
there exist $2$ $($resp. $20$, $194$$)$ cases of 
algebraic $k$-tori with 
$H^1(k,{\rm Pic}\, \overline{X})\simeq(\bZ/2\bZ)^{\oplus 2}$ 
$($resp. $H^1(k,{\rm Pic}\, \overline{X})\simeq\bZ/2\bZ$, 
$H^1(k,{\rm Pic}\, \overline{X})=0$$)$. 
\end{theorem}
\begin{theorem}[see Theorem \ref{thmain1-5p} for the detailed statement]\label{thmain1-5}
Let $k$ be a field, $T$ be an algebraic $k$-torus of dimension $5$ 
and $X$ be a smooth $k$-compactification of $T$. 
Among the $($at most$)$ 
$3003$ cases of not retract rational algebraic $k$-tori $T$,
there exist $11$ $($resp. $263$, $2729$$)$ cases of 
algebraic $k$-tori with 
$H^1(k,{\rm Pic}\, \overline{X})\simeq(\bZ/2\bZ)^{\oplus 2}$ 
$($resp. $H^1(k,{\rm Pic}\, \overline{X})\simeq\bZ/2\bZ$, 
$H^1(k,{\rm Pic}\, \overline{X})=0$$)$. 
\end{theorem}
Note that Hoshi and Yamasaki \cite[Chapter 7]{HY17} 
showed the vanishing $H^1(k,{\rm Pic}\, \overline{X})\simeq H^1(G,[\widehat{T}]^{fl})=0$ for any Bravais group $G$ of dimension $n\leq 6$ 
(see also \cite{Vos83}, \cite[Section 8]{Vos98}). 
There exists $1$ (resp. $5$, $14$, $64$, $189$, $841$) 
Bravais group of dimension $n=1$ (resp. $2$, $3$, $4$, $5$, $6$) 
(see \cite[Example 4.16]{HY17}).\\ 

Let $G$ be a finite group and $M$ be a $G$-lattice. We define 
\begin{align*}
\Sha^i_\omega(G,M):={\rm Ker}\left\{H^i(G,M)\xrightarrow{{\rm res}}\bigoplus_{g\in G}H^i(\langle g\rangle,M)\right\}\quad (i\geq 1) .
\end{align*}
The following is a theorem of Colliot-Th\'{e}l\`{e}ne and Sansuc \cite{CTS87}: 
\begin{theorem}[{Colliot-Th\'{e}l\`{e}ne and Sansuc \cite[Proposition 9.5 (ii)]{CTS87}, see also \cite[Proposition 9.8]{San81} and \cite[page 98]{Vos98}}]
Let $k$ be a field 
with ${\rm char}\, k=0$
and $K/k$ be a finite Galois extension 
with Galois group $G={\rm Gal}(K/k)$. 
Let $T$ be an algebraic $k$-torus which splits over $K$ and 
$X$ be a smooth $k$-compactification of $T$. 
Then we have 
\begin{align*}
\Sha^2_\omega(G,\widehat{T})\simeq 
H^1(G,{\rm Pic}\, X_K)\simeq {\rm Br}(X)/{\rm Br}(k)
\end{align*}
where 
${\rm Br}(X)$ is the \'etale cohomological Brauer Group of $X$ 
$($it is the same as the Azumaya-Brauer group of $X$ 
for such $X$, see \cite[page 199]{CTS87}$)$. 
\end{theorem}

In other words, for $G$-lattice $M=\widehat{T}$, 
we have 
$H^1(k,{\rm Pic}\, \overline{X})\simeq H^1(G,{\rm Pic}\, X_K)\simeq 
H^1(G,[M]^{fl})\simeq \Sha^2_\omega(G,M)\simeq {\rm Br}(X)/{\rm Br}(k)$ 
(for the flabby class $[M]^{fl}$ of $M$, see Section \ref{S3}).
Hence Theorem \ref{thKun1}, 
Theorem \ref{thmain1-4} and Theorem \ref{thmain1-5} compute 
$H^1(G,[M]^{fl})\simeq\Sha^2_\omega(G,M)\simeq {\rm Br}(X)/{\rm Br}(k)$ 
where $M=\widehat{T}$. 
We also see  
${\rm Br}_{\rm nr}(k(X)/k)={\rm Br}(X)\subset {\rm Br}(k(X))$ 
(see Colliot-Th\'{e}l\`{e}ne \cite[Theorem 5.11]{CTS07}, 
Saltman \cite[Proposition 10.5]{Sal99}).\\ 

Let $k$ be a 
global field, 
$K/k$ be a finite extension and 
$\bA_K^\times$ be the idele group of $K$. 
We say that {\it the Hasse norm principle holds for $K/k$} 
if $(N_{K/k}(\bA_K^\times)\cap k^\times)/N_{K/k}(K^\times)=1$ 
where $N_{K/k}$ is the norm map. 

Hasse \cite[Satz, page 64]{Has31} proved that 
the Hasse norm principle holds for any cyclic extension $K/k$ 
but does not hold for bicyclic extension $\bQ(\sqrt{-39},\sqrt{-3})/\bQ$. 
For Galois extensions $K/k$, Tate \cite{Tat67} gave the following theorem:
\begin{theorem}[{Tate \cite[page 198]{Tat67}}]\label{thTate}
Let $k$ be a global field, $K/k$ be a finite Galois extension 
with Galois group ${\rm Gal}(K/k)\simeq G$. 
Let $V_k$ be the set of all places of $k$ 
and $G_v$ be the decomposition group of $G$ at $v\in V_k$. 
Then we have 
\begin{align*}
(N_{K/k}(\bA_K^\times)\cap k^\times)/N_{K/k}(K^\times)\simeq 
{\rm Coker}\left\{\bigoplus_{v\in V_k}\widehat H^{-3}(G_v,\bZ)\xrightarrow{\rm cores}\widehat H^{-3}(G,\bZ)\right\}
\end{align*}
where $\widehat H$ is the Tate cohomology. 
In particular, the Hasse norm principle holds for $K/k$ 
if and only if the restriction map 
$H^3(G,\bZ)\xrightarrow{\rm res}\bigoplus_{v\in V_k}H^3(G_v,\bZ)$ 
is injective. 
\end{theorem}

Let $S_n$ (resp. $A_n$, $D_n$, $C_n$) be the symmetric 
(resp. the alternating, the dihedral, the cyclic) group 
of degree $n$ of order $n!$ (resp. $n!/2$, $2n$, $n$). 
Let $V_4\simeq C_2\times C_2$ be the Klein four group. 

If $G\simeq C_n$ is cyclic, then 
$\widehat H^{-3}(G,\bZ)\simeq H^3(G,\bZ)\simeq H^1(G,\bZ)=0$ 
and hence the Hasse's original theorem follows. 
If there exists a place $v$ of $k$ such that $G_v=G$, then 
the Hasse norm principle also holds for $K/k$. 
For example, the Hasse norm principle holds for $K/k$ with 
$G\simeq V_4$ if and only if 
there exists a place $v$ of $k$ such that $G_v=V_4$ because 
$H^3(V_4,\bZ)\simeq\bZ/2\bZ$ and $H^3(C_2,\bZ)=0$. 
The Hasse norm principle holds for $K/k$ with 
$G\simeq (C_2)^3$ if and only if {\rm (i)} 
there exists a place $v$ of $k$ such that $G_v=G$ 
or {\rm (ii)} 
there exist places $v_1,v_2,v_3$ of $k$ such that 
$G_{v_i}\simeq V_4$ and 
$H^3(G,\bZ)\xrightarrow{\rm res}
H^3(G_{v_1},\bZ)\oplus H^3(G_{v_2},\bZ)\oplus H^3(G_{v_3},\bZ)$ 
is an isomorphism because 
$H^3(G,\bZ)\simeq(\bZ/2\bZ)^{\oplus 3}$ and 
$H^3(V_4,\bZ)\simeq \bZ/2\bZ$. 

Ono \cite{Ono63} established the relationship 
between the Hasse norm principle for $K/k$ 
and the Hasse principle for all torsors under 
the norm one torus $R^{(1)}_{K/k}(\bG_m)$ of $K/k$: 
\begin{theorem}[{Ono \cite[page 70]{Ono63}, see also Platonov \cite[page 44]{Pla82}, Kunyavskii \cite[Remark 3]{Kun84}, Platonov and Rapinchuk \cite[page 307]{PR94}}]\label{thOno}
Let $k$ be a global field and $K/k$ be a finite extension. 
Then 
\begin{align*}
\Sha(R^{(1)}_{K/k}(\bG_m))\simeq (N_{K/k}(\bA_K^\times)\cap k^\times)/N_{K/k}(K^\times).
\end{align*}
In particular, $\Sha(R^{(1)}_{K/k}(\bG_m))=0$ if and only if 
the Hasse norm principle holds for $K/k$. 
\end{theorem}

The Hasse norm principle for Galois extensions $K/k$ 
was investigated by Gerth \cite{Ger77}, \cite{Ger78} and 
Gurak \cite{Gur78a}, \cite{Gur78b}, \cite{Gur80} (see also \cite[pages 308--309]{PR94}), etc. 
Gurak \cite{Gur78a} showed that 
the Hasse norm principle holds for Galois extension $K/k$ 
if all the Sylow subgroups of ${\rm Gal}(K/k)$ are cyclic. 
Note that this also follows from Theorem \ref{thOno} 
and the retract $k$-rationality of 
$T=R^{(1)}_{K/k}(\bG_m)$ due to Endo and Miyata \cite[Theorem 2.3]{EM75}. 

However, for non-Galois extension $K/k$, 
very little is known about the Hasse norm principle. 
Bartels \cite{Bar81a} (resp. \cite{Bar81b}) 
showed that the Hasse norm principle for $K/k$ holds 
when $[K:k]$ is prime (resp. ${\rm Gal}(L/k)\simeq D_n$). 
The former case also follows from Theorem \ref{thOno} and 
the retract $k$-rationality of 
$T=R^{(1)}_{K/k}(\bG_m)$ due to 
Colliot-Th\'{e}l\`{e}ne and Sansuc \cite[Proposition 9.1]{CTS87} 
(see Theorem \ref{thS}). 

%
\begin{theorem}[{Voskresenskii and Kunyavskii \cite{VK84}, see also Voskresenskii \cite[Theorem 4, Corollary]{Vos88}}]\label{thKV84}
Let $k$ be a number field, $K/k$ be a finite extension of degree $n$
and $L/k$ be the Galois closure of $K/k$ with 
${\rm Gal}(L/k)\simeq S_n$; the symmetric group of degree $n$. 
Let $T=R^{(1)}_{K/k}(\bG_m)$ be the norm one torus of $K/k$ 
and $X$ be a smooth $k$-compactification of $T$. 
Then $H^1(S_n,{\rm Pic}\,X_L)=0$. 
In particular, $T$ has the weak approximation property 
and the Hasse norm principle holds for $K/k$. 
\end{theorem}
\begin{theorem}[{Macedo \cite{Mac20}}]\label{thMac}
Let $k$ be a number field, $K/k$ be a finite extension of degree $n\geq 5$
and $L/k$ be the Galois closure of $K/k$ 
with ${\rm Gal}(L/k)\simeq A_n$; the alternating group of 
degree $n\geq 5$. 
Let $T=R^{(1)}_{K/k}(\bG_m)$ be the norm one torus of $K/k$. 
Then $\Sha^2_\omega(A_n,\widehat{T})=0$.  
In particular, $T$ has the weak approximation property 
and the Hasse norm principle holds for $K/k$. 
\end{theorem}

\begin{remark}
Applying Theorem \ref{thV} to $T=R^{(1)}_{K/k}(\bG_m)$,  
it follows from Theorem \ref{thOno} that 
$H^1(k,{\rm Pic}\,\overline{X})=0$ if and only if 
$A(T)=0$ and $\Sha(T)=0$, 
i.e. 
$T$ has the weak approximation property and 
the Hasse norm principle holds for $K/k$. 
In the algebraic language, 
the latter condition $\Sha(T)=0$ means that 
for the corresponding norm hypersurface $f(x_1,\ldots,x_n)=b$, 
it has a $k$-rational point 
if and only if it has a $k_v$-rational point 
for any valuation $v$ of $k$ where 
$f\in k[x_1,\ldots,x_n]$ is the polynomial of total 
degree $n$ defined by the norm map $N_{K/k}:K^\times\to k^\times$ 
and $b\in k^\times$ 
(see \cite[Example 4, page 122]{Vos98}).
\end{remark}

Let $nTm$ be the $m$-th transitive subgroup of $S_n$ 
up to conjugacy 
(see Butler and McKay \cite{BM83}, \cite{GAP}). 

Let $k$ be a number field, 
$K/k$ be a field extension of degree $n$ 
and $L/k$ be the Galois closure of $K/k$ with ${\rm Gal}(L/k)\simeq G$. 
Then we may regard $G$ as the transitive subgroup $G=nTm\leq S_n$. 
Let $v$ be a place of $k$ and $G_v$ be the decomposition group of $G$ at $v$. 
Using Theorem \ref{thKun1}, 
Kunyavskii \cite{Kun84} gave a necessary and sufficient condition 
for the Hasse norm principle for $n=4$:

\begin{theorem}[{Kunyavskii \cite[page 1899]{Kun84}}]\label{thKun2}
Let $k$ be a number field, 
$K/k$ be a field extension of degree $4$ 
and $L/k$ be the Galois closure of $K/k$. 
Let $G={\rm Gal}(L/k)=4Tm$ $(1\leq m\leq 5)$ 
be a transitive subgroup of $S_4$ 
and $H={\rm Gal}(L/K)$ with $[G:H]=4$. 
Let $T=R^{(1)}_{K/k}(\bG_m)$ be the norm one torus of $K/k$. 
Then $A(T)\simeq\Sha(T)=0$ 
except for $4T2\simeq V_4$ and $4T4\simeq A_4$. 
For $4T2\simeq V_4$ and $4T4\simeq A_4$, 
either {\rm (i)} $A(T)=0$ and $\Sha(T)\simeq\bZ/2\bZ$ 
or {\rm (ii)} $A(T)\simeq\bZ/2\bZ$ and $\Sha(T)=0$, 
and the following conditions are equivalent:\\ 
{\rm (ii)} $A(T)\simeq\bZ/2\bZ$ and $\Sha(T)=0$;\\
{\rm (iii)} there exists a place $v$ of $k$ 
$($which ramifies in $L$$)$ such that $V_4\leq G_v$. 
\end{theorem}

Drakokhrust and Platonov \cite{DP87} gave a necessary and sufficient condition 
for the Hasse norm principle for $n=6$ ($G=6Tm\ (1\leq m\leq 16))$: 

\begin{theorem}[{Drakokhrust and Platonov \cite[Lemma 12, Proposition 6, Lemma 13]{DP87}}]\label{thDP}
Let $k$ be a number field, 
$K/k$ be a field extension of degree $6$ 
and $L/k$ be the Galois closure of $K/k$. 
Let $G={\rm Gal}(L/k)=6Tm$ $(1\leq m\leq 16)$ 
be a transitive subgroup of $S_6$ 
and $H={\rm Gal}(L/K)$ with $[G:H]=6$. 
Let $T=R^{(1)}_{K/k}(\bG_m)$ be the norm one torus of $K/k$. 
Then $\Sha(T)=0$ 
except for $6T4\simeq A_4$ and $6T12\simeq A_5$. 
For $6T4\simeq A_4$ and $6T12\simeq A_5$, 
{\rm (i)} 
$\Sha(T)\leq \bZ/2\bZ$; and 
{\rm (ii)} $\Sha(T)=0$ if and only if 
there exists a place $v$ of $k$ 
$($which ramifies in $L$$)$ such that $V_4\leq G_v$. 
\end{theorem}

The number of transitive subgroups $nTm$ of $S_n$ 
$(2\leq n\leq 15)$ up to conjugacy is given as follows 
(see Butler and McKay \cite{BM83} for $n\leq 11$, 
Royle \cite{Roy87} for $n=12$, 
Butler \cite{But93} for $n=14,15$ 
and \cite{GAP}):\\

\begin{center}
\begin{tabular}{r|rrrrrrrrrrrrrrr}
$n$ & $2$ & $3$ & $4$ & $5$ & $6$ & $7$ & $8$ & $9$ & $10$ & $11$ & 
$12$ & $13$ & $14$ & $15$\\\hline
$\#$ of $nTm$ & $1$ & $2$ & $5$ & $5$ & $16$ & $7$ & $50$ & $34$ & $45$ & $8$ &
$301$ & $9$ & $63$ & $104$\vspace*{3mm}\\
\end{tabular}
\end{center}

The following theorem which is one of the main results of this paper 
classifies the norm one tori 
$T=R^{(1)}_{K/k}(\bG_m)$ 
with non-vanishing 
$H^1(k,{\rm Pic}\, \overline{X})\neq 0$ 
for $[K:k]=n\leq 15$ and $n\neq 12$. 

\begin{theorem}\label{thmain2}
Let $2\leq n\leq 15$ be an integer with $n\neq 12$. 
Let $k$ be a field, 
$K/k$ be a separable field extension of degree $n$ 
and $L/k$ be the Galois closure of $K/k$. 
Assume that $G={\rm Gal}(L/k)=nTm$ is a transitive subgroup of $S_n$ 
and $H={\rm Gal}(L/K)$ with $[G:H]=n$. 
Let $T=R^{(1)}_{K/k}(\bG_m)$ be the norm one torus of $K/k$ of dimension $n-1$ 
and $X$ be a smooth $k$-compactification of $T$. 
Then $H^1(k,{\rm Pic}\, \overline{X})\neq 0$ 
if and only if $G$ is given as in Table $1$. 
In particular, if $k$ is a number field and 
$L/k$ is an unramified extension, then $A(T)=0$ and 
$H^1(k,{\rm Pic}\,\overline{X})\simeq \Sha(T)$. 
\end{theorem}
%
%
\begin{center}
\vspace*{1mm}
Table $1$: $H^1(k,{\rm Pic}\, \overline{X})\simeq H^1(G,[J_{G/H}]^{fl})\neq 0$ 
where $G=nTm$ with $2\leq n\leq 15$ and $n\neq 12$\vspace*{2mm}\\
\begin{tabular}{lc} 
$G$ & $H^1(k,{\rm Pic}\, \overline{X})$ 
$\simeq H^1(G,[J_{G/H}]^{fl})$\\\hline
$4T2\simeq V_4$ & $\bZ/2\bZ$\\
$4T4\simeq A_4$ & $\bZ/2\bZ$\\\hline
$6T4\simeq A_4$ & $\bZ/2\bZ$\\
$6T12\simeq A_5$ & $\bZ/2\bZ$\\\hline
$8T2\simeq C_4\times C_2$ & $\bZ/2\bZ$\\
$8T3\simeq (C_2)^3$ & $(\bZ/2\bZ)^{\oplus 3}$\\
$8T4\simeq D_4$ & $\bZ/2\bZ$\\
$8T9\simeq D_4\times C_2$ & $\bZ/2\bZ$\\
$8T11\simeq (C_4\times C_2)\rtimes C_2$ & $\bZ/2\bZ$\\
$8T13\simeq A_4\times C_2$ & $\bZ/2\bZ$\\
$8T14\simeq S_4$ & $\bZ/2\bZ$\\
$8T15\simeq C_8\rtimes V_4$ & $\bZ/2\bZ$\\
$8T19\simeq (C_2)^3\rtimes C_4$ & $\bZ/2\bZ$\\
$8T21\simeq (C_2)^3\rtimes C_4$ & $\bZ/2\bZ$\\
$8T22\simeq (C_2)^3\rtimes V_4$ & $\bZ/2\bZ$\\
$8T31\simeq ((C_2)^4\rtimes C_2)\rtimes C_2$ & $\bZ/2\bZ$\\
$8T32\simeq ((C_2)^3\rtimes V_4)\rtimes C_3$ & $\bZ/2\bZ$\\
$8T37\simeq \PSL_3(\bF_2)\simeq \PSL_2(\bF_7)$ & $\bZ/2\bZ$\\
$8T38\simeq (((C_2)^4\rtimes C_2)\rtimes C_2)\rtimes C_3$ & $\bZ/2\bZ$\\\hline
$9T2\simeq (C_3)^2$ & $\bZ/3\bZ$\\
$9T5\simeq (C_3)^2\rtimes C_2$ & $\bZ/3\bZ$\\
$9T7\simeq (C_3)^2\rtimes C_3$ & $\bZ/3\bZ$\\
$9T9\simeq (C_3)^2\rtimes C_4$ & $\bZ/3\bZ$\\
$9T11\simeq (C_3)^2\rtimes C_6$ & $\bZ/3\bZ$\\
$9T14\simeq (C_3)^2\rtimes Q_8$ & $\bZ/3\bZ$\\
$9T23\simeq ((C_3)^2\rtimes Q_8)\rtimes C_3$ & $\bZ/3\bZ$\\\hline
$10T7\simeq A_5$ & $\bZ/2\bZ$\\
$10T26\simeq \PSL_2(\bF_9)\simeq A_6$ & $\bZ/2\bZ$\\
$10T32\simeq S_6$ & $\bZ/2\bZ$\\\hline
$14T30\simeq \PSL_2(\bF_{13})$ & $\bZ/2\bZ$\\\hline
$15T9\simeq (C_5)^2\rtimes C_3$ & $\bZ/5\bZ$\\
$15T14\simeq (C_5)^2\rtimes S_3$ & $\bZ/5\bZ$\\\hline
\end{tabular}
\end{center}~\\
\begin{remark}
In Table $1$, only the abelian groups of prime exponent $p$ appear 
as $H^1(k,{\rm Pic}\,\overline{X})$. 
However, 
we find that $H^1(k,{\rm Pic}\,\overline{X})\simeq \bZ/4\bZ$ 
for $G=12T31\simeq (C_4)^2\rtimes C_3$ 
and $G=12T57\simeq ((C_4\times C_2)\rtimes C_4)\rtimes C_3$
by using the same technique as in the proof of Theorem \ref{thmain2}. 
\end{remark}

Additionally, by using the same method of Theorem \ref{thmain2}, 
we obtain the vanishing 
$H^1(k,{\rm Pic}\, \overline{X})=0$ 
for the $5$ Mathieu groups $M_n\leq S_n$ where $n=11,12,22,23,24$ 
(see Dixon and Mortimer \cite[Chapter 6]{DM96}, Gorenstein, Lyons and Solomon \cite[Chapter 5]{GLS98} for the $5$ Mathieu groups): 
\begin{theorem}\label{thmain2M}
Let $k$ be a field, 
$K/k$ be a separable field extension of degree $n$ 
and $L/k$ be the Galois closure of $K/k$. 
Assume that $G={\rm Gal}(L/k)=M_n\leq S_n$ $(n=11,12,22,23,24)$ 
is the Mathieu group of degree $n$ 
and $H={\rm Gal}(L/K)$ with $[G:H]=n$. 
Let $T=R^{(1)}_{K/k}(\bG_m)$ be the norm one torus of $K/k$ of dimension $n-1$ 
and $X$ be a smooth $k$-compactification of $T$. 
Then $H^1(k,{\rm Pic}\, \overline{X})=0$. 
In particular, if $k$ is a number field, then $A(T)=0$ and $\Sha(T)=0$. 
\end{theorem}
Let $Z(G)$ be the center of a group $G$, 
$[G,G]$ be the commutator group of $G$ 
and ${\rm Syl}_p(G)$ be a $p$-Sylow subgroup of $G$ 
where $p$ is a prime. 
Let ${\rm Orb}_{G}(i)$ be the orbit of $1\leq i\leq n$ 
under the action of $G\leq S_n$. 

By Theorem \ref{thmain2}, we obtain the following theorem 
which 
gives a necessary and sufficient condition 
for the Hasse norm principle for $K/k$ where 
$[K:k]=n\leq 15$ and $n\neq 12$. 
Note that 
a place $v$ of $k$ with non-cyclic decomposition group $G_v$ 
as in Theorem \ref{thmain3} must be ramified in $L$ because 
if $v$ is unramified, then $G_v$ is cyclic. 
\begin{theorem}\label{thmain3}
Let $2\leq n\leq 15$ be an integer with $n\neq 12$. 
Let $k$ be a number field, 
$K/k$ be a field extension of degree $n$ 
and $L/k$ be the Galois closure of $K/k$. 
Assume that $G={\rm Gal}(L/k)=nTm$ 
is a transitive subgroup of $S_n$, 
$H={\rm Gal}(L/K)$ with $[G:H]=n$ 
and $G_v$ is the decomposition group of $G$ at a place $v$ of $k$. 
Let $T=R^{(1)}_{K/k}(\bG_m)$ be the norm one torus of $K/k$ 
of dimension $n-1$ and $X$ be a smooth $k$-compactification of $T$. 
Then $A(T)\simeq\Sha(T)=0$ 
except for the cases in Table $1$. 
For the cases in Table $1$ except for $G=8T3$, 
either {\rm (a)} 
$A(T)=0$ and $\Sha(T)\simeq H^1(k,{\rm Pic}\,\overline{X})$ 
or {\rm (b)} 
$A(T)\simeq H^1(k,{\rm Pic}\,\overline{X})$ and $\Sha(T)=0$. 
For $G=8Tm$ $(m=9,11,15,19,22,32)$, 
we assume that $H$ is the stabilizer of one of the letters in $G$. 
Then 
a necessary and sufficient condition for $\Sha(T)=0$ is given 
as in Table $2$. 
\end{theorem}
\begin{center}
\vspace*{1mm}
Table $2$: $\Sha(T)=0$ for $T=R^{(1)}_{K/k}(\bG_m)$ 
and $G={\rm Gal}(L/k)=nTm$ as in Table $1$\vspace*{2mm}\\
\renewcommand{\arraystretch}{1.05}
\begin{tabular}{lc} 
$G$ & $\Sha(T)=0$ if and only if there exists a place $v$ of $k$ such that\\\hline
$4T2\simeq V_4$ & \multirow{2}{*}{$V_4\leq G_v$ 
\begin{minipage}{5.5cm}
~\vspace*{-1mm}\\
{\rm (Tate \cite{Tat67} for $4T2\simeq V_4$)}\\
{\rm (Kunyavskii \cite{Kun84} for $4T4\simeq A_4$)}
\end{minipage}}\vspace*{1mm}\\
$4T4\simeq A_4$ & \\\hline
$6T4\simeq A_4$ & \multirow{2}{*}{$V_4\leq G_v$ {\rm (Drakokhrust and Platonov \cite{DP87})}}\\
$6T12\simeq A_5$ & \\\hline
$8T3\simeq (C_2)^3$ & see the second paragraph after Theorem \ref{thTate} {\rm (Tate \cite{Tat67})}\\\cline{2-2}
$8T4\simeq D_4$ & \multirow{4}{*}{$V_4\leq G_v$ {\rm (Tate \cite{Tat67} for $8T4\simeq D_4$)}}\\
$8T13\simeq A_4\times C_2$ & \\
$8T14\simeq S_4$ & \\
$8T37\simeq \PSL_2(\bF_7)$ & \\\cline{2-2}
$8T2\simeq C_4\times C_2$ & \multirow{2}{*}{$G_v=G$ {\rm (Tate \cite{Tat67} for $8T2\simeq C_4\times C_2$)}}\\
$8T21\simeq (C_2)^3\rtimes C_4$ & \\\cline{2-2}
 & \multirow{4}{*}{
\begin{minipage}{10.6cm}
~\vspace*{-1mm}\\
{\rm (i)} 
$V_4\leq G_v$ where $V_4\cap [{\rm Syl}_2(G),{\rm Syl}_2(G)]=1$ 
with ${\rm Syl}_2(G)\lhd G$ 
$($equivalently, $|{\rm Orb}_{V_4}(i)|=4$ for any $1\leq i\leq 8$ 
and $V_4\cap Z(G)=1$$)$, 
{\rm (ii)} 
$C_4\times C_2\leq G_v$ where 
$(C_4\times C_2)\cap [{\rm Syl}_2(G),{\rm Syl}_2(G)]\simeq C_2$ 
$($equivalently, 
$C_4\times C_2$ is transitive in $S_8$$)$ 
or {\rm (iii)} 
$(C_2)^3\rtimes C_4\leq G_v$
\end{minipage}\vspace*{1mm}}\\
$8T31\simeq (C_2)^4\rtimes V_4$ & \\
$8T38\simeq 8T31\rtimes C_3$ & \\
 & \\\hline
$8T9\simeq D_4\times C_2$ & 
\begin{minipage}{10.6cm}
~\vspace*{-1mm}\\
{\rm (i)} $V_4\leq G_v$ where 
$|{\rm Orb}_{V_4}(i)|=4$ for any $1\leq i\leq 8$ 
and $V_4\cap [G,G]=1$; 
or {\rm (ii)} 
$C_4\times C_2\leq G_v$
\end{minipage}\vspace*{1mm}\\\cline{2-2}
$8T11\simeq Q_8\rtimes C_2$ & 
$C_4\times C_2\leq G_v$ where $C_4\times C_2$ is transitive in $S_8$\\\cline{2-2}
$8T15\simeq C_8\rtimes V_4$ & 
\begin{minipage}{10.6cm}
~\vspace*{-1mm}\\
{\rm (i)} 
$V_4\leq G_v$ where 
$|{\rm Orb}_{V_4}(i)|=2$ for any $1\leq i\leq 8$ 
and $V_4\cap [G,G]=1$ 
$($equivalently, 
$|{\rm Orb}_{V_4}(i)|=2$ for any $1\leq i\leq 8$ 
and $V_4$ is not in $A_8$$)$ 
or {\rm (ii)} 
$C_4\times C_2\leq G_v$ where 
$(C_4\times C_2)\cap [G,G]\simeq C_2$ 
$($equivalently, 
$C_4\times C_2$ is transitive in $S_8$$)$
\end{minipage}\vspace*{1mm}\\\cline{2-2}
$8T19\simeq (C_2)^3\rtimes C_4$ & 
\begin{minipage}{10.6cm}
~\vspace*{-1mm}\\
{\rm (i)} 
$V_4\leq G_v$ where 
$V_4\cap Z(G)=1$ and 
$V_4\cap Z^2(G)\simeq C_2$ 
with the upper central series $1\leq Z(G)\leq Z^2(G)\leq G$ 
$($equivalently, 
$|{\rm Orb}_{V_4}(i)|=4$ for any $1\leq i\leq 8$ 
and $V_4\cap Z(G)=1$$)$; 
or {\rm (ii)} 
$C_4\times C_2\leq G_v$
where 
$C_4\times C_2$ is not transitive in $S_8$ 
or 
$[G,G]\leq C_4\times C_2$
\end{minipage}\vspace*{1mm}\\\cline{2-2}
$8T22\simeq (C_2)^3\rtimes V_4$ & 
\multirow{2}{*}{
\begin{minipage}{10.6cm}
~\vspace*{-1mm}\\
{\rm (i)} 
$V_4\leq G_v$ where 
$|{\rm Orb}_{V_4}(i)|=4$ for any $1\leq i\leq 8$ 
and $V_4\cap Z(G)=1$ 
or {\rm (ii)} 
$C_4\times C_2\leq G_v$ where 
$C_4\times C_2$ is transitive in $S_8$
\end{minipage}\vspace*{1mm}}\\
$8T32\simeq 8T22\rtimes C_3$ & \\\hline
$9T2\simeq (C_3)^2$ & \\
$9T5\simeq (C_3)^2\rtimes C_2$ & \\
$9T7\simeq (C_3)^2\rtimes C_3$ & \\
$9T9\simeq (C_3)^2\rtimes C_4$ & $(C_3)^2\leq G_v$ {\rm (Tate \cite{Tat67} for $9T2\simeq (C_3)^2$)}\\
$9T11\simeq (C_3)^2\rtimes C_6$ & \\
$9T14\simeq (C_3)^2\rtimes Q_8$ & \\
$9T23\simeq 9T14\rtimes C_3$ & \\\hline
$10T7\simeq A_5$ & $V_4\leq G_v$\\\cline{2-2}
$10T26\simeq \PSL_2(\bF_9)$ & $D_4\leq G_v$\\\cline{2-2}
$10T32\simeq S_6$ & 
\begin{minipage}{10.6cm}
~\vspace*{-1mm}\\
{\rm (i)} 
$V_4\leq G_v$ where 
$N_{\widetilde{G}}(V_4)\simeq C_8\rtimes (C_2\times C_2)$ 
for the normalizer $N_{\widetilde{G}}(V_4)$ of $V_4$ in $\widetilde{G}$ 
with the normalizer $\widetilde{G}=N_{S_{10}}(G)\simeq {\rm Aut}(G)$ of $G$ in $S_{10}$
$($equivalently, $|{\rm Orb}_{V_4}(i)|=2$ for any $1\leq i\leq 10$$)$ or  
{\rm (ii)} $D_4\leq G_v$ where $D_4\leq [G,G]\simeq A_6$
\end{minipage}\vspace*{1mm}\\\hline
$14T30\simeq \PSL_2(\bF_{13})$ & $V_4\leq G_v$\\\hline
$15T9\simeq (C_5)^2\rtimes C_3$ & \multirow{2}{*}{$(C_5)^2\leq G_v$}\\
$15T14\simeq (C_5)^2\rtimes S_3$ & \\\hline
\end{tabular}
\renewcommand{\arraystretch}{1}
\end{center}~\\

\newpage
We organize this paper as follows. 
In Section \ref{S2}, 
we prepare some basic definitions and known results about 
the rationality problem for norm one tori. 
In Section \ref{S3}, 
we recall our basic tool ``flabby resolution of $G$-lattices'' 
to investigate algebraic $k$-tori. 
In Section \ref{S4}, we give the proof of 
Theorem \ref{thmain1-4} and Theorem \ref{thmain1-5}. 
In Section \ref{S5}, 
the proofs of Theorem \ref{thmain2} and Theorem \ref{thmain2M} are given. 
In Section \ref{S6}, 
we prove Theorem \ref{thmain3} by using 
Drakokhrust and Platonov's method for 
the Hasse norm principle for $K/k$. 
In Section \ref{S7}, we will give an application of 
Theorem \ref{thmain1-4}, Theorem \ref{thmain1-5} and Theorem \ref{thmain2} 
to obtain the group $T(k)/R$ of $R$-equivalence classes 
over a local field $k$ via the formula 
of Colliot-Th\'{e}l\`{e}ne and Sansuc. 
In Section \ref{S8}, 
we also give an application of Theorem \ref{thmain1-4}, 
Theorem \ref{thmain1-5} and Theorem \ref{thmain3} 
to evaluate the Tamagawa number 
$\tau(T)$ over a number field $k$ via Ono's formula.
In Section \ref{S9}, 
we will give GAP computations 
of $H^1(G,J_{G/H})$ as the appendix of this paper. 
GAP algorithms will be given in Section \ref{S10} 
which are also available from 
{\tt https://www.math.kyoto-u.ac.jp/\~{}yamasaki/Algorithm/Norm1ToriHNP}.
\begin{acknowledgments}
We would like to thank Ming-chang Kang, Shizuo Endo and Boris Kunyavskii 
for giving us useful and valuable comments. 
We also thank the referees for very careful reading of the manuscript. 
This paper is greatly improved by their helpful comments and suggestions. 
In particular, one of them tell us the recent papers 
Macedo and Newton \cite{MN} and Macedo \cite{Mac} 
as a convergence of interests. 
It may be interesting to compare this paper with them. 
\end{acknowledgments}

\section{Rationality problem for norm one tori}\label{S2}

Let $k$ be a field, $K/k$ be a separable field extension of degree $n$ 
and $L/k$ be the Galois closure of $K/k$. 
Let $G={\rm Gal}(L/k)$ and $H={\rm Gal}(L/K)$ with $[G:H]=n$. 
The Galois group $G$ may be regarded as a transitive subgroup of 
the symmetric group $S_n$ of degree $n$ via an injection $G\to S_n$ 
which is derived from the action of $G$ on the left cosets 
$\{g_1H,\ldots,g_nH\}$ by $g(g_iH)=(gg_i)H$ for any $g\in G$ 
and we may assume that 
$H$ is the stabilizer of one of the letters in $G$, 
i.e. $L=k(\theta_1,\ldots,\theta_n)$ and $K=k(\theta_i)$ for some 
$1\leq i\leq n$.
The norm one torus $R^{(1)}_{K/k}(\bG_m)$ has the 
Chevalley module $J_{G/H}$ as its character module 
where $J_{G/H}=(I_{G/H})^\circ={\rm Hom}_\bZ(I_{G/H},\bZ)$ 
is the dual lattice of $I_{G/H}={\rm Ker}\, \varepsilon$ and 
$\varepsilon : \bZ[G/H]\rightarrow \bZ$ is the augmentation map 
(see \cite[Section 4.8]{Vos98}). 
We have the exact sequence 
\begin{align*}
0\rightarrow \bZ\rightarrow \bZ[G/H]\rightarrow J_{G/H}\rightarrow 0
\end{align*}
and rank $J_{G/H}=n-1$. 
Write $J_{G/H}=\oplus_{1\leq i\leq n-1}\bZ u_i$. 
We define the action of $G$ on $L(x_1,\ldots,x_{n-1})$ by 
$\sigma(x_i)=\prod_{j=1}^n x_j^{a_{i,j}} (1\leq i\leq n)$ 
for any $\sigma\in G$, when $\sigma (u_i)=\sum_{j=1}^n a_{i,j} u_j$ 
$(a_{i,j}\in\bZ)$. 
Then the invariant field $L(x_1,\ldots,x_{n-1})^G$ 
may be identified with the function field of the norm 
one torus $R^{(1)}_{K/k}(\bG_m)$ (see \cite[Section 1]{EM75}). 

Let $T=R^{(1)}_{K/k}(\bG_m)$ be the norm one torus of $K/k$. 
The rationality problem for norm one tori is investigated 
by \cite{EM75}, \cite{CTS77}, \cite{Hur84}, \cite{CTS87}, 
\cite{LeB95}, \cite{CK00}, \cite{LL00}, \cite{Flo}, \cite{End11}, 
\cite{HY17}, \cite{HY21}, \cite{HHY20}. 

\begin{theorem}[{Colliot-Th\'{e}l\`{e}ne and Sansuc \cite[Proposition 9.1]{CTS87}, 
\cite[Theorem 3.1]{LeB95}, 
\cite[Proposition 0.2]{CK00}, \cite{LL00}, 
Endo \cite[Theorem 4.1]{End11}, see also 
\cite[Remark 4.2 and Theorem 4.3]{End11}}]\label{thS}
Let $K/k$ be a non-Galois separable field extension 
of degree $n$ and $L/k$ be the Galois closure of $K/k$. 
Assume that ${\rm Gal}(L/k)=S_n$, $n\geq 3$, 
and ${\rm Gal}(L/K)=S_{n-1}$ is the stabilizer of one of the letters 
in $S_n$. 
Then we have:\\
{\rm (i)}\ 
$R^{(1)}_{K/k}(\bG_m)$ is retract $k$-rational 
if and only if $n$ is a prime;\\
{\rm (ii)}\ 
$R^{(1)}_{K/k}(\bG_m)$ is $($stably$)$ $k$-rational 
if and only if $n=3$.
\end{theorem}
\begin{theorem}[Endo {\cite[Theorem 4.4]{End11}, Hoshi and Yamasaki \cite[Corollary 1.11]{HY17}}]\label{thA}
Let $K/k$ be a non-Galois separable field extension 
of degree $n$ and $L/k$ be the Galois closure of $K/k$. 
Assume that ${\rm Gal}(L/k)=A_n$, $n\geq 4$, 
and ${\rm Gal}(L/K)=A_{n-1}$ is the stabilizer of one of the letters 
in $A_n$. 
Then we have:\\
{\rm (i)}\ 
$R^{(1)}_{K/k}(\bG_m)$ is retract $k$-rational 
if and only if $n$ is a prime.\\
{\rm (ii)}\ $R^{(1)}_{K/k}(\bG_m)$ is stably $k$-rational 
if and only if $n=5$.
\end{theorem}

A necessary and sufficient condition for the classification 
of stably/retract rational norm one tori $T=R^{(1)}_{K/k}(\bG_m)$ 
with $[K:k]=n\leq 15$, 
but with one exception 
$G=9T27\simeq \PSL_2(\bF_8)$ for the stable rationality, 
was given in 
Hoshi and Yamasaki \cite{HY21} (for the case $n$ is a prime number or 
the case $n\leq 10$)
and Hasegawa, Hoshi and Yamasaki \cite{HHY20} (for $n=12,14,15$). 

\section{Strategy: flabby resolution of $G$-lattices}\label{S3}
We recall some basic facts of the theory of flabby (flasque) $G$-lattices
(see \cite{CTS77}, \cite{Swa83}, \cite[Chapter 2]{Vos98}, \cite[Chapter 2]{Lor05}, \cite{Swa10}). 
Recall also that we may take the $G$-lattice $\widehat{T}$ 
for an algebraic $k$-torus $T$ (see Section \ref{S1}). 

\begin{definition}
Let $G$ be a finite group and $M$ be a $G$-lattice 
(i.e. finitely generated $\bZ[G]$-module which is $\bZ$-free 
as an abelian group). \\
{\rm (i)} $M$ is called a {\it permutation} $G$-lattice 
if $M$ has a $\bZ$-basis permuted by $G$, 
i.e. $M\simeq \oplus_{1\leq i\leq m}\bZ[G/H_i]$ 
for some subgroups $H_1,\ldots,H_m$ of $G$.\\
{\rm (ii)} $M$ is called a {\it stably permutation} $G$-lattice 
if $M\oplus P\simeq P^\prime$ 
for some permutation $G$-lattices $P$ and $P^\prime$.\\
{\rm (iii)} $M$ is called {\it invertible} (or {\it permutation projective}) 
if it is a direct summand of a permutation $G$-lattice, 
i.e. $P\simeq M\oplus M^\prime$ for some permutation $G$-lattice 
$P$ and a $G$-lattice $M^\prime$.\\
{\rm (iv)} $M$ is called {\it flabby} (or {\it flasque}) if $\widehat H^{-1}(H,M)=0$ 
for any subgroup $H$ of $G$ where $\widehat H$ is the Tate cohomology.\\
{\rm (v)} $M$ is called {\it coflabby} (or {\it coflasque}) if $H^1(H,M)=0$
for any subgroup $H$ of $G$.
\end{definition}

\begin{definition}[{see \cite[Section 1]{EM75}, \cite[Section 4.7]{Vos98}}]
Let $\cC(G)$ be the category of all $G$-lattices. 
Let $\cS(G)$ be the full subcategory of $\cC(G)$ of all permutation $G$-lattices 
and $\cD(G)$ be the full subcategory of $\cC(G)$ of all invertible $G$-lattices.
Let 
\begin{align*}
\cH^i(G)=\{M\in \cC(G)\mid \widehat H^i(H,M)=0\ {\rm for\ any}\ H\leq G\}\ (i=\pm 1)
\end{align*}
be the class of ``$\widehat H^i$-vanish'' $G$-lattices 
where $\widehat H^i$ is the Tate cohomology. 
Then we have the inclusions 
$\cS(G)\subset \cD(G)\subset \cH^i(G)\subset \cC(G)$ $(i=\pm 1)$. 
\end{definition}

\begin{definition}\label{defCM}
We say that two $G$-lattices $M_1$ and $M_2$ are {\it similar} 
if there exist permutation $G$-lattices $P_1$ and $P_2$ such that 
$M_1\oplus P_1\simeq M_2\oplus P_2$. 
We denote the similarity class of $M$ by $[M]$. 
The set of similarity classes $\cC(G)/\cS(G)$ becomes a 
commutative monoid 
(with respect to the sum $[M_1]+[M_2]:=[M_1\oplus M_2]$ 
and the zero $0=[P]$ where $P\in \cS(G)$). 
\end{definition}
\begin{theorem}[{Endo and Miyata \cite[Lemma 1.1]{EM75}, Colliot-Th\'el\`ene and Sansuc \cite[Lemma 3]{CTS77}, 
see also \cite[Lemma 8.5]{Swa83}, \cite[Lemma 2.6.1]{Lor05}}]\label{thEM}
For any $G$-lattice $M$,
there exists a short exact sequence of $G$-lattices
$0 \rightarrow M \rightarrow P \rightarrow F \rightarrow 0$
where $P$ is permutation and $F$ is flabby.
\end{theorem}
\begin{definition}\label{defFlabby}
The exact sequence $0 \rightarrow M \rightarrow P \rightarrow F \rightarrow 0$ 
as in Theorem \ref{thEM} is called a {\it flabby resolution} of the $G$-lattice $M$.
$\rho_G(M)=[F] \in \cC(G)/\cS(G)$ is called {\it the flabby class} of $M$,
denoted by $[M]^{fl}=[F]$.
Note that $[M]^{fl}$ is well-defined: 
if $[M]=[M^\prime]$, $[M]^{fl}=[F]$ and $[M^\prime]^{fl}=[F^\prime]$
then $F \oplus P_1 \simeq F^\prime \oplus P_2$
for some permutation $G$-lattices $P_1$ and $P_2$,
and therefore $[F]=[F^\prime]$ (cf. \cite[Lemma 8.7]{Swa83}). 
We say that $[M]^{fl}$ is {\it invertible} if 
$[M]^{fl}=[E]$ for some invertible $G$-lattice $E$. 
\end{definition}

For $G$-lattice $M$, 
it is not difficult to see 
\begin{align*}
\textrm{permutation}\ \ 
\Rightarrow\ \ 
&\textrm{stably\ permutation}\ \ 
\Rightarrow\ \ 
\textrm{invertible}\ \ 
\Rightarrow\ \ 
\textrm{flabby\ and\ coflabby}\\
&\hspace*{8mm}\Downarrow\hspace*{34mm} \Downarrow\\
&\hspace*{7mm}[M]^{fl}=0\hspace*{10mm}\Rightarrow\hspace*{5mm}[M]^{fl}\ 
\textrm{is\ invertible}.
\end{align*}

The above implications in each step cannot be reversed 
(see, for example, \cite[Section 1]{HY17}). 

Let $T$ be an algebraic $k$-torus and 
$\widehat{T}={\rm Hom}(T,\bG_m)$ be the character module of $T$. 
Then $\widehat{T}$ becomes a $G$-lattice where 
$G={\rm Gal}(L/k)$ is the Galois group of $L/k$ 
and $L$ is the minimal splitting field of $T$. 
%
The flabby class $\rho_G(\widehat{T})=[\widehat{T}]^{fl}$ 
plays crucial role in the rationality problem for $T$ 
as follows (see Voskresenskii's fundamental book \cite[Section 4.6]{Vos98} and Kunyavskii \cite{Kun07}, see also e.g. Swan \cite{Swa83}, 
Kunyavskii \cite[Section 2]{Kun90}, 
Lemire, Popov and Reichstein \cite[Section 2]{LPR06}, 
Kang \cite{Kan12}, Yamasaki \cite{Yam12}, Hoshi and Yamasaki \cite{HY17}):
\begin{theorem}[{Endo and Miyata \cite{EM73}, Voskresenskii \cite{Vos74}, Saltman \cite{Sal84}}]\label{thEM73}
Let $T$ and $T^\prime$ be algebraic $k$-tori 
with the same minimal splitting field $L$. 
Then we have:\\
{\rm (i)} $(${\rm Endo and Miyata} \cite[Theorem 1.6]{EM73}$)$ 
$[\widehat{T}]^{fl}=0$ if and only if $T$ is stably $k$-rational;\\
{\rm (ii)} $(${\rm Voskresenskii} \cite[Theorem 2]{Vos74}$)$ 
$[\widehat{T}]^{fl}=[\widehat{T}^\prime]^{fl}$ if and only if 
$T$ and $T^\prime$ are stably $k$-isomorphic;\\
{\rm (iii)} $(${\rm Saltman} \cite[Theorem 3.14]{Sal84}$)$ 
$[\widehat{T}]^{fl}$ is invertible if and only if $T$ is 
retract $k$-rational.
\end{theorem}
For norm one tori $T=R^{(1)}_{K/k}(\bG_m)$, 
recall that $\widehat{T}=J_{G/H}$. 
Hence we have  
\begin{align*}
[J_{G/H}]^{fl}=0\,
\ \ \Rightarrow\ \ [J_{G/H}]^{fl}\ \textrm{is\ invertible}
\ \ \Rightarrow\ \  H^1(G,[J_{G/H}]^{fl})=0\,
\ \ \Rightarrow\ \  A(T)=0\ \textrm{and}\ \Sha(T)=0
\end{align*}
where the last implication holds over a global field $k$ 
(see also Colliot-Th\'{e}l\`{e}ne and Sansuc \cite[page 29]{CTS77}).  
The last conditions mean that 
$T$ has the weak approximation property and 
the Hasse norm principle holds for $K/k$ (see Section \ref{S1}). 
In particular, it follows from Theorem \ref{thS} that 
$H^1(G,[J_{G/H}]^{fl})=0$ and hence 
$A(T)=0$ and $\Sha(T)=0$ when $G=pTm\leq S_p$ is a transitive 
subgroup of $S_p$ of prime degree $p$ 
and $H\leq G$ with $[G:H]=p$ 
(see \cite[Lemma 2.17]{HY17} and 
also the first paragraph of Section \ref{S5}). 

\section{{Proof of Theorem \ref{thmain1-4}} and {Theorem \ref{thmain1-5}}}\label{S4}
We will give the proof of Theorem \ref{thmain1-4p} and 
Theorem \ref{thmain1-5p} which are detailed statements of 
Theorem \ref{thmain1-4} and Theorem \ref{thmain1-5} respectively: 
\begin{theorem}\label{thmain1-4p}
Let $k$ be a 
field, $T$ be an algebraic $k$-torus of dimension $4$ 
and $X$ be a smooth $k$-compactification of $T$. 
Among the $($at most$)$ 
$216$ cases of not retract rational algebraic $k$-tori $T$,
there exist $2$ $($resp. $20$, $194$$)$ cases of 
algebraic $k$-tori with 
$H^1(k,{\rm Pic}\, \overline{X})\simeq(\bZ/2\bZ)^{\oplus 2}$ 
$($resp. $H^1(k,{\rm Pic}\, \overline{X})\simeq\bZ/2\bZ$, 
$H^1(k,{\rm Pic}\, \overline{X})=0$$)$. 
Moreover, for the character module $\widehat{T}\simeq M_G$ of $T$ 
with $H^1(k,{\rm Pic}\,\overline{X})\simeq H^1(G,[M_G]^{fl})$, we have\\
{\rm (i)} $H^1(G,[M_G]^{fl})\simeq(\bZ/2\bZ)^{\oplus 2}$ 
if and only if the {\rm GAP ID} of $G$ is one of 
$(4,32,1,2)$ and $(4,33,3,1)$ 
where $M_G$ is an indecomposable $G$-lattice of rank $4$ and 
$G$ is isomorphic to 
$Q_8$  and $\SL_2(\bF_3)$ respectively;\\
{\rm (ii)} $H^1(G,[M_G]^{fl})\simeq\bZ/2\bZ$ 
if and only if\\
{\rm (ii-1)} the {\rm GAP ID} of $G$ is one of 
$(4,5,1,12)$, $(4,5,2,8)$, $(4,6,2,10)$, $(4,12,2,6)$, $(4,12,4,12)$, 
$(4,12,5,10)$, $(4,18,1,3)$, $(4,18,4,4)$, $(4,32,2,2)$, $(4,32,3,2)$, 
$(4,32,4,2)$, $(4,32,6,2)$, $(4,33,5,1)$, $(4,33,6,1)$, $(4,33,9,1)$ 
where $M_G$ is an indecomposable $G$-lattice of rank $4$ 
and $G$ is isomorphic to 
$V_4$, $(C_2)^3$, $(C_2)^3$, $C_4\times C_2$, $D_4$, 
$C_2\times D_4$, $C_4\times C_2$, $C_2\times D_4$, $Q_{16}$, $QD_8$, 
$(C_4\times C_2)\rtimes C_2$, $C_8\rtimes V_4$, 
$((C_4\times C_2)\rtimes C_2)\rtimes C_3$, 
$\GL_2(\bF_3)$, $\GL_2(\bF_3)\rtimes C_2$
respectively; or\\
{\rm (ii-2)} the {\rm GAP ID} of $G$ is one of  
$(4,4,3,6)$, $(4,5,1,9)$, $(4,6,2,9)$, $(4,24,1,5)$, $(4,25,2,4)$ 
where $M_G$ is a decomposable $G$-lattice of rank $4=3+1$ 
and $G$ is isomorphic to 
$V_4$, $V_4$, $(C_2)^3$, $A_4$, $C_2\times A_4$ respectively. 
\end{theorem}
\begin{proof}
It follows from \cite[Theorem 1.9]{HY17} that 
among the $710$ cases of $4$-dimensional 
algebraic $k$-tori, 
there exist $216$ cases of algebraic $k$-tori which 
are not retract $k$-rational. 
Because if $T$ is retract $k$-rational, 
then $H^1(k,{\rm Pic}\,\overline{X})\simeq H^1(G,[M_G]^{fl})=0$, 
we should check only the $216$ cases. 
The GAP IDs of such $216$ groups $G\leq \GL_4(\bZ)$ with 
$[M_G]^{fl}$ is not invertible, are given in \cite[Tables 3, 4]{HY17} 
(see \cite[Chapter 3]{HY17} for the explanation of GAP ID). 
They are also given in 
\cite[Example 10.1]{HY17} as the lists {\tt N4} (resp. {\tt N31}) 
when $M_G$ is indecomposable (resp. decomposable with rank $4=3+1$) 
and available from\\ 
{\tt https://www.math.kyoto-u.ac.jp/\~{}yamasaki/Algorithm/MultInvField/NonInv.dat}.

Then we apply the function 
{\tt FlabbyResolutionLowRank($G$).actionF} 
(see also \cite[Algorithm 4.1]{HHY20}) 
which returns a suitable flabby class $F$ of $M_G$ $([F]=[M_G]^{fl})$
with low rank by using the backtracking techniques. 
The function {\tt H1} may compute the group $H^1(G,F)$ 
(see Example \ref{exN4}).  
The related functions are available from\\ 
{\tt https://www.math.kyoto-u.ac.jp/\~{}yamasaki/Algorithm/RatProbNorm1Tori/}.
\end{proof}
\begin{theorem}\label{thmain1-5p}
Let $k$ be a 
field, $T$ be an algebraic $k$-torus of dimension $5$ 
and $X$ be a smooth $k$-compactification of $T$. 
Among the $($at most$)$ 
$3003$ cases of not retract rational algebraic $k$-tori $T$,
there exist $11$ $($resp. $263$, $2729$$)$ cases of 
algebraic $k$-tori with 
$H^1(k,{\rm Pic}\, \overline{X})\simeq(\bZ/2\bZ)^{\oplus 2}$ 
$($resp. $H^1(k,{\rm Pic}\, \overline{X})\simeq\bZ/2\bZ$, 
$H^1(k,{\rm Pic}\, \overline{X})=0$$)$. 
Moreover, for the character module $\widehat{T}\simeq M_G$ of $T$ 
with $H^1(k,{\rm Pic}\,\overline{X})\simeq H^1(G,[M_G]^{fl})$, we have\\
{\rm (i)} $H^1(G,[M_G]^{fl})\simeq(\bZ/2\bZ)^{\oplus 2}$ 
if and only if\\
{\rm (i-1)} the {\rm CARAT ID} of $G$ is one of the $6$ triples 
$(5,31,26)$, $(5,31,27)$, $(5,664,2)$, $(5,669,2)$, $(5,670,2)$, $(5,773,4)$ 
where $M_G$ is an indecomposable $G$-lattice of rank $5$ and 
$G$ is isomorphic to 
$(C_2)^3$, $(C_2)^3$, $C_2\times Q_8$, $(C_4\times C_2)\rtimes C_2$, 
$(C_4\times C_2)\rtimes C_2$, $Q_8$ respectively; or\\
{\rm (i-2)} the {\rm CARAT ID} of $G$ is one of the $5$ triples 
$(5,664,1)$, $(5,773,3)$, $(5,774,3)$, $(5,691,1)$, $(5,730,1)$
where $M_G$ is a decomposable $G$-lattice of rank $5=4+1$ and 
$G$ is isomorphic to $C_2\times Q_8$, $Q_8$, $Q_8$, $\SL_2(\bF_3)$, 
$C_2\times \SL_2(\bF_3)$ respectively.\\
{\rm (ii)} $H^1(G,[M_G]^{fl})\simeq\bZ/2\bZ$ 
if and only if\\
{\rm (ii-1)} the {\rm CARAT ID} of $G$ is one of the $141$ triples 
as in Example \ref{exN5} 
where $M_G$ is an indecomposable $G$-lattice of rank $5$;\\
{\rm (ii-2)} the {\rm CARAT ID} of $G$ is one of the $73$ triples 
as in Example \ref{exN5} 
where $M_G$ is a decomposable $G$-lattice of rank $5=4+1$;\\
{\rm (ii-3)} the {\rm CARAT ID} of $G$ is one of the $36$ triples 
as in Example \ref{exN5} 
where $M_G$ is a decomposable $G$-lattice of rank $5=3+2$; or\\
{\rm (ii-4)} the {\rm CARAT ID} of $G$ is one of the $13$ triples 
as in Example \ref{exN5} 
where $M_G$ is a decomposable $G$-lattice of rank $5=3+1+1$.
\end{theorem}
\begin{proof}
The method is the same as in the proof of Theorem \ref{thmain1-4p}. 
By \cite[Theorem 1.12]{HY17}, among the $6079$ cases of $5$-dimensional 
algebraic $k$-tori, 
there exist $3003$ cases of algebraic $k$-tori which 
are not retract $k$-rational.  
The CARAT IDs of such $3003$ groups $G\leq \GL_5(\bZ)$ with 
$[M_G]^{fl}$ is not invertible, 
are given in \cite[Tables 12, 13, 14, 15]{HY17}.  
They are also given in 
\cite[Example 4.12 and Example 11.1]{HY17} 
as the lists {\tt N5}, {\tt N41}, {\tt N32}, {\tt N311}
when $M_G$ is indecomposable (resp. decomposable with rank $5=4+1$, 
$5=3+1$, $5=3+1+1$) 
and available from\\
{\tt https://www.math.kyoto-u.ac.jp/\~{}yamasaki/Algorithm/MultInvField/NonInv5.dat}. 

Then we apply the functions 
{\tt FlabbyResolutionLowRank($G$).actionF} 
in \cite[Algorithm 4.1]{HHY20} and {\tt H1} to get $H^1(G,[M_G]^{fl})$ 
(see Example \ref{exN5} and also the proof of Theorem \ref{thmain1-4p}). 
\end{proof}

\smallskip
\begin{example}[{Classification of $H^1(G,[M_{G}]^{fl})\neq 0$ for $G\leq {\rm GL}_3(\bZ)$}]\label{exN3}
~{}\vspace*{-4mm}\\
{\small 
\begin{verbatim}
gap> Read("FlabbyResolutionFromBase.gap");
gap> Read("NonInv.dat");
# N3 is the list of GAP IDs (Crystcat IDs) of indecomposable lattice of rank 3
# whose flabby class [M_G]^fl is not invertible [HY17, Example 4.12]
gap> N3; 
[ [ 3, 3, 1, 3 ], [ 3, 3, 3, 3 ], [ 3, 3, 3, 4 ], [ 3, 4, 3, 2 ], [ 3, 4, 4, 2 ],
  [ 3, 4, 6, 3 ], [ 3, 4, 7, 2 ], [ 3, 7, 1, 2 ], [ 3, 7, 2, 2 ], [ 3, 7, 2, 3 ],
  [ 3, 7, 3, 2 ], [ 3, 7, 3, 3 ], [ 3, 7, 4, 2 ], [ 3, 7, 5, 2 ], [ 3, 7, 5, 3 ] ]
gap> Length(N3); # there exist 15 not retract rational tori in dim=3 [HY17, Table 1]
15
gap> N3g:=List(N3,x->MatGroupZClass(x[1],x[2],x[3],x[4]));;
gap> List(N3g,StructureDescription);
[ "C2 x C2", "C2 x C2 x C2", "C2 x C2 x C2", "C4 x C2", "D8", "D8", "C2 x D8", 
"A4", "C2 x A4", "C2 x A4", "S4", "S4", "S4", "C2 x S4", "C2 x S4" ]
gap> N3gF:=List(N3g,x->FlabbyResolutionLowRank(x).actionF);;
gap> N3H1F:=List(N3gF,x->Filtered(H1(x),y->y>1)); # H1(F)
[ [ 2 ], [  ], [  ], [  ], [  ], [  ], [  ], [ 2 ], [  ], [  ], [  ], [  ], [  ], [  ], [  ] ]
gap> N3H1FC2:=Filtered([1..Length(N3gF)],x->N3H1F[x]=[2]);
[ 1, 8 ]
gap> List(N3H1FC2,x->N3[x]); # GAP ID's of F with H1(F)=C2
[ [ 3, 3, 1, 3 ], [ 3, 7, 1, 2 ] ]
gap> List(N3H1FC2,x->StructureDescription(N3g[x]));
[ "C2 x C2", "A4" ]
\end{verbatim}
}
\end{example}

\smallskip
\begin{example}[{Classification of $H^1(G,[M_{G}]^{fl})\neq 0$ for $G\leq {\rm GL}_4(\bZ)$}]\label{exN4}
~{}\vspace*{-4mm}\\
{\small
\begin{verbatim}
gap> Read("FlabbyResolutionFromBase.gap");
gap> Read("NonInv.dat");
# N4 is the list of GAP IDs (Crystcat IDs) of indecomposable lattice of rank 4
# whose flabby class [M_G]^fl is not invertible [HY17, Example 4.12]
gap> N4; 
[ [ 4, 5, 1, 12 ], [ 4, 5, 2, 5 ], [ 4, 5, 2, 8 ], [ 4, 5, 2, 9 ], [ 4, 6, 1, 6 ],
  [ 4, 6, 1, 11 ], [ 4, 6, 2, 6 ], [ 4, 6, 2, 10 ], [ 4, 6, 2, 12 ], [ 4, 6, 3, 4 ],
  [ 4, 6, 3, 7 ], [ 4, 6, 3, 8 ], [ 4, 12, 2, 5 ], [ 4, 12, 2, 6 ], [ 4, 12, 3, 11 ],
  [ 4, 12, 4, 10 ], [ 4, 12, 4, 11 ], [ 4, 12, 4, 12 ], [ 4, 12, 5, 8 ], [ 4, 12, 5, 9 ],
  [ 4, 12, 5, 10 ], [ 4, 12, 5, 11 ], [ 4, 13, 1, 5 ], [ 4, 13, 2, 5 ], [ 4, 13, 3, 5 ],
  [ 4, 13, 4, 5 ], [ 4, 13, 5, 4 ], [ 4, 13, 5, 5 ], [ 4, 13, 6, 5 ], [ 4, 13, 7, 9 ],
  [ 4, 13, 7, 10 ], [ 4, 13, 7, 11 ], [ 4, 13, 8, 5 ], [ 4, 13, 8, 6 ], [ 4, 13, 9, 4 ],
  [ 4, 13, 9, 5 ], [ 4, 13, 10, 4 ], [ 4, 13, 10, 5 ], [ 4, 18, 1, 3 ], [ 4, 18, 2, 4 ],
  [ 4, 18, 2, 5 ], [ 4, 18, 3, 5 ], [ 4, 18, 3, 6 ], [ 4, 18, 3, 7 ], [ 4, 18, 4, 4 ],
  [ 4, 18, 4, 5 ], [ 4, 18, 5, 5 ], [ 4, 18, 5, 6 ], [ 4, 18, 5, 7 ], [ 4, 19, 1, 2 ],
  [ 4, 19, 2, 2 ], [ 4, 19, 3, 2 ], [ 4, 19, 4, 3 ], [ 4, 19, 4, 4 ], [ 4, 19, 5, 2 ],
  [ 4, 19, 6, 2 ], [ 4, 22, 1, 1 ], [ 4, 22, 2, 1 ], [ 4, 22, 3, 1 ], [ 4, 22, 4, 1 ],
  [ 4, 22, 5, 1 ], [ 4, 22, 5, 2 ], [ 4, 22, 6, 1 ], [ 4, 22, 7, 1 ], [ 4, 22, 8, 1 ],
  [ 4, 22, 9, 1 ], [ 4, 22, 10, 1 ], [ 4, 22, 11, 1 ], [ 4, 24, 2, 4 ], [ 4, 24, 2, 6 ],
  [ 4, 24, 4, 4 ], [ 4, 24, 5, 4 ], [ 4, 24, 5, 6 ], [ 4, 25, 1, 3 ], [ 4, 25, 2, 3 ],
  [ 4, 25, 2, 5 ], [ 4, 25, 3, 3 ], [ 4, 25, 4, 3 ], [ 4, 25, 5, 3 ], [ 4, 25, 5, 5 ],
  [ 4, 25, 6, 3 ], [ 4, 25, 6, 5 ], [ 4, 25, 7, 3 ], [ 4, 25, 8, 3 ], [ 4, 25, 9, 3 ],
  [ 4, 25, 9, 5 ], [ 4, 25, 10, 3 ], [ 4, 25, 10, 5 ], [ 4, 25, 11, 3 ], [ 4, 25, 11, 5 ],
  [ 4, 29, 1, 1 ], [ 4, 29, 1, 2 ], [ 4, 29, 2, 1 ], [ 4, 29, 3, 1 ], [ 4, 29, 3, 2 ],
  [ 4, 29, 3, 3 ], [ 4, 29, 4, 1 ], [ 4, 29, 4, 2 ], [ 4, 29, 5, 1 ], [ 4, 29, 6, 1 ],
  [ 4, 29, 7, 1 ], [ 4, 29, 7, 2 ], [ 4, 29, 8, 1 ], [ 4, 29, 8, 2 ], [ 4, 29, 9, 1 ],
  [ 4, 32, 1, 2 ], [ 4, 32, 2, 2 ], [ 4, 32, 3, 2 ], [ 4, 32, 4, 2 ], [ 4, 32, 5, 2 ],
  [ 4, 32, 5, 3 ], [ 4, 32, 6, 2 ], [ 4, 32, 7, 2 ], [ 4, 32, 8, 2 ], [ 4, 32, 9, 4 ],
  [ 4, 32, 9, 5 ], [ 4, 32, 10, 2 ], [ 4, 32, 11, 2 ], [ 4, 32, 11, 3 ], [ 4, 32, 12, 2 ],
  [ 4, 32, 13, 3 ], [ 4, 32, 13, 4 ], [ 4, 32, 14, 3 ], [ 4, 32, 14, 4 ], [ 4, 32, 15, 2 ],
  [ 4, 32, 16, 2 ], [ 4, 32, 16, 3 ], [ 4, 32, 17, 2 ], [ 4, 32, 18, 2 ], [ 4, 32, 18, 3 ],
  [ 4, 32, 19, 2 ], [ 4, 32, 19, 3 ], [ 4, 32, 20, 2 ], [ 4, 32, 20, 3 ], [ 4, 32, 21, 2 ],
  [ 4, 32, 21, 3 ], [ 4, 33, 1, 1 ], [ 4, 33, 3, 1 ], [ 4, 33, 4, 1 ], [ 4, 33, 5, 1 ],
  [ 4, 33, 6, 1 ], [ 4, 33, 7, 1 ], [ 4, 33, 8, 1 ], [ 4, 33, 9, 1 ], [ 4, 33, 10, 1 ],
  [ 4, 33, 11, 1 ], [ 4, 33, 12, 1 ], [ 4, 33, 13, 1 ], [ 4, 33, 14, 1 ], [ 4, 33, 14, 2 ],
  [ 4, 33, 15, 1 ], [ 4, 33, 16, 1 ] ]
gap> Length(N4); # there exist 152 not retract rational tori in dim=4 [HY17, Table 4]
152
gap> N4g:=List(N4,x->MatGroupZClass(x[1],x[2],x[3],x[4]));;
gap> N4gF:=List(N4g,x->FlabbyResolutionLowRank(x).actionF);;
gap> N4H1F:=List(N4gF,x->Filtered(H1(x),y->y>1));;
gap> Collected(N4H1F);
[ [ [  ], 135 ], [ [ 2 ], 15 ], [ [ 2, 2 ], 2 ] ]

gap> N4H1FC2xC2:=Filtered([1..Length(N4H1F)],x->N4H1F[x]=[2,2]);
[ 106, 138 ]
gap> List(N4H1FC2xC2,x->N4[x]); # GAP ID's of F with H1(F)=C2xC2
[ [ 4, 32, 1, 2 ], [ 4, 33, 3, 1 ] ]
gap> List(N4H1FC2xC2,x->StructureDescription(N4g[x])); 
[ "Q8", "SL(2,3)" ]

gap> N4H1FC2:=Filtered([1..Length(N4H1F)],x->N4H1F[x]=[2]);
[ 1, 3, 8, 14, 18, 21, 39, 45, 107, 108, 109, 112, 140, 141, 144 ]
gap> List(N4H1FC2,x->N4[x]); # GAP ID's of F with H1(F)=C2
[ [ 4, 5, 1, 12 ], [ 4, 5, 2, 8 ], [ 4, 6, 2, 10 ], [ 4, 12, 2, 6 ], [ 4, 12, 4, 12 ],
  [ 4, 12, 5, 10 ], [ 4, 18, 1, 3 ], [ 4, 18, 4, 4 ], [ 4, 32, 2, 2 ], [ 4, 32, 3, 2 ],
  [ 4, 32, 4, 2 ], [ 4, 32, 6, 2 ], [ 4, 33, 5, 1 ], [ 4, 33, 6, 1 ], [ 4, 33, 9, 1 ] ]
gap> List(N4H1FC2,x->StructureDescription(N4g[x]));
[ "C2 x C2", "C2 x C2 x C2", "C2 x C2 x C2", "C4 x C2", "D8", "C2 x D8", "C4 x C2",
  "C2 x D8", "C8 : C2", "QD16", "(C4 x C2) : C2", "C8 : (C2 x C2)",
  "((C4 x C2) : C2) : C3", "GL(2,3)", "GL(2,3) : C2" ]

# N31 is the list of GAP IDs (Crystcat IDs) of decomposable lattice of rank 4=3+1 
# whose flabby class [M_G]^fl is not invertible [HY17, Example 4.12]
gap> N31;
[ [ 4, 4, 3, 6 ], [ 4, 4, 4, 4 ], [ 4, 4, 4, 6 ], [ 4, 5, 1, 9 ], [ 4, 5, 2, 4 ],
  [ 4, 5, 2, 7 ], [ 4, 6, 1, 4 ], [ 4, 6, 1, 8 ], [ 4, 6, 2, 4 ], [ 4, 6, 2, 8 ],
  [ 4, 6, 2, 9 ], [ 4, 6, 3, 3 ], [ 4, 6, 3, 6 ], [ 4, 7, 3, 2 ], [ 4, 7, 4, 3 ],
  [ 4, 7, 5, 2 ], [ 4, 7, 7, 2 ], [ 4, 12, 2, 4 ], [ 4, 12, 3, 7 ], [ 4, 12, 4, 6 ],
  [ 4, 12, 4, 8 ], [ 4, 12, 4, 9 ], [ 4, 12, 5, 6 ], [ 4, 12, 5, 7 ], [ 4, 13, 1, 3 ],
  [ 4, 13, 2, 4 ], [ 4, 13, 3, 4 ], [ 4, 13, 4, 3 ], [ 4, 13, 5, 3 ], [ 4, 13, 6, 3 ],
  [ 4, 13, 7, 6 ], [ 4, 13, 7, 7 ], [ 4, 13, 7, 8 ], [ 4, 13, 8, 3 ], [ 4, 13, 8, 4 ],
  [ 4, 13, 9, 3 ], [ 4, 13, 10, 3 ], [ 4, 24, 1, 5 ], [ 4, 24, 2, 3 ], [ 4, 24, 2, 5 ],
  [ 4, 24, 3, 5 ], [ 4, 24, 4, 3 ], [ 4, 24, 4, 5 ], [ 4, 24, 5, 3 ], [ 4, 24, 5, 5 ],
  [ 4, 25, 1, 2 ], [ 4, 25, 1, 4 ], [ 4, 25, 2, 4 ], [ 4, 25, 3, 2 ], [ 4, 25, 3, 4 ],
  [ 4, 25, 4, 4 ], [ 4, 25, 5, 2 ], [ 4, 25, 5, 4 ], [ 4, 25, 6, 2 ], [ 4, 25, 6, 4 ],
  [ 4, 25, 7, 2 ], [ 4, 25, 7, 4 ], [ 4, 25, 8, 2 ], [ 4, 25, 8, 4 ], [ 4, 25, 9, 4 ],
  [ 4, 25, 10, 2 ], [ 4, 25, 10, 4 ], [ 4, 25, 11, 2 ], [ 4, 25, 11, 4 ] ]
gap> Length(N31); # there exist 64 not retract rational tori in dim=4=3+1 [HY17, Table 3]
64
gap> N31g:=List(N31,x->MatGroupZClass(x[1],x[2],x[3],x[4]));;
gap> N31gF:=List(N31g,x->FlabbyResolutionLowRank(x).actionF);;
gap> N31H1F:=List(N31gF,x->Filtered(H1(x),y->y>1));;
gap> Collected(N31H1F);
[ [ [  ], 59 ], [ [ 2 ], 5 ] ]
gap> N31H1FC2:=Filtered([1..Length(N31H1F)],x->N31H1F[x]=[2]);
[ 1, 4, 11, 38, 48 ]

gap> List(N31H1FC2,x->N31[x]); # GAP ID's of F with H1(F)=C2
[ [ 4, 4, 3, 6 ], [ 4, 5, 1, 9 ], [ 4, 6, 2, 9 ], [ 4, 24, 1, 5 ], [ 4, 25, 2, 4 ] ]
gap> List(N31H1FC2,x->StructureDescription(N31g[x]));
[ "C2 x C2", "C2 x C2", "C2 x C2 x C2", "A4", "C2 x A4" ]
\end{verbatim}
}
\end{example}

\smallskip
\begin{example}[{Classification of $H^1(G,[M_{G}]^{fl})\neq 0$ for $G\leq {\rm GL}_5(\bZ)$}]\label{exN5}
~{}\vspace*{-4mm}\\
{\small 
\begin{verbatim}
gap> Read("FlabbyResolutionFromBase.gap");
gap> Read("caratnumber.gap");
gap> Read("NonInv5.dat");
# N5 is the list of CARAT IDs of indecomposable lattice of rank 5 
# whose flabby class [M_G]^fl is not invertible [HY17, Example 4.12]
gap> N5g:=List(N5,x->CaratMatGroupZClass(x[1],x[2],x[3]));;
gap> Length(N5g); # there exist 1141 not retract rational tori in dim=5 [HY17, Table 15]
1141
gap> N5gF:=List(N5g,x->FlabbyResolutionLowRank(x).actionF);;
gap> N5H1F:=List(N5gF,x->Filtered(H1(x),y->y>1));;
gap> Collected(N5H1F);
[ [ [  ], 994 ], [ [ 2 ], 141 ], [ [ 2, 2 ], 6 ] ]

gap> N5H1FC2xC2:=Filtered([1..Length(N5H1F)],x->N5H1F[x]=[2,2]);
[ 69, 70, 906, 913, 915, 1064 ]
gap> List(N5H1FC2xC2,x->N5[x]); # CARAT ID's of F with H1(F)=C2xC2
[ [ 5, 31, 26 ], [ 5, 31, 27 ], [ 5, 664, 2 ], [ 5, 669, 2 ], [ 5, 670, 2 ], [ 5, 773, 4 ] ]
gap> List(N5H1FC2xC2,x->StructureDescription(N5g[x]));
[ "C2 x C2 x C2", "C2 x C2 x C2", "C2 x Q8", "(C4 x C2) : C2", "(C4 x C2) : C2", "Q8" ]

gap> N5H1FC2:=Filtered([1..Length(N5H1F)],x->N5H1F[x]=[2]);
[ 3, 4, 5, 6, 8, 11, 19, 20, 27, 36, 37, 42, 43, 61, 63, 67, 71, 72, 74, 78, 79, 86, 88, 89,
  96, 99, 100, 103, 115, 116, 128, 129, 130, 131, 142, 143, 158, 159, 160, 173, 174, 178,
  179, 185, 186, 187, 188, 191, 193, 199, 200, 221, 222, 238, 242, 243, 253, 254, 288, 292,
  293, 316, 317, 318, 324, 327, 331, 333, 334, 337, 339, 348, 358, 362, 375, 376, 378, 389,
  401, 403, 404, 406, 407, 410, 414, 419, 423, 425, 440, 470, 480, 495, 511, 523, 540, 573,
  588, 590, 591, 592, 593, 595, 596, 597, 606, 680, 715, 723, 762, 852, 853, 854, 855, 908,
  909, 912, 916, 918, 921, 922, 948, 957, 961, 964, 970, 971, 973, 974, 976, 980, 982, 984,
  1037, 1060, 1065, 1114, 1115, 1116, 1117, 1129, 1130 ]
gap> List(N5H1FC2,x->N5[x]); # CARAT ID's of F with H1(F)=C2
[ [ 5, 18, 23 ], [ 5, 19, 17 ], [ 5, 20, 14 ], [ 5, 20, 17 ], [ 5, 21, 17 ], [ 5, 24, 23 ],
  [ 5, 25, 27 ], [ 5, 25, 28 ], [ 5, 26, 21 ], [ 5, 26, 40 ], [ 5, 26, 41 ], [ 5, 27, 14 ],
  [ 5, 27, 15 ], [ 5, 30, 24 ], [ 5, 30, 28 ], [ 5, 31, 18 ], [ 5, 31, 31 ], [ 5, 31, 32 ],
  [ 5, 31, 36 ], [ 5, 31, 44 ], [ 5, 31, 45 ], [ 5, 32, 36 ], [ 5, 32, 44 ], [ 5, 32, 51 ],
  [ 5, 39, 5 ], [ 5, 71, 19 ], [ 5, 71, 22 ], [ 5, 71, 25 ], [ 5, 72, 34 ], [ 5, 72, 36 ],
  [ 5, 73, 32 ], [ 5, 73, 34 ], [ 5, 73, 36 ], [ 5, 73, 37 ], [ 5, 75, 34 ], [ 5, 75, 36 ],
  [ 5, 76, 49 ], [ 5, 76, 50 ], [ 5, 76, 51 ], [ 5, 78, 12 ], [ 5, 78, 15 ], [ 5, 78, 28 ],
  [ 5, 78, 31 ], [ 5, 79, 12 ], [ 5, 79, 15 ], [ 5, 79, 17 ], [ 5, 79, 18 ], [ 5, 79, 31 ],
  [ 5, 79, 36 ], [ 5, 80, 12 ], [ 5, 80, 15 ], [ 5, 83, 15 ], [ 5, 83, 17 ], [ 5, 86, 9 ],
  [ 5, 87, 9 ], [ 5, 87, 11 ], [ 5, 88, 34 ], [ 5, 88, 36 ], [ 5, 93, 9 ], [ 5, 94, 9 ],
  [ 5, 94, 11 ], [ 5, 99, 23 ], [ 5, 99, 24 ], [ 5, 99, 25 ], [ 5, 100, 12 ], [ 5, 100, 23 ],
  [ 5, 100, 28 ], [ 5, 101, 17 ], [ 5, 101, 18 ], [ 5, 102, 9 ], [ 5, 102, 17 ],
  [ 5, 105, 5 ], [ 5, 109, 5 ], [ 5, 109, 14 ], [ 5, 112, 5 ], [ 5, 112, 7 ], [ 5, 113, 4 ],
  [ 5, 116, 20 ], [ 5, 118, 18 ], [ 5, 119, 4 ], [ 5, 119, 5 ], [ 5, 119, 10 ],
  [ 5, 119, 12 ], [ 5, 120, 5 ], [ 5, 120, 14 ], [ 5, 121, 13 ], [ 5, 122, 9 ],
  [ 5, 122, 15 ], [ 5, 127, 11 ], [ 5, 134, 9 ], [ 5, 136, 18 ], [ 5, 140, 23 ],
  [ 5, 142, 14 ], [ 5, 143, 23 ], [ 5, 148, 5 ], [ 5, 154, 15 ], [ 5, 160, 4 ],
  [ 5, 160, 7 ], [ 5, 161, 5 ], [ 5, 161, 7 ], [ 5, 162, 5 ], [ 5, 224, 9 ], [ 5, 227, 11 ],
  [ 5, 232, 14 ], [ 5, 242, 9 ], [ 5, 526, 11 ], [ 5, 534, 11 ], [ 5, 536, 13 ],
  [ 5, 546, 11 ], [ 5, 580, 2 ], [ 5, 604, 2 ], [ 5, 604, 4 ], [ 5, 605, 2 ], [ 5, 665, 4 ],
  [ 5, 666, 4 ], [ 5, 668, 2 ], [ 5, 670, 3 ], [ 5, 671, 2 ], [ 5, 672, 2 ], [ 5, 673, 2 ],
  [ 5, 704, 3 ], [ 5, 706, 8 ], [ 5, 708, 2 ], [ 5, 709, 3 ], [ 5, 713, 2 ], [ 5, 714, 2 ],
  [ 5, 715, 2 ], [ 5, 716, 2 ], [ 5, 717, 2 ], [ 5, 719, 2 ], [ 5, 720, 2 ], [ 5, 721, 2 ],
  [ 5, 763, 3 ], [ 5, 770, 2 ], [ 5, 774, 4 ], [ 5, 948, 1 ], [ 5, 948, 2 ], [ 5, 948, 3 ],
  [ 5, 948, 4 ], [ 5, 952, 1 ], [ 5, 952, 3 ] ]
gap> List(N5H1FC2,x->StructureDescription(N5g[x]));
[ "C2 x C2", "C2 x C2", "C2 x C2 x C2", "C2 x C2 x C2", "C2 x C2 x C2", "C2 x C2 x C2",
  "C2 x C2 x C2 x C2", "C2 x C2 x C2 x C2", "C2 x C2 x C2 x C2", "C2 x C2 x C2 x C2",
  "C2 x C2 x C2 x C2", "C2 x C2 x C2 x C2", "C2 x C2 x C2 x C2", "C2 x C2 x C2",
  "C2 x C2 x C2", "C2 x C2 x C2", "C2 x C2 x C2", "C2 x C2 x C2", "C2 x C2 x C2",
  "C2 x C2 x C2", "C2 x C2 x C2", "C2 x C2 x C2", "C2 x C2 x C2", "C2 x C2 x C2", "D8",
  "C2 x D8", "C2 x D8", "C2 x D8", "C2 x D8", "C2 x D8", "C2 x D8", "C2 x D8", "C2 x D8",
  "C2 x D8", "C2 x D8", "C2 x D8", "C2 x D8", "C2 x D8", "C2 x D8", "C2 x D8", "C2 x D8",
  "C2 x D8", "C2 x D8", "C2 x D8", "C2 x D8", "C2 x D8", "C2 x D8", "C2 x D8", "C2 x D8",
  "C2 x D8", "C2 x D8", "C4 x C2 x C2", "C4 x C2 x C2", "C4 x C2 x C2", "C4 x C2 x C2",
  "C4 x C2 x C2", "C2 x C2 x D8", "C2 x C2 x D8", "C2 x C2 x D8", "C2 x C2 x D8",
  "C2 x C2 x D8", "D8", "D8", "D8", "D8", "D8", "D8", "C4 x C2", "C4 x C2", "C4 x C2",
  "C4 x C2", "C4 : C4", "(C4 x C2) : C2", "(C4 x C2) : C2", "C4 x C2 x C2", "C4 x C2 x C2",
  "C4 x C4", "(C4 x C2) : C2", "(C4 x C2) : C2", "C2 x D8", "C2 x D8", "C2 x D8", "C2 x D8",
  "C2 x D8", "C2 x D8", "C2 x D8", "C2 x D8", "C2 x D8", "C4 x D8", "(C4 x C4) : C2",
  "(C2 x C2 x C2 x C2) : C2", "(C4 x C2 x C2) : C2", "(C2 x C2 x C2 x C2) : C2",
  "(C4 x C2 x C2) : C2", "C2 x C2 x D8", "D8 x D8", "C4 x C2", "C4 x C2", "C4 x C2",
  "C4 x C2", "C4 x C2", "C6 x C2", "D12", "D12", "C2 x C2 x S3", "C2 x C2 x A4", "C2 x S4",
  "C2 x S4", "C2 x C2 x S4", "A4", "C2 x A4", "C2 x A4", "C2 x A4", "C8 : C2", "C8 : C2",
  "(C4 x C2) : C2", "(C4 x C2) : C2", "(C4 x C2) : C2", "QD16", "QD16", "(C2 x C2 x C2) : C4",
  "(C2 x C2 x C2) : C4", "(C2 x C2 x C2) : (C2 x C2)", "(C2 x C2 x C2) : (C2 x C2)",
  "C8 : (C2 x C2)", "C8 : (C2 x C2)", "C8 : (C2 x C2)", "C8 : (C2 x C2)", "C8 : (C2 x C2)",
  "C2 x ((C4 x C2) : C2)", "C2 x (C8 : C2)", "C2 x QD16", "((C2 x C2 x C2) : (C2 x C2)) : C2",
  "C2 x (C8 : (C2 x C2))", "Q8", "C2 x A5", "C2 x A5", "C2 x A5", "C2 x A5", "A5", "A5" ]

# N41 is the list of CARAT IDs of decomposable lattice of rank 5=4+1 
# whose flabby class [M_G]^fl is not invertible [HY17, Example 4.12]
gap> N41g:=List(N41,x->CaratMatGroupZClass(x[1],x[2],x[3]));;
gap> Length(N41g); # there exist 768 not retract rational tori in dim=5=4+1 [HY17, Table 14]
768
gap> N41gF:=List(N41g,x->FlabbyResolutionLowRank(x).actionF);;
gap> N41H1F:=List(N41gF,x->Filtered(H1(x),y->y>1));;
gap> Collected(N41H1F);
[ [ [  ], 690 ], [ [ 2 ], 73 ], [ [ 2, 2 ], 5 ] ]

gap> N41H1FC2xC2:=Filtered([1..Length(N41H1F)],x->N41H1F[x]=[2,2]);
[ 589, 590, 591, 720, 721 ]
gap> List(N41H1FC2xC2,x->N41[x]);
[ [ 5, 664, 1 ], [ 5, 773, 3 ], [ 5, 774, 3 ], [ 5, 691, 1 ], [ 5, 730, 1 ] ]
gap> List(N41H1FC2xC2,x->StructureDescription(N41g[x]));
[ "C2 x Q8", "Q8", "Q8", "SL(2,3)", "C2 x SL(2,3)" ]

gap> N41H1FC2:=Filtered([1..Length(N41H1F)],x->N41H1F[x]=[2]);
[ 1, 2, 3, 4, 9, 10, 11, 12, 13, 14, 36, 37, 38, 39, 40, 41, 42, 74, 75, 76, 77, 93, 94, 95,
  96, 97, 112, 113, 114, 115, 116, 117, 118, 254, 255, 256, 257, 281, 282, 283, 284, 285,
  592, 593, 594, 595, 596, 597, 598, 599, 600, 601, 602, 603, 604, 605, 610, 611, 612, 613,
  614, 615, 616, 727, 728, 729, 730, 731, 732, 738, 739, 740, 741 ]
gap> List(N41H1FC2,x->N41[x]); # CARAT ID's of F with H1(F)=C2
[ [ 5, 18, 18 ], [ 5, 18, 21 ], [ 5, 19, 10 ], [ 5, 32, 23 ], [ 5, 20, 10 ], [ 5, 20, 13 ],
  [ 5, 25, 14 ], [ 5, 30, 14 ], [ 5, 31, 16 ], [ 5, 31, 29 ], [ 5, 21, 10 ], [ 5, 24, 18 ],
  [ 5, 24, 21 ], [ 5, 26, 19 ], [ 5, 31, 22 ], [ 5, 31, 25 ], [ 5, 32, 30 ], [ 5, 66, 5 ],
  [ 5, 83, 7 ], [ 5, 101, 4 ], [ 5, 102, 4 ], [ 5, 63, 12 ], [ 5, 65, 12 ], [ 5, 76, 31 ],
  [ 5, 99, 5 ], [ 5, 100, 5 ], [ 5, 48, 12 ], [ 5, 71, 8 ], [ 5, 72, 26 ], [ 5, 75, 26 ],
  [ 5, 78, 26 ], [ 5, 79, 26 ], [ 5, 88, 26 ], [ 5, 112, 3 ], [ 5, 160, 3 ], [ 5, 161, 3 ],
  [ 5, 162, 3 ], [ 5, 119, 3 ], [ 5, 120, 11 ], [ 5, 121, 11 ], [ 5, 122, 14 ],
  [ 5, 148, 3 ], [ 5, 665, 3 ], [ 5, 666, 3 ], [ 5, 667, 3 ], [ 5, 720, 1 ], [ 5, 672, 1 ],
  [ 5, 673, 1 ], [ 5, 674, 1 ], [ 5, 675, 1 ], [ 5, 721, 1 ], [ 5, 668, 1 ], [ 5, 669, 1 ],
  [ 5, 670, 1 ], [ 5, 671, 1 ], [ 5, 719, 1 ], [ 5, 713, 1 ], [ 5, 714, 1 ], [ 5, 715, 1 ],
  [ 5, 716, 1 ], [ 5, 717, 1 ], [ 5, 718, 1 ], [ 5, 770, 1 ], [ 5, 731, 1 ], [ 5, 732, 1 ],
  [ 5, 775, 1 ], [ 5, 733, 1 ], [ 5, 734, 1 ], [ 5, 776, 1 ], [ 5, 682, 1 ], [ 5, 780, 1 ],
  [ 5, 781, 1 ], [ 5, 783, 1 ] ]
gap> List(N41H1FC2,x->StructureDescription(N41g[x]));
[ "C2 x C2", "C2 x C2", "C2 x C2", "C2 x C2 x C2", "C2 x C2 x C2", "C2 x C2 x C2",
  "C2 x C2 x C2 x C2", "C2 x C2 x C2", "C2 x C2 x C2", "C2 x C2 x C2", "C2 x C2 x C2",
  "C2 x C2 x C2", "C2 x C2 x C2", "C2 x C2 x C2 x C2", "C2 x C2 x C2", "C2 x C2 x C2",
  "C2 x C2 x C2", "C4 x C2", "C4 x C2 x C2", "C4 x C2", "C4 x C2", "D8", "D8", "C2 x D8",
  "D8", "D8", "C2 x D8", "C2 x D8", "C2 x D8", "C2 x D8", "C2 x D8", "C2 x D8",
  "C2 x C2 x D8", "C4 x C2 x C2", "C4 x C2", "C4 x C2", "C4 x C2", "C2 x D8", "C2 x D8",
  "C2 x D8", "C2 x D8", "C2 x C2 x D8", "C8 : C2", "C8 : C2", "C8 : C2", "C2 x (C8 : C2)",
  "QD16", "QD16", "QD16", "QD16", "C2 x QD16", "(C4 x C2) : C2", "(C4 x C2) : C2",
  "(C4 x C2) : C2", "(C4 x C2) : C2", "C2 x ((C4 x C2) : C2)", "C8 : (C2 x C2)",
  "C8 : (C2 x C2)", "C8 : (C2 x C2)", "C8 : (C2 x C2)", "C8 : (C2 x C2)", "C8 : (C2 x C2)",
  "C2 x (C8 : (C2 x C2))", "((C4 x C2) : C2) : C3", "((C4 x C2) : C2) : C3",
  "C2 x (((C4 x C2) : C2) : C3)", "GL(2,3)", "GL(2,3)", "C2 x GL(2,3)", "C2 x (GL(2,3) : C2)",
  "(((C4 x C2) : C2) : C3) : C2", "(((C4 x C2) : C2) : C3) : C2",
  "(((C4 x C2) : C2) : C3) : C2" ]

# N32 is the list of CARAT IDs of decomposable lattice of rank 5=3+2 
# whose flabby class [M_G]^fl is not invertible [HY17, Example 4.12]
gap> N32g:=List(N32,x->CaratMatGroupZClass(x[1],x[2],x[3]));;
gap> Length(N32g); # there exist 849 not retract rational tori in dim=5=3+2 [HY17, Table 13]
849
gap> N32gF:=List(N32g,x->FlabbyResolutionLowRank(x).actionF);;
gap> N32H1F:=List(N32gF,x->Filtered(H1(x),y->y>1));;
gap> Collected(N32H1F);
[ [ [  ], 813 ], [ [ 2 ], 36 ] ]]

gap> N32H1FC2:=Filtered([1..Length(N32H1F)],x->N32H1F[x]=[2]);
[ 1, 2, 3, 4, 5, 6, 7, 8, 9, 10, 11, 12, 13, 14, 15, 16, 17, 18, 19, 20, 21, 
  22, 23, 24, 25, 578, 579, 580, 581, 582, 583, 584, 585, 586, 587, 588 ]
gap> List(N32H1FC2,x->N32[x]); # CARAT ID's of F with H1(F)=C2
[ [ 5, 14, 8 ], [ 5, 18, 19 ], [ 5, 24, 19 ], [ 5, 26, 20 ], [ 5, 31, 17 ], 
  [ 5, 32, 17 ], [ 5, 78, 8 ], [ 5, 78, 27 ], [ 5, 80, 8 ], [ 5, 86, 5 ], 
  [ 5, 93, 5 ], [ 5, 100, 11 ], [ 5, 102, 8 ], [ 5, 224, 4 ], [ 5, 227, 5 ], 
  [ 5, 228, 3 ], [ 5, 232, 4 ], [ 5, 232, 9 ], [ 5, 237, 3 ], [ 5, 242, 4 ], 
  [ 5, 242, 14 ], [ 5, 247, 3 ], [ 5, 247, 7 ], [ 5, 253, 4 ], 
  [ 5, 259, 3 ], [ 5, 520, 17 ], [ 5, 525, 2 ], [ 5, 560, 3 ], 
  [ 5, 566, 3 ], [ 5, 580, 1 ], [ 5, 590, 1 ], [ 5, 605, 1 ], [ 5, 620, 1 ], 
  [ 5, 629, 1 ], [ 5, 634, 1 ], [ 5, 634, 3 ] ]
gap> List(N32H1FC2,x->StructureDescription(N32g[x]));
[ "C2 x C2", "C2 x C2", "C2 x C2 x C2", "C2 x C2 x C2 x C2", "C2 x C2 x C2", 
  "C2 x C2 x C2", "C2 x D8", "C2 x D8", "C2 x D8", "C4 x C2 x C2", 
  "C2 x C2 x D8", "D8", "C4 x C2", "C6 x C2", "D12", "C6 x C2", "D12", 
  "D12", "C6 x C2 x C2", "C2 x C2 x S3", "C2 x C2 x S3", "C2 x C2 x S3", 
  "C2 x C2 x S3", "C2 x C2 x S3", "C2 x C2 x C2 x S3", "C2 x A4", 
  "C2 x C2 x A4", "C4 x A4", "A4 x D8", "A4", "C2 x A4 x S3", "C2 x A4", 
  "C3 x A4", "C6 x A4", "A4 x S3", "A4 x S3" ]

# N311 is the list of CARAT IDs of decomposable lattice of rank 5=3+1+1
# whose flabby class [M_G]^fl is not invertible [HY17, Example 4.12]
gap> N311g:=List(N311,x->CaratMatGroupZClass(x[1],x[2],x[3]));;
gap> Length(N311g); # there exist 245 not retract rational tori in dim=5=3+1+1 [HY17, Table 12]
245
gap> N311gF:=List(N311g,x->FlabbyResolutionLowRank(x).actionF);;
gap> N311H1F:=List(N311gF,x->Filtered(H1(x),y->y>1));;
gap> Collected(N311H1F);
[ [ [  ], 232 ], [ [ 2 ], 13 ] ]

gap> N311H1FC2:=Filtered([1..Length(N311H1F)],x->N311H1F[x]=[2]);
[ 1, 2, 3, 4, 5, 6, 7, 8, 9, 164, 165, 166, 167 ]
gap> List(N311H1FC2,x->N311[x]); # CARAT ID's of F with H1(F)=C2
[ [ 5, 11, 4 ], [ 5, 14, 4 ], [ 5, 18, 7 ], [ 5, 19, 5 ], [ 5, 21, 5 ], 
  [ 5, 24, 7 ], [ 5, 26, 3 ], [ 5, 31, 4 ], [ 5, 32, 10 ], [ 5, 502, 6 ], 
  [ 5, 505, 1 ], [ 5, 520, 16 ], [ 5, 525, 1 ] ]
gap> List(N311H1FC2,x->StructureDescription(N311g[x]));
[ "C2 x C2", "C2 x C2", "C2 x C2", "C2 x C2", "C2 x C2 x C2", 
  "C2 x C2 x C2", "C2 x C2 x C2 x C2", "C2 x C2 x C2", "C2 x C2 x C2", 
  "A4", "C2 x A4", "C2 x A4", "C2 x C2 x A4" ]
\end{verbatim}
}
\end{example}

\section{{Proof of Theorem \ref{thmain2}} and {Theorem \ref{thmain2M}}}\label{S5}

Let $T=R^{(1)}_{K/k}(\bG_m)$ be a norm one torus of $K/k$. 
We have the character module $\widehat{T}=J_{G/H}$ of $T$ 
and then 
$H^1(k,{\rm Pic}\,\overline{X})\simeq H^1(G,[J_{G/H}]^{fl})$ 
(see Section \ref{S2}). 
We may assume that 
$H$ is the stabilizer of one of the letters in $G$, 
i.e. $L=k(\theta_1,\ldots,\theta_n)$ and $K=k(\theta_i)$ for some 
$1\leq i\leq n$. 
In order to compute $H^1(G,[J_{G/H}]^{fl})$, 
we apply the functions 
{\tt Norm1TorusJ($n,m$)} and 
{\tt FlabbyResolutionLowRankFromGroup($G$,$nTm$).actionF} 
in \cite[Algorithm 4.1]{HHY20}. 
{\tt Norm1TorusJ($n,m$)} returns 
$J_{G/H}$ for $G=nTm\leq S_n$ 
and $H$ is the stabilizer of one of the letters in $G$ 
and {\tt FlabbyResolutionLowRankFromGroup($G$,$nTm$).actionF} 
returns a suitable
flabby class $F=[J_{G/H}]^{fl}$  
with low rank by using the backtracking techniques 
for $G=nTm\leq S_n$.\\

{\it Proof of Theorem \ref{thmain2}}. 

For $G=nTm$ $(2\leq n\leq 11)$, 
the computation is described in Example \ref{exH1F}. 

For $G=nTm$ $(12\leq n\leq 15)$, 
it needs much time and computer resources (memory) in computations. 
At present, we do not know the complete solutions for $n=12$. 

For $13\leq n\leq 15$, 
we wish to compute $H^1(G,[J_{G/H}]^{fl})$ for each of the cases 
$G=13Tm$ $(1\leq m\leq 9)$, 
$G=14Tm$ $(1\leq m\leq 63)$, 
$G=15Tm$ $(1\leq m\leq 104)$. 
This is achievable 
except for the each last two groups 
$G=13T8\simeq A_{13}$, $13T9\simeq S_{13}$, 
$14T62\simeq A_{14}$, $14T63\simeq S_{14}$, 
$15T103\simeq A_{15}$, $15T104\simeq S_{15}$ 
because of the computer resources reason (see Example \ref{exH1F13to15}). 
However, for the exceptional cases, 
we already know that $H^1(G,[J_{G/H}]^{fl})=0$
by Theorem \ref{thKV84} and Theorem \ref{thMac}.

The last assertion follows from Theorem \ref{thV}.\qed\\

{\it Proof of Theorem \ref{thmain2M}}. 

For $G=11T6\simeq M_{11}\leq S_{11}$ 
and $G=23T5\simeq M_{23}\leq S_{23}$, 
it follows from Theorem \ref{thS} that 
$H^1(G,[J_{G/H}]^{fl})=0$ (see also the paragraph after Theorem \ref{thEM73}). 
For $G=24T24680\simeq M_{24}\leq S_{24}$, 
we know that the Schur multiplier of $G$ vanishes: 
$M(G)\simeq H^3(G,\bZ)=0$ (see Mazet \cite{Maz82}). 
We know that the Mathieu groups are simple groups 
and a subgroup $H\leq M_{24}$ with $[M_{24}:H]=24$ 
is isomorphic to $M_{23}$ which is the stabilizer of one of the letters 
in $M_{24}$ (see e.g. \cite[Exercises 6.8.8]{DM96}). 
Hence it follows from 
$0=H^{ab}\simeq H^2(H,\bZ)\simeq H^2(G,\bZ[G/H])\to 
H^2(G,J_{G/H})\xrightarrow{\delta} H^3(G,\bZ)=0$ 
that $H^2(G,J_{G/H})=0$. 
Thus we have 
$H^1(G,[J_{G/H}]^{fl})\simeq\Sha^2_\omega(G,J_{G/H})=0$. 

For $G=12T295\simeq M_{12}\leq S_{12}$ and $G=22T38\simeq M_{22}\leq S_{22}$,  
the computation is described in Example \ref{exH1FM12M22}.
\qed\\

Some related functions for Example \ref{exH1F}, 
Example \ref{exH1F13to15} and Example \ref{exH1FM12M22} 
are available from\\ 
{\tt https://www.math.kyoto-u.ac.jp/\~{}yamasaki/Algorithm/RatProbNorm1Tori/}.

\smallskip
\begin{example}[{Computation of $H^1(G,[J_{G/H}]^{fl})$ where $G=nTm$ $(n\leq 11)$}]\label{exH1F}
~{}\vspace*{-4mm}\\
{\small 
\begin{verbatim}
gap> Read("FlabbyResolutionFromBase.gap");
gap> for n in [2..11] do for m in [1..NrTransitiveGroups(n)] do
> F:=FlabbyResolutionLowRankFromGroup(Norm1TorusJ(n,m),TransitiveGroup(n,m)).actionF;
> Print([[n,m],Length(F.1),Filtered(H1(F),x->x>1)],"\n");od;Print("\n");od;
[ [ 2, 1 ], 1, [  ] ]

[ [ 3, 1 ], 1, [  ] ]
[ [ 3, 2 ], 4, [  ] ]

[ [ 4, 1 ], 1, [  ] ]
[ [ 4, 2 ], 5, [ 2 ] ]
[ [ 4, 3 ], 7, [  ] ]
[ [ 4, 4 ], 9, [ 2 ] ]
[ [ 4, 5 ], 15, [  ] ]

[ [ 5, 1 ], 1, [  ] ]
[ [ 5, 2 ], 6, [  ] ]
[ [ 5, 3 ], 16, [  ] ]
[ [ 5, 4 ], 16, [  ] ]
[ [ 5, 5 ], 16, [  ] ]

[ [ 6, 1 ], 1, [  ] ]
[ [ 6, 2 ], 7, [  ] ]
[ [ 6, 3 ], 9, [  ] ]
[ [ 6, 4 ], 10, [ 2 ] ]
[ [ 6, 5 ], 21, [  ] ]
[ [ 6, 6 ], 10, [  ] ]
[ [ 6, 7 ], 19, [  ] ]
[ [ 6, 8 ], 19, [  ] ]
[ [ 6, 9 ], 27, [  ] ]
[ [ 6, 10 ], 27, [  ] ]
[ [ 6, 11 ], 19, [  ] ]
[ [ 6, 12 ], 10, [ 2 ] ]
[ [ 6, 13 ], 27, [  ] ]
[ [ 6, 14 ], 31, [  ] ]
[ [ 6, 15 ], 60, [  ] ]
[ [ 6, 16 ], 60, [  ] ]

[ [ 7, 1 ], 1, [  ] ]
[ [ 7, 2 ], 8, [  ] ]
[ [ 7, 3 ], 15, [  ] ]
[ [ 7, 4 ], 36, [  ] ]
[ [ 7, 5 ], 15, [  ] ]
[ [ 7, 6 ], 36, [  ] ]
[ [ 7, 7 ], 36, [  ] ]

[ [ 8, 1 ], 1, [  ] ]
[ [ 8, 2 ], 9, [ 2 ] ]
[ [ 8, 3 ], 17, [ 2, 2, 2 ] ]
[ [ 8, 4 ], 9, [ 2 ] ]
[ [ 8, 5 ], 9, [  ] ]
[ [ 8, 6 ], 11, [  ] ]
[ [ 8, 7 ], 21, [  ] ]
[ [ 8, 8 ], 11, [  ] ]
[ [ 8, 9 ], 21, [ 2 ] ]
[ [ 8, 10 ], 21, [  ] ]
[ [ 8, 11 ], 21, [ 2 ] ]
[ [ 8, 12 ], 25, [  ] ]
[ [ 8, 13 ], 19, [ 2 ] ]
[ [ 8, 14 ], 13, [ 2 ] ]
[ [ 8, 15 ], 43, [ 2 ] ]
[ [ 8, 16 ], 29, [  ] ]
[ [ 8, 17 ], 43, [  ] ]
[ [ 8, 18 ], 91, [  ] ]
[ [ 8, 19 ], 51, [ 2 ] ]
[ [ 8, 20 ], 29, [  ] ]
[ [ 8, 21 ], 49, [ 2 ] ]
[ [ 8, 22 ], 49, [ 2 ] ]
[ [ 8, 23 ], 31, [  ] ]
[ [ 8, 24 ], 31, [  ] ]
[ [ 8, 25 ], 49, [  ] ]
[ [ 8, 26 ], 67, [  ] ]
[ [ 8, 27 ], 29, [  ] ]
[ [ 8, 28 ], 83, [  ] ]
[ [ 8, 29 ], 99, [  ] ]
[ [ 8, 30 ], 67, [  ] ]
[ [ 8, 31 ], 49, [ 2 ] ]
[ [ 8, 32 ], 61, [ 2 ] ]
[ [ 8, 33 ], 99, [  ] ]
[ [ 8, 34 ], 123,[  ] ]
[ [ 8, 35 ], 99, [  ] ]
[ [ 8, 36 ], 49, [  ] ]
[ [ 8, 37 ], 49, [ 2 ] ]
[ [ 8, 38 ], 61, [ 2 ] ]
[ [ 8, 39 ], 211, [  ] ]
[ [ 8, 40 ], 115, [  ] ]
[ [ 8, 41 ], 123, [  ] ]
[ [ 8, 42 ], 123, [  ] ]
[ [ 8, 43 ], 91, [  ] ]
[ [ 8, 44 ], 211, [  ] ]
[ [ 8, 45 ], 123, [  ] ]
[ [ 8, 46 ], 123, [  ] ]
[ [ 8, 47 ], 123, [  ] ]
[ [ 8, 48 ], 483, [  ] ]
[ [ 8, 49 ], 539, [  ] ]
[ [ 8, 50 ], 539, [  ] ]

[ [ 9, 1 ], 1, [  ] ]
[ [ 9, 2 ], 10, [ 3 ] ]
[ [ 9, 3 ], 10, [  ] ]
[ [ 9, 4 ], 13, [  ] ]
[ [ 9, 5 ], 28, [ 3 ] ]
[ [ 9, 6 ], 31, [  ] ]
[ [ 9, 7 ], 31, [ 3 ] ]
[ [ 9, 8 ], 28, [  ] ]
[ [ 9, 9 ], 28, [ 3 ] ]
[ [ 9, 10 ], 70, [  ] ]
[ [ 9, 11 ], 70, [ 3 ] ]
[ [ 9, 12 ], 61, [  ] ]
[ [ 9, 13 ], 40, [  ] ]
[ [ 9, 14 ], 64, [ 3 ] ]
[ [ 9, 15 ], 64, [  ] ]
[ [ 9, 16 ], 34, [  ] ]
[ [ 9, 17 ], 31, [  ] ]
[ [ 9, 18 ], 70, [  ] ]
[ [ 9, 19 ], 64, [  ] ]
[ [ 9, 20 ], 61, [  ] ]
[ [ 9, 21 ], 70, [  ] ]
[ [ 9, 22 ], 40, [  ] ]
[ [ 9, 23 ], 88, [ 3 ] ]
[ [ 9, 24 ], 70, [  ] ]
[ [ 9, 25 ], 40, [  ] ]
[ [ 9, 26 ], 88, [  ] ]
[ [ 9, 27 ], 64, [  ] ]
[ [ 9, 28 ], 40, [  ] ]
[ [ 9, 29 ], 70, [  ] ]
[ [ 9, 30 ], 70, [  ] ]
[ [ 9, 31 ], 70, [  ] ]
[ [ 9, 32 ], 232, [  ] ]
[ [ 9, 33 ], 1744, [  ] ]
[ [ 9, 34 ], 1744, [  ] ]

[ [ 10, 1 ], 1, [  ] ]
[ [ 10, 2 ], 11, [  ] ]
[ [ 10, 3 ], 13, [  ] ]
[ [ 10, 4 ], 13, [  ] ]
[ [ 10, 5 ], 31, [  ] ]
[ [ 10, 6 ], 53, [  ] ]
[ [ 10, 7 ], 26, [ 2 ] ]
[ [ 10, 8 ], 36, [  ] ]
[ [ 10, 9 ], 63, [  ] ]
[ [ 10, 10 ], 63, [  ] ]
[ [ 10, 11 ], 31, [  ] ]
[ [ 10, 12 ], 31, [  ] ]
[ [ 10, 13 ], 36, [  ] ]
[ [ 10, 14 ], 36, [  ] ]
[ [ 10, 15 ], 51, [  ] ]
[ [ 10, 16 ], 51, [  ] ]
[ [ 10, 17 ], 83, [  ] ]
[ [ 10, 18 ], 83, [  ] ]
[ [ 10, 19 ], 83, [  ] ]
[ [ 10, 20 ], 83, [  ] ]
[ [ 10, 21 ], 63, [  ] ]
[ [ 10, 22 ], 31, [  ] ]
[ [ 10, 23 ], 51, [  ] ]
[ [ 10, 24 ], 61, [  ] ]
[ [ 10, 25 ], 61, [  ] ]
[ [ 10, 26 ], 46, [ 2 ] ]
[ [ 10, 27 ], 83, [  ] ]
[ [ 10, 28 ], 83, [  ] ]
[ [ 10, 29 ], 61, [  ] ]
[ [ 10, 30 ], 91, [  ] ]
[ [ 10, 31 ], 67, [  ] ]
[ [ 10, 32 ], 46, [ 2 ] ]
[ [ 10, 33 ], 83, [  ] ]
[ [ 10, 34 ], 61, [  ] ]
[ [ 10, 35 ], 91, [  ] ]
[ [ 10, 36 ], 61, [  ] ]
[ [ 10, 37 ], 61, [  ] ]
[ [ 10, 38 ], 61, [  ] ]
[ [ 10, 39 ], 61, [  ] ]
[ [ 10, 40 ], 83, [  ] ]
[ [ 10, 41 ], 83, [  ] ]
[ [ 10, 42 ], 83, [  ] ]
[ [ 10, 43 ], 83, [  ] ]
[ [ 10, 44 ], 378, [  ] ]
[ [ 10, 45 ], 378, [  ] ]

[ [ 11, 1 ], 1, [  ] ]
[ [ 11, 2 ], 12, [  ] ]
[ [ 11, 3 ], 45, [  ] ]
[ [ 11, 4 ], 100, [  ] ]
[ [ 11, 5 ], 56, [  ] ]
[ [ 11, 6 ], 100, [  ] ]
[ [ 11, 7 ], 100, [  ] ]
[ [ 11, 8 ], 100, [  ] ]
\end{verbatim}
}
\end{example}

\smallskip
\begin{example}[{Computation of $H^1(G,[J_{G/H}]^{fl})$ where $G=13Tm$ $(1\leq m\leq 9, m\neq 8,9)$, $G=14Tm$ $(1\leq m\leq 63, m\neq 62,63)$, $G=15Tm$ $(1\leq m\leq 104, m\neq 103, 104)$}]\label{exH1F13to15}
~{}\vspace*{-4mm}\\
{\small  
\begin{verbatim}
gap> Read("FlabbyResolutionFromBase.gap");
gap> for m in [1..NrTransitiveGroups(13)-2] do 
> F:=FlabbyResolutionLowRankFromGroup(Norm1TorusJ(13,m),TransitiveGroup(13,m)).actionF; 
> Print([[n,i],Length(F.1),Filtered(H1(F),x->x>1)],"\n");od;Print("\n");od;
[ [ 13, 1 ], 1, [  ] ]
[ [ 13, 2 ], 14, [  ] ]
[ [ 13, 3 ], 27, [  ] ]
[ [ 13, 4 ], 40, [  ] ]
[ [ 13, 5 ], 66, [  ] ]
[ [ 13, 6 ], 144, [  ] ]
[ [ 13, 7 ], 40, [  ] ]

gap> for m in [1..NrTransitiveGroups(14)-2] do 
> F:=FlabbyResolutionLowRankFromGroup(Norm1TorusJ(14,m),TransitiveGroup(14,m)).actionF; 
> Print([[14,m],Length(F.1),Filtered(H1(F),x->x>1)],"\n");od;
[ [ 14, 1 ], 1, [  ] ]
[ [ 14, 2 ], 15, [  ] ]
[ [ 14, 3 ], 17, [  ] ]
[ [ 14, 4 ], 31, [  ] ]
[ [ 14, 5 ], 31, [  ] ]
[ [ 14, 6 ], 50, [  ] ]
[ [ 14, 7 ], 57, [  ] ]
[ [ 14, 8 ], 101, [  ] ]
[ [ 14, 9 ], 78, [  ] ]
[ [ 14, 10 ], 64, [  ] ]
[ [ 14, 11 ], 86, [  ] ]
[ [ 14, 12 ], 115, [  ] ]
[ [ 14, 13 ], 115, [  ] ]
[ [ 14, 14 ], 129, [  ] ]
[ [ 14, 15 ], 129, [  ] ]
[ [ 14, 16 ], 31, [  ] ]
[ [ 14, 17 ], 92, [  ] ]
[ [ 14, 18 ], 92, [  ] ]
[ [ 14, 19 ], 31, [  ] ]
[ [ 14, 20 ], 115, [  ] ]
[ [ 14, 21 ], 78, [  ] ]
[ [ 14, 22 ], 171, [  ] ]
[ [ 14, 23 ], 171, [  ] ]
[ [ 14, 24 ], 171, [  ] ]
[ [ 14, 25 ], 171, [  ] ]
[ [ 14, 26 ], 129, [  ] ]
[ [ 14, 27 ], 99, [  ] ]
[ [ 14, 28 ], 99, [  ] ]
[ [ 14, 29 ], 78, [  ] ]
[ [ 14, 30 ], 92, [ 2 ] ]
[ [ 14, 31 ], 171, [  ] ]
[ [ 14, 32 ], 171, [  ] ]
[ [ 14, 33 ], 92, [  ] ]
[ [ 14, 34 ], 92, [  ] ]
[ [ 14, 35 ], 92, [  ] ]
[ [ 14, 36 ], 171, [  ] ]
[ [ 14, 37 ], 171, [  ] ]
[ [ 14, 38 ], 99, [  ] ]
[ [ 14, 39 ], 183, [  ] ]
[ [ 14, 40 ], 127, [  ] ]
[ [ 14, 41 ], 127,[  ] ]
[ [ 14, 42 ], 92, [  ] ]
[ [ 14, 43 ], 92, [  ] ]
[ [ 14, 44 ], 99, [  ] ]
[ [ 14, 45 ], 171, [  ] ]
[ [ 14, 46 ], 57, [  ] ]
[ [ 14, 47 ], 57, [  ] ]
[ [ 14, 48 ], 127, [  ] ]
[ [ 14, 49 ], 57, [  ] ]
[ [ 14, 50 ], 92, [  ] ]
[ [ 14, 51 ], 92, [  ] ]
[ [ 14, 52 ], 129, [  ] ]
[ [ 14, 53 ], 127, [  ] ]
[ [ 14, 54 ], 127, [  ] ]
[ [ 14, 55 ], 127, [  ] ]
[ [ 14, 56 ], 127, [  ] ]
[ [ 14, 57 ], 127, [  ] ]
[ [ 14, 58 ], 171, [  ] ]
[ [ 14, 59 ], 171, [  ] ]
[ [ 14, 60 ], 171, [  ] ]
[ [ 14, 61 ], 171, [  ] ]

gap> for m in [1..NrTransitiveGroups(15)-2] do 
> F:=FlabbyResolutionLowRankFromGroup(Norm1TorusJ(15,m),TransitiveGroup(15,m)).actionF; 
> Print([[15,m],Length(F.1),Filtered(H1(F),x->x>1)],"\n");od;
[ [ 15, 1 ], 1, [  ] ]
[ [ 15, 2 ], 16, [  ] ]
[ [ 15, 3 ], 14, [  ] ]
[ [ 15, 4 ], 21, [  ] ]
[ [ 15, 5 ], 21, [  ] ]
[ [ 15, 6 ], 39, [  ] ]
[ [ 15, 7 ], 17, [  ] ]
[ [ 15, 8 ], 36, [  ] ]
[ [ 15, 9 ], 79, [ 5 ] ]
[ [ 15, 10 ], 36, [  ] ]
[ [ 15, 11 ], 27, [  ] ]
[ [ 15, 12 ], 94, [  ] ]
[ [ 15, 13 ], 82, [  ] ]
[ [ 15, 14 ], 97, [ 5 ] ]
[ [ 15, 15 ], 51, [  ] ]
[ [ 15, 16 ], 36, [  ] ]
[ [ 15, 17 ], 127, [  ] ]
[ [ 15, 18 ], 97, [  ] ]
[ [ 15, 19 ], 124, [  ] ]
[ [ 15, 20 ], 81, [  ] ]
[ [ 15, 21 ], 66, [  ] ]
[ [ 15, 22 ], 39, [  ] ]
[ [ 15, 23 ], 27, [  ] ]
[ [ 15, 24 ], 36, [  ] ]
[ [ 15, 25 ], 79, [  ] ]
[ [ 15, 26 ], 96, [  ] ]
[ [ 15, 27 ], 127, [  ] ]
[ [ 15, 28 ], 81, [  ] ]
[ [ 15, 29 ], 27, [  ] ]
[ [ 15, 30 ], 94, [  ] ]
[ [ 15, 31 ], 169, [  ] ]
[ [ 15, 32 ], 154, [  ] ]
[ [ 15, 33 ], 111, [  ] ]
[ [ 15, 34 ], 186, [  ] ]
[ [ 15, 35 ], 201, [  ] ]
[ [ 15, 36 ], 96, [  ] ]
[ [ 15, 37 ], 199, [  ] ]
[ [ 15, 38 ], 124, [  ] ]
[ [ 15, 39 ], 94, [  ] ]
[ [ 15, 40 ], 169, [  ] ]
[ [ 15, 41 ], 186, [  ] ]
[ [ 15, 42 ], 201, [  ] ]
[ [ 15, 43 ], 201, [  ] ]
[ [ 15, 44 ], 111, [  ] ]
[ [ 15, 45 ], 201, [  ] ]
[ [ 15, 46 ], 186, [  ] ]
[ [ 15, 47 ], 156, [  ] ]
[ [ 15, 48 ], 169, [  ] ]
[ [ 15, 49 ], 199, [  ] ]
[ [ 15, 50 ], 94, [  ] ]
[ [ 15, 51 ], 169, [  ] ]
[ [ 15, 52 ], 201, [  ] ]
[ [ 15, 53 ], 456, [  ] ]
[ [ 15, 54 ], 201, [  ] ]
[ [ 15, 55 ], 201, [  ] ]
[ [ 15, 56 ], 186, [  ] ]
[ [ 15, 57 ], 124, [  ] ]
[ [ 15, 58 ], 199, [  ] ]
[ [ 15, 59 ], 124, [  ] ]
[ [ 15, 60 ], 169, [  ] ]
[ [ 15, 61 ], 471, [  ] ]
[ [ 15, 62 ], 471, [  ] ]
[ [ 15, 63 ], 456, [  ] ]
[ [ 15, 64 ], 201, [  ] ]
[ [ 15, 65 ], 199, [  ] ]
[ [ 15, 66 ], 199, [  ] ]
[ [ 15, 67 ], 124, [  ] ]
[ [ 15, 68 ], 199, [  ] ]
[ [ 15, 69 ], 456, [  ] ]
[ [ 15, 70 ], 471, [  ] ]
[ [ 15, 71 ], 111, [  ] ]
[ [ 15, 72 ], 156, [  ] ]
[ [ 15, 73 ], 199, [  ] ]
[ [ 15, 74 ], 199, [  ] ]
[ [ 15, 75 ], 124, [  ] ]
[ [ 15, 76 ], 471, [  ] ]
[ [ 15, 77 ], 471, [  ] ]
[ [ 15, 78 ], 456, [  ] ]
[ [ 15, 79 ], 201, [  ] ]
[ [ 15, 80 ], 201, [  ] ]
[ [ 15, 81 ], 111, [  ] ]
[ [ 15, 82 ], 199, [  ] ]
[ [ 15, 83 ], 471, [  ] ]
[ [ 15, 84 ], 201, [  ] ]
[ [ 15, 85 ], 201, [  ] ]
[ [ 15, 86 ], 201, [  ] ]
[ [ 15, 87 ], 201, [  ] ]
[ [ 15, 88 ], 471, [  ] ]
[ [ 15, 89 ], 471, [  ] ]
[ [ 15, 90 ], 471, [  ] ]
[ [ 15, 91 ], 471, [  ] ]
[ [ 15, 92 ], 124, [  ] ]
[ [ 15, 93 ], 471, [  ] ]
[ [ 15, 94 ], 199, [  ] ]
[ [ 15, 95 ], 124, [  ] ]
[ [ 15, 96 ], 199, [  ] ]
[ [ 15, 97 ], 199, [  ] ]
[ [ 15, 98 ], 124, [  ] ]
[ [ 15, 99 ], 199, [  ] ]
[ [ 15, 100 ], 199, [  ] ]
[ [ 15, 101 ], 124, [  ] ]
[ [ 15, 102 ], 199, [  ] ]
\end{verbatim}
}
\end{example}

\smallskip
\begin{example}[{Computation of $H^1(G,[J_{G/H}]^{fl})=0$ where $G=12T295\simeq M_{12}$ and $G=22T38\simeq M_{22}$}]\label{exH1FM12M22}
~{}\vspace*{-4mm}\\
{\small 
\begin{verbatim}
gap> Read("FlabbyResolutionFromBase.gap");

gap> G:=TransitiveGroup(12,295);
M(12)
gap> F:=FlabbyResolutionLowRankFromGroup(Norm1TorusJ(12,295),G).actionF;
<matrix group with 2 generators>
gap> [[12,295],Length(F.1),Filtered(H1(F),x->x>1)];
[ [ 12, 295 ], 814, [  ] ]

gap> G:=TransitiveGroup(22,38);
t22n38
gap> StructureDescription(G);
"M22"
gap> F:=FlabbyResolutionLowRankFromGroup(Norm1TorusJ(22,38),G).actionF;
<matrix group with 2 generators>
gap> [[22,38],Length(F.1),Filtered(H1(F),x->x>1)];
[ [ 22, 38 ], 672, [  ] ]
\end{verbatim}
}
\end{example}
%
\section{Proof of Theorem \ref{thmain3}}\label{S6}

Let $k$ be a number field, $K/k$ be a finite extension, 
$\bA_K^\times$ be the idele group of $K$ and 
$L/k$ be the Galois closure of $K/k$. 
Let $G={\rm Gal}(L/k)=nTm$ be a transitive subgroup of $S_n$ 
and $H={\rm Gal}(L/K)$ with $[G:H]=n$. 

For $x,y\in G$, we denote $[x,y]=x^{-1}y^{-1}xy$ the commutator of 
$x$ and $y$, and $[G,G]$ the commutator group of $G$. 
Let $V_k$ be the set of all places of $k$ 
and $G_v$ be the decomposition group of $G$ at $v\in V_k$. 

\begin{definition}[{Drakokhrust and Platonov \cite[page 350]{PD85a}, \cite[page 300]{DP87}}]
Let $k$ be a number field, 
$L\supset K\supset k$ be a tower of finite extensions 
where $L$ is normal over $k$. 

We call the group 
\begin{align*}
{\rm Obs}(K/k)=(N_{K/k}(\bA_K^\times)\cap k^\times)/N_{K/k}(K^\times)
\end{align*}
{\it the total obstruction to the Hasse norm principle for $K/k$} 
and 
\begin{align*}
{\rm Obs}_1(L/K/k)=\left(N_{K/k}(\bA_K^\times)\cap k^\times\right)/\left((N_{L/k}(\bA_L^\times)\cap k^\times)N_{K/k}(K^\times)\right)
\end{align*}
{\it the first obstruction to the Hasse norm principle for $K/k$ 
corresponding to the tower 
$L\supset K\supset k$}. 
\end{definition}

Note that (i) 
${\rm Obs}(K/k)=1$ if and only if 
the Hasse norm principle holds for $K/k$; 
and (ii) ${\rm Obs}_1(L/K/k)
={\rm Obs}(K/k)/(N_{L/k}(\bA_L^\times)\cap k^\times)$. 

Drakokhrust and Platonov gave a formula 
for computing the first obstruction ${\rm Obs}_1(L/K/k)$: 

\begin{theorem}[{Drakokhrust and Platonov \cite[page 350]{PD85a}, \cite[pages 789--790]{PD85b}, \cite[Theorem 1]{DP87}}]\label{thDP2}
Let $k$ be a number field, 
$L\supset K\supset k$ be a tower of finite extensions 
where 
$L$ is normal over $k$.  
Let $G={\rm Gal}(L/k)$ and $H={\rm Gal}(L/K)$. 
Then 
\begin{align*}
{\rm Obs}_1(L/K/k)\simeq 
{\rm Ker}\, \psi_1/\varphi_1({\rm Ker}\, \psi_2)
\end{align*}
where in the the commutative diagram 
\begin{align*}
\begin{CD}
H/[H,H] @>\psi_1 >> G/[G,G]\\
@AA\varphi_1 A @AA\varphi_2 A\\
\displaystyle{\bigoplus_{v\in V_k}\left(\bigoplus_{w\mid v} H_w/[H_w,H_w]\right)} @>\psi_2 >> 
\displaystyle{\bigoplus_{v\in V_k} G_v/[G_v,G_v]}, 
\end{CD}
\end{align*}
$\psi_1$, $\varphi_1$ and $\varphi_2$ are defined 
by the inclusions $H\subset G$, $H_w\subset H$ and $G_v\subset G$ respectively, and 
\begin{align*}
\psi_2(h[H_{w},H_{w}])=x^{-1}hx[G_v,G_v]
\end{align*}
for $h\in H_{w}=H\cap x^{-1}hx[G_v,G_v]$ $(x\in G)$.
\end{theorem}


Let $\psi_2^{v}$ be the restriction of $\psi_2$ to the subgroup 
$\bigoplus_{w\mid v} H_w/[H_w,H_w]$ with respect to $v\in V_k$ 
and $\psi_2^{\rm nr}$ (resp. $\psi_2^{\rm r}$) be 
the restriction of $\psi_2$ to the unramified (resp. the ramified) 
places $v$ of $k$. 
\begin{proposition}[{Drakokhrust and Platonov \cite{DP87}}]\label{propDP}
Let $k$, 
$L\supset K\supset k$, 
$G$ and $H$ be as in Theorem \ref{thDP2}.\\
{\rm (i)} $($\cite[Lemma 1]{DP87}$)$ 
Places $w_i\mid v$ of $K$ are in one-to-one correspondence 
with the set of double cosets in the decomposition 
$G=\cup_{i=1}^{r_v} Hx_iG_v$ where $H_{w_i}=H\cap x_iG_vx_i^{-1}$;\\
{\rm (ii)} $($\cite[Lemma 2]{DP87}$)$ 
If $G_{v_1}\leq G_{v_2}$, then $\varphi_1({\rm Ker}\,\psi_2^{v_1})\subset \varphi_1({\rm Ker}\,\psi_2^{v_2})$;\\
{\rm (iii)} $($\cite[Theorem 2]{DP87}$)$ 
$\varphi_1({\rm Ker}\,\psi_2^{\rm nr})=\Phi^G(H)/[H,H]$ 
where $\Phi^G(H)=\langle [h,x]\mid h\in H\cap xHx^{-1}, x\in G\rangle$;\\
{\rm (iv)} $($\cite[Lemma 8]{DP87}$)$ If $[K:k]=p^r$ $(r\geq 1)$ 
and ${\rm Obs}(K_p/k_p)=1$ where $k_p=L^{G_p}$, $K_p=L^{H_p}$, 
$G_p$ and $H_p\leq H\cap G_p$ are $p$-Sylow subgroups of $G$ and $H$ 
respectively, then ${\rm Obs}(K/k)=1$.
\end{proposition}

\begin{remark}
The inverse direction of Proposition \ref{propDP} (iv) 
does not hold in general. 
For example, if $n=8$, $G=8T13\simeq A_4\times C_2$ and 
there exists a place $v$ of $k$ such that $G_v\simeq V_4$, 
then ${\rm Obs}(K/k)=1$ but $G_2=8T3\simeq (C_2)^3$ 
and ${\rm Obs}(K_2/k_2)\neq 1$ may occur 
(see Theorem \ref{thmain3} and Table $2$).
\end{remark}

\begin{theorem}[{Drakokhrust and Platonov \cite[Theorem 3, Corollary 1]{DP87}}]\label{thDP87}
Let $k$, 
$L\supset K\supset k$, 
$G$ and $H$ be as in Theorem \ref{thDP2}. 
Let $H_i\leq G_i\leq G$ $(1\leq i\leq m)$, 
$H_i\leq H\cap G_i$, 
$k_i=L^{G_i}$ and $K_i=L^{H_i}$. 
If ${\rm Obs}(K_i/k_i)=1$ for all $1\leq i\leq m$ and 
\begin{align*}
\bigoplus_{i=1}^m \widehat{H}^{-3}(G_i,\bZ)\xrightarrow{\rm cores} 
\widehat{H}^{-3}(G,\bZ)
\end{align*}
is surjective, 
then ${\rm Obs}(K/k)={\rm Obs}_1(L/K/k)$. 
In particular, 
if $[K:k]=n$ is square-free, 
then ${\rm Obs}(K/k)={\rm Obs}_1(L/K/k)$.
\end{theorem}

We note that if $L/k$ is an unramified extension, 
then $A(T)=0$ and $H^1(G,[J_{G/H}]^{fl})\simeq \Sha(T)\simeq 
{\rm Obs}(K/k)$ where $T=R^{(1)}_{K/k}(\bG_m)$ 
(see Theorem \ref{thV} and Theorem \ref{thOno}). 
If, in addition, ${\rm Obs}(K/k)={\rm Obs}_1(L/K/k)$ 
(e.g. $[K:k]=6,10,14,15$; square-free, see Theorem \ref{thDP87}), 
then ${\rm Obs}(K/k)={\rm Obs}_1(L/K/k)=
{\rm Ker}\, \psi_1/\varphi_1({\rm Ker}\, \psi_2^{\rm nr})\simeq 
{\rm Ker}\, \psi_1/(\Phi^G(H)/[H,H])$ 
(see Proposition \ref{propDP} (iii)). 

\begin{theorem}[{Drakokhrust \cite[Theorem 1]{Dra89}, see also Opolka \cite[Satz 3]{Opo80} for the existence of $\widetilde{L}$}]\label{thDra89}
Let $k$, 
$L\supset K\supset k$, 
$G$ and $H$ be as in Theorem \ref{thDP2}. 
Assume that $\widetilde{L}\supset L\supset k$ is 
a tower of Galois extensions with 
$\widetilde{G}={\rm Gal}(\widetilde{L}/k)$ 
and $\widetilde{H}={\rm Gal}(\widetilde{L}/K)$ 
which correspond to a central extension 
$1\to A\to \widetilde{G}\to G\to 1$ with 
$A\cap[\widetilde{G},\widetilde{G}]\simeq M(G)=H^2(G,\bC^\times)$; 
the Schur multiplier of $G$ 
$($this is equivalent to 
the inflation 
$M(G)\to M(\widetilde{G})$ being the zero map, 
see {\rm Beyl and Tappe \cite[Proposition 2.13, page 85]{BT82}}$)$. 
Then 
${\rm Obs}(K/k)={\rm Obs}_1(\widetilde{L}/K/k)$. 
In particular, if $\widetilde{G}$ is a Schur cover of $G$, 
i.e. $A\simeq M(G)$, then ${\rm Obs}(K/k)={\rm Obs}_1(\widetilde{L}/K/k)$. 
\end{theorem}

Indeed, Drakokhrust \cite[Theorem 1]{Dra89} shows that 
${\rm Obs}(K/k)\simeq 
{\rm Ker}\, \widetilde{\psi}_1/\widetilde{\varphi}_1({\rm Ker}\, \widetilde{\psi}_2)$ where the maps $\widetilde{\psi}_1, \widetilde{\psi}_2$ and $\widetilde{\varphi}_1$ are defined as in 
\cite[page 31, the paragraph before Proposition 1]{Dra89}. 
The proof of \cite[Proposition 1]{Dra89} shows that 
this group is the same as ${\rm Obs}_1(\widetilde{L}/K/k)$ 
(see also \cite[Lemma 2, Lemma 3 and Lemma 4]{Dra89}).\\

We made the following functions 
of GAP (\cite{GAP}) which will be used 
in the proof of Theorem \ref{thmain3}. \\

{\tt FirstObstructionN($G,H$).ker} returns 
the list $[l_1, [l_2, l_3]]$ where 
$l_1$ is the abelian invariant of the numerator of the first obstruction 
${\rm Ker}\, \psi_1=\langle y_1,\ldots,y_t\rangle$ 
with respect to $G$, $H$ as in Theorem \ref{thDP2}, 
$l_2=[e_1,\ldots,e_m]$ 
is the abelian invariant of 
$H^{ab}=H/[H,H]=\langle x_1,\ldots,x_m\rangle$ with $e_i={\rm order}(x_i)$ 
and 
$l_3=[l_{3,1},\ldots,l_{3,t}]$, 
$l_{3,i}=[r_{i,1},\ldots,r_{i,m}]$ is the list with 
$y_i=x_1^{r_{i,1}}\cdots x_m^{r_{i,m}}$ 
for $H\leq G\leq S_n$.

{\tt FirstObstructionN($G$).ker} returns the same 
as {\tt FirstObstructionN($G,H$).ker} where $H={\rm Stab}_1(G)$ 
is the stabilizer of $1$ in $G\leq S_n$.\\

{\tt FirstObstructionDnr($G,H$).Dnr} returns 
the list $[l_1, [l_2, l_3]]$ where 
$l_1$ is the abelian invariant of the unramified part 
of the denominator of the first obstruction 
$\varphi_1({\rm Ker}\, \psi_2^{\rm nr})=\Phi^G(H)/[H,H]=\langle y_1,\ldots,y_t\rangle$ 
with respect to $G$, $H$ as in Proposition \ref{propDP} (iii), 
$l_2=[e_1,\ldots,e_m]$ 
is the abelian invariant of 
$H^{ab}=H/[H,H]=\langle x_1,\ldots,x_m\rangle$ with $e_i={\rm order}(x_i)$ 
and 
$l_3=[l_{3,1},\ldots,l_{3,t}]$, 
$l_{3,i}=[r_{i,1},\ldots,r_{i,m}]$ is the list with 
$y_i=x_1^{r_{i,1}}\cdots x_m^{r_{i,m}}$ 
for $H\leq G\leq S_n$.

{\tt FirstObstructionDnr($G$).Dnr} returns the same 
as {\tt FirstObstructionDnr($G,H$).Dnr} where $H={\rm Stab}_1(G)$ 
is the stabilizer of $1$ in $G\leq S_n$.\\

{\tt FirstObstructionDr($G,G_v,H$).Dr} returns 
the list $[l_1, [l_2, l_3]]$ where 
$l_1$ is the abelian invariant of the ramified part 
of the denominator of the first obstruction 
$\varphi_1({\rm Ker}\, \psi_2^v)=\langle y_1,\ldots,y_t\rangle$ 
with respect to $G$, $G_v$, $H$ as in Theorem \ref{thDP2}, 
$l_2=[e_1,\ldots,e_m]$ 
is the abelian invariant of 
$H^{ab}=H/[H,H]=\langle x_1,\ldots,x_m\rangle$ with $e_i={\rm order}(x_i)$ 
and 
$l_3=[l_{3,1},\ldots,l_{3,t}]$, 
$l_{3,i}=[r_{i,1},\ldots,r_{i,m}]$ is the list with 
$y_i=x_1^{r_{i,1}}\cdots x_m^{r_{i,m}}$ 
for $G_v, H\leq G\leq S_n$.

{\tt FirstObstructionDr($G,G_v$).Dr} returns the same 
as {\tt FirstObstructionDr($G,G_v,H$).Dr} where $H={\rm Stab}_1(G)$ 
is the stabilizer of $1$ in $G\leq S_n$.\\

{\tt SchurCoverG($G$).SchurCover} 
(resp. {\tt SchurCoverG($G$).epi}) 
returns 
one of the Schur covers $\widetilde{G}$ of $G$ 
(resp. the surjective map $\pi$) 
in a central extension 
$1\to A\to \widetilde{G}\xrightarrow{\pi} G\to 1$ 
with $A\simeq M(G)$; Schur multiplier of $G$ 
(see Karpilovsky \cite[page 16]{Kap87}). 
The Schur covers $\widetilde{G}$ are 
stem extensions, i.e. $A\leq Z(\widetilde{G})\cap [\widetilde{G},\widetilde{G}]$, of the maximal size. 
This function is based on the built-in function 
{\tt EpimorphismSchurCover} in GAP.\\

{\tt MinimalStemExtensions($G$)[$j$].MinimalStemExtension} 
(resp. {\tt MinimalStemExtensions($G$)[$j$].epi}) 
returns 
the $j$-th minimal stem extension $\overline{G}=\widetilde{G}/A^\prime$, 
i.e. $\overline{A}\leq Z(\overline{G})\cap [\overline{G},\overline{G}]$, 
of $G$ provided by the Schur cover $\widetilde{G}$ of $G$ 
via {\tt SchurCoverG($G$).SchurCover} 
where $A^\prime$ is the $j$-th maximal subgroup of $A=M(G)$
(resp. the surjective map $\overline{\pi}$) 
in the commutative diagram 
\begin{align*}
\begin{CD}
1 @>>> A=M(G) @>>> \widetilde{G}@>\pi>> G@>>> 1\\
  @. @VVV @VVV @|\\
1 @>>> \overline{A}=A/A^\prime @>>> \overline{G}=\widetilde{G}/A^\prime @>\overline{\pi}>> G@>>> 1
\end{CD}
\end{align*}
(see Robinson \cite[Exercises 11.4]{Rob96}). 
This function is based on the built-in function 
{\tt EpimorphismSchurCover} in GAP.\\

{\tt ResolutionNormalSeries(LowerCentralSeries($G$),} {\tt$n+1$)} 
(resp. {\tt ResolutionNormalSeries(DerivedSeries} {\tt ($G$),$n+1$)}, 
{\tt ResolutionFiniteGroup($G$,$n+1$)}) returns 
a free resolution $RG$ of $G$ 
when $G$ is nilpotent (resp. solvable, finite). 
This function is the built-in function of 
HAP (\cite{HAP}) in GAP (\cite{GAP}) .\\

{\tt ResHnZ($RG,RH,n$).HnGZ} (resp. {\tt ResHnZ($RG,RH,n$).HnHZ}) returns 
the abelian invariants of $H^n(G,\bZ)$ (resp. $H^n(H,\bZ)$) 
with respect to Smith normal form, 
for free resolutions $RG$ and $RH$ of $G$ and $H$ respectively.\\

{\tt ResHnZ($RG,RH,n$).Res} returns 
the list $L=[l_1,\ldots,l_s]$ where 
$H^n(G,\bZ)=\langle x_1,\ldots,x_s\rangle\xrightarrow{\rm res} 
H^n(H,\bZ)=\langle y_1,\ldots,y_t\rangle$, 
${\rm res}(x_i)=\prod_{j=1}^t y_j^{l_{i,j}}$ 
and $l_i=[l_{i,1},\ldots,l_{i,t}]$ 
for free resolutions $RG$ and $RH$ of $G$ and $H$ respectively.\\

{\tt ResHnZ($RG,RH,n$).Ker} returns 
the list $L=[l_1,[l_2,l_3]]$ 
where $l_1$ is the abelian invariant of 
${\rm Ker}\{H^n(G,\bZ)$ $\xrightarrow{\rm res}$ 
$H^n(H,\bZ)\}=\langle y_1,\ldots,y_t\rangle$, 
$l_2=[d_1,\ldots,d_s]$ is the abelian invariant of $H^n(G,\bZ)=\langle x_1,\ldots,x_s\rangle$ with $d_i={\rm ord}(x_i)$
and $l_3=[l_{3,1},\ldots,l_{3,t}]$, 
$l_{3,j}=[r_{j,1},\ldots,r_{j,s}]$ is the list with 
$y_j=x_1^{r_{j,1}}\cdots x_s^{r_{j,s}}$ 
for free resolutions $RG$ and $RH$ of $G$ and $H$ respectively.\\

{\tt ResHnZ($RG,RH,n$).Coker} returns 
the list $L=[l_1,[l_2,l_3]]$ 
where $l_1=[e_1,\ldots,e_t]$ is the abelian invariant of 
${\rm Coker}\{H^n(G,\bZ)$ $\xrightarrow{\rm res}$ 
$H^n(H,\bZ)\}=\langle \overline{y_1},\ldots,\overline{y_t}\rangle$ 
with $e_j={\rm ord}(\overline{y_j})$, 
$l_2=[d_1,\ldots,d_s]$ is the abelian invariant of $H^n(H,\bZ)=\langle x_1,\ldots,x_s\rangle$ with $d_i={\rm ord}(x_i)$ 
and $l_3=[l_{3,1},\ldots,l_{3,t}]$, 
$l_{3,j}=[r_{j,1},\ldots,r_{j,s}]$ is the list with 
$\overline{y_j}=\overline{x_1}^{r_{j,1}}\cdots \overline{x_s}^{r_{j,s}}$ 
for free resolutions $RG$ and $RH$ of $G$ and $H$ respectively.\\

{\tt KerResH3Z($G,H$)} returns the list $L=[l_1,[l_2,l_3]]$ 
where $l_1$ is the abelian invariant of 
${\rm Ker}\{H^3(G,\bZ)\xrightarrow{\rm res}\oplus_{i=1}^{m^\prime} H^3(G_i,\bZ)\}=\langle y_1,\ldots,y_t\rangle$ 
where $H_i\leq G_i\leq G$, $H_i\leq H\cap G_i$, $[G_i:H_i]=n$ 
and 
the action of $G_i$ on $\bZ[G_i/H_i]$ may be regarded as $nTm$ 
$(n\leq 15, n\neq 12)$ which is not in Table $1$, 
$l_2=[d_1,\ldots,d_s]$ is the abelian invariant of $H^3(G,\bZ)=\langle x_1,\ldots,x_s\rangle$ with $d_{i^\prime}={\rm ord}(x_{i^\prime})$ 
and $l_3=[l_{3,1},\ldots,l_{3,t}]$, 
$l_{3,j}=[r_{j,1},\ldots,r_{j,s}]$ is the list with 
$y_j=x_1^{r_{j,1}}\cdots x_s^{r_{j,s}}$ 
for groups $G$ and $H$ (cf. Theorem \ref{thDra89}). \\ 

The functions above are available from\\
{\tt https://www.math.kyoto-u.ac.jp/\~{}yamasaki/Algorithm/Norm1ToriHNP}.\\

{\it Proof of Theorem \ref{thmain3}.} 

Let $G={\rm Gal}(L/k)=nTm\leq S_n$ be 
the $m$-th transitive subgroup of $S_n$ and 
$H={\rm Gal}(L/K)\leq G$ with $[G:H]=n$. 
Let $V_k$ be the set of all places of $k$ 
and $G_v$ be the decomposition group of $G$ at $v\in V_k$. 

We split the proof into the following cases:\\
(1) $G=8Tm$ $(m=2,3,4,13,14,21,31,37,38)$,\\
(2) $G=8Tm$ $(m=9,11,15,19,22,32)$,\\
(3) $G=9Tm$ $(m=2,5,7,9,11,14,23)$,\\
(4) $G=10Tm$ $(m=7,26,32)$,\\ 
(5) $G=14T30$,\\
(6) $G=15Tm$ $(m=9,14)$. 

For the reader's convenience, 
we also give the GAP computations to the known cases: 
$G=4T2\simeq V_4, 4T4\simeq A_4, 6T4\simeq A_4, 6T12\simeq A_5$ 
(see Example 6.9 and Example 6.10). 

In order to prove the statement of the theorem, 
we may assume that 
$H={\rm Stab}_1(G)$ is the stabilizer of $1$ in $G$, 
i.e. $L=k(\theta_1,\ldots,\theta_n)$ and $K=L^H=k(\theta_1)$, 
without loss of generality 
except for the cases 
(2) $G=8Tm$ $(m=9,11,15,19,22,32)$ and $G=10T32\simeq S_6$ 
because 
the center $Z(G)$ and the commutator group $[G^\prime,G^\prime]$ 
where $G^\prime\leq G$ is a characteristic subgroup 
in the statement of the theorem, are characteristic subgroups of $G$, i.e. 
invariants under the automorphisms of $G$. 

For the cases (2) $G=8Tm$, 
by the assumption of the statement of the theorem, 
we may assume that 
$H={\rm Stab}_1(G)$ is the stabilizer of $1$ in $G$ 
because (the multi-set) $\{{\rm Orb}_{G^\prime}(i)\mid 1\leq i\leq n\}$ 
$(G^\prime\leq G)$ 
is invariant under the conjugacy actions of $G$, 
i.e. inner automorphisms of $G$. 

For the case $G=10T32\simeq S_6$, 
there exist exactly $10$ subgroups 
$H\leq G$ with $[G:H]=10$ which are conjugate in $G$.
Hence we may assume that $H={\rm Stab}_1(G)$ without loss of generality. 

By Theorem \ref{thV} and Theorem \ref{thmain2}, 
it is enough to give a necessary and sufficient condition for $\Sha(T)=0$.\\

{\rm (1)} $n=8$: $G=8Tm$ $(m=2,3,4,13,14,21,31,37,38)$. 
Applying the functions 
{\tt FirstObstructionN($G$)} and 
{\tt FirstObstructionDnr($G$)}, 
we have ${\rm Obs}_1(L/K/k)=1$ except for 
$G=8T21\simeq (C_2)^3\rtimes C_4$. 
For $G=8T21$, 
we obtain that ${\rm Obs}_1(L/K/k)\simeq\bZ/2\bZ$. 

(1-1) The case $G=8T3\simeq (C_2)^3$. 
This case follows from Theorem \ref{thTate} because $H=1$. 
See also Example \ref{ex8} 
and the second paragraph after Theorem \ref{thTate}. 

(1-2) The case $G=8T21\simeq (C_2)^3\rtimes C_4$. 
We have $H=H^{ab}\simeq C_2\times C_2$. 
Applying 
{\tt FirstObstructionN($G$)} and 
{\tt FirstObstructionDnr($G$)}, 
we obtain that 
${\rm Ker}\, \psi_1/\varphi_1({\rm Ker}\, \psi_2^{\rm nr})$ $\simeq$ $\bZ/2\bZ$. 
By Theorem \ref{thmain2}, we get 
${\rm Obs}(K/k)={\rm Obs}_1(L/K/k)$ 
when $L/k$ is unramified (see the paragraph after Theorem \ref{thDP87}). 
Use 
Theorem \ref{thDP87}. 
Applying the function 
{\tt KerResH3Z(G,H)}, 
we see that ${\rm Ker}\{H^3(G,\bZ)\xrightarrow{\rm res}\oplus_{i=1}^{m^\prime} H^3(G_i,\bZ)\}=0$ and hence 
$\oplus_{i=1}^{m^\prime} \widehat{H}^{-3}(G_i,\bZ)\xrightarrow{\rm cores} 
\widehat{H}^{-3}(G,\bZ)$ is surjective.  
It follows from Theorem \ref{thDP87} that 
${\rm Obs}(K/k)={\rm Obs}_1(L/K/k)$. 
Applying the function {\tt FirstObstructionDr($G,G^\prime$)} 
for all subgroups $G^\prime\leq G$, 
we find that ${\rm Obs}_1(L/K/k)$ $=$ $1$ 
if and only if there exists $v\in V_k$ such that $G_v=G$ 
(see Example \ref{ex8}).

(1-3) The case $G=8Tm$ $(m=2,4,13,14,37)$. 
Because ${\rm Obs}_1(K/k)=1$, we just apply Theorem \ref{thDra89}. 
We have the Schur multiplier $M(G)\simeq\bZ/2\bZ$ 
for $G=8Tm$ $(m=2,4,13,14,37)$. 

(1-3-1) The case $G=8T2\simeq C_4\times C_2$ (see also Theorem \ref{thTate} because $H=1$). 
Apply Theorem \ref{thDra89}. 
We obtain a Schur cover 
$1\to M(G)\simeq\bZ/2\bZ\to \widetilde{G}\xrightarrow{\pi} G\to 1$ 
with $\widetilde{G}\simeq (C_4\times C_2)\rtimes C_2$, 
$\widetilde{H}\simeq C_2$ 
and ${\rm Obs}(K/k)={\rm Obs}_1(\widetilde{L}/K/k)$. 
By Theorem \ref{thmain2}, 
${\rm Ker}\, \widetilde{\psi}_1/\widetilde{\varphi}_1({\rm Ker}\, \widetilde{\psi}_2^{\rm nr})\simeq \bZ/2\bZ$ 
(see the paragraph after Theorem \ref{thDP87}). 
By applying 
{\tt FirstObstructionDr($\widetilde{G},\widetilde{G}^\prime,\widetilde{H}$)} 
for all subgroups $\widetilde{G}^\prime\leq \widetilde{G}$, 
we obtain that ${\rm Obs}_1(\widetilde{L}/K/k)=1$ 
if and only if there exists $v\in V_k$ such that $\widetilde{G}_v=\widetilde{G}$ if and only if there exists $v\in V_k$ such that $G_v=G$ 
(see Example \ref{ex8}).

(1-3-2) The case $G=8T4\simeq D_4$ 
(see also Theorem \ref{thTate} because $H=1$). 
Apply Theorem \ref{thDra89}. 
We obtain a Schur cover 
$1\to M(G)\simeq\bZ/2\bZ\to \widetilde{G}\xrightarrow{\pi} G\to 1$ 
with $\widetilde{G}\simeq D_8$, $\widetilde{H}\simeq C_2$ 
and ${\rm Obs}(K/k)={\rm Obs}_1(\widetilde{L}/K/k)$. 
By Theorem \ref{thmain2}, 
${\rm Ker}\, \widetilde{\psi}_1/\widetilde{\varphi}_1({\rm Ker}\, \widetilde{\psi}_2^{\rm nr})\simeq \bZ/2\bZ$. 
By applying 
{\tt FirstObstructionDr($\widetilde{G},\widetilde{G}^\prime,\widetilde{H}$)} 
for all subgroups $\widetilde{G}^\prime\leq \widetilde{G}$, 
we obtain that ${\rm Obs}_1(\widetilde{L}/K/k)=1$ 
if and only if there exists $v\in V_k$ such that $D_4\leq\widetilde{G}_v$ if and only if there exists $v\in V_k$ such that $V_4\leq G_v$ 
(see Example \ref{ex8}).

(1-3-3) The case $G=8T13\simeq A_4\times C_2$. 
We have $H\simeq C_3$. 
Apply Theorem \ref{thDra89}. 
We obtain a Schur cover 
$1\to M(G)\simeq\bZ/2\bZ\to \widetilde{G}\xrightarrow{\pi} G\to 1$ 
with $\widetilde{G}\simeq ((C_4\times C_2)\rtimes C_2)\rtimes C_3$, 
$\widetilde{H}\simeq C_6$ 
and ${\rm Obs}(K/k)={\rm Obs}_1(\widetilde{L}/K/k)$. 
By Theorem \ref{thmain2}, 
${\rm Ker}\, \widetilde{\psi}_1/\widetilde{\varphi}_1({\rm Ker}\, \widetilde{\psi}_2^{\rm nr})\simeq \bZ/2\bZ$. 
By applying 
{\tt FirstObstructionDr($\widetilde{G},\widetilde{G}^\prime,\widetilde{H}$)} 
for all subgroups $\widetilde{G}^\prime\leq \widetilde{G}$, 
we obtain that ${\rm Obs}_1(\widetilde{L}/K/k)=1$ 
if and only if there exists $v\in V_k$ such that $V_4\leq G_v$ 
(see Example \ref{ex8}).

(1-3-4) The case $G=8T14\simeq S_4$. 
We have $H\simeq C_3$. 
Apply Theorem \ref{thDra89}. 
We obtain a Schur cover 
$1\to M(G)\simeq\bZ/2\bZ\to \widetilde{G}\xrightarrow{\pi} G\to 1$ 
with $\widetilde{G}\simeq GL_2(\bF_3)$ 
and $\widetilde{H}\simeq C_6$. 
By Theorem \ref{thmain2}, 
${\rm Ker}\, \widetilde{\psi}_1/\widetilde{\varphi}_1({\rm Ker}\, \widetilde{\psi}_2^{\rm nr})\simeq \bZ/2\bZ$. 
By applying 
{\tt FirstObstructionDr($\widetilde{G},\widetilde{G}^\prime,\widetilde{H}$)} 
for all subgroups $\widetilde{G}^\prime\leq \widetilde{G}$, 
we obtain that ${\rm Obs}_1(\widetilde{L}/K/k)=1$ 
if and only if there exists $v\in V_k$ such that $V_4\leq G_v$ 
(see Example \ref{ex8}).

(1-3-5) The case $G=8T37\simeq \PSL_3(\bF_2)\simeq \PSL_2(\bF_7)$. 
We have $H\simeq C_7\rtimes C_3$. 
We obtain a Schur cover 
$1\to M(G)\simeq\bZ/2\bZ\to \widetilde{G}\xrightarrow{\pi} G\to 1$ 
with $\widetilde{G}\simeq \SL_2(\bF_7)$, 
$\widetilde{H}\simeq C_2\times (C_7\rtimes C_3)$ 
and ${\rm Obs}(K/k)={\rm Obs}_1(\widetilde{L}/K/k)$. 
By Theorem \ref{thmain2}, 
${\rm Ker}\, \widetilde{\psi}_1/\widetilde{\varphi}_1({\rm Ker}\, \widetilde{\psi}_2^{\rm nr})\simeq \bZ/2\bZ$. 
By applying 
{\tt FirstObstructionDr($\widetilde{G},\widetilde{G}^\prime,\widetilde{H}$)} 
for all subgroups $\widetilde{G}^\prime\leq \widetilde{G}$, 
we obtain that ${\rm Obs}_1(\widetilde{L}/K/k)=1$ 
if and only if there exists $v\in V_k$ such that $V_4\leq G_v$ 
(see Example \ref{ex8}).

(1-4) The case $G=8Tm$ $(m=31,38)$. 
Applying 
{\tt FirstObstructionN($G$)} and 
{\tt FirstObstructionDnr($G$)}, we have ${\rm Obs}_1(L/K/k)=1$. 
For $G=8T31$ (resp. $G=8T38$), 
we have $M(G)\simeq(\bZ/2\bZ)^{\oplus 4}$ 
(resp. $M(G)\simeq(\bZ/2\bZ)^{\oplus 2}$). 
Hence we take a minimal stem extension 
$\overline{G}=\widetilde{G}/A^\prime$, 
i.e. $\overline{A}\leq Z(\overline{G})\cap [\overline{G},\overline{G}]$, 
of $G$ in the commutative diagram 
\begin{align*}
\begin{CD}
1 @>>> A=M(G) @>>> \widetilde{G}@>\pi>> G@>>> 1\\
  @. @VVV @VVV @|\\
1 @>>> \overline{A}=A/A^\prime @>>> \overline{G}=\widetilde{G}/A^\prime @>\overline{\pi}>> G@>>> 1
\end{CD}
\end{align*}
with $\overline{A}\simeq\bZ/2\bZ$ 
via the function 
{\tt MinimalStemExtensions($G$)[$j$].MinimalStemExtension}. 
Then we apply Theorem \ref{thDP87} 
instead of Theorem \ref{thDra89}. 

(1-4-1) The case $G=8T31\simeq ((C_2)^4\rtimes C_2)\rtimes C_2$. 
We have $H\simeq (C_2)^3$ and $M(G)\simeq(\bZ/2\bZ)^{\oplus 4}$. 
Applying the function 
{\tt MinimalStemExtensions($G$)[$j$].MinimalStemExtension}, 
we get the minimal stem extensions 
$\overline{G}_1,\ldots,\overline{G}_{15}$ of $G$. 
Use 
Theorem \ref{thDP87}. 
Applying 
{\tt KerResH3Z(G,H)}, 
we see that ${\rm Ker}\{H^3(\overline{G}_1,\bZ)\xrightarrow{\rm res}\oplus_{i=1}^{m^\prime} H^3(G_i,\bZ)\}=0$ but 
${\rm Ker}\{H^3(\overline{G}_j,\bZ)\xrightarrow{\rm res}\oplus_{i=1}^{m^\prime} H^3(G_i,\bZ)\}\simeq\bZ/2\bZ$ for $j\in J:=\{j\mid 2\leq j\leq 15\}$. 
Because 
$\oplus_{i=1}^{m^\prime} \widehat{H}^{-3}(G_i,\bZ)\xrightarrow{\rm cores} 
\widehat{H}^{-3}(\overline{G}_1,\bZ)$ is surjective, 
it follows from Theorem \ref{thDP87} that 
${\rm Obs}(K/k)={\rm Obs}_1(\overline{L}_1/K/k)$. 
We also checked that 
${\rm Ker}\, \overline{\psi}_1/\overline{\varphi}_1({\rm Ker}\, \overline{\psi}_1^{\rm nr})$ $\simeq$ $\bZ/2\bZ$ for $\overline{G}_1$ and 
${\rm Ker}\, \overline{\psi}_1/\overline{\varphi}_1({\rm Ker}\, \overline{\psi}_2^{\rm nr})=0$ for $\overline{G}_j$ $(j\in J)$. 
This implies that 
${\rm Obs}(K/k)\neq {\rm Obs}_1(\overline{L}_j/K/k)$ 
when $\overline{L}_j/k$ is unramified 
for $j\in J$. 
Apply 
{\tt FirstObstructionDr($\overline{G}_1,\overline{G}_1^\prime,\overline{H}_1$)} 
for all subgroups $\overline{G}_1^\prime\leq \overline{G}_1$. 
We find that ${\rm Obs}_1(\overline{L}_1/K/k)$ $=$ $1$ 
if and only if 
there exists $v\in V_k$ such that 
{\rm (i)} 
$V_4\leq G_v$ where $V_4\cap [G,G]=1$  
$($equivalently, $|{\rm Orb}_{V_4}(i)|=4$ for any $1\leq i\leq 8$ 
and $V_4\cap Z(G)=1$$)$, 
{\rm (ii)} 
$C_4\times C_2\leq G_v$ where 
$(C_4\times C_2)\cap [G,G]\simeq C_2$ 
$($equivalently, 
$C_4\times C_2$ is transitive in $S_8$$)$ 
or {\rm (iii)} 
$(C_2)^3\rtimes C_4\leq G_v$ 
(see Details \ref{det6.8} and Example \ref{ex8}).

(1-4-2) The case $G=8T38\simeq (((C_2)^4\rtimes C_2)\rtimes C_2)\rtimes C_3$. 
We have $H\simeq C_2\times A_4$ and $M(G)\simeq(\bZ/2\bZ)^{\oplus 2}$. 
Applying the function 
{\tt MinimalStemExtensions($G$)[$j$].MinimalStemExtension}, 
we get the minimal stem extensions 
$\overline{G}_1,\overline{G}_2,\overline{G}_{3}$ of $G$. 
Use 
Theorem \ref{thDP87}. 
Applying 
{\tt KerResH3Z(G,H)}, 
we see that ${\rm Ker}\{H^3(\overline{G}_{2},\bZ)\xrightarrow{\rm res}\oplus_{i=1}^{m^\prime} H^3(G_i,\bZ)\}=0$ but 
${\rm Ker}\{H^3(\overline{G}_j,\bZ)\xrightarrow{\rm res}\oplus_{i=1}^{m^\prime} H^3(G_i,\bZ)\}\simeq\bZ/2\bZ$ for $j\in J:=\{1,3\}$. 
We have that 
$\oplus_{i=1}^{m^\prime} \widehat{H}^{-3}(G_i,\bZ)\xrightarrow{\rm cores} 
\widehat{H}^{-3}(\overline{G}_{2},\bZ)$ is surjective. 
By Theorem \ref{thDP87}, 
${\rm Obs}(K/k)={\rm Obs}_1(\overline{L}_{2}/K/k)$. 
We also checked that 
${\rm Ker}\, \overline{\psi}_1/\overline{\varphi}_1({\rm Ker}\, \overline{\psi}_1^{\rm nr})$ $\simeq$ $\bZ/2\bZ$ for $\overline{G}_{2}$ and 
${\rm Ker}\, \overline{\psi}_1/\overline{\varphi}_1({\rm Ker}\, \overline{\psi}_2^{\rm nr})=0$ for $\overline{G}_j$ $(j\in J)$. 
This implies that 
${\rm Obs}(K/k)\neq {\rm Obs}_1(\overline{L}_j/K/k)$ 
when $\overline{L}_j/k$ is unramified 
for $j\in J$. 
Apply 
{\tt FirstObstructionDr($\overline{G}_2,\overline{G}_2^\prime,\overline{H}_2$)} 
for all subgroups $\overline{G}_2^\prime\leq \overline{G}_2$. 
We find that ${\rm Obs}_1(\overline{L}_2/K/k)$ $=$ $1$ 
if and only if 
there exists $v\in V_k$ such that 
{\rm (i)} 
$V_4\leq G_v$ where $V_4\cap [{\rm Syl}_2(G),{\rm Syl}_2(G)]=1$ 
with ${\rm Syl}_2(G)\lhd G$ 
$($equivalently, $|{\rm Orb}_{V_4}(i)|=4$ for any $1\leq i\leq 8$ 
and $V_4\cap Z(G)=1$$)$, 
{\rm (ii)} 
$C_4\times C_2\leq G_v$ where 
$(C_4\times C_2)\cap [{\rm Syl}_2(G),{\rm Syl}_2(G)]\simeq C_2$ 
$($equivalently, 
$C_4\times C_2$ is transitive in $S_8$$)$ 
or {\rm (iii)} 
$(C_2)^3\rtimes C_4\leq G_v$ 
(see Details \ref{det6.8} and Example \ref{ex8}).\\

{\rm (2) $n=8$: $G=8Tm$ $(m=9,11,15,19,22,32)$}. 
We assume that $H={\rm Stab}_1(G)$ by the assumption. 
Applying {\tt FirstObstructionN($G$)}, 
we have ${\rm Obs}_1(L/K/k)=1$. 
For the cases 
$G=8Tm$ $(m=9,11,15,19,22,32)$, we also find that 
$(\bZ/2\bZ)^{\oplus 2}\leq M(G)\leq (\bZ/2\bZ)^{\oplus 5}$. 
Hence we take a minimal stem extension 
$\overline{G}=\widetilde{G}/A^\prime$, 
i.e. $\overline{A}\leq Z(\overline{G})\cap [\overline{G},\overline{G}]$, 
of $G$ in the commutative diagram 
\begin{align*}
\begin{CD}
1 @>>> A=M(G) @>>> \widetilde{G}@>\pi>> G@>>> 1\\
  @. @VVV @VVV @|\\
1 @>>> \overline{A}=A/A^\prime @>>> \overline{G}=\widetilde{G}/A^\prime @>\overline{\pi}>> G@>>> 1
\end{CD}
\end{align*}
with $\overline{A}\simeq\bZ/2\bZ$ 
via the function 
{\tt MinimalStemExtensions($G$)[$j$].MinimalStemExtension}. 
Then we apply Theorem \ref{thDP87} 
instead of Theorem \ref{thDra89} as in the case (1-4). 

(2-1) The case $G=8T9\simeq D_4\times C_2$. 
We have $H\simeq C_2$. 
We obtain that the Schur multiplier $M(G)\simeq(\bZ/2\bZ)^{\oplus 3}$. 
By applying {\tt MinimalStemExtensions($G$)[$j$].MinimalStemExtension}, 
we obtain the minimal stem extensions 
$\overline{G}_1,\ldots,\overline{G}_7$ of $G$. 
Use 
Theorem \ref{thDP87}. 
Applying 
{\tt KerResH3Z(G,H)}, 
we see that ${\rm Ker}\{H^3(\overline{G}_2,\bZ)\xrightarrow{\rm res}\oplus_{i=1}^{m^\prime} H^3(G_i,\bZ)\}=0$ but 
${\rm Ker}\{H^3(\overline{G}_j,\bZ)\xrightarrow{\rm res}\oplus_{i=1}^{m^\prime} H^3(G_i,\bZ)\}=\bZ/2\bZ$ for $j\in J:=\{1,3,4,5,6,7\}$. 
Because 
$\oplus_{i=1}^{m^\prime} \widehat{H}^{-3}(G_i,\bZ)\xrightarrow{\rm cores} 
\widehat{H}^{-3}(\overline{G}_2,\bZ)$ is surjective, 
it follows from Theorem \ref{thDP87} that 
${\rm Obs}(K/k)={\rm Obs}_1(\overline{L}_2/K/k)$. 
We also checked that 
${\rm Ker}\, \overline{\psi}_1/\overline{\varphi}_1({\rm Ker}\, \overline{\psi}_2^{\rm nr})\simeq\bZ/2\bZ$ for $\overline{G}_2$ and 
${\rm Ker}\, \overline{\psi}_1/\overline{\varphi}_1({\rm Ker}\, \overline{\psi}_2^{\rm nr})=0$ for $\overline{G}_j$ $(j\in J)$. 
Hence 
${\rm Obs}(K/k)\neq {\rm Obs}_1(\overline{L}_j/K/k)$ 
when $\overline{L}_j/k$ is unramified 
for $j\in J$. 
Apply 
{\tt FirstObstructionDr($\overline{G}_2,\overline{G}_2^\prime,\overline{H}_2$)} 
for all subgroups $\overline{G}_2^\prime\leq \overline{G}_2$. 
We find that ${\rm Obs}_1(\overline{L}_2/K/k)$ $=$ $1$ 
if and only if 
there exists $v\in V_k$ such that 
{\rm (i)} $V_4\leq G_v$ where 
$|{\rm Orb}_{V_4}(i)|=4$ for any $1\leq i\leq 8$ 
and $V_4\cap [G,G]=1$; 
or {\rm (ii)} 
$C_4\times C_2\leq G_v$ 
(see Details \ref{det6.8} and Example \ref{ex8-2}).

(2-2) The case $G=8T11\simeq (C_4\times C_2)\rtimes C_2$.
We have $H\simeq C_2$ and $M(G)\simeq(\bZ/2\bZ)^{\oplus 3}$. 
Applying 
{\tt MinimalStemExtensions($G$)[$j$].MinimalStemExtension}. 
We get the minimal stem extensions 
$\overline{G}_1,\overline{G}_2,\overline{G}_3$ of $G$. 
Use 
Theorem \ref{thDP87}. 
Applying 
{\tt KerResH3Z(G,H)}, 
we see that ${\rm Ker}\{H^3(\overline{G}_2,\bZ)\xrightarrow{\rm res}\oplus_{i=1}^{m^\prime} H^3(G_i,\bZ)\}=0$ but 
${\rm Ker}\{H^3(\overline{G}_j,\bZ)\xrightarrow{\rm res}\oplus_{i=1}^{m^\prime} H^3(G_i,\bZ)\}=\bZ/2\bZ$ for $j\in J:=\{1,3\}$. 
Because 
$\oplus_{i=1}^{m^\prime} \widehat{H}^{-3}(G_i,\bZ)\xrightarrow{\rm cores} 
\widehat{H}^{-3}(\overline{G}_2,\bZ)$ is surjective, 
it follows from Theorem \ref{thDP87} that 
${\rm Obs}(K/k)={\rm Obs}_1(\overline{L}_2/K/k)$. 
We also checked that 
${\rm Ker}\, \overline{\psi}_1/\overline{\varphi}_1({\rm Ker}\, \overline{\psi}_2^{\rm nr})$ $\simeq$ $\bZ/2\bZ$ for $\overline{G}_2$ and 
${\rm Ker}\, \overline{\psi}_1/\overline{\varphi}_1({\rm Ker}\, \overline{\psi}_2^{\rm nr})=0$ for $\overline{G}_j$ $(j\in J)$. 
This implies that 
${\rm Obs}(K/k)\neq {\rm Obs}_1(\overline{L}_j/K/k)$ 
when $\overline{L}_j/k$ is unramified 
for $j\in J$. 
Apply 
{\tt FirstObstructionDr($\overline{G}_2,\overline{G}_2^\prime,\overline{H}_2$)} 
for all subgroups $\overline{G}_2^\prime\leq \overline{G}_2$. 
We find that ${\rm Obs}_1(\overline{L}_2/K/k)$ $=$ $1$ 
if and only if 
there exists $v\in V_k$ such that 
$C_4\times C_2\leq G_v$ where 
$C_4\times C_2$ is transitive in $S_8$ 
(see Details \ref{det6.8} and Example \ref{ex8-2}).


(2-3) The case $G=8T15\simeq C_8\rtimes V_4$.
We have $H\simeq V_4$ and $M(G)\simeq(\bZ/2\bZ)^{\oplus 2}$. 
Applying the function 
{\tt MinimalStemExtensions($G$)[$j$].MinimalStemExtension}, 
we get the minimal stem extensions 
$\overline{G}_1,\overline{G}_2,\overline{G}_3$ of $G$. 
Use 
Theorem \ref{thDP87}. 
Applying 
{\tt KerResH3Z(G,H)}, 
we see that ${\rm Ker}\{H^3(\overline{G}_1,\bZ)\xrightarrow{\rm res}\oplus_{i=1}^{m^\prime} H^3(G_i,\bZ)\}=0$ but 
${\rm Ker}\{H^3(\overline{G}_j,\bZ)\xrightarrow{\rm res}\oplus_{i=1}^{m^\prime} H^3(G_i,\bZ)\}=\bZ/2\bZ$ for $j\in J:=\{2,3\}$. 
Because 
$\oplus_{i=1}^{m^\prime} \widehat{H}^{-3}(G_i,\bZ)\xrightarrow{\rm cores} 
\widehat{H}^{-3}(\overline{G}_1,\bZ)$ is surjective, 
it follows from Theorem \ref{thDP87} that 
${\rm Obs}(K/k)={\rm Obs}_1(\overline{L}_1/K/k)$. 
We also checked that 
${\rm Ker}\, \overline{\psi}_1/\overline{\varphi}_1({\rm Ker}\, \overline{\psi}_1^{\rm nr})$ $\simeq$ $\bZ/2\bZ$ for $\overline{G}_1$ and 
${\rm Ker}\, \overline{\psi}_1/\overline{\varphi}_1({\rm Ker}\, \overline{\psi}_2^{\rm nr})=0$ for $\overline{G}_j$ $(j\in J)$. 
This implies that 
${\rm Obs}(K/k)\neq {\rm Obs}_1(\overline{L}_j/K/k)$ 
when $\overline{L}_j/k$ is unramified 
for $j\in J$. 
Apply 
{\tt FirstObstructionDr($\overline{G}_1,\overline{G}_1^\prime,\overline{H}_1$)} 
for all subgroups $\overline{G}_1^\prime\leq \overline{G}_1$. 
We find that ${\rm Obs}_1(\overline{L}_1/K/k)$ $=$ $1$ 
if and only if 
there exists $v\in V_k$ such that 
{\rm (i)} 
$V_4\leq G_v$ where 
$|{\rm Orb}_{V_4}(i)|=2$ for any $1\leq i\leq 8$ 
and $V_4\cap [G,G]=1$ 
$($equivalently, 
$|{\rm Orb}_{V_4}(i)|=2$ for any $1\leq i\leq 8$ 
and $V_4$ is not in $A_8$$)$ 
or {\rm (ii)} 
$C_4\times C_2\leq G_v$ where 
$(C_4\times C_2)\cap [G,G]\simeq C_2$ 
$($equivalently, 
$C_4\times C_2$ is transitive in $S_8$$)$ 
(see Details \ref{det6.8} and Example \ref{ex8-2}).

(2-4) The case $G=8T19\simeq (C_2)^3\rtimes C_4$.
We have $H\simeq C_4$ and $M(G)\simeq(\bZ/2\bZ)^{\oplus 2}$. 
Applying the function 
{\tt MinimalStemExtensions($G$)[$j$].MinimalStemExtension}, 
we get the minimal stem extensions 
$\overline{G}_1,\overline{G}_2,\overline{G}_3$ of $G$. 
Use 
Theorem \ref{thDP87}. 
Applying 
{\tt KerResH3Z(G,H)}, 
we see that ${\rm Ker}\{H^3(\overline{G}_3,\bZ)\xrightarrow{\rm res}\oplus_{i=1}^{m^\prime} H^3(G_i,\bZ)\}=0$ but 
${\rm Ker}\{H^3(\overline{G}_j,\bZ)\xrightarrow{\rm res}\oplus_{i=1}^{m^\prime} H^3(G_i,\bZ)\}=\bZ/2\bZ$ for $j\in J:=\{1,2\}$. 
Because 
$\oplus_{i=1}^{m^\prime} \widehat{H}^{-3}(G_i,\bZ)\xrightarrow{\rm cores} 
\widehat{H}^{-3}(\overline{G}_3,\bZ)$ is surjective, 
it follows from Theorem \ref{thDP87} that 
${\rm Obs}(K/k)={\rm Obs}_1(\overline{L}_3/K/k)$. 
We also checked that 
${\rm Ker}\, \overline{\psi}_1/\overline{\varphi}_1({\rm Ker}\, \overline{\psi}_1^{\rm nr})$ $\simeq$ $\bZ/2\bZ$ for $\overline{G}_3$ and 
${\rm Ker}\, \overline{\psi}_1/\overline{\varphi}_1({\rm Ker}\, \overline{\psi}_2^{\rm nr})=0$ for $\overline{G}_j$ $(j\in J)$. 
This implies that 
${\rm Obs}(K/k)\neq {\rm Obs}_1(\overline{L}_j/K/k)$ 
when $\overline{L}_j/k$ is unramified 
for $j\in J$. 
Apply 
{\tt FirstObstructionDr($\overline{G}_3,\overline{G}_3^\prime,\overline{H}_3$)} 
for all subgroups $\overline{G}_3^\prime\leq \overline{G}_3$. 
We find that ${\rm Obs}_1(\overline{L}_3/K/k)$ $=$ $1$ 
if and only if 
there exists $v\in V_k$ such that 
{\rm (i)} 
$V_4\leq G_v$ where 
$V_4\cap Z(G)=1$ and 
$V_4\cap Z^2(G)\simeq C_2$ 
with the upper central series $1\leq Z(G)\leq Z^2(G)\leq G$ of $G$ 
$($equivalently, 
$|{\rm Orb}_{V_4}(i)|=4$ for any $1\leq i\leq 8$ 
and $V_4\cap Z(G)=1$$)$; 
or {\rm (ii)} 
$C_4\times C_2\leq G_v$
where 
$C_4\times C_2$ is not transitive in $S_8$ 
or 
$[G,G]\leq C_4\times C_2$ 
(see Details \ref{det6.8} and Example \ref{ex8-2}).

(2-5) The case $G=8T22\simeq (C_2)^3\rtimes V_4$.
We have $H\simeq V_4$ and $M(G)\simeq(\bZ/2\bZ)^{\oplus 5}$. 
Applying the function 
{\tt MinimalStemExtensions($G$)[$j$].MinimalStemExtension}, 
we get the minimal stem extensions 
$\overline{G}_1,\ldots,\overline{G}_{31}$ of $G$. 
Use 
Theorem \ref{thDP87}. 
Applying 
{\tt KerResH3Z(G,H)}, 
we see that ${\rm Ker}\{H^3(\overline{G}_{16},\bZ)\xrightarrow{\rm res}\oplus_{i=1}^{m^\prime} H^3(G_i,\bZ)\}=0$ but 
${\rm Ker}\{H^3(\overline{G}_j,\bZ)\xrightarrow{\rm res}\oplus_{i=1}^{m^\prime} H^3(G_i,\bZ)\}\neq 0$ for $j\in J:=\{j\mid 1\leq j\leq 31, j\neq 16\}$. 
Because 
$\oplus_{i=1}^{m^\prime} \widehat{H}^{-3}(G_i,\bZ)\xrightarrow{\rm cores} 
\widehat{H}^{-3}(\overline{G}_{16},\bZ)$ is surjective, 
it follows from Theorem \ref{thDP87} that 
${\rm Obs}(K/k)={\rm Obs}_1(\overline{L}_{16}/K/k)$. 
We also checked that 
${\rm Ker}\, \overline{\psi}_1/\overline{\varphi}_1({\rm Ker}\, \overline{\psi}_1^{\rm nr})$ $\simeq$ $\bZ/2\bZ$ for $\overline{G}_{16}$ and 
${\rm Ker}\, \overline{\psi}_1/\overline{\varphi}_1({\rm Ker}\, \overline{\psi}_2^{\rm nr})=0$ for $\overline{G}_j$ $(j\in J)$. 
This implies that 
${\rm Obs}(K/k)\neq {\rm Obs}_1(\overline{L}_j/K/k)$ 
when $\overline{L}_j/k$ is unramified 
for $j\in J$. 
Apply 
{\tt FirstObstructionDr($\overline{G}_{16},\overline{G}_{16}^\prime,\overline{H}_{16}$)} 
for all subgroups $\overline{G}_{16}^\prime\leq \overline{G}_{16}$. 
We find that ${\rm Obs}_1(\overline{L}_{16}/K/k)$ $=$ $1$ 
if and only if 
there exists $v\in V_k$ such that 
{\rm (i)} 
$V_4\leq G_v$ where 
$|{\rm Orb}_{V_4}(i)|=4$ for any $1\leq i\leq 8$ 
and $V_4\cap Z(G)=1$ 
or {\rm (ii)} 
$C_4\times C_2\leq G_v$ where 
$C_4\times C_2$ is transitive in $S_8$. 
(see Details \ref{det6.8} and Example \ref{ex8-2}).

(2-6) The case $G=8T32\simeq ((C_2)^3\rtimes V_4)\rtimes C_3$. 
We have $H\simeq A_4$ and $M(G)\simeq(\bZ/2\bZ)^{\oplus 3}$. 
Applying the function 
{\tt MinimalStemExtensions($G$)[$j$].MinimalStemExtension}, 
we get the minimal stem extensions 
$\overline{G}_1,\ldots,\overline{G}_{7}$ of $G$. 
Use 
Theorem \ref{thDP87}. 
Applying 
{\tt KerResH3Z(G,H)}, 
we see that ${\rm Ker}\{H^3(\overline{G}_{5},\bZ)\xrightarrow{\rm res}\oplus_{i=1}^{m^\prime} H^3(G_i,\bZ)\}=0$ but 
${\rm Ker}\{H^3(\overline{G}_j,\bZ)\xrightarrow{\rm res}\oplus_{i=1}^{m^\prime} H^3(G_i,\bZ)\}\simeq\bZ/2\bZ$ for $j\in J:=\{j\mid 2\leq j\leq 7\}$. 
Because 
$\oplus_{i=1}^{m^\prime} \widehat{H}^{-3}(G_i,\bZ)\xrightarrow{\rm cores} 
\widehat{H}^{-3}(\overline{G}_{5},\bZ)$ is surjective, 
it follows from Theorem \ref{thDP87} that 
${\rm Obs}(K/k)={\rm Obs}_1(\overline{L}_{5}/K/k)$. 
We also checked that 
${\rm Ker}\, \overline{\psi}_1/\overline{\varphi}_1({\rm Ker}\, \overline{\psi}_1^{\rm nr})$ $\simeq$ $\bZ/2\bZ$ for $\overline{G}_{5}$ and 
${\rm Ker}\, \overline{\psi}_1/\overline{\varphi}_1({\rm Ker}\, \overline{\psi}_2^{\rm nr})=0$ for $\overline{G}_j$ $(j\in J)$. 
This implies that 
${\rm Obs}(K/k)\neq {\rm Obs}_1(\overline{L}_j/K/k)$ 
when $\overline{L}_j/k$ is unramified 
for $j\in J$. 
Apply 
{\tt FirstObstructionDr($\overline{G}_1,\overline{G}_1^\prime,\overline{H}_1$)} 
for all subgroups $\overline{G}_1^\prime\leq \overline{G}_1$. 
We find that ${\rm Obs}_1(\overline{L}_1/K/k)$ $=$ $1$ 
if and only if 
there exists $v\in V_k$ such that 
{\rm (i)} 
$V_4\leq G_v$ where 
$|{\rm Orb}_{V_4}(i)|=4$ for any $1\leq i\leq 8$ 
and $V_4\cap Z(G)=1$ 
or {\rm (ii)} 
$C_4\times C_2\leq G_v$ where 
$C_4\times C_2$ is transitive in $S_8$ 
(see Details \ref{det6.8} and Example \ref{ex8-2}).

{\rm (3)} $n=9$: $G=9Tm$ $(m=2,5,7,9,11,14,23)$. 
Applying {\tt FirstObstructionN($G$)}, 
we have ${\rm Obs}_1(L/K/k)=1$ for each cases. 
We apply Theorem \ref{thDra89} for 
$m=2,5,9,11,14,23$. 
We see that $M(G)\simeq \bZ/3\bZ$ for $m=2,5,9,11,14,23$ 
and $M(G)\simeq (\bZ/3\bZ)^{\oplus 2}$ for $m=7$. 
Then for $m=7$ we apply Theorem \ref{thDP87} 
instead of Theorem \ref{thDra89} as in the case (2) $n=8$. 

(3-1) The case $G=9T2\simeq (C_3)^2$ (see also Theorem \ref{thTate} because $H=1$). 
Apply Theorem \ref{thDra89}. 
We obtain a Schur cover 
$1\to M(G)\simeq\bZ/3\bZ\to \widetilde{G}\xrightarrow{\pi} G\to 1$ 
with $\widetilde{G}\simeq (C_3)^2\rtimes C_3$ and $\widetilde{H}\simeq C_3$
and ${\rm Obs}(K/k)={\rm Obs}_1(\widetilde{L}/K/k)$. 
By Theorem \ref{thmain2}, 
${\rm Ker}\, \widetilde{\psi}_1/\widetilde{\varphi}_1({\rm Ker}\, \widetilde{\psi}_2^{\rm nr})\simeq \bZ/3\bZ$. 
By applying 
{\tt FirstObstructionDr($\widetilde{G},\widetilde{G}^\prime,\widetilde{H}$)} 
for all subgroups $\widetilde{G}^\prime\leq \widetilde{G}$, 
we obtain that ${\rm Obs}_1(\widetilde{L}/K/k)=1$ 
if and only if there exists $v\in V_k$ such that $\widetilde{G}_v=\widetilde{G}$ if and only if there exists $v\in V_k$ such that $G_v=G$ 
(see Example \ref{ex9}).  

(3-2) The case $G=9T5\simeq (C_3)^2\rtimes C_2$.
We have $H\simeq C_2$.  
Apply Theorem \ref{thDra89} as in the case (2) $n=8$. 
We obtain a Schur cover 
$1\to M(G)\simeq\bZ/3\bZ\to \widetilde{G}\xrightarrow{\pi} G\to 1$ 
with $\widetilde{G}\simeq ((C_3)^2\rtimes C_3)\rtimes C_2$ and $\widetilde{H}\simeq C_6$ 
and ${\rm Obs}(K/k)={\rm Obs}_1(\widetilde{L}/K/k)$. 
By Theorem \ref{thmain2}, 
${\rm Ker}\, \widetilde{\psi}_1/\widetilde{\varphi}_1({\rm Ker}\, \widetilde{\psi}_2^{\rm nr})\simeq \bZ/3\bZ$. 
By applying 
{\tt FirstObstructionDr($\widetilde{G},\widetilde{G}^\prime,\widetilde{H}$)} 
for all subgroups $\widetilde{G}^\prime\leq \widetilde{G}$, 
we obtain that ${\rm Obs}_1(\widetilde{L}/K/k)=1$ 
if and only if there exists $v\in V_k$ such that $(C_3)^2\rtimes C_3\leq \widetilde{G}_v$ if and only if there exists $v\in V_k$ such that $(C_3)^2\leq G_v$ 
(see Example \ref{ex9}). 

(3-3) The case $G=9T7\simeq (C_3)^2\rtimes C_3$. 
We have $H\simeq C_3$ and $M(G)\simeq(\bZ/3\bZ)^{\oplus 2}$. 
Applying the function 
{\tt MinimalStemExtensions($G$)[$j$].MinimalStemExtension}, 
we get the minimal stem extensions 
$\overline{G}_1,\ldots,\overline{G}_{4}$ of $G$. 
Use 
Theorem \ref{thDP87}. 
Applying 
{\tt KerResH3Z(G,H)}, 
we obtain that ${\rm Ker}\{H^3(\overline{G}_{1},\bZ)\xrightarrow{\rm res}\oplus_{i=1}^{m^\prime} H^3(G_i,\bZ)\}=0$ but 
${\rm Ker}\{H^3(\overline{G}_j,\bZ)\xrightarrow{\rm res}\oplus_{i=1}^{m^\prime} H^3(G_i,\bZ)\}\simeq\bZ/3\bZ$ for $j\in J:=\{2,3,4\}$. 
Because 
$\oplus_{i=1}^{m^\prime} \widehat{H}^{-3}(G_i,\bZ)\xrightarrow{\rm cores} 
\widehat{H}^{-3}(\overline{G}_{1},\bZ)$ is surjective, 
it follows from Theorem \ref{thDP87} that 
${\rm Obs}(K/k)={\rm Obs}_1(\overline{L}_{1}/K/k)$. 
We also have  
${\rm Ker}\, \overline{\psi}_1/\overline{\varphi}_1({\rm Ker}\, \overline{\psi}_1^{\rm nr})$ $\simeq$ $\bZ/2\bZ$ for $\overline{G}_{1}$ and 
${\rm Ker}\, \overline{\psi}_1/\overline{\varphi}_1({\rm Ker}\, \overline{\psi}_2^{\rm nr})=0$ for $\overline{G}_j$ $(j\in J)$. 
This implies that 
${\rm Obs}(K/k)\neq {\rm Obs}_1(\overline{L}_j/K/k)$ 
when $\overline{L}_j/k$ is unramified 
for $j\in J$. 
Apply 
{\tt FirstObstructionDr($\overline{G}_1,\overline{G}_1^\prime,\overline{H}_1$)} 
for all subgroups $\overline{G}_1^\prime\leq \overline{G}_1$. 
We find that ${\rm Obs}_1(\overline{L}_1/K/k)$ $=$ $1$ 
if and only if 
there exists $v\in V_k$ such that 
$(C_3)^2\leq \overline{G}_v$ 
if and only if there exists $v\in V_k$ such that $(C_3)^2\leq G_v$ 
(see Example \ref{ex9}). 

(3-4) The case $G=9T9\simeq (C_3)^2\rtimes C_4$.
We have $H\simeq C_4$. 
Apply Theorem \ref{thDra89}. 
We obtain a Schur cover 
$1\to M(G)\simeq\bZ/3\bZ\to \widetilde{G}\xrightarrow{\pi} G\to 1$ 
with $\widetilde{G}\simeq ((C_3)^2\rtimes C_3)\rtimes C_4$, 
$\widetilde{H}\simeq C_{12}$ 
and ${\rm Obs}(K/k)={\rm Obs}_1(\widetilde{L}/K/k)$. 
By Theorem \ref{thmain2}, 
${\rm Ker}\, \widetilde{\psi}_1/\widetilde{\varphi}_1({\rm Ker}\, \widetilde{\psi}_2^{\rm nr})\simeq \bZ/3\bZ$. 
By applying 
{\tt FirstObstructionDr($\widetilde{G},\widetilde{G}^\prime,\widetilde{H}$)} 
for all subgroups $\widetilde{G}^\prime\leq \widetilde{G}$, 
we obtain that ${\rm Obs}_1(\widetilde{L}/K/k)=1$ 
if and only if there exists $v\in V_k$ such that $(C_3)^2\rtimes C_3\leq \widetilde{G}_v$ if and only if there exists $v\in V_k$ such that $(C_3)^2\leq G_v$ 
(see Example \ref{ex9}). 

(3-5) The case $G=9T11\simeq (C_3)^2\rtimes C_6$.
We have $H\simeq C_6$. 
Apply Theorem \ref{thDra89}. 
We obtain a Schur cover 
$1\to M(G)\simeq\bZ/3\bZ\to \widetilde{G}\xrightarrow{\pi} G\to 1$ 
with $\widetilde{G}\simeq ((C_3)^3\rtimes C_3)\rtimes C_2$, 
$\widetilde{H}\simeq C_6\times C_3$ 
and ${\rm Obs}(K/k)={\rm Obs}_1(\widetilde{L}/K/k)$. 
By Theorem \ref{thmain2}, 
${\rm Ker}\, \widetilde{\psi}_1/\widetilde{\varphi}_1({\rm Ker}\, \widetilde{\psi}_2^{\rm nr})\simeq \bZ/3\bZ$. 
By applying 
{\tt FirstObstructionDr($\widetilde{G},\widetilde{G}^\prime,\widetilde{H}$)} 
for all subgroups $\widetilde{G}^\prime\leq \widetilde{G}$, 
we obtain that ${\rm Obs}_1(\widetilde{L}/K/k)=1$ 
if and only if 
there exists $v\in V_k$ such that $(C_3)^2\leq G_v$ 
(see Example \ref{ex9}). 

(3-6) The case $G=9T14\simeq (C_3)^2\rtimes Q_8$. 
We have $H\simeq Q_8$. 
Apply Theorem \ref{thDra89}. 
We obtain a Schur cover 
$1\to M(G)\simeq\bZ/3\bZ\to \widetilde{G}\xrightarrow{\pi} G\to 1$ 
with $\widetilde{G}\simeq (((C_3)^2\rtimes C_3)\rtimes Q_8)\rtimes C_3$, 
$\widetilde{H}\simeq C_3\times \SL_2(\bF_3)$ 
and ${\rm Obs}(K/k)={\rm Obs}_1(\widetilde{L}/K/k)$. 
By Theorem \ref{thmain2}, 
${\rm Ker}\, \widetilde{\psi}_1/\widetilde{\varphi}_1({\rm Ker}\, \widetilde{\psi}_2^{\rm nr})\simeq \bZ/3\bZ$. 
By applying 
{\tt FirstObstructionDr($\widetilde{G},\widetilde{G}^\prime,\widetilde{H}$)} 
for all subgroups $\widetilde{G}^\prime\leq \widetilde{G}$, 
we obtain that ${\rm Obs}_1(\widetilde{L}/K/k)=1$ 
if and only if 
there exists $v\in V_k$ such that $(C_3)^2\rtimes C_3\leq \widetilde{G}_v$ 
if and only if 
there exists $v\in V_k$ such that $(C_3)^2\leq G_v$ 
(see Example \ref{ex9}). 

(3-7) The case $G=9T23\simeq ((C_3)^2\rtimes Q_8)\rtimes C_3$. 
We have $H\simeq \SL_2(\bF_3)$. 
Apply Theorem \ref{thDra89}. 
We obtain a Schur cover 
$1\to M(G)\simeq\bZ/3\bZ\to \widetilde{G}\xrightarrow{\pi} G\to 1$ 
with $\widetilde{G}\simeq (((C_3)^2\rtimes C_3)\rtimes Q_8)\rtimes C_3$, 
$\widetilde{H}\simeq C_3\times \SL_2(\bF_3)$ 
and ${\rm Obs}(K/k)={\rm Obs}_1(\widetilde{L}/K/k)$. 
By Theorem \ref{thmain2}, 
${\rm Ker}\, \widetilde{\psi}_1/\widetilde{\varphi}_1({\rm Ker}\, \widetilde{\psi}_2^{\rm nr})\simeq \bZ/3\bZ$. 
By applying 
{\tt FirstObstructionDr($\widetilde{G},\widetilde{G}^\prime,\widetilde{H}$)} 
for all subgroups $\widetilde{G}^\prime\leq \widetilde{G}$, 
we obtain that ${\rm Obs}_1(\widetilde{L}/K/k)=1$ 
if and only if 
there exists $v\in V_k$ such that $(C_3)^2\leq G_v$ 
(see Example \ref{ex9}). \\

{\rm (4)} $n=10$: $G\simeq 10T7\simeq A_5$, 
$G\simeq 10T26\simeq A_6$ and 
$G\simeq 10T32\simeq S_6$. 
By Theorem \ref{thDP87}, we have 
${\rm Obs}(K/k)={\rm Obs}_1(L/K/k)$. 

{\rm (4-1)} The case $G=10T7\simeq A_5$.  
We have $H\simeq S_3$ and $H^{ab}\simeq C_2$. 
It follows from Theorem \ref{thmain2} that 
${\rm Ker}\, \psi_1\simeq {\rm Ker}\, \psi_1/\varphi_1({\rm Ker}\, \psi_2^{\rm nr})\simeq\bZ/2\bZ$. 
Apply {\tt FirstObstructionDr($G,G^\prime$)} for all subgroups 
$G^\prime\leq G$. 
We get that ${\rm Obs}_1(L/K/k)=1$ if and only if 
there exists $v\in V_k$ such that 
$V_4\leq G_v$ 
(see Example \ref{ex10}).

{\rm (4-2)} The case $G=10T26\simeq A_6$.   
We have $H\simeq (C_3)^2\rtimes C_4$ and $H^{ab}\simeq C_4$. 
Applying {\tt FirstObstructionN($G$)} 
and {\tt FirstObstructionDnr($G$)}, 
we obtain that 
${\rm Ker}\, \psi_1\simeq\bZ/4\bZ$ and 
${\rm Ker}\, \psi_1/\varphi_1({\rm Ker}\, \psi_2^{\rm nr})\simeq\bZ/2\bZ$. 
Apply {\tt FirstObstructionDr($G,G^\prime$)} for all subgroups 
$G^\prime\leq G$. 
We get that ${\rm Obs}_1(L/K/k)=1$ if and only if 
there exists $v\in V_k$ such that 
$D_4\leq G_v$ 
(see Example \ref{ex10}). 

{\rm (4-3)} The case $G=10T32\simeq S_6$.   
We have $H\simeq (S_3)^2\rtimes C_2$ and $H^{ab}\simeq C_2\times C_2$. 
Applying {\tt FirstObstructionN($G$)} 
and {\tt FirstObstructionDnr($G$)}, 
we obtain that 
${\rm Ker}\, \psi_1\simeq{\rm Ker}\, \psi_1/\varphi_1({\rm Ker}\, \psi_2^{\rm nr})\simeq\bZ/2\bZ$. 
We also apply the function 
{\tt FirstObstructionDr($G,G^\prime$)} 
for all subgroups $G^\prime\leq G$. 
We obtain that ${\rm Obs}_1(L/K/k)=1$ if and only if 
there exists $v\in V_k$ such that 
{\rm (i)} 
$V_4\leq G_v$ where 
$N_{\widetilde{G}}(V_4)\simeq C_8\rtimes (C_2\times C_2)$ 
for the normalizer $N_{\widetilde{G}}(V_4)$ of $V_4$ in $\widetilde{G}$ 
with the normalizer $\widetilde{G}=N_{S_{10}}(G)\simeq {\rm Aut}(G)$ of $G$ in $S_{10}$
$($equivalently, $|{\rm Orb}_{V_4}(i)|=2$ for any $1\leq i\leq 10$$)$ or  
{\rm (ii)} $D_4\leq G_v$ where $D_4\leq [G,G]\simeq A_6$ 
(see Details \ref{det6.8} and Example \ref{ex10}).\\ 

{\rm (5)} $n=14$: $G=14T30\simeq \PSL_2(\bF_{13})$. 
By Theorem \ref{thDP87}, we obtain that 
${\rm Obs}(K/k)={\rm Obs}_1(L/K/k)$.  
We have $H\simeq C_{13}\rtimes C_6$ and $H^{ab}\simeq C_6$. 
Applying {\tt FirstObstructionN($G$)} 
and {\tt FirstObstructionDnr($G$)}, 
we obtain that 
${\rm Ker}\, \psi_1\simeq\bZ/6\bZ$ and 
${\rm Ker}\, \psi_1/\varphi_1({\rm Ker}\, \psi_2^{\rm nr})\simeq\bZ/2\bZ$. 
Apply {\tt FirstObstructionDr($G,G^\prime$)} for all subgroups 
$G^\prime\leq G$. 
We get that ${\rm Obs}_1(L/K/k)=1$ if and only if 
there exists $v\in V_k$ such that 
$V_4\leq G_v$ 
(see Example \ref{ex14}).\\ 

{\rm (6)} $n=15$: $G=15T9\simeq (C_5)^2\rtimes C_3$ and 
$G\simeq 15T14\simeq (C_5)^2\rtimes S_3$. 
By Theorem \ref{thDP87}, we obtain that 
${\rm Obs}(K/k)={\rm Obs}_1(L/K/k)$. 

(6-1) The case $G=15T9\simeq (C_5)^2\rtimes C_3$.   
We have $H\simeq H^{ab}\simeq C_5$. 
Applying {\tt FirstObstructionN($G$)} 
and {\tt FirstObstructionDnr($G$)}, 
we have 
${\rm Ker}\, \psi_1\simeq
{\rm Ker}\, \psi_1/\varphi_1({\rm Ker}\, \psi_2^{\rm nr})\simeq\bZ/5\bZ$. 
Apply {\tt FirstObstructionDr($G,G^\prime$)} for all subgroups 
$G^\prime\leq G$. 
We get that ${\rm Obs}_1(L/K/k)=1$ if and only if 
there exists $v\in V_k$ such that 
$(C_5)^2\leq G_v$ 
(see Example \ref{ex15}). 

(6-2) The case $G=15T14\simeq (C_5)^2\rtimes S_3$. 
We have $H\simeq H^{ab}\simeq C_{10}$. 
Applying {\tt FirstObstructionN($G$)} 
and {\tt FirstObstructionDnr($G$)}, 
we obtain that 
${\rm Ker}\, \psi_1\simeq
{\rm Ker}\, \psi_1/\varphi_1({\rm Ker}\, \psi_2^{\rm nr})\simeq\bZ/5\bZ$. 
We apply the function 
{\tt FirstObstructionDr($G,G^\prime$)} for all subgroups 
$G^\prime\leq G$. 
We get that ${\rm Obs}_1(L/K/k)=1$ if and only if 
there exists $v\in V_k$ such that 
$(C_5)^2\leq G_v$ 
(see Example \ref{ex15}). 
\qed
\begin{remark}
By the proof of Theorem \ref{thmain3}, 
for $n=8$ (resp. $n=9)$, 
there exists $\widetilde{L}\supset L$ with $[\widetilde{L}:L]=2$ 
(resp. $[\widetilde{L}:L]=3$) 
such that ${\rm Obs}(K/k)={\rm Obs}_1(\widetilde{L}/K/k)$ 
although ${\rm Obs}(K/k)\neq {\rm Obs}_1(L/K/k)$ 
when $L/k$ is unramified 
and ${\rm Obs}_1(L/K/k)=1$ except for the case $G=8T21$ with 
${\rm Obs}_1(L/K/k)\simeq \bZ/2\bZ$.
\end{remark}

\begin{details}[{The cases (1-4) $G=8T31$, $8T38$, 
(2) $G=8Tm$ $(m=9,11,15,19,22,32)$ and 
(4-3) $G=10T32$ in Theorem \ref{thmain3}}]\label{det6.8}
We take generators of $G$ and $H={\rm Stab}_1(G)$ in Theorem \ref{thmain3} 
(Table $2$) 
and give more details for the cases {\rm (1-4)} $G=8T31$, $8T38$,  
{\rm (2)} $G=8Tm$ $(m=9,11,15,19,22,32)$ 
and (4-3) $G=10T32$. 

(1-4) 
For  
$G=8T31\simeq ((C_2)^4\rtimes C_2)\rtimes C_2$ 
and $G^\prime=8T38\simeq (((C_2)^4)\rtimes C_2)\rtimes C_2)\rtimes C_3$, 
we take $G=\langle g_1,g_2,g_3\rangle={\rm Syl}_2(G^\prime)\lhd 
G^\prime=\langle g_1,g_2,g_3^\prime\rangle$ 
and $H={\rm Stab}_1(G)=\langle g_1,(2,6),(3,7)\rangle\simeq (C_2)^3
\leq H^\prime={\rm Stab}_1(G^\prime)=\langle g_1,(2,6),(2,8,3)(4,7,6)\rangle\simeq C_2\times A_4$ 
where 
$g_1=(4,8)$, 
$g_2=(1,8)(2,3)(4,5)(6,7)$, 
$g_3=(1,3)(2,8)(4,6)(5,7)$, 
$g_3^\prime=(1,2,3)(5,6,7)$. 
There exist $61$ subgroups $V_4\leq G$ 
and $14$ of them satisfy $|{\rm Orb}_{V_4}(i)|=4$ ($1\leq i\leq 8$): 
\begin{align*}
V_4^{(1)}&=\{1,\sigma_1, \sigma_2, \sigma_3\},\ 
V_4^{(2)}=\{1,\sigma_1, \sigma_6, \sigma_7\},\ 
V_4^{(3)}=\{1,\tau_1, \tau_2, \tau_6\},\ 
V_4^{(4)}=\{1,\tau_1, \sigma_4, \tau_5\},\\
V_4^{(5)}&=\{1,\sigma_2, \tau_4, \sigma_7\},\ 
V_4^{(6)}=\{1,\tau_2, \sigma_4, \sigma_5)\},\ 
V_4^{(7)}=\{1,\sigma_3, \tau_4, \sigma_6\},\ 
V_4^{(8)}=\{1,\sigma_5, \tau_5, \tau_6\},\\ 
V_4^{(9)}&=\{1,\sigma_1, \tau_3, \tau_4\},\ 
V_4^{(10)}=\{1,\tau_1, \tau_3, \sigma_5)\},\ 
V_4^{(11)}=\{1,\sigma_2, \tau_3, \sigma_6\},\\
V_4^{(12)}&=\{1,\tau_2, \tau_3, \tau_5\},\ 
V_4^{(13)}=\{1,\sigma_3, \tau_3, \sigma_7\},\ 
V_4^{(14)}=\{1,\sigma_4, \tau_3, \tau_6\}
\end{align*}
where 
$\sigma_1=(1,2)(3,4)(5,6)(7,8)$, 
$\sigma_2=(1,3)(2,4)(5,7)(6,8)$, 
$\sigma_3=(1,4)(2,3)(5,8)(6,7)$, 
$\sigma_4=(1,4)(2,7)(3,6)(5,8)$, 
$\sigma_5=(1,6)(2,5)(3,4)(7,8)$,
$\sigma_6=(1,7)(2,8)(3,5)(4,6)$, 
$\sigma_7=(1,8)(2,7)(3,6)(4,5)$, 
$\tau_1=(1,2)(3,8)(4,7)(5,6)$, 
$\tau_2=(1,3)(2,8)(4,6)(5,7)$, 
$\tau_3=(1,5)(2,6)(3,7)(4,8)$, 
$\tau_4=(1,6)(2,5)(3,8)(4,7)$, 
$\tau_5=(1,7)(2,4)(3,5)(6,8)$, 
$\tau_6=(1,8)(2,3)(4,5)(6,7)$. 
Note that $Z(G)=\langle\tau_3\rangle\simeq C_2$ 
and the first $8$ groups 
$V_4^{(i)}$ $(1 \leq i\leq 8)$ satisfy 
$V_4\cap Z(G)=1$ as appearing in Theorem  \ref{thmain3} (Table $2$)  and 
the last  $6$ groups $V_4^{(i)}$ $(9\leq i\leq 14)$ 
satisfy $V_4\cap Z(G)\simeq C_2$. 
On the other hand, 
there exist $15$ subgroups $C_4\times C_2\leq G$: 
\begin{align*}
G^{(1)}&=\langle (1,4,5,8)(2,3,6,7), \sigma_1, \tau_3\rangle,\\
G^{(2)}&=\langle (1,8,5,4)(2,3,6,7), \tau_1, \tau_3\rangle,\\
G^{(3)}&=\langle (1,8,5,4)(2,3,6,7), (1,2,5,6)(3,4,7,8), \tau_3\rangle,\\
G^{(4)}&=\langle (1,6,5,2)(3,8,7,4), \tau_6, \tau_3\rangle,\\
G^{(5)}&=\langle (1,4,5,8)(2,3,6,7), (1,2,5,6)(3,8,7,4), \tau_3\rangle,\\
G^{(6)}&=\langle (1,6,5,2)(3,4,7,8), \sigma_3, \tau_3\rangle,\\
G^{(7)}&=\langle (1,2)(3,8,7,4)(5,6), (1,5)(2,6), (3,7)(4,8)\rangle,\\
G^{(8)}&=\langle (1,6,5,2)(3,8)(4,7), (1,5)(2,6), (3,7)(4,8)\rangle,\\
G^{(9)}&=\langle (1,6,5,2)(3,8,7,4), (1,5)(2,6), (3,7)(4,8)\rangle,\\
G^{(10)}&=\langle (1,7)(2,4,6,8)(3,5), (1,5)(3,7), (2,6)(4,8)\rangle,\\
G^{(11)}&=\langle (1,7,5,3)(2,8)(4,6), (1,5)(3,7), (2,6)(4,8)\rangle,\\
G^{(12)}&=\langle (1,7,5,3)(2,4,6,8), (1,5) (3,7), (2,6)(4,8)\rangle,\\
G^{(13)}&=\langle (1,4,5,8)(2,3)(6,7), (1,5)(4,8), (2,6)(3,7)\rangle,\\
G^{(14)}&=\langle (1,4)(2,3,6,7)(5,8), (1,5)(4,8), (2,6)(3,7)\rangle,\\
G^{(15)}&=\langle (1,4,5,8)(2,3,6,7), (1,5)(4,8), (2,6)(3,7)\rangle.
\end{align*}
Note that $[G,G]=\langle(1,5)(4,8), (2,6)(4,8), (2,6)(3,7)\rangle\simeq 
(C_2)^3$ and 
the first $6$ groups $G^{(i)}$ $(1\leq i\leq 6)$ 
satisfy $G^{(i)}\cap [G,G]=\langle\tau_3\rangle\simeq C_2$ 
and are transitive in $S_8$ 
as appearing in Theorem  \ref{thmain3} (Table $2$) 
and the last $9$ groups $G^{(i)}$ $(7\leq i\leq 15)$ 
satisfy $G^{(i)}\cap [G,G]\simeq V_4$ 
and are not transitive in $S_8$. 
Also, 
there exist $3$ subgroups $(C_2)^3\rtimes V_4\leq G$. 
They are all transitive in $S_8$ (see Example \ref{ex8}). 

(2-1) 
For  $G=8T9\simeq D_4\times C_2$, 
we take $G=\langle g_1,g_2,g_3,g_4\rangle$ 
and $H={\rm Stab}_1(G)=\langle g_4\rangle\simeq C_2$ 
where 
$g_1=(1,8)(2,3)(4,5)(6,7)$, 
$g_2=(1,3)(2,8)(4,6)(5,7)$, 
$g_3=(1,5)(2,6)(3,7)(4,8)$, 
$g_4=(4,5)(6,7)$. 
There exist $13$ subgroups $V_4\leq G$ and $8$ of them 
satisfy $|{\rm Orb}_{V_4}(i)|=4$ $(1\leq i\leq 8)$: 
\begin{align*}
&V_4^{(1)}=\{1,\sigma_1,\sigma_3,\sigma_6\},\ 
V_4^{(2)}=\{1,\sigma_1,\sigma_4,\sigma_5\},\ 
V_4^{(3)}=\{1,\sigma_2,\sigma_3,\sigma_5\},\ 
V_4^{(4)}=\{1,\sigma_2,\sigma_4,\sigma_6\},\\ 
&V_4^{(5)}=\{1,\sigma_1,\sigma_2,\tau_3\},\ 
V_4^{(6)}=\{1,\sigma_3,\sigma_4,\tau_3\},\ 
V_4^{(7)}=\{1,\sigma_5,\sigma_6,\tau_3\},\ 
V_4^{(8)}=\{1,\tau_1,\tau_2,\tau_3\}
\end{align*}
where 
$\sigma_1=(1,2)(3,8)(4,7)(5,6)$, $\sigma_2=(1,3)(2,8)(4,6)(5,7)$, 
$\sigma_3=(1,4)(2,7)(3,6)(5,8)$, $\sigma_4=(1,5)(2,6)(3,7)(4,8)$, 
$\sigma_5=\sigma_1 \sigma_4=(1,6)(2,5)(3,4)(7,8)$, 
$\sigma_6=\sigma_1 \sigma_3=(1,7)(2,4)(3,5)(6,8)$, 
$\tau_1=(1,2)(3,8)(4,6)(5,7)$, 
$\tau_2=(1,3)(2,8)$ $(4,7)(5,6)$, 
$\tau_3=(1,8)(2,3)(4,5)(6,7)$. 
Note that $[G,G]=\langle\tau_3\rangle\simeq C_2$ and 
the first half $V_4^{(i)}$ $(1\leq i\leq 4)$ 
satisfy $V_4^{(i)}\cap [G,G]=1$ 
as appearing in Theorem  \ref{thmain3} (Table $2$) 
although  
the second half $V_4^{(i)}$ $(5\leq i\leq 8)$ 
satisfy $[G,G]\leq V_4^{(i)}$. 
On the other hand, 
there exists the unique subgroup $C_4\times C_2\leq G$ 
(see Example \ref{ex8-2}). 

(2-2) 
For $G=8T11\simeq (C_4\times C_2)\rtimes C_2$, 
we take $G=\langle g_1,g_2,g_3\rangle$ 
and $H={\rm Stab}_1(G)=\langle (2,6)(4,8)\rangle\simeq C_2$ 
where 
$g_1=(1,5)(3,7)$, 
$g_2=(1,3,5,7)(2,4,6,8)$, 
$g_3=(1,4,5,8)(2,3,6,7)$. 
There exist $3$ subgroups 
$C_4\times C_2\leq G$: 
\begin{align*}
&G^{(1)}=\langle g_2, (1,2)(3,4)(5,6)(7,8)\rangle,\\ 
&G^{(2)}=\langle g_2, (1,8)(2,3)(4,5)(6,7)\rangle,\\ 
&G^{(3)}=\langle g_2, (2,6)(4,8)\rangle.
\end{align*}
The first two groups $G^{(1)}$ and $G^{(2)}$ 
are transitive in $S_8$ as appearing in Theorem  \ref{thmain3} (Table $2$) 
although  
the last one $G^{(3)}$ is not transitive in $S_8$ (see Example \ref{ex8-2}).

(2-3) 
For $G=8T15\simeq C_8\rtimes V_4$, 
we take $G=\langle g_1,g_2,g_3\rangle$ 
and $H={\rm Stab}_1(G)=\langle (2,8)(3,7)(4,6), (2,4)(3,7)(6,8)\rangle
\simeq V_4$ 
where 
$g_1=(1,2,3,4,5,6,7,8)$,  
$g_2=(1,5)(3,7)$, 
$g_3=(1,6)(2,5)(3,4)(7,8)$.
There exist $15$ subgroups $V_4\leq G$ 
and $5$ of them satisfy $|{\rm Orb}_{V_4}(i)|=2$ ($1\leq i\leq 8$): 
\begin{align*}
V_4^{(1)}&=\{1,\sigma_1,\tau_1,\tau_2\},\ 
V_4^{(2)}=\{1,\sigma_2,\tau_3,\tau_4\},\ 
V_4^{(3)}=\{1,\sigma_2,\tau_5,\tau_6\},\ 
V_4^{(4)}=\{1,\sigma_1,\tau_7,\tau_8\},\\ 
V_4^{(5)}&=\{1,\sigma_1,\sigma_2,\sigma_1\sigma_2\}
\end{align*}
where 
$\sigma_1=(1,5)(3,7)$, 
$\sigma_2=(2,6)(4,8)$, 
$\tau_1=(2,4)(3,7)(6,8)$, 
$\tau_2=(1,5)(2,4)(6,8)$, 
$\tau_3=(1,3)(4,8)(5,7)$, 
$\tau_4=(1,3)(2,6)(5,7)$, 
$\tau_5=(1,7)(3,5)(4,8)$, 
$\tau_6=(1,7)(2,6)(3,5)$, 
$\tau_7=(2,8)(3,7)(4,6)$, 
$\tau_8=(1,5)(2,8)(4,6)$. 
Note that $[G,G]=\langle (1,3,5,7),(2,4,6,8)\rangle\simeq C_4$ 
and 
the first four groups $V_4^{(i)}$ $(1\leq i\leq 4)$  
satisfy that $V_4^{(i)}\cap [G,G]=1$ 
and $V_4$ is not in $A_8$ 
as appearing in Theorem  \ref{thmain3} (Table $2$) 
although  
the last one $V_4^{(5)}$ satisfy that 
$V_4^{(5)}\cap [G,G]\simeq C_2$ 
and $V_4$ is in $A_8$. 
On the other hand, 
there exist $3$ subgroups $C_4\times C_2\leq G$: 
\begin{align*}
G^{(1)}&=\langle (1,3,5,7)(2,8,6,4), (1,2)(3,8)(4,7)(5,6)\rangle,\\ 
G^{(2)}&=\langle (1,3,5,7)(2,8,6,4), (1,4)(2,3)(5,8)(6,7)\rangle,\\ 
G^{(3)}&=\langle (1,3,5,7)(2,4,6,8),(2,6)(4,8)\rangle.
\end{align*}
The first two groups $G^{(i)}$ $(i=1,2)$ 
satisfy $G^{(i)}\cap [G,G]\simeq C_2$ $(i=1,2)$ 
which is transitive in $S_8$ 
as appearing in Theorem  \ref{thmain3} (Table $2$) 
although 
the last one $G^{(3)}$ satisfy 
$G^{(i)}\cap [G,G]\simeq C_4$ 
which is not transitive in $S_8$ (see Example \ref{ex8-2}). 

(2-4) 
For $G=8T19\simeq (C_2)^3\rtimes C_4$, 
we take $G=\langle g_1,g_2,g_3,g_4\rangle$ 
and $H={\rm Stab}_1(G)=\langle (2,8)(4,5,6,7)\rangle\simeq C_4$ 
where 
$g_1=(1,8)(2,3)(4,5)(6,7)$, 
$g_2=(1,3)(2,8)(4,6)(5,7)$, 
$g_3=(1,5)(2,6)(3,7)(4,8)$, 
$g_4=(1,3)(4,5,6,7)$. 
There exist $13$ subgroups $V_4\leq G$ 
and $8$ of them satisfy $|{\rm Orb}_{V_4}(i)|=4$ ($1\leq i\leq 8$): 
\begin{align*}
V_4^{(1)}&=\{1,\sigma_1,\sigma_2,\sigma_5\},\ 
V_4^{(2)}=\{1,\sigma_1,\sigma_3,\sigma_4\},\ 
V_4^{(3)}=\{1,\sigma_2,\sigma_3,\sigma_6\},\ 
V_4^{(4)}=\{1,\sigma_4,\sigma_5,\sigma_6\},\\
V_4^{(5)}&=\{1,\tau_1,\tau_2,\tau_3\},\ 
V_4^{(6)}=\{1,\tau_2,\sigma_1,\sigma_6\},\ 
V_4^{(7)}=\{1,\tau_2,\sigma_2,\sigma_4\},\ 
V_4^{(8)}=\{1,\tau_2,\sigma_3,\sigma_5\}
\end{align*} 
where 
$\sigma_1=(1,2)(3,8)(4,7)(5,6)$, 
$\sigma_2=(1,4)(2,7)(3,6)(5,8)$, 
$\sigma_3=(1,5)(2,6)(3,7)(4,8)$, 
$\sigma_4=(1,6)(2,5)(3,4)(7,8)$, 
$\sigma_5=(1,7)(2,4)(3,5)(6,8)$, 
$\sigma_6=(1,8)(2,3)(4,5)(6,7)$, 
$\tau_1=(1,2)(3,8)(4,5)(6,7)$, 
$\tau_2=(1,3)(2,8)(4,6)(5,7)$, 
$\tau_3=(1,8)(2,3)(4,7)(5,6)$. 
Note that $Z(G)=\langle \tau_2\rangle\simeq C_2$ 
and 
the first four groups $V_4^{(i)}$ $(1\leq i\leq 4)$  
satisfy $V_4^{(i)}\cap Z(G)=1$ and $V_4^{(i)}\cap Z^2(G)\simeq C_2$ 
with the upper central series $1\leq Z(G)\leq Z^2(G)\leq G$ of $G$ 
as appearing in Theorem  \ref{thmain3} (Table $2$) 
although  
the last four groups $V_4^{(i)}$ $(5\leq i\leq 8)$ 
satisfy $V_4^{(i)}\cap Z(G)\simeq C_2$. 
On the other hand, 
there exist $5$ subgroups $C_4\times C_2\leq G$: 
\begin{align*}
G^{(1)}&=\langle (2,8)(4,5,6,7), (1,3)(2,8)\rangle,\\ 
G^{(2)}&=\langle (1,2,3,8)(5,7), (4,6)(5,7)\rangle,\\ 
G^{(3)}&=\langle (1,5,3,7)(2,6,8,4),(1,8)(2,3)(4,5)(6,7)\rangle,\\ 
G^{(4)}&=\langle (1,4,2,5)(4,6,8,7),(1,3)(2,8)(4,6)(5,7)\rangle,\\ 
G^{(5)}&=\langle (1,4,8,7)(2,5,3,6),(1,3)(2,8)(4,6)(5,7)\rangle.
\end{align*}
Note that $[G,G]=\langle (1,3)(2,8)(4,6)(5,7),(1,8)(2,3)(4,5)(6,7)
\rangle\simeq V_4$ and the first two groups $G^{(1)}$ and $G^{(2)}$ 
are not transitive in $S_8$ and the third one  
$G^{(3)}$ is transitive in $S_8$ 
which satisfy $[G,G]\leq G^{(3)}$
as appearing in Theorem  \ref{thmain3} (Table $2$) 
although 
the last two $G^{(4)}$ and $G^{(5)}$ are transitive in $S_8$ 
which satisfy $G^{(i)}\cap [G,G] \simeq C_2$ $(4\leq i\leq 5)$ 
(see Example \ref{ex8-2}).

(2-5), (2-6)  For  
$G=8T22\simeq (C_2)^3\rtimes V_4$ 
and $G^\prime=8T32\simeq ((C_2)^3\rtimes V_4)\rtimes C_3$, 
we take $G=\langle g_1,g_2,g_3,g_4,g_5\rangle$, 
$H={\rm Stab}_1(G)=\langle g_4,g_5\rangle\simeq V_4$, 
$G^\prime=\langle g_1,g_2,g_3,g_4^\prime,g_5^\prime\rangle$ 
and 
$H^\prime={\rm Stab}_1(G^\prime)=\langle g_5^\prime, 
(2,3,8)(4,7,5)\rangle\simeq A_4$ 
where 
$g_1=(1,8)(2,3)(4,5)(6,7)$, 
$g_2=(1,3)(2,8)(4,6)(5,7)$, 
$g_3=(1,5)(2,6)(3,7)(4,8)$, 
$g_4=(2,3)(4,5)$, 
$g_5=(2,3)(6,7)$, 
$g_4^\prime=(1,2,3)(4,6,5)$, 
$g_5^\prime=(2,5)(3,4)$.
There exist $33$ subgroups $V_4\leq G$ 
and $14$ of them satisfy $|{\rm Orb}_{V_4}(i)|=4$ ($1\leq i\leq 8$): 
\begin{align*}
V_4^{(1)}&=\{1,\sigma_1, \sigma_3, \sigma_5\},\ 
V_4^{(2)}=\{1,\sigma_1, \sigma_4, \sigma_6\},\ 
V_4^{(3)}=\{1,\tau_1, \tau_2, \tau_5\},\ 
V_4^{(4)}=\{1,\tau_1, \tau_4, \sigma_5\},\\ 
V_4^{(5)}&=\{1,\tau_3, \tau_2, \sigma_5\},\ 
V_4^{(6)}=\{1,\tau_3, \tau_4, \tau_5\},\ 
V_4^{(7)}=\{1,\sigma_2, \sigma_3, \sigma_6\},\ 
V_4^{(8)}=\{1,\sigma_2, \sigma_4, \sigma_5\},\\ 
V_4^{(9)}&=\{1,\sigma_1, \sigma_2, \tau_6\},\ 
V_4^{(10)}=\{1,\tau_1, \tau_3, \tau_6\},\ 
V_4^{(11)}=\{1,\sigma_3, \sigma_4, \tau_6\},\\ 
V_4^{(12)}&=\{1,\tau_2, \tau_3, \tau_6\},\ 
V_4^{(13)}=\{1,\sigma_5, \sigma_6, \tau_6\},\ 
V_4^{(14)}=\{1,\sigma_5, \tau_5, \tau_6\}
\end{align*}
where
$\sigma_1=(1,2)(3,8)(4,6)(5,7)$, 
$\sigma_2=(1,3)(2,8)(4,7)(5,6)$, 
$\sigma_3=(1,4)(2,6)(3,7)(5,8)$, 
$\sigma_4=(1,5)(2,7)(3,6)(4,8)$, 
$\sigma_5=(1,6)(2,4)(3,5)(7,8)$, 
$\sigma_6=(1,7)(2,5)(3,4)(6,8)$, 
$\tau_1=(1,2)(3,8)(4,7)(5,6)$, 
$\tau_2=(1,3)(2,8)(4,6)(5,7)$
$\tau_3=(1,4)(2,7)(3,6)(5,8)$, 
$\tau_4=(1,5)(2,6)(3,7)(4,8)$, 
$\tau_5=(1,7)(2,4)(3,5)(6,8)$, 
$\tau_6=(1,8)(2,3)(4,5)(6,7)$. 
Note that $Z(G)=\langle\tau_6\rangle\simeq C_2$ 
and the first $8$ groups $V_4^{(i)}$ $(1\leq i\leq 8)$ 
satisfy $V_4\cap Z(G)=1$ as appearing in Theorem  \ref{thmain3} (Table $2$) and 
the last $6$ groups  
$V_4^{(i)}$ $(9\leq i\leq 14)$ 
satisfy $V_4\cap Z(G)\simeq C_2$. 
On the other hand, 
there exist $9$ subgroups $C_4\times C_2\leq G$:
\begin{align*}
G^{(1)}&=\langle (1,4,8,5)(2,6,3,7), \sigma_1, \tau_6\rangle,\\
G^{(2)}&=\langle (1,6,8,7)(2,5,3,4), \tau_1, \tau_6\rangle,\\
G^{(3)}&=\langle (1,2,8,3)(4,6,5,7), \sigma_3, \tau_6\rangle,\\ 
G^{(4)}&=\langle (1,4,8,5)(2,6,3,7), (1,2,8,3)(4,6,5,7), \tau_6\rangle,\\
G^{(5)}&=\langle (1,2,8,3)(4,7,5,6), \tau_3, \tau_6\rangle,\\ 
G^{(6)}&=\langle (1,2,8,3)(4,7,5,6), \sigma_5, \tau_6\rangle,\\
G^{(7)}&=\langle (1,2,8,3)(4,6,5,7), (4,5)(6,7), \tau_6\rangle,\\
G^{(8)}&=\langle (1,4,8,5)(2,6,3,7), (2,3)(6,7), \tau_6\rangle,\\
G^{(9)}&=\langle (1,6,8,7)(2,5,3,4), (2,3)(4,5), \tau_6\rangle.
\end{align*}
The first $6$ groups $G^{(i)}$ $(1\leq i\leq 6)$ are 
transitive in $S_8$ as appearing in Theorem  \ref{thmain3} (Table $2$) 
and the last $3$ groups $G^{(i)}$ $(7\leq i\leq 9)$
are not transitive in $S_8$ (see Example \ref{ex8-2}).

(4-3) 
For  $G=10T32\simeq S_6$, 
we take $G=\langle g_1,g_2,g_3,g_4\rangle$ and 
$H={\rm Stab}_1(G)=\langle g_4,(3,10)(6,9)(7,8), 
(2,4)(3,7)$ $(6,9)(8,10)\rangle\simeq (S_3)^2\rtimes C_2$ 
where 
$g_1=(1,2,10)(3,4,5)(6,7,8)$, 
$g_2=(1,3,2,6)(4,5,8,7)$, 
$g_3=(1,2)(4,7)(5,8)(9,10)$, 
$g_4=(3,6)(4,7)(5,8)$. 
There exist $165$ subgroups $V_4$ of $G$ 
and $45=({10 \atop 2})$ groups 
$\langle (a,b)(c,d)(e,f),(a,b)(g,h)(i,j)\rangle$ $\simeq$ $V_4$ 
$(\{a,b,c,d,e,f,g,h,i,j\}$ $=$ $\{1,2,3,4,5,6,7,8,9,10\})$ 
of them satisfy 
$N_{\widetilde{G}}(V_4)\simeq C_8\rtimes (C_2\times C_2)$ 
where $\widetilde{G}=N_{S_{10}}(G)\simeq {\rm Aut}(G)$ is 
the normalizer of $G$ in $S_{10}$ 
$($equivalently, $|{\rm Orb}_{V_4}(i)|=2$ for any $1\leq i\leq 10$$)$ 
as appearing in Theorem \ref{thmain3} (Table $2$). 
There exist $180$ subgroups $D_4$ of $G$ 
and $45$ groups of them satisfy $D_4\leq [G,G]\simeq A_6$ 
(see Example \ref{ex10}). 
\end{details}

\smallskip
\begin{example}[$G=4T2\simeq V_4$ and $G=4T4\simeq A_4$]\label{ex4}
~{}\vspace*{-2mm}\\

Case $G=4T2\simeq V_4$. 
{\small 
\begin{verbatim}
gap> Read("HNP.gap");
gap> G:=TransitiveGroup(4,2); # G=4T2=V4
E(4) = 2[x]2
gap> H:=Stabilizer(G,1); # H=1
Group(())
gap> FirstObstructionN(G).ker; # Obs1N=1
[ [  ], [ [  ], [  ] ] ]
gap> SchurMultPcpGroup(G); # M(G)=C2: Schur multiplier of G
[ 2 ]
gap> ScG:=SchurCoverG(G);
rec( SchurCover := Group([ (2,3), (1,2)(3,4) ]), 
  epi := [ (2,3), (1,2)(3,4) ] -> [ (1,2)(3,4), (1,4)(2,3) ], Tid := [ 4, 3 ] )
gap> StructureDescription(TransitiveGroup(4,3));
"D8"
gap> tG:=ScG.SchurCover; # tG=G~=D4 is a Schur cover of G 
Group([ (2,3), (1,2)(3,4) ])
gap> tH:=PreImage(ScG.epi,H); # tH=H~=C2
Group([ (1,4)(2,3) ])
gap> FirstObstructionN(tG,tH).ker; # Obs1N~=C2
[ [ 2 ], [ [ 2 ], [ [ 1 ] ] ] ]
gap> FirstObstructionDnr(tG,tH).Dnr; # Obs1Dnr~=1
[ [  ], [ [ 2 ], [  ] ] ]
gap> tGs:=AllSubgroups(tG);
[ Group(()), Group([ (2,3) ]), Group([ (1,4) ]), Group([ (1,4)(2,3) ]), 
  Group([ (1,2)(3,4) ]), Group([ (1,3)(2,4) ]), Group([ (1,4), (2,3) ]), 
  Group([ (1,2)(3,4), (1,4)(2,3) ]), Group([ (1,3,4,2), (1,4)(2,3) ]), 
  Group([ (1,4), (2,3), (1,2)(3,4) ]) ]
gap> List(tGs,StructureDescription);
[ "1", "C2", "C2", "C2", "C2", "C2", "C2 x C2", "C2 x C2", "C4", "D8" ]
gap> List(tGs,x->FirstObstructionDr(tG,x,tH).Dr[1]);
[ [  ], [  ], [  ], [  ], [  ], [  ], [  ], [  ], [  ], [ 2 ] ]
gap> List(tGs,x->StructureDescription(Image(ScG.epi,x)));
[ "1", "C2", "C2", "1", "C2", "C2", "C2", "C2", "C2", "C2 x C2" ]
\end{verbatim}
}~\\\vspace*{-4mm}

Case $G=4T4\simeq A_4$.
{\small 
\begin{verbatim}
gap> Read("HNP.gap");
gap> G:=TransitiveGroup(4,4); # G=4T4
A4
gap> H:=Stabilizer(G,1); # H=C3
Group([ (2,3,4) ])
gap> FirstObstructionN(G).ker; # Obs1N=1
[ [  ], [ [ 3 ], [  ] ] ]
gap> SchurMultPcpGroup(G); # M(G)=C2: Schur multiplier of G
[ 2 ]
gap> ScG:=SchurCoverG(G);
rec( SchurCover := Group([ (1,2)(3,5,7,4,6,8), (1,3,6,2,4,5)(7,8) ]), 
  epi := [ (1,2)(3,5,7,4,6,8), (1,3,6,2,4,5)(7,8) ] -> [ (1,2,3), (2,3,4) ], 
  Tid := [ 8, 12 ] )
gap> StructureDescription(TransitiveGroup(8,12));
"SL(2,3)"
gap> tG:=ScG.SchurCover; # tG=G~=SL(2,3) is a Schur cover of G 
Group([ (1,2)(3,5,7,4,6,8), (1,3,6,2,4,5)(7,8) ])
gap> tH:=PreImage(ScG.epi,H); # tH=H~=C6
Group([ (1,4,6)(2,3,5), (1,2)(3,4)(5,6)(7,8) ])
gap> StructureDescription(tH);
"C6"
gap> FirstObstructionN(tG,tH).ker; # Obs1N~=C2
[ [ 2 ], [ [ 6 ], [ 3 ] ] ]
gap> FirstObstructionDnr(tG,tH).Dnr; # Obs1Dnr~=1
[ [  ], [ [ 6 ], [  ] ] ]
gap> tGs:=AllSubgroups(tG);;
gap> Length(tGs);
15
gap> List(tGs,StructureDescription);
[ "1", "C2", "C3", "C3", "C3", "C3", "C4", "C4", "C4", "C6", "C6", "C6", 
  "C6", "Q8", "SL(2,3)" ]
gap> List(tGs,x->FirstObstructionDr(tG,x,tH).Dr[1]);
[ [  ], [  ], [  ], [  ], [  ], [  ], [  ], [  ], [  ], [  ], [  ], [  ], 
  [  ], [ 2 ], [ 2 ] ]
gap> List(tGs,x->StructureDescription(Image(ScG.epi,x)));
[ "1", "1", "C3", "C3", "C3", "C3", "C2", "C2", "C2", "C3", 
  "C3", "C3", "C3", "C2 x C2", "A4" ]
\end{verbatim}
}
\end{example}

\smallskip
\begin{example}[$G=6T4\simeq A_4$ and $G=6T12\simeq A_5$]\label{ex6}
~{}\vspace*{-2mm}\\

Case $G=6T4\simeq A_4$. 
{\small 
\begin{verbatim}
gap> Read("HNP.gap");
gap> G:=TransitiveGroup(6,4); # G=6T4=A4
A_4(6) = [2^2]3
gap> H:=Stabilizer(G,1); # H=C2
Group([ (2,5)(3,6) ])
gap> FirstObstructionN(G).ker; # Obs1N=C2
[ [ 2 ], [ [ 2 ], [ [ 1 ] ] ] ]
gap> FirstObstructionDnr(G).Dnr; # Obs1Dnr=1
[ [  ], [ [ 2 ], [  ] ] ]
gap> Gs:=AllSubgroups(G);;
gap> Length(Gs);
10
gap> List(Gs,StructureDescription);
[ "1", "C2", "C2", "C2", "C3", "C3", "C3", "C3", "C2 x C2", "A4" ]
gap> List(Gs,x->FirstObstructionDr(G,x).Dr[1]);
[ [  ], [  ], [  ], [  ], [  ], [  ], [  ], [  ], [ 2 ], [ 2 ] ]
\end{verbatim}
}~\\\vspace*{-4mm}

Case $G=6T12\simeq A_5$.  
{\small 
\begin{verbatim}
gap> Read("HNP.gap");
gap> G:=TransitiveGroup(6,12); # G=6T12=A5
L(6) = PSL(2,5) = A_5(6)
gap> H:=Stabilizer(G,1); # H=D5
Group([ (2,4,3,6,5), (3,6)(4,5) ])
gap> StructureDescription(H);
"D10"
gap> FirstObstructionN(G).ker; # Obs1N=C2
[ [ 2 ], [ [ 2 ], [ [ 1 ] ] ] ]
gap> FirstObstructionDnr(G).Dnr; # Obs1Dnr=1
[ [  ], [ [ 2 ], [  ] ] ]
gap> Gs:=AllSubgroups(G);;
gap> Length(Gs);
59
gap> GsHNPfalse:=Filtered(Gs,x->FirstObstructionDr(G,x).Dr[1]=[]);;
gap> GsHNPtrue:=Filtered(Gs,x->FirstObstructionDr(G,x).Dr[1]=[2]);;
gap> List([GsHNPfalse,GsHNPtrue],Length);
[ 48, 11 ]                                                    ^
gap> Collected(List(GsHNPfalse,x->StructureDescription(x)));
[ [ "1", 1 ], [ "C2", 15 ], [ "C3", 10 ], [ "C5", 6 ], [ "D10", 6 ], 
  [ "S3", 10 ] ]
gap> Collected(List(GsHNPtrue,x->StructureDescription(x)));
[ [ "A4", 5 ], [ "A5", 1 ], [ "C2 x C2", 5 ] ]
\end{verbatim}
}
\end{example}

\smallskip
\begin{example}[$G=8Tm$ $(m=2,3,4,13,14,21,31,37,38)$]\label{ex8}
~{}\vspace*{-2mm}\\

(1-1) $G=8T3\simeq (C_2)^3$. 
{\small 
\begin{verbatim}
gap> Read("HNP.gap");
gap> G:=TransitiveGroup(8,3); # G=8T3=C2xC2xC2
E(8)=2[x]2[x]2
gap> H:=Stabilizer(G,1); # H=1
Group(())
gap> FirstObstructionN(G).ker; # Obs1N=1
[ [  ], [ [  ], [  ] ] ]
gap> SchurMultPcpGroup(G); # M(G)=C2xC2xC2: Schur multiplier of G
[ 2, 2, 2 ]
gap> ScG:=SchurCoverG(G);
rec( SchurCover := Group([ (2,3)(4,6)(9,10)(11,13)(12,15)(14,16), (1,2)(3,5)
  (4,7)(6,8)(10,11)(12,14), (2,4)(3,6)(10,12)(11,14) ]), 
  epi := [ (2,3)(4,6)(9,10)(11,13)(12,15)(14,16), 
      (1,2)(3,5)(4,7)(6,8)(10,11)(12,14), (2,4)(3,6)(10,12)(11,14) ] -> 
    [ (1,5)(2,6)(3,7)(4,8), (1,3)(2,8)(4,6)(5,7), (1,8)(2,3)(4,5)(6,7) ] )
gap> tG:=ScG.SchurCover; # tG=G~ is a Schur cover of G 
Group([ (2,3)(4,6)(9,10)(11,13)(12,15)(14,16), (1,2)(3,5)(4,7)(6,8)(10,11)
(12,14), (2,4)(3,6)(10,12)(11,14) ])
gap> tH:=PreImage(ScG.epi,H); # tH=H~=C2xC2xC2
Group([ (1,5)(2,3)(4,6)(7,8)(9,13)(10,11)(12,14)(15,16), (9,15)(10,12)(11,14)
(13,16), (1,7)(2,4)(3,6)(5,8) ])
gap> IdSmallGroup(tG);
[ 64, 73 ]
gap> StructureDescription(tG);
"(C2 x C2 x D8) : C2"
gap> StructureDescription(tH);
"C2 x C2 x C2"
gap> FirstObstructionN(tG,tH).ker; # Obs1N~=C2xC2xC2
[ [ 2, 2, 2 ], [ [ 2, 2, 2 ], [ [ 1, 0, 0 ], [ 0, 1, 0 ], [ 0, 0, 1 ] ] ] ]
gap> FirstObstructionDnr(tG,tH).Dnr; # Obs1Dnr~=1
[ [  ], [ [ 2, 2, 2 ], [  ] ] ]
gap> tGs:=AllSubgroups(tG);;
gap> Length(tGs);
317
gap> Collected(List(tGs,x->FirstObstructionDr(tG,x,tH).Dr));
[ [ [ [  ], [ [ 2, 2, 2 ], [  ] ] ], 213 ], 
  [ [ [ 2 ], [ [ 2, 2, 2 ], [ [ 0, 0, 1 ] ] ] ], 29 ], 
  [ [ [ 2 ], [ [ 2, 2, 2 ], [ [ 0, 1, 0 ] ] ] ], 29 ], 
  [ [ [ 2 ], [ [ 2, 2, 2 ], [ [ 0, 1, 1 ] ] ] ], 5 ], 
  [ [ [ 2 ], [ [ 2, 2, 2 ], [ [ 1, 0, 0 ] ] ] ], 29 ], 
  [ [ [ 2 ], [ [ 2, 2, 2 ], [ [ 1, 0, 1 ] ] ] ], 5 ], 
  [ [ [ 2 ], [ [ 2, 2, 2 ], [ [ 1, 1, 0 ] ] ] ], 5 ], 
  [ [ [ 2 ], [ [ 2, 2, 2 ], [ [ 1, 1, 1 ] ] ] ], 1 ], 
  [ [ [ 2, 2, 2 ], [ [ 2, 2, 2 ], [ [ 1, 0, 0 ], 
  [ 0, 1, 0 ], [ 0, 0, 1 ] ] ] ], 1 ] ]
gap> tGsHNPfalse0:=Filtered(tGs,x->FirstObstructionDr(tG,x,tH).Dr[1]=[]);;
gap> Length(tGsHNPfalse0);
213
gap> tGsHNPtrue0:=Filtered(tGs,x->FirstObstructionDr(tG,x,tH).Dr[1]=[2,2,2]);;
gap> Length(tGsHNPtrue0);
1
gap> Collected(List(tGsHNPfalse0,x->StructureDescription(Image(ScG.epi,x))));
[ [ "1", 16 ], [ "C2", 197 ] ]
gap> Collected(List(tGsHNPtrue0,x->StructureDescription(Image(ScG.epi,x))));
[ [ "C2 x C2 x C2", 1 ] ]
\end{verbatim}
}~\\\vspace*{-4mm}

(1-2) $G=8T21\simeq (C_2)^3\rtimes C_4$. 
{\small 
\begin{verbatim}
gap> Read("HNP.gap");
gap> G:=TransitiveGroup(8,21); # G=8T21=(C2xC2xC2):C4
1/2[2^4]E(4)=[1/4.dD(4)^2]2
gap> H:=Stabilizer(G,1); # H=C2xC2
Group([ (2,6)(4,8), (3,7)(4,8) ])
gap> FirstObstructionN(G).ker; # Obs1N=C2
[ [ 2 ], [ [ 2, 2 ], [ [ 1, 0 ] ] ] ]
gap> FirstObstructionDnr(G).Dnr; # Obs1Dnr=1 => Obs=Obs1=C2 if unramified 
[ [  ], [ [ 2, 2 ], [  ] ] ]
gap> KerResH3Z(G,H); # Obs=Obs1
[ [  ], [ [ 2, 2 ], [  ] ] ]
gap> Gs:=AllSubgroups(G);;
gap> Length(Gs);
50
gap> GsHNPfalse:=Filtered(Gs,x->FirstObstructionDr(G,x,H).Dr[1]=[]);;
gap> Length(GsHNPfalse);
49
gap> GsHNPtrue:=Filtered(Gs,x->FirstObstructionDr(G,x,H).Dr[1]=[2]);;
gap> Length(GsHNPtrue);
1
gap> Collected(List(GsHNPfalse,StructureDescription));
[ [ "(C4 x C2) : C2", 2 ], [ "1", 1 ], [ "C2", 11 ], [ "C2 x C2", 13 ], 
  [ "C2 x C2 x C2", 2 ], [ "C2 x D8", 1 ], [ "C4", 10 ], [ "C4 x C2", 5 ], 
  [ "D8", 4 ] ]
gap> Collected(List(GsHNPtrue,StructureDescription));
[ [ "(C2 x C2 x C2) : C4", 1 ] ]
\end{verbatim}
}~\\\vspace*{-4mm}

(1-3-1) $G=8T2\simeq C_4\times C_2$. 
{\small 
\begin{verbatim}
gap> Read("HNP.gap");
gap> G:=TransitiveGroup(8,2); # G=8T2=C4xC2
4[x]2
gap> H:=Stabilizer(G,1); # H=1
Group(())
gap> FirstObstructionN(G).ker; # Obs1N=1
[ [  ], [ [  ], [  ] ] ]
gap> SchurMultPcpGroup(G); # M(G)=C2: Schur multiplier of G
[ 2 ]
gap> ScG:=SchurCoverG(G);
rec( SchurCover := Group([ (2,4)(3,6), (1,2,5,3)(4,7,6,8) ]), 
  Tid := [ 8, 10 ], epi := [ (2,4)(3,6), (1,2,5,3)(4,7,6,8) ] -> 
    [ (1,5)(2,6)(3,7)(4,8), (1,2,3,8)(4,5,6,7) ] )
gap> tG:=ScG.SchurCover; # tG=G~=(C4xC2):C2 is a Schur cover of G 
Group([ (2,4)(3,6), (1,2,5,3)(4,7,6,8) ])
gap> StructureDescription(TransitiveGroup(8,10));
"(C4 x C2) : C2"
gap> tH:=PreImage(ScG.epi,H); # tH=H~=C2
Group([ (1,8)(2,4)(3,6)(5,7) ])
gap> FirstObstructionN(tG,tH).ker; # Obs1N~=C2
[ [ 2 ], [ [ 2 ], [ [ 1 ] ] ] ]
gap> FirstObstructionDnr(tG,tH).Dnr; # Obs1Dnr~=1
[ [  ], [ [ 2 ], [  ] ] ]
gap> tGs:=AllSubgroups(tG);;
gap> Length(tGs);
23
gap> tGsHNPfalse:=Filtered(tGs,x->FirstObstructionDr(tG,x,tH).Dr[1]=[]);;
gap> Length(tGsHNPfalse);
22
gap> tGsHNPtrue:=Filtered(tGs,x->FirstObstructionDr(tG,x,tH).Dr[1]=[2]);;
gap> Length(tGsHNPtrue);
1
gap> Collected(List(tGsHNPfalse,x->StructureDescription(Image(ScG.epi,x))));
[ [ "1", 2 ], [ "C2", 9 ], [ "C2 x C2", 5 ], [ "C4", 6 ] ]
gap> Collected(List(tGsHNPtrue,x->StructureDescription(Image(ScG.epi,x))));
[ [ "C4 x C2", 1 ] ]
\end{verbatim}
}~\\\vspace*{-4mm}

(1-3-2) $G=8T4\simeq D_4$. 
{\small 
\begin{verbatim}
gap> Read("HNP.gap");
gap> G:=TransitiveGroup(8,4); # G=8T4=D4
D_8(8)=[4]2
gap> H:=Stabilizer(G,1); # H=1
Group(())
gap> FirstObstructionN(G).ker; # Obs1N=1
[ [  ], [ [  ], [  ] ] ]
gap> SchurMultPcpGroup(G); # M(G)=C2: Schur multiplier of G
[ 2 ]
gap> ScG:=SchurCoverG(G);
rec( SchurCover := Group([ (2,3)(4,5)(6,7), (1,2,4,6,8,7,5,3) ]), 
  Tid := [ 8, 6 ], epi := [ (2,3)(4,5)(6,7), (1,2,4,6,8,7,5,3) ] -> 
    [ (1,6)(2,5)(3,4)(7,8), (1,2,3,8)(4,5,6,7) ] )
gap> tG:=ScG.SchurCover; # tG=G~=D8 is a Schur cover of G
Group([ (2,3)(4,5)(6,7), (1,2,4,6,8,7,5,3) ])
gap> StructureDescription(tG);
"D16"
gap> tH:=PreImage(ScG.epi,H); # tH=H~=C2
Group([ (1,8)(2,7)(3,6)(4,5) ])
gap> FirstObstructionN(tG,tH).ker; # Obs1N~=C2
[ [ 2 ], [ [ 2 ], [ [ 1 ] ] ] ]
gap> FirstObstructionDnr(tG,tH).Dnr; # Obs1Dnr~=1
[ [  ], [ [ 2 ], [  ] ] ]
gap> tGs:=AllSubgroups(tG);;
gap> Length(tGs);
19
gap> tGsHNPfalse:=Filtered(tGs,x->FirstObstructionDr(tG,x,tH).Dr[1]=[]);;
gap> Length(tGsHNPfalse);
16
gap> tGsHNPtrue:=Filtered(tGs,x->FirstObstructionDr(tG,x,tH).Dr[1]=[2]);;
gap> Length(tGsHNPtrue);
3
gap> Collected(List(tGsHNPfalse,StructureDescription));
[ [ "1", 1 ], [ "C2", 9 ], [ "C2 x C2", 4 ], [ "C4", 1 ], [ "C8", 1 ] ]
gap> Collected(List(tGsHNPtrue,StructureDescription));
[ [ "D16", 1 ], [ "D8", 2 ] ]
gap> Collected(List(tGsHNPfalse,x->StructureDescription(Image(ScG.epi,x))));
[ [ "1", 2 ], [ "C2", 13 ], [ "C4", 1 ] ]
gap> Collected(List(tGsHNPtrue,x->StructureDescription(Image(ScG.epi,x))));
[ [ "C2 x C2", 2 ], [ "D8", 1 ] ]
\end{verbatim}
}~\\\vspace*{-4mm}

(1-3-3) $G=8T13\simeq A_4\times C_2$. 
{\small 
\begin{verbatim}
gap> Read("HNP.gap");
gap> G:=TransitiveGroup(8,13); # G=8T13=A4xC2
E(8):3=A(4)[x]2
gap> H:=Stabilizer(G,1); # H=C3
Group([ (2,3,8)(4,7,5) ])
gap> FirstObstructionN(G).ker; # Obs1N=1
[ [  ], [ [ 3 ], [  ] ] ]
gap> SchurMultPcpGroup(G); # M(G)=C2: Schur multiplier of G
[ 2 ]
gap> ScG:=SchurCoverG(G);
rec( SchurCover := Group([ (1,2,3)(4,6,7)(5,8,9)(10,14,15)(11,16,17)(12,18,19)
(13,20,21)(22,23,24), (2,4)(3,5)(6,10)(7,11)(8,12)(9,13)(15,18)(16,22)(17,20)
(19,23) ]), Tid := [ 24, 21 ],  epi := [ (1,2,3)(4,6,7)(5,8,9)(10,14,15)
(11,16,17)(12,18,19)(13,20,21)(22,23,24), (2,4)(3,5)(6,10)(7,11)(8,12)(9,13)
(15,18)(16,22)(17,20)(19,23) ] -> [ (2,8,3)(4,5,7), (1,5)(2,6)(3,7)(4,8) ] )
gap> tG:=ScG.SchurCover; # tG=G~ is a Schur cover of G
Group([ (1,2,3)(4,6,7)(5,8,9)(10,14,15)(11,16,17)(12,18,19)(13,20,21)(22,23,24), 
(2,4)(3,5)(6,10)(7,11)(8,12)(9,13)(15,18)(16,22)(17,20)(19,23) ])
gap> tH:=PreImage(ScG.epi,H); # tH=H~=C6
Group([ (1,3,2)(4,7,6)(5,9,8)(10,15,14)(11,17,16)(12,19,18)(13,21,20)(22,24,23), 
(1,24)(2,22)(3,23)(4,16)(5,19)(6,17)(7,11)(8,12)(9,18)(10,20)(13,15)(14,21) ])
gap> StructureDescription(TransitiveGroup(24,21));
"((C4 x C2) : C2) : C3"
gap> StructureDescription(tH);
"C6"
gap> FirstObstructionN(tG,tH).ker; # Obs1N~=C2
[ [ 2 ], [ [ 6 ], [ [ 3 ] ] ] ]
gap> FirstObstructionDnr(tG,tH).Dnr; # Obs1Dnr~=1
[ [  ], [ [ 6 ], [  ] ] ]
gap> tGs:=AllSubgroups(tG);;
gap> Length(tGs);
37
gap> tGsHNPfalse:=Filtered(tGs,x->FirstObstructionDr(tG,x,tH).Dr[1]=[]);;
gap> Length(tGsHNPfalse);
30
gap> tGsHNPtrue:=Filtered(tGs,x->FirstObstructionDr(tG,x,tH).Dr[1]=[2]);;
gap> Length(tGsHNPtrue);
7
gap> Collected(List(tGsHNPfalse,x->StructureDescription(Image(ScG.epi,x))));
[ [ "1", 2 ], [ "C2", 13 ], [ "C2 x C2", 3 ], [ "C3", 8 ], [ "C6", 4 ] ]
gap> Collected(List(tGsHNPtrue,x->StructureDescription(Image(ScG.epi,x))));
[ [ "A4", 1 ], [ "C2 x A4", 1 ], [ "C2 x C2", 4 ], [ "C2 x C2 x C2", 1 ] ]
gap> pi:=NaturalHomomorphismByNormalSubgroup(G,Centre(G));
[ (1,8)(2,3)(4,5)(6,7), (1,3)(2,8)(4,6)(5,7), (1,5)(2,6)(3,7)(4,8), 
(1,2,3)(4,6,5) ] -> [ f3, f2, f2*f3, f1*f2*f3 ]
gap> Collected(List(tGsHNPtrue,x->StructureDescription(Image(pi,Image(ScG.epi,x)))));
[ [ "A4", 2 ], [ "C2 x C2", 5 ] ]
gap> Collected(List(tGsHNPfalse,x->StructureDescription(Image(pi,Image(ScG.epi,x)))));
[ [ "1", 3 ], [ "C2", 15 ], [ "C3", 12 ] ]
gap> Collected(List(tGsHNPtrue,x->StructureDescription(Image(pi,Image(ScG.epi,x)))));
[ [ "A4", 2 ], [ "C2 x C2", 5 ] ]
\end{verbatim}
}~\\\vspace*{-4mm}

(1-3-4) $G=8T14\simeq S_4$. 
{\small 
\begin{verbatim}
gap> Read("HNP.gap");
gap> G:=TransitiveGroup(8,14); # G=8T14=S4
S(4)[1/2]2=1/2(S_4[x]2)
gap> H:=Stabilizer(G,1); # H=C3
Group([ (2,8,3)(4,7,6) ])
gap> FirstObstructionN(G).ker; # Obs1N=C3
[ [ 3 ], [ [ 3 ], [ [ 1 ] ] ] ]
gap> FirstObstructionDnr(G).Dnr; # Osb1Dr=C3
[ [ 3 ], [ [ 3 ], [ [ 1 ] ] ] ]
gap> SchurMultPcpGroup(G); # M(G)=C2: Schur multiplier of G
[ 2 ]
gap> ScG:=SchurCoverG(G); 
rec( SchurCover := Group([ (2,4)(3,6)(5,8), (1,2,5,7,4,3)(6,8) ]), Tid := [ 8, 23 ],
  epi := [ (2,4)(3,6)(5,8), (1,2,5,7,4,3)(6,8) ] -> 
    [ (1,4)(2,6)(3,7)(5,8), (2,8,3)(4,7,6) ] )
gap> tG:=ScG.SchurCover; # tG=G~=SL(2,3) is a Schur cover of G
Group([ (2,4)(3,6)(5,8), (1,2,5,7,4,3)(6,8) ])
gap> tH:=PreImage(ScG.epi,H); # tH=H~=C6
Group([ (1,4,5)(2,3,7), (1,7)(2,4)(3,5)(6,8) ])
gap> StructureDescription(tG); 
"GL(2,3)"
gap> StructureDescription(tH); 
"C6"
gap> FirstObstructionN(tG,tH).ker; # Obs1N~=C6
[ [ 6 ], [ [ 6 ], [ [ 1 ] ] ] ]
gap> FirstObstructionDnr(tG,tH).Dnr; # Obs1Dnr~=C3
[ [ 3 ], [ [ 6 ], [ [ 2 ] ] ] ]
gap> tGs:=AllSubgroups(tG);;
gap> Length(tGs);
55
gap> tGsHNPfalse1:=Filtered(tGs,x->FirstObstructionDr(tG,x,tH).Dr[1]=[]);;
gap> tGsHNPfalse2:=Filtered(tGs,x->FirstObstructionDr(tG,x,tH).Dr[1]=[3]);;
gap> tGsHNPtrue1:=Filtered(tGs,x->FirstObstructionDr(tG,x,tH).Dr[1]=[2]);;
gap> tGsHNPtrue2:=Filtered(tGs,x->FirstObstructionDr(tG,x,tH).Dr[1]=[6]);;
gap> List([tGsHNPfalse1,tGsHNPfalse2,tGsHNPtrue1,tGsHNPtrue2],Length);
[ 26, 20, 7, 2 ]
gap> Sum(last);
55
gap> Collected(List(tGsHNPfalse1,x->StructureDescription(Image(ScG.epi,x))));
[ [ "1", 2 ], [ "C2", 21 ], [ "C4", 3 ] ]
gap> Collected(List(tGsHNPfalse2,x->StructureDescription(Image(ScG.epi,x))));
[ [ "C3", 8 ], [ "S3", 12 ] ]
gap> Collected(List(tGsHNPtrue1,x->StructureDescription(Image(ScG.epi,x))));
[ [ "C2 x C2", 4 ], [ "D8", 3 ] ]
gap> Collected(List(tGsHNPtrue2,x->StructureDescription(Image(ScG.epi,x))));
[ [ "A4", 1 ], [ "S4", 1 ] ]
\end{verbatim}
}~\\\vspace*{-4mm}

(1-3-5) $G=8T37\simeq \PSL_3(\bF_2)\simeq \PSL_2(\bF_7)$. 
{\small 
\begin{verbatim}
gap> Read("HNP.gap");
gap> G:=TransitiveGroup(8,37); # G=8T37=PSL(3,2)=PSL(2,7)
L(8)=PSL(2,7)
gap> H:=Stabilizer(G,1); # H=C7:C3
Group([ (2,3,6)(5,8,7), (3,7,8)(4,5,6) ])
gap> StructureDescription(H);
"C7 : C3"
gap> FirstObstructionN(G).ker; # Obs1N=C3
[ [ 3 ], [ [ 3 ], [ [ 1 ] ] ] ]
gap> FirstObstructionDnr(G).Dnr; # Obs1Dnr=C3
[ [ 3 ], [ [ 3 ], [ [ 1 ] ] ] ]
gap> SchurMultPcpGroup(G); # M(G)=C2: Schur multiplier of G
[ 2 ]
gap> ScG:=SchurCoverG(G);
rec( SchurCover := Group([ (1,2,4,8)(3,6,9,12)(5,10,14,11)(7,13,15,16), 
   (1,3,7,4,9,15)(2,5,11,8,14,10)(6,12)(13,16) ]), Tid := [ 16, 715 ], 
  epi := [ (1,2,4,8)(3,6,9,12)(5,10,14,11)(7,13,15,16), 
      (1,3,7,4,9,15)(2,5,11,8,14,10)(6,12)(13,16) ] -> 
    [ (1,2)(3,5)(4,7)(6,8), (2,6,7)(3,5,4) ] )
gap> tG:=ScG.SchurCover; # tG=G=SL(2,7) is a Schur cover of G
Group([ (1,2,4,8)(3,6,9,12)(5,10,14,11)(7,13,15,16), (1,3,7,4,9,15)
(2,5,11,8,14,10)(6,12)(13,16) ])
gap> tH:=PreImage(ScG.epi,H); # tH=H~=C2x(C7:C3)
Group([ (1,11,16)(3,5,15)(4,10,13)(7,9,14), (1,13,14,4,16,5)(2,10,7,8,11,15)
(3,9)(6,12) ])
gap> StructureDescription(tG);
"SL(2,7)"
gap> StructureDescription(tH);
"C2 x (C7 : C3)"
gap> FirstObstructionN(tG,tH).ker; # Obs1N~=C6
[ [ 6 ], [ [ 6 ], [ [ 1 ] ] ] ]
gap> FirstObstructionDnr(tG,tH).Dnr; # Obs1Dnr~=C3
[ [ 3 ], [ [ 6 ], [ [ 2 ] ] ] ]
gap> tGs:=AllSubgroups(tG);;
gap> Length(tGs);
224
gap> tGsHNPfalse1:=Filtered(tGs,x->FirstObstructionDr(tG,x,tH).Dr[1]=[]);;
gap> tGsHNPfalse2:=Filtered(tGs,x->FirstObstructionDr(tG,x,tH).Dr[1]=[3]);;
gap> tGsHNPtrue1:=Filtered(tGs,x->FirstObstructionDr(tG,x,tH).Dr[1]=[2]);;
gap> tGsHNPtrue2:=Filtered(tGs,x->FirstObstructionDr(tG,x,tH).Dr[1]=[6]);;
gap> List([tGsHNPfalse1,tGsHNPfalse2,tGsHNPtrue1,tGsHNPtrue2],Length);
[ 60, 100, 35, 29 ]
gap> Sum(last);
224
gap> Collected(List(tGsHNPfalse1,x->StructureDescription(Image(ScG.epi,x))));
[ [ "1", 2 ], [ "C2", 21 ], [ "C4", 21 ], [ "C7", 16 ] ]
gap> Collected(List(tGsHNPfalse2,x->StructureDescription(Image(ScG.epi,x))));
[ [ "C3", 56 ], [ "C7 : C3", 16 ], [ "S3", 28 ] ]
gap> Collected(List(tGsHNPtrue1,x->StructureDescription(Image(ScG.epi,x))));
[ [ "C2 x C2", 14 ], [ "D8", 21 ] ]
gap> Collected(List(tGsHNPtrue2,x->StructureDescription(Image(ScG.epi,x))));
[ [ "A4", 14 ], [ "PSL(3,2)", 1 ], [ "S4", 14 ] ]
\end{verbatim}
}~\\\vspace*{-4mm}

(1-4-1) $G=8T31\simeq ((C_2)^4\rtimes C_2)\rtimes C_2$. 
{\small 
\begin{verbatim}
gap> Read("HNP.gap");
gap> G:=TransitiveGroup(8,31); # G=8T31=(C2^4:C2):C2
[2^4]E(4)
gap> GeneratorsOfGroup(G);
[ (4,8), (1,8)(2,3)(4,5)(6,7), (1,3)(2,8)(4,6)(5,7) ]
gap> H:=Stabilizer(G,1); # H=C2xC2xC2
Group([ (4,8), (2,6), (3,7) ])
gap> FirstObstructionN(G).ker; # Obs1N=C2xC2
[ [ 2, 2 ], [ [ 2, 2, 2 ], [ [ 1, 0, 1 ], [ 0, 1, 1 ] ] ] ]
gap> FirstObstructionDnr(G).Dnr; # Obs1Dnr=C2xC2
[ [ 2, 2 ], [ [ 2, 2, 2 ], [ [ 1, 0, 1 ], [ 0, 1, 1 ] ] ] ]
gap> SchurMultPcpGroup(G); # M(G)=C2xC2xC2: Schur multiplier of G
[ 2, 2, 2, 2 ]
gap> cGs:=MinimalStemExtensions(G);; # 15 minimal stem extensions
gap> for cG in cGs do
> bG:=cG.MinimalStemExtension;
> bH:=PreImage(cG.epi,H);
> Print(KerResH3Z(bG,bH));
> od;
[ [  ], [ [ 2, 2, 2, 2, 2 ], [  ] ] ]
[ [ 2 ], [ [ 2, 2, 2, 2 ], [ [ 0, 0, 1, 0 ] ] ] ]
[ [ 2 ], [ [ 2, 2, 2 ], [ [ 1, 0, 0 ] ] ] ]
[ [ 2 ], [ [ 2, 2, 2, 2 ], [ [ 0, 0, 1, 0 ] ] ] ]
[ [ 2 ], [ [ 2, 2, 2, 2 ], [ [ 0, 0, 1, 0 ] ] ] ]
[ [ 2 ], [ [ 2, 2, 2, 2 ], [ [ 0, 0, 0, 1 ] ] ] ]
[ [ 2 ], [ [ 2, 2, 2 ], [ [ 0, 0, 1 ] ] ] ]
[ [ 2 ], [ [ 2, 2, 2, 2 ], [ [ 0, 1, 0, 1 ] ] ] ]
[ [ 2 ], [ [ 2, 2, 2 ], [ [ 1, 0, 1 ] ] ] ]
[ [ 2 ], [ [ 2, 2, 2, 2 ], [ [ 0, 0, 1, 1 ] ] ] ]
[ [ 2 ], [ [ 2, 2, 4 ], [ [ 0, 0, 2 ] ] ] ]
[ [ 2 ], [ [ 2, 2, 2, 2 ], [ [ 0, 0, 1, 1 ] ] ] ]
[ [ 2 ], [ [ 2, 2, 2 ], [ [ 0, 1, 1 ] ] ] ]
[ [ 2 ], [ [ 2, 2, 2 ], [ [ 1, 1, 1 ] ] ] ]
[ [ 2 ], [ [ 2, 2, 2 ], [ [ 1, 1, 1 ] ] ] ]
gap> for cG in cGs do
> bG:=cG.MinimalStemExtension;
> bH:=PreImage(cG.epi,H);
> Print(FirstObstructionN(bG,bH).ker[1]);
> Print(FirstObstructionDnr(bG,bH).Dnr[1]);
> Print("\n");
> od;
[ 2, 2, 2 ][ 2, 2 ]
[ 2, 2 ][ 2, 2 ]
[ 2, 2 ][ 2, 2 ]
[ 2, 2 ][ 2, 2 ]
[ 2, 2 ][ 2, 2 ]
[ 2, 2 ][ 2, 2 ]
[ 2, 2 ][ 2, 2 ]
[ 2, 2 ][ 2, 2 ]
[ 2, 2 ][ 2, 2 ]
[ 2, 2 ][ 2, 2 ]
[ 2, 2 ][ 2, 2 ]
[ 2, 2 ][ 2, 2 ]
[ 2, 2 ][ 2, 2 ]
[ 2, 2 ][ 2, 2 ]
[ 2, 2 ][ 2, 2 ]
gap> cG:=cGs[1];;
gap> bG:=cG.MinimalStemExtension; # bG=G- is a minimal stem extension of G
<permutation group of size 128 with 7 generators>
gap> bH:=PreImage(cG.epi,H); # bH=H-
<permutation group of size 16 with 4 generators>
gap> FirstObstructionN(bG,bH).ker; # Obs1N-=C2xC2xC2
[ [ 2, 2, 2 ], 
  [ [ 2, 2, 2, 2 ], [ [ 1, 0, 1, 0 ], [ 0, 1, 1, 0 ], [ 0, 0, 0, 1 ] ] ] ]
gap> FirstObstructionDnr(bG,bH).Dnr; # Obs1Dnr-=C2xC2
[ [ 2, 2 ], [ [ 2, 2, 2, 2 ], [ [ 1, 0, 1, 0 ], [ 0, 1, 1, 0 ] ] ] ]
gap> bGs:=AllSubgroups(bG);;
gap> Length(bGs);
896
gap> bGsHNPfalse:=Filtered(bGs,x->Filtered(FirstObstructionDr(bG,x,bH).Dr[2][2],
> y->y[4]=1)=[]);;
gap> Length(bGsHNPfalse);
855
gap> bGsHNPtrue:=Filtered(bGs,x->Filtered(FirstObstructionDr(bG,x,bH).Dr[2][2],
> y->y[4]=1)<>[]);;
gap> Length(bGsHNPtrue);
41
gap> Collected(List(bGsHNPfalse,x->StructureDescription(Image(cG.epi,x))));
[ [ "(C2 x C2 x C2 x C2) : C2", 19 ], [ "(C4 x C2) : C2", 45 ], [ "1", 2 ], 
  [ "C2", 73 ], [ "C2 x C2", 241 ], [ "C2 x C2 x C2", 154 ], 
  [ "C2 x C2 x C2 x C2", 17 ], [ "C2 x D8", 57 ], [ "C4", 54 ], 
  [ "C4 x C2", 45 ], [ "D8", 146 ], [ "Q8", 2 ] ]
gap> Collected(List(bGsHNPtrue,x->StructureDescription(Image(cG.epi,x))));
[ [ "((C2 x C2 x C2 x C2) : C2) : C2", 1 ], 
  [ "(C2 x C2 x C2) : (C2 x C2)", 1 ], [ "(C2 x C2 x C2) : C4", 11 ], 
  [ "(C4 x C2) : C2", 6 ], [ "C2 x C2", 8 ], [ "C2 x C2 x C2", 2 ], 
  [ "C2 x D8", 6 ], [ "C4 x C2", 6 ] ]
gap> GsHNPfalse:=Set(bGsHNPfalse,x->Image(cG.epi,x));;
gap> Length(GsHNPfalse);
192
gap> GsHNPtrue:=Set(bGsHNPtrue,x->Image(cG.epi,x));;
gap> Length(GsHNPtrue);
33
gap> Intersection(GsHNPfalse,GsHNPtrue);
[  ]
gap> GsHNPtrueMin:=Filtered(GsHNPtrue,x->Length(Filtered(GsHNPtrue,
> y->IsSubgroup(x,y)))=1);
[ Group([ (), (1,5)(2,6), (2,6)(4,8), (2,6)(3,7), (1,6)(2,5)(3,8)(4,7), (1,3)
  (2,8,6,4)(5,7) ]), Group([ (), (1,5)(2,6), (2,6)(4,8), (2,6)(3,7), (1,3)
  (2,8)(4,6)(5,7), (1,8,5,4)(2,3)(6,7) ]), Group([ (), (1,5)(2,6), (2,6)
  (4,8), (2,6)(3,7), (1,8,5,4)(2,3,6,7), (1,3)(2,8,6,4)(5,7) ]), Group([ (1,3)
  (2,4)(5,7)(6,8), (), (1,4)(2,3)(5,8)(6,7) ]), Group([ (1,7,5,3)
  (2,8,6,4), (), (1,8,5,4)(2,7,6,3), (1,5)(2,6)(3,7)(4,8) ]), Group([ (1,7)
  (2,8)(3,5)(4,6), (), (1,8)(2,7)(3,6)(4,5) ]), Group([ (1,3)(2,8)(4,6)
  (5,7), (), (1,8)(2,3)(4,5)(6,7) ]), Group([ (1,7,5,3)(2,4,6,8), (), (1,8,5,
   4)(2,3,6,7), (1,5)(2,6)(3,7)(4,8) ]), Group([ (1,7)(2,4)(3,5)
  (6,8), (), (1,4)(2,7)(3,6)(5,8) ]), Group([ (1,3)(2,4)(5,7)(6,8), (), (1,8,
   5,4)(2,3,6,7), (1,5)(2,6)(3,7)(4,8) ]), Group([ (1,7,5,3)(2,8,6,4), (), (1,
   4)(2,7)(3,6)(5,8), (1,5)(2,6)(3,7)(4,8) ]), Group([ (1,3)(2,8)(4,6)
  (5,7), (), (1,8,5,4)(2,7,6,3), (1,5)(2,6)(3,7)(4,8) ]), Group([ (1,7,5,3)
  (2,4,6,8), (), (1,4)(2,3)(5,8)(6,7), (1,5)(2,6)(3,7)(4,8) ]), Group([ (1,3)
  (2,4)(5,7)(6,8), (), (1,8)(2,7)(3,6)(4,5) ]), Group([ (1,3)(2,8)(4,6)
  (5,7), (), (1,4)(2,7)(3,6)(5,8) ]), Group([ (1,7)(2,8)(3,5)(4,6), (), (1,4)
  (2,3)(5,8)(6,7) ]), Group([ (1,7)(2,4)(3,5)(6,8), (), (1,8)(2,3)(4,5)
  (6,7) ]) ]
gap> Length(GsHNPtrueMin);
17
gap> List(GsHNPtrueMin,IdSmallGroup);
[ [ 32, 6 ], [ 32, 6 ], [ 32, 6 ], [ 4, 2 ], [ 8, 2 ], [ 4, 2 ], [ 4, 2 ], 
  [ 8, 2 ], [ 4, 2 ], [ 8, 2 ], [ 8, 2 ], [ 8, 2 ], [ 8, 2 ], [ 4, 2 ], 
  [ 4, 2 ], [ 4, 2 ], [ 4, 2 ] ]
gap> Collected(List(GsHNPfalse,x->Filtered(GsHNPtrueMin,y->IsSubgroup(x,y))));
[ [ [  ], 192 ] ]
gap> Gs:=AllSubgroups(G);;
gap> Length(Gs);
225
gap> GsC2xC2:=Filtered(Gs,x->IdSmallGroup(x)=[4,2]);;
gap> Length(GsC2xC2);
61
gap> GsC4xC2:=Filtered(Gs,x->IdSmallGroup(x)=[8,2]);;
gap> Length(GsC4xC2);
15
gap> Gs32_6:=Filtered(Gs,x->IdSmallGroup(x)=[32,6]);;
gap> Length(Gs32_6);
3
gap> GsHNPfalseC2xC2:=Filtered(GsHNPfalse,x->IdSmallGroup(x)=[4,2]);
[ Group([ (3,7)(4,8), (3,7) ]), Group([ (2,6)(4,8), (2,6)(4,8), (2,6) ]), 
  Group([ (2,6)(3,7), (2,6)(3,7)(4,8) ]), Group([ (1,5)(4,8), (1,5)
  (4,8), (4,8) ]), Group([ (1,5)(3,7), (4,8) ]), Group([ (1,5)(2,6)
  (4,8), (1,5)(2,6) ]), Group([ (1,5)(2,6)(3,7)(4,8), (1,5)(2,6)(3,7)
  (4,8), (4,8) ]), Group([ (2,6)(3,7), (2,6)(3,7), (2,6) ]), Group([ (2,6)
  (4,8), (2,6)(4,8), (3,7) ]), Group([ (1,5)(3,7), (1,5)(3,7), (3,7) ]), 
  Group([ (3,7), (1,5)(4,8) ]), Group([ (1,5)(2,6), (3,7) ]), Group([ (1,5)
  (2,6)(3,7)(4,8), (1,5)(2,6)(3,7)(4,8), (3,7) ]), Group([ (3,7)
  (4,8), (2,6) ]), Group([ (3,7)(4,8), (2,6)(3,7) ]), Group([ (1,2)(3,4)(5,6)
  (7,8), (3,7)(4,8), (3,7)(4,8) ]), Group([ (3,7)(4,8), (1,5) ]), 
  Group([ (1,5)(4,8), (1,5)(4,8), (3,7)(4,8) ]), Group([ (1,5)(2,6)(3,7)
  (4,8), (1,5)(2,6)(3,7)(4,8), (3,7)(4,8) ]), Group([ (1,5)(2,6)(4,8), (3,7)
  (4,8) ]), Group([ (1,6)(2,5)(3,8)(4,7), (3,7)(4,8), (3,7)(4,8) ]), 
  Group([ (1,5)(2,6), (1,5) ]), Group([ (1,5)(4,8), (2,6) ]), Group([ (1,5)
  (3,7), (1,5)(3,7), (2,6) ]), Group([ (1,5)(2,6)(3,7)(4,8), (1,5)(2,6)(3,7)
  (4,8), (2,6) ]), Group([ (1,3)(2,8)(4,6)(5,7), (2,6)(4,8), (2,6)(4,8) ]), 
  Group([ (2,6)(4,8), (2,6)(4,8), (1,5) ]), Group([ (1,5)(4,8), (1,5)
  (4,8), (1,5)(2,6) ]), Group([ (1,5)(2,6)(3,7)(4,8), (1,5)(2,6)(3,7)
  (4,8), (2,6)(4,8) ]), Group([ (2,6)(4,8), (2,6)(4,8), (1,5)(2,6)(3,7) ]), 
  Group([ (1,7)(2,4)(3,5)(6,8), (2,6)(4,8) ]), Group([ (2,6)(3,7), (1,4)(2,3)
  (5,8)(6,7) ]), Group([ (2,6)(3,7), (2,6)(3,7), (1,5) ]), Group([ (1,5)(2,6)
  (3,7)(4,8), (1,5)(4,8), (2,6)(3,7) ]), Group([ (2,6)(3,7), (2,6)(3,7), (1,5)
  (2,6) ]), Group([ (2,6)(3,7), (2,6)(3,7), (1,5)(2,6)(4,8) ]), Group([ (1,8)
  (2,3)(4,5)(6,7), (2,6)(3,7) ]), Group([ (1,5)(2,6)(3,7)(4,8), (1,5)(2,6)
  (3,7)(4,8), (1,5) ]), Group([ (1,5)(4,8), (1,5)(4,8), (2,6)(3,7)(4,8) ]), 
  Group([ (1,5)(3,7), (1,5)(3,7), (2,6)(3,7)(4,8) ]), Group([ (1,5)
  (2,6), (1,5)(2,6), (2,6)(3,7)(4,8) ]), Group([ (1,2)(3,4)(5,6)(7,8), (1,5)
  (2,6), (1,5)(2,6) ]), Group([ (1,6)(2,5)(3,8)(4,7), (1,5)(2,6)(3,7)
  (4,8), (1,5)(2,6)(3,7)(4,8) ]), Group([ (1,6)(2,5)(3,8)(4,7), (1,5)
  (2,6), (1,5)(2,6) ]), Group([ (1,2)(3,8)(4,7)(5,6), (1,5)(2,6)(3,7)
  (4,8), (1,5)(2,6)(3,7)(4,8) ]), Group([ (1,3)(2,4)(5,7)(6,8), (1,5)
  (3,7), (1,5)(3,7) ]), Group([ (1,3)(2,4)(5,7)(6,8), (1,5)(2,6)(3,7)
  (4,8), (1,5)(2,6)(3,7)(4,8) ]), Group([ (1,3)(2,8)(4,6)(5,7), (1,5)
  (3,7), (1,5)(3,7) ]), Group([ (1,3)(2,8)(4,6)(5,7), (1,5)(2,6)(3,7)
  (4,8), (1,5)(2,6)(3,7)(4,8) ]), Group([ (1,8)(2,3)(4,5)(6,7), (1,5)
  (4,8) ]), Group([ (1,8)(2,7)(3,6)(4,5), (1,5)(2,6)(3,7)(4,8) ]), 
  Group([ (1,5)(4,8), (1,4)(2,7)(3,6)(5,8), (1,5)(4,8) ]), Group([ (1,5)(2,6)
  (3,7)(4,8), (1,4)(2,7)(3,6)(5,8), (1,5)(2,6)(3,7)(4,8) ]) ]
gap> Length(GsHNPfalseC2xC2);
53
gap> GsHNPtrueC2xC2:=Filtered(GsHNPtrue,x->IdSmallGroup(x)=[4,2]);
[ Group([ (1,3)(2,4)(5,7)(6,8), (), (1,4)(2,3)(5,8)(6,7) ]), Group([ (1,7)
  (2,8)(3,5)(4,6), (), (1,8)(2,7)(3,6)(4,5) ]), Group([ (1,3)(2,8)(4,6)
  (5,7), (), (1,8)(2,3)(4,5)(6,7) ]), Group([ (1,7)(2,4)(3,5)(6,8), (), (1,4)
  (2,7)(3,6)(5,8) ]), Group([ (1,3)(2,4)(5,7)(6,8), (), (1,8)(2,7)(3,6)
  (4,5) ]), Group([ (1,3)(2,8)(4,6)(5,7), (), (1,4)(2,7)(3,6)(5,8) ]), 
  Group([ (1,7)(2,8)(3,5)(4,6), (), (1,4)(2,3)(5,8)(6,7) ]), Group([ (1,7)
  (2,4)(3,5)(6,8), (), (1,8)(2,3)(4,5)(6,7) ]) ]
gap> Length(GsHNPtrueC2xC2);
8
gap> Collected(List(GsHNPfalseC2xC2,x->List(Orbits(x),Length)));
[ [ [ 2, 2 ], 6 ], [ [ 2, 2, 2 ], 16 ], [ [ 2, 2, 2, 2 ], 13 ], 
  [ [ 2, 2, 4 ], 1 ], [ [ 2, 4, 2 ], 5 ], [ [ 4, 2, 2 ], 6 ], [ [ 4, 4 ], 6 ] ]
gap> Collected(List(GsHNPtrueC2xC2,x->List(Orbits(x),Length)));
[ [ [ 4, 4 ], 8 ] ]
gap> GsHNPfalse44C2xC2:=Filtered(GsHNPfalseC2xC2,
> x->List(Orbits(x,[1..8]),Length)=[4,4]);
[ Group([ (1,6)(2,5)(3,8)(4,7), (1,5)(2,6)(3,7)(4,8), (1,5)(2,6)(3,7)
  (4,8) ]), Group([ (1,2)(3,8)(4,7)(5,6), (1,5)(2,6)(3,7)(4,8), (1,5)(2,6)
  (3,7)(4,8) ]), Group([ (1,3)(2,4)(5,7)(6,8), (1,5)(2,6)(3,7)(4,8), (1,5)
  (2,6)(3,7)(4,8) ]), Group([ (1,3)(2,8)(4,6)(5,7), (1,5)(2,6)(3,7)
  (4,8), (1,5)(2,6)(3,7)(4,8) ]), Group([ (1,8)(2,7)(3,6)(4,5), (1,5)(2,6)
  (3,7)(4,8) ]), Group([ (1,5)(2,6)(3,7)(4,8), (1,4)(2,7)(3,6)(5,8), (1,5)
  (2,6)(3,7)(4,8) ]) ]
gap> List(GsHNPtrueC2xC2,Elements);
[ [ (), (1,2)(3,4)(5,6)(7,8), (1,3)(2,4)(5,7)(6,8), (1,4)(2,3)(5,8)(6,7) ], 
  [ (), (1,2)(3,4)(5,6)(7,8), (1,7)(2,8)(3,5)(4,6), (1,8)(2,7)(3,6)(4,5) ], 
  [ (), (1,2)(3,8)(4,7)(5,6), (1,3)(2,8)(4,6)(5,7), (1,8)(2,3)(4,5)(6,7) ], 
  [ (), (1,2)(3,8)(4,7)(5,6), (1,4)(2,7)(3,6)(5,8), (1,7)(2,4)(3,5)(6,8) ], 
  [ (), (1,3)(2,4)(5,7)(6,8), (1,6)(2,5)(3,8)(4,7), (1,8)(2,7)(3,6)(4,5) ], 
  [ (), (1,3)(2,8)(4,6)(5,7), (1,4)(2,7)(3,6)(5,8), (1,6)(2,5)(3,4)(7,8) ], 
  [ (), (1,4)(2,3)(5,8)(6,7), (1,6)(2,5)(3,8)(4,7), (1,7)(2,8)(3,5)(4,6) ], 
  [ (), (1,6)(2,5)(3,4)(7,8), (1,7)(2,4)(3,5)(6,8), (1,8)(2,3)(4,5)(6,7) ] ]
gap> List(GsHNPfalseC2xC2,Elements);
[ [ (), (4,8), (3,7), (3,7)(4,8) ], [ (), (4,8), (2,6), (2,6)(4,8) ], 
  [ (), (4,8), (2,6)(3,7), (2,6)(3,7)(4,8) ], [ (), (4,8), (1,5), (1,5)(4,8) ], 
  [ (), (4,8), (1,5)(3,7), (1,5)(3,7)(4,8) ], 
  [ (), (4,8), (1,5)(2,6), (1,5)(2,6)(4,8) ], 
  [ (), (4,8), (1,5)(2,6)(3,7), (1,5)(2,6)(3,7)(4,8) ], 
  [ (), (3,7), (2,6), (2,6)(3,7) ], 
  [ (), (3,7), (2,6)(4,8), (2,6)(3,7)(4,8) ], 
  [ (), (3,7), (1,5), (1,5)(3,7) ], 
  [ (), (3,7), (1,5)(4,8), (1,5)(3,7)(4,8) ], 
  [ (), (3,7), (1,5)(2,6), (1,5)(2,6)(3,7) ], 
  [ (), (3,7), (1,5)(2,6)(4,8), (1,5)(2,6)(3,7)(4,8) ], 
  [ (), (3,7)(4,8), (2,6), (2,6)(3,7)(4,8) ], 
  [ (), (3,7)(4,8), (2,6)(4,8), (2,6)(3,7) ], 
  [ (), (3,7)(4,8), (1,2)(3,4)(5,6)(7,8), (1,2)(3,8)(4,7)(5,6) ], 
  [ (), (3,7)(4,8), (1,5), (1,5)(3,7)(4,8) ], 
  [ (), (3,7)(4,8), (1,5)(4,8), (1,5)(3,7) ], 
  [ (), (3,7)(4,8), (1,5)(2,6), (1,5)(2,6)(3,7)(4,8) ], 
  [ (), (3,7)(4,8), (1,5)(2,6)(4,8), (1,5)(2,6)(3,7) ], 
  [ (), (3,7)(4,8), (1,6)(2,5)(3,4)(7,8), (1,6)(2,5)(3,8)(4,7) ], 
  [ (), (2,6), (1,5), (1,5)(2,6) ], 
  [ (), (2,6), (1,5)(4,8), (1,5)(2,6)(4,8) ], 
  [ (), (2,6), (1,5)(3,7), (1,5)(2,6)(3,7) ], 
  [ (), (2,6), (1,5)(3,7)(4,8), (1,5)(2,6)(3,7)(4,8) ], 
  [ (), (2,6)(4,8), (1,3)(2,4)(5,7)(6,8), (1,3)(2,8)(4,6)(5,7) ], 
  [ (), (2,6)(4,8), (1,5), (1,5)(2,6)(4,8) ], 
  [ (), (2,6)(4,8), (1,5)(4,8), (1,5)(2,6) ], 
  [ (), (2,6)(4,8), (1,5)(3,7), (1,5)(2,6)(3,7)(4,8) ], 
  [ (), (2,6)(4,8), (1,5)(3,7)(4,8), (1,5)(2,6)(3,7) ], 
  [ (), (2,6)(4,8), (1,7)(2,4)(3,5)(6,8), (1,7)(2,8)(3,5)(4,6) ], 
  [ (), (2,6)(3,7), (1,4)(2,3)(5,8)(6,7), (1,4)(2,7)(3,6)(5,8) ], 
  [ (), (2,6)(3,7), (1,5), (1,5)(2,6)(3,7) ], 
  [ (), (2,6)(3,7), (1,5)(4,8), (1,5)(2,6)(3,7)(4,8) ], 
  [ (), (2,6)(3,7), (1,5)(3,7), (1,5)(2,6) ], 
  [ (), (2,6)(3,7), (1,5)(3,7)(4,8), (1,5)(2,6)(4,8) ], 
  [ (), (2,6)(3,7), (1,8)(2,3)(4,5)(6,7), (1,8)(2,7)(3,6)(4,5) ], 
  [ (), (2,6)(3,7)(4,8), (1,5), (1,5)(2,6)(3,7)(4,8) ], 
  [ (), (2,6)(3,7)(4,8), (1,5)(4,8), (1,5)(2,6)(3,7) ], 
  [ (), (2,6)(3,7)(4,8), (1,5)(3,7), (1,5)(2,6)(4,8) ], 
  [ (), (2,6)(3,7)(4,8), (1,5)(3,7)(4,8), (1,5)(2,6) ], 
  [ (), (1,2)(3,4)(5,6)(7,8), (1,5)(2,6), (1,6)(2,5)(3,4)(7,8) ], 
  [ (), (1,2)(3,4)(5,6)(7,8), (1,5)(2,6)(3,7)(4,8), (1,6)(2,5)(3,8)(4,7) ], 
  [ (), (1,2)(3,8)(4,7)(5,6), (1,5)(2,6), (1,6)(2,5)(3,8)(4,7) ], 
  [ (), (1,2)(3,8)(4,7)(5,6), (1,5)(2,6)(3,7)(4,8), (1,6)(2,5)(3,4)(7,8) ], 
  [ (), (1,3)(2,4)(5,7)(6,8), (1,5)(3,7), (1,7)(2,4)(3,5)(6,8) ], 
  [ (), (1,3)(2,4)(5,7)(6,8), (1,5)(2,6)(3,7)(4,8), (1,7)(2,8)(3,5)(4,6) ], 
  [ (), (1,3)(2,8)(4,6)(5,7), (1,5)(3,7), (1,7)(2,8)(3,5)(4,6) ], 
  [ (), (1,3)(2,8)(4,6)(5,7), (1,5)(2,6)(3,7)(4,8), (1,7)(2,4)(3,5)(6,8) ], 
  [ (), (1,4)(2,3)(5,8)(6,7), (1,5)(4,8), (1,8)(2,3)(4,5)(6,7) ], 
  [ (), (1,4)(2,3)(5,8)(6,7), (1,5)(2,6)(3,7)(4,8), (1,8)(2,7)(3,6)(4,5) ], 
  [ (), (1,4)(2,7)(3,6)(5,8), (1,5)(4,8), (1,8)(2,7)(3,6)(4,5) ], 
  [ (), (1,4)(2,7)(3,6)(5,8), (1,5)(2,6)(3,7)(4,8), (1,8)(2,3)(4,5)(6,7) ] ]
gap> List(GsHNPfalse44C2xC2,Elements);
[ [ (), (1,2)(3,4)(5,6)(7,8), (1,5)(2,6)(3,7)(4,8), (1,6)(2,5)(3,8)(4,7) ], 
  [ (), (1,2)(3,8)(4,7)(5,6), (1,5)(2,6)(3,7)(4,8), (1,6)(2,5)(3,4)(7,8) ], 
  [ (), (1,3)(2,4)(5,7)(6,8), (1,5)(2,6)(3,7)(4,8), (1,7)(2,8)(3,5)(4,6) ], 
  [ (), (1,3)(2,8)(4,6)(5,7), (1,5)(2,6)(3,7)(4,8), (1,7)(2,4)(3,5)(6,8) ], 
  [ (), (1,4)(2,3)(5,8)(6,7), (1,5)(2,6)(3,7)(4,8), (1,8)(2,7)(3,6)(4,5) ], 
  [ (), (1,4)(2,7)(3,6)(5,8), (1,5)(2,6)(3,7)(4,8), (1,8)(2,3)(4,5)(6,7) ] ]
gap> ZG:=Centre(G);
Group([ (1,5)(2,6)(3,7)(4,8) ])
gap> List(GsHNPtrueC2xC2,x->Intersection(x,ZG));
[ Group(()), Group(()), Group(()), Group(()), Group(()), Group(()), 
  Group(()), Group(()) ]
gap> List(GsHNPfalse44C2xC2,x->Intersection(x,ZG));
[ Group([ (1,5)(2,6)(3,7)(4,8) ]), Group([ (1,5)(2,6)(3,7)(4,8) ]), 
  Group([ (1,5)(2,6)(3,7)(4,8) ]), Group([ (1,5)(2,6)(3,7)(4,8) ]), 
  Group([ (1,5)(2,6)(3,7)(4,8) ]), Group([ (1,5)(2,6)(3,7)(4,8) ]) ]
gap> DG:=DerivedSubgroup(G);
Group([ (1,5)(4,8), (2,6)(4,8), (3,7)(4,8) ])
gap> StructureDescription(DG);
"C2 x C2 x C2"
gap> Collected(List(GsHNPfalseC2xC2,x->Order(Intersection(DG,x))));
[ [ 2, 46 ], [ 4, 7 ] ]
gap> Collected(List(GsHNPtrueC2xC2,x->Order(Intersection(DG,x))));
[ [ 1, 8 ] ]
gap> GsHNPfalseC4xC2:=Filtered(GsHNPfalse,x->IdSmallGroup(x)=[8,2]);
[ Group([ (1,2)(3,4,7,8)(5,6), (1,5)(2,6)(3,7)(4,8), (1,5)(2,6)(3,7)
  (4,8), (1,5)(2,6) ]), Group([ (1,2,5,6)(3,8)(4,7), (1,5)(2,6)(3,7)
  (4,8), (1,5)(2,6)(3,7)(4,8), (3,7)(4,8) ]), Group([ (1,2,5,6)
  (3,8,7,4), (1,5)(2,6)(3,7)(4,8), (1,5)(2,6)(3,7)(4,8), (3,7)(4,8) ]), 
  Group([ (1,3)(2,8,6,4)(5,7), (1,5)(2,6)(3,7)(4,8), (1,5)(2,6)(3,7)
  (4,8), (2,6)(4,8) ]), Group([ (1,3,5,7)(2,4)(6,8), (1,5)(2,6)(3,7)
  (4,8), (1,5)(2,6)(3,7)(4,8), (1,5)(3,7) ]), Group([ (1,7,5,3)
  (2,8,6,4), (1,5)(2,6)(3,7)(4,8), (1,5)(2,6)(3,7)(4,8), (2,6)(4,8) ]), 
  Group([ (1,5)(2,6)(3,7)(4,8), (1,8,5,4)(2,3)(6,7), (1,5)(4,8), (2,6)
  (3,7) ]), Group([ (1,4)(2,7,6,3)(5,8), (1,5)(4,8), (2,6)(3,7) ]), 
  Group([ (1,5)(4,8), (1,8,5,4)(2,3,6,7), (1,5)(2,6)(3,7)(4,8) ]) ]
gap> Length(GsHNPfalseC4xC2);
9
gap> GsHNPtrueC4xC2:=Filtered(GsHNPtrue,x->IdSmallGroup(x)=[8,2]);
[ Group([ (1,7,5,3)(2,8,6,4), (), (1,8,5,4)(2,7,6,3), (1,5)(2,6)(3,7)
  (4,8) ]), Group([ (1,7,5,3)(2,4,6,8), (), (1,8,5,4)(2,3,6,7), (1,5)(2,6)
  (3,7)(4,8) ]), Group([ (1,3)(2,4)(5,7)(6,8), (), (1,8,5,4)(2,3,6,7), (1,5)
  (2,6)(3,7)(4,8) ]), Group([ (1,7,5,3)(2,8,6,4), (), (1,4)(2,7)(3,6)
  (5,8), (1,5)(2,6)(3,7)(4,8) ]), Group([ (1,3)(2,8)(4,6)(5,7), (), (1,8,5,4)
  (2,7,6,3), (1,5)(2,6)(3,7)(4,8) ]), Group([ (1,7,5,3)(2,4,6,8), (), (1,4)
  (2,3)(5,8)(6,7), (1,5)(2,6)(3,7)(4,8) ]) ]
gap> Length(GsHNPtrueC4xC2);
6
gap> Collected(List(GsHNPfalseC4xC2,x->List(Orbits(x),Length)));
[ [ [ 4, 4 ], 9 ] ]
gap> Collected(List(GsHNPtrueC4xC2,x->List(Orbits(x),Length)));
[ [ [ 8 ], 6 ] ]
gap> Collected(List(GsHNPfalseC4xC2,x->Order(Intersection(DG,x))));
[ [ 4, 9 ] ]
gap> Collected(List(GsHNPtrueC4xC2,x->Order(Intersection(DG,x))));
[ [ 2, 6 ] ]
\end{verbatim}
}~\\\vspace*{-4mm}

(1-4-2) $G=8T38\simeq (((C_2)^4\rtimes C_2)\rtimes C_2)\rtimes C_3$. 
{\small 
\begin{verbatim}
gap> Read("HNP.gap");
gap> G:=TransitiveGroup(8,38); # G=8T38=((C2^4:C2):C2):C3
[2^4]A(4)
gap> GeneratorsOfGroup(G);
[ (4,8), (1,8)(2,3)(4,5)(6,7), (1,2,3)(5,6,7) ]
gap> H:=Stabilizer(G,1); # H=C2xA4
Group([ (4,8), (2,6), (2,8,3)(4,7,6) ])
gap> FirstObstructionN(G).ker; # Obs1N=1
[ [  ], [ [ 6 ], [  ] ] ]
gap> SchurMultPcpGroup(G); # M(G)=C2xC2: Schur multiplier of G
[ 2, 2 ]
gap> cGs:=MinimalStemExtensions(G);; # 3 minimal stem extensions
gap> for cG in cGs do
> bG:=cG.MinimalStemExtension;
> bH:=PreImage(cG.epi,H);
> Print(KerResH3Z(bG,bH));
> Print("\n");
> od;
[ [ 2 ], [ [ 2 ], [ [ 1 ] ] ] ]
[ [  ], [ [ 2, 2 ], [  ] ] ]
[ [ 2 ], [ [ 2 ], [ [ 1 ] ] ] ]
gap> for cG in cGs do
> bG:=cG.MinimalStemExtension;
> bH:=PreImage(cG.epi,H);
> Print(FirstObstructionN(bG,bH).ker[1]);
> Print(FirstObstructionDnr(bG,bH).Dnr[1]);
> Print("\n");
> od;
[  ][  ]
[ 2 ][  ]
[  ][  ]
gap> cG:=cGs[2];;
gap> bG:=cG.MinimalStemExtension; # bG=G- is a minimal stem extension of G
<permutation group of size 384 with 8 generators>
gap> bH:=PreImage(cG.epi,H); # bH=H-
<permutation group of size 48 with 3 generators>
gap> FirstObstructionN(bG,bH).ker; # Obs1N-=C2
[ [ 2 ], [ [ 2, 6 ], [ [ 1, 0 ] ] ] ]
gap> FirstObstructionDnr(bG,bH).Dnr; # Obs1Dnr-=1
[ [  ], [ [ 2, 6 ], [  ] ] ]
gap> bGs:=AllSubgroups(bG);;
gap> Length(bGs);
1002
gap> bGsHNPfalse:=Filtered(bGs,x->FirstObstructionDr(bG,x,bH).Dr[1]=[]);;
gap> Length(bGsHNPfalse);
951
gap> bGsHNPtrue:=Filtered(bGs,x->FirstObstructionDr(bG,x,bH).Dr[1]=[2]);;
gap> Length(bGsHNPtrue);
51
gap> Collected(List(bGsHNPfalse,x->StructureDescription(Image(cG.epi,x))));
[ [ "(C2 x C2 x C2 x C2) : C2", 3 ], [ "(C4 x C2) : C2", 45 ], [ "1", 2 ], 
  [ "A4", 8 ], [ "C2", 57 ], [ "C2 x A4", 36 ], [ "C2 x C2", 193 ], 
  [ "C2 x C2 x A4", 20 ], [ "C2 x C2 x C2", 138 ], 
  [ "C2 x C2 x C2 x C2", 17 ], [ "C2 x D8", 9 ], [ "C3", 32 ], [ "C4", 54 ], 
  [ "C4 x C2", 45 ], [ "C6", 144 ], [ "C6 x C2", 80 ], [ "D8", 42 ], 
  [ "Q8", 10 ], [ "SL(2,3)", 16 ] ]
gap> Collected(List(bGsHNPtrue,x->StructureDescription(Image(cG.epi,x))));
[ [ "(((C2 x C2 x C2 x C2) : C2) : C2) : C3", 1 ], 
  [ "((C2 x C2 x C2 x C2) : C2) : C2", 1 ], 
  [ "((C2 x C2 x C2) : (C2 x C2)) : C3", 1 ], 
  [ "(C2 x C2 x C2) : (C2 x C2)", 1 ], [ "(C2 x C2 x C2) : C4", 3 ], 
  [ "(C4 x C2) : C2", 6 ], [ "A4", 8 ], [ "C2 x A4", 8 ], [ "C2 x C2", 8 ], 
  [ "C2 x C2 x C2", 2 ], [ "C2 x D8", 6 ], [ "C4 x C2", 6 ] ]
gap> GsHNPfalse:=Set(bGsHNPfalse,x->Image(cG.epi,x));;
gap> Length(GsHNPfalse);
300
gap> GsHNPtrue:=Set(bGsHNPtrue,x->Image(cG.epi,x));;
gap> Length(GsHNPtrue);
51
gap> Intersection(GsHNPfalse,GsHNPtrue);
[  ]
gap> GsHNPtrueMin:=Filtered(GsHNPtrue,x->Length(Filtered(GsHNPtrue,
> y->IsSubgroup(x,y)))=1);
[ Group([ (1,2)(3,4)(5,6)(7,8), (1,3)(2,8,6,4)(5,7), (), (2,6)(4,8), (1,5)
  (3,7), (1,5)(4,8) ]), Group([ (1,3)(2,4)(5,7)(6,8), (1,8,5,4)(2,3)
  (6,7), (), (1,5)(4,8), (2,6)(3,7), (3,7)(4,8) ]), Group([ (1,4)(2,3)(5,8)
  (6,7), (1,2)(3,8,7,4)(5,6), (), (3,7)(4,8), (1,5)(2,6), (2,6)(4,8) ]), 
  Group([ (1,3)(2,4)(5,7)(6,8), (1,4)(2,3)(5,8)(6,7), () ]), Group([ (1,2)
  (3,4)(5,6)(7,8), (1,7,5,3)(2,8,6,4), (), (1,5)(2,6)(3,7)(4,8) ]), 
  Group([ (1,7)(2,8)(3,5)(4,6), (1,8)(2,7)(3,6)(4,5), () ]), Group([ (1,3)
  (2,8)(4,6)(5,7), (1,8)(2,3)(4,5)(6,7), () ]), Group([ (1,2)(3,8)(4,7)
  (5,6), (1,7,5,3)(2,4,6,8), (), (1,5)(2,6)(3,7)(4,8) ]), Group([ (1,7)(2,4)
  (3,5)(6,8), (1,4)(2,7)(3,6)(5,8), () ]), Group([ (1,3)(2,4)(5,7)
  (6,8), (1,8,5,4)(2,3,6,7), (), (1,5)(2,6)(3,7)(4,8) ]), Group([ (1,8)(2,3)
  (4,5)(6,7), (1,2,5,6)(3,4,7,8), (), (1,5)(2,6)(3,7)(4,8) ]), Group([ (1,3)
  (2,8)(4,6)(5,7), (1,4,5,8)(2,3,6,7), (), (1,5)(2,6)(3,7)(4,8) ]), 
  Group([ (1,4)(2,3)(5,8)(6,7), (1,2,5,6)(3,8,7,4), (), (1,5)(2,6)(3,7)
  (4,8) ]), Group([ (1,3)(2,4)(5,7)(6,8), (1,8)(2,7)(3,6)(4,5), () ]), 
  Group([ (1,3)(2,8)(4,6)(5,7), (1,4)(2,7)(3,6)(5,8), () ]), Group([ (1,6)
  (2,5)(3,8)(4,7), (1,7)(2,8)(3,5)(4,6), () ]), Group([ (1,6)(2,5)(3,4)
  (7,8), (1,7)(2,4)(3,5)(6,8), () ]) ]
gap> Length(GsHNPtrueMin);
17
gap> List(GsHNPtrueMin,IdSmallGroup);
[ [ 32, 6 ], [ 32, 6 ], [ 32, 6 ], [ 4, 2 ], [ 8, 2 ], [ 4, 2 ], [ 4, 2 ], 
  [ 8, 2 ], [ 4, 2 ], [ 8, 2 ], [ 8, 2 ], [ 8, 2 ], [ 8, 2 ], [ 4, 2 ], 
  [ 4, 2 ], [ 4, 2 ], [ 4, 2 ] ]
gap> Collected(List(GsHNPfalse,x->Filtered(GsHNPtrueMin,y->IsSubgroup(x,y))));
[ [ [  ], 300 ] ]
gap> Gs:=AllSubgroups(G);;
gap> Length(Gs);
351
gap> GsC2xC2:=Filtered(Gs,x->IdSmallGroup(x)=[4,2]);;
gap> Length(GsC2xC2);
61
gap> GsC4xC2:=Filtered(Gs,x->IdSmallGroup(x)=[8,2]);;
gap> Length(GsC4xC2);
15
gap> Gs32_6:=Filtered(Gs,x->IdSmallGroup(x)=[32,6]);;
gap> Length(Gs32_6);
3
gap> GsHNPfalseC2xC2:=Filtered(GsHNPfalse,x->IdSmallGroup(x)=[4,2]);
[ Group([ (3,7), (4,8) ]), Group([ (2,6), (4,8) ]), Group([ (), (4,8), (2,6)
  (3,7) ]), Group([ (4,8), (1,5)(4,8) ]), Group([ (4,8), (1,5)(3,7) ]), 
  Group([ (1,5)(2,6), (4,8) ]), Group([ (4,8), (1,5)(2,6)(3,7)(4,8) ]), 
  Group([ (), (2,6), (3,7) ]), Group([ (), (3,7), (2,6)(4,8) ]), 
  Group([ (), (1,5), (3,7) ]), Group([ (3,7), (1,5)(4,8) ]), 
  Group([ (3,7), (1,5)(2,6) ]), Group([ (), (3,7), (1,5)(2,6)(4,8) ]), 
  Group([ (), (2,6), (3,7)(4,8) ]), Group([ (), (2,6)(4,8), (3,7)(4,8) ]), 
  Group([ (1,2)(3,4)(5,6)(7,8), (), (3,7)(4,8) ]), Group([ (3,7)
  (4,8), (1,5) ]), Group([ (3,7)(4,8), (1,5)(4,8) ]), Group([ (), (3,7)
  (4,8), (1,5)(2,6) ]), Group([ (1,5)(2,6)(4,8), (3,7)(4,8) ]), Group([ (1,6)
  (2,5)(3,8)(4,7), (), (3,7)(4,8) ]), Group([ (), (1,5), (2,6) ]), 
  Group([ (1,5)(4,8), (2,6) ]), Group([ (2,6), (1,5)(3,7) ]), 
  Group([ (), (2,6), (1,5)(3,7)(4,8) ]), Group([ (1,3)(2,4)(5,7)
  (6,8), (), (2,6)(4,8) ]), Group([ (), (1,5), (2,6)(4,8) ]), 
  Group([ (), (2,6)(4,8), (1,5)(2,6) ]), Group([ (), (2,6)(4,8), (1,5)
  (3,7) ]), Group([ (2,6)(4,8), (1,5)(2,6)(3,7) ]), Group([ (1,7)(2,8)(3,5)
  (4,6), (), (2,6)(4,8) ]), Group([ (1,4)(2,3)(5,8)(6,7), (), (2,6)(3,7) ]), 
  Group([ (), (1,5), (2,6)(3,7) ]), Group([ (), (1,5)(4,8), (2,6)(3,7) ]), 
  Group([ (), (1,5)(2,6), (2,6)(3,7) ]), Group([ (1,5)(3,7)(4,8), (2,6)
  (3,7) ]), Group([ (1,8)(2,3)(4,5)(6,7), (), (2,6)(3,7) ]), Group([ (2,6)
  (3,7)(4,8), (1,5) ]), Group([ (), (1,5)(4,8), (2,6)(3,7)(4,8) ]), 
  Group([ (), (1,5)(3,7), (2,6)(3,7)(4,8) ]), Group([ (), (1,5)(2,6), (2,6)
  (3,7)(4,8) ]), Group([ (1,2)(3,4)(5,6)(7,8), (), (1,5)(2,6) ]), 
  Group([ (1,2)(3,4)(5,6)(7,8), (), (1,5)(2,6)(3,7)(4,8) ]), Group([ (1,2)
  (3,8)(4,7)(5,6), (), (1,5)(2,6) ]), Group([ (1,2)(3,8)(4,7)(5,6), (), (1,5)
  (2,6)(3,7)(4,8) ]), Group([ (1,3)(2,4)(5,7)(6,8), (), (1,5)(3,7) ]), 
  Group([ (1,3)(2,4)(5,7)(6,8), (), (1,5)(2,6)(3,7)(4,8) ]), Group([ (1,3)
  (2,8)(4,6)(5,7), (), (1,5)(3,7) ]), Group([ (1,3)(2,8)(4,6)(5,7), (), (1,5)
  (2,6)(3,7)(4,8) ]), Group([ (1,4)(2,3)(5,8)(6,7), (), (1,5)(4,8) ]), 
  Group([ (1,4)(2,3)(5,8)(6,7), (), (1,5)(2,6)(3,7)(4,8) ]), Group([ (1,8)
  (2,7)(3,6)(4,5), (), (1,5)(4,8) ]), Group([ (1,8)(2,3)(4,5)(6,7), (), (1,5)
  (2,6)(3,7)(4,8) ]) ]
gap> Length(GsHNPfalseC2xC2);
53
gap> GsHNPtrueC2xC2:=Filtered(GsHNPtrue,x->IdSmallGroup(x)=[4,2]);
[ Group([ (1,3)(2,4)(5,7)(6,8), (1,4)(2,3)(5,8)(6,7), () ]), Group([ (1,7)
  (2,8)(3,5)(4,6), (1,8)(2,7)(3,6)(4,5), () ]), Group([ (1,3)(2,8)(4,6)
  (5,7), (1,8)(2,3)(4,5)(6,7), () ]), Group([ (1,7)(2,4)(3,5)(6,8), (1,4)(2,7)
  (3,6)(5,8), () ]), Group([ (1,3)(2,4)(5,7)(6,8), (1,8)(2,7)(3,6)
  (4,5), () ]), Group([ (1,3)(2,8)(4,6)(5,7), (1,4)(2,7)(3,6)(5,8), () ]), 
  Group([ (1,6)(2,5)(3,8)(4,7), (1,7)(2,8)(3,5)(4,6), () ]), Group([ (1,6)
  (2,5)(3,4)(7,8), (1,7)(2,4)(3,5)(6,8), () ]) ]
gap> Length(GsHNPtrueC2xC2);
8
gap> Collected(List(GsHNPfalseC2xC2,x->List(Orbits(x),Length)));
[ [ [ 2, 2 ], 6 ], [ [ 2, 2, 2 ], 16 ], [ [ 2, 2, 2, 2 ], 13 ], 
  [ [ 2, 2, 4 ], 1 ], [ [ 2, 4, 2 ], 5 ], [ [ 4, 2, 2 ], 6 ], [ [ 4, 4 ], 6 ] ]
gap> Collected(List(GsHNPtrueC2xC2,x->List(Orbits(x),Length)));
[ [ [ 4, 4 ], 8 ] ]
gap> GsHNPfalse44C2xC2:=Filtered(GsHNPfalseC2xC2,
> x->List(Orbits(x,[1..8]),Length)=[4,4]);
[ Group([ (1,2)(3,4)(5,6)(7,8), (), (1,5)(2,6)(3,7)(4,8) ]), Group([ (1,2)
  (3,8)(4,7)(5,6), (), (1,5)(2,6)(3,7)(4,8) ]), Group([ (1,3)(2,4)(5,7)
  (6,8), (), (1,5)(2,6)(3,7)(4,8) ]), Group([ (1,3)(2,8)(4,6)(5,7), (), (1,5)
  (2,6)(3,7)(4,8) ]), Group([ (1,4)(2,3)(5,8)(6,7), (), (1,5)(2,6)(3,7)
  (4,8) ]), Group([ (1,8)(2,3)(4,5)(6,7), (), (1,5)(2,6)(3,7)(4,8) ]) ]
gap> List(GsHNPtrueC2xC2,Elements);
[ [ (), (1,2)(3,4)(5,6)(7,8), (1,3)(2,4)(5,7)(6,8), (1,4)(2,3)(5,8)(6,7) ], 
  [ (), (1,2)(3,4)(5,6)(7,8), (1,7)(2,8)(3,5)(4,6), (1,8)(2,7)(3,6)(4,5) ], 
  [ (), (1,2)(3,8)(4,7)(5,6), (1,3)(2,8)(4,6)(5,7), (1,8)(2,3)(4,5)(6,7) ], 
  [ (), (1,2)(3,8)(4,7)(5,6), (1,4)(2,7)(3,6)(5,8), (1,7)(2,4)(3,5)(6,8) ], 
  [ (), (1,3)(2,4)(5,7)(6,8), (1,6)(2,5)(3,8)(4,7), (1,8)(2,7)(3,6)(4,5) ], 
  [ (), (1,3)(2,8)(4,6)(5,7), (1,4)(2,7)(3,6)(5,8), (1,6)(2,5)(3,4)(7,8) ], 
  [ (), (1,4)(2,3)(5,8)(6,7), (1,6)(2,5)(3,8)(4,7), (1,7)(2,8)(3,5)(4,6) ], 
  [ (), (1,6)(2,5)(3,4)(7,8), (1,7)(2,4)(3,5)(6,8), (1,8)(2,3)(4,5)(6,7) ] ]
gap> List(GsHNPfalseC2xC2,Elements);
[ [ (), (4,8), (3,7), (3,7)(4,8) ], [ (), (4,8), (2,6), (2,6)(4,8) ], 
  [ (), (4,8), (2,6)(3,7), (2,6)(3,7)(4,8) ], [ (), (4,8), (1,5), (1,5)(4,8) ],  
  [ (), (4,8), (1,5)(3,7), (1,5)(3,7)(4,8) ], 
  [ (), (4,8), (1,5)(2,6), (1,5)(2,6)(4,8) ], 
  [ (), (4,8), (1,5)(2,6)(3,7), (1,5)(2,6)(3,7)(4,8) ], 
  [ (), (3,7), (2,6), (2,6)(3,7) ], 
  [ (), (3,7), (2,6)(4,8), (2,6)(3,7)(4,8) ], 
  [ (), (3,7), (1,5), (1,5)(3,7) ], 
  [ (), (3,7), (1,5)(4,8), (1,5)(3,7)(4,8) ], 
  [ (), (3,7), (1,5)(2,6), (1,5)(2,6)(3,7) ], 
  [ (), (3,7), (1,5)(2,6)(4,8), (1,5)(2,6)(3,7)(4,8) ], 
  [ (), (3,7)(4,8), (2,6), (2,6)(3,7)(4,8) ], 
  [ (), (3,7)(4,8), (2,6)(4,8), (2,6)(3,7) ], 
  [ (), (3,7)(4,8), (1,2)(3,4)(5,6)(7,8), (1,2)(3,8)(4,7)(5,6) ], 
  [ (), (3,7)(4,8), (1,5), (1,5)(3,7)(4,8) ], 
  [ (), (3,7)(4,8), (1,5)(4,8), (1,5)(3,7) ], 
  [ (), (3,7)(4,8), (1,5)(2,6), (1,5)(2,6)(3,7)(4,8) ], 
  [ (), (3,7)(4,8), (1,5)(2,6)(4,8), (1,5)(2,6)(3,7) ], 
  [ (), (3,7)(4,8), (1,6)(2,5)(3,4)(7,8), (1,6)(2,5)(3,8)(4,7) ], 
  [ (), (2,6), (1,5), (1,5)(2,6) ], 
  [ (), (2,6), (1,5)(4,8), (1,5)(2,6)(4,8) ], 
  [ (), (2,6), (1,5)(3,7), (1,5)(2,6)(3,7) ], 
  [ (), (2,6), (1,5)(3,7)(4,8), (1,5)(2,6)(3,7)(4,8) ], 
  [ (), (2,6)(4,8), (1,3)(2,4)(5,7)(6,8), (1,3)(2,8)(4,6)(5,7) ], 
  [ (), (2,6)(4,8), (1,5), (1,5)(2,6)(4,8) ], 
  [ (), (2,6)(4,8), (1,5)(4,8), (1,5)(2,6) ], 
  [ (), (2,6)(4,8), (1,5)(3,7), (1,5)(2,6)(3,7)(4,8) ], 
  [ (), (2,6)(4,8), (1,5)(3,7)(4,8), (1,5)(2,6)(3,7) ], 
  [ (), (2,6)(4,8), (1,7)(2,4)(3,5)(6,8), (1,7)(2,8)(3,5)(4,6) ], 
  [ (), (2,6)(3,7), (1,4)(2,3)(5,8)(6,7), (1,4)(2,7)(3,6)(5,8) ], 
  [ (), (2,6)(3,7), (1,5), (1,5)(2,6)(3,7) ], 
  [ (), (2,6)(3,7), (1,5)(4,8), (1,5)(2,6)(3,7)(4,8) ], 
  [ (), (2,6)(3,7), (1,5)(3,7), (1,5)(2,6) ], 
  [ (), (2,6)(3,7), (1,5)(3,7)(4,8), (1,5)(2,6)(4,8) ], 
  [ (), (2,6)(3,7), (1,8)(2,3)(4,5)(6,7), (1,8)(2,7)(3,6)(4,5) ], 
  [ (), (2,6)(3,7)(4,8), (1,5), (1,5)(2,6)(3,7)(4,8) ], 
  [ (), (2,6)(3,7)(4,8), (1,5)(4,8), (1,5)(2,6)(3,7) ], 
  [ (), (2,6)(3,7)(4,8), (1,5)(3,7), (1,5)(2,6)(4,8) ], 
  [ (), (2,6)(3,7)(4,8), (1,5)(3,7)(4,8), (1,5)(2,6) ], 
  [ (), (1,2)(3,4)(5,6)(7,8), (1,5)(2,6), (1,6)(2,5)(3,4)(7,8) ], 
  [ (), (1,2)(3,4)(5,6)(7,8), (1,5)(2,6)(3,7)(4,8), (1,6)(2,5)(3,8)(4,7) ], 
  [ (), (1,2)(3,8)(4,7)(5,6), (1,5)(2,6), (1,6)(2,5)(3,8)(4,7) ], 
  [ (), (1,2)(3,8)(4,7)(5,6), (1,5)(2,6)(3,7)(4,8), (1,6)(2,5)(3,4)(7,8) ], 
  [ (), (1,3)(2,4)(5,7)(6,8), (1,5)(3,7), (1,7)(2,4)(3,5)(6,8) ], 
  [ (), (1,3)(2,4)(5,7)(6,8), (1,5)(2,6)(3,7)(4,8), (1,7)(2,8)(3,5)(4,6) ], 
  [ (), (1,3)(2,8)(4,6)(5,7), (1,5)(3,7), (1,7)(2,8)(3,5)(4,6) ], 
  [ (), (1,3)(2,8)(4,6)(5,7), (1,5)(2,6)(3,7)(4,8), (1,7)(2,4)(3,5)(6,8) ], 
  [ (), (1,4)(2,3)(5,8)(6,7), (1,5)(4,8), (1,8)(2,3)(4,5)(6,7) ], 
  [ (), (1,4)(2,3)(5,8)(6,7), (1,5)(2,6)(3,7)(4,8), (1,8)(2,7)(3,6)(4,5) ], 
  [ (), (1,4)(2,7)(3,6)(5,8), (1,5)(4,8), (1,8)(2,7)(3,6)(4,5) ], 
  [ (), (1,4)(2,7)(3,6)(5,8), (1,5)(2,6)(3,7)(4,8), (1,8)(2,3)(4,5)(6,7) ] ]
gap> List(GsHNPfalse44C2xC2,Elements);
[ [ (), (1,2)(3,4)(5,6)(7,8), (1,5)(2,6)(3,7)(4,8), (1,6)(2,5)(3,8)(4,7) ], 
  [ (), (1,2)(3,8)(4,7)(5,6), (1,5)(2,6)(3,7)(4,8), (1,6)(2,5)(3,4)(7,8) ], 
  [ (), (1,3)(2,4)(5,7)(6,8), (1,5)(2,6)(3,7)(4,8), (1,7)(2,8)(3,5)(4,6) ], 
  [ (), (1,3)(2,8)(4,6)(5,7), (1,5)(2,6)(3,7)(4,8), (1,7)(2,4)(3,5)(6,8) ], 
  [ (), (1,4)(2,3)(5,8)(6,7), (1,5)(2,6)(3,7)(4,8), (1,8)(2,7)(3,6)(4,5) ], 
  [ (), (1,4)(2,7)(3,6)(5,8), (1,5)(2,6)(3,7)(4,8), (1,8)(2,3)(4,5)(6,7) ] ]
gap> ZG:=Centre(G);
Group([ (1,5)(2,6)(3,7)(4,8) ])
gap> List(GsHNPtrueC2xC2,x->Intersection(x,ZG));
[ Group(()), Group(()), Group(()), Group(()), Group(()), Group(()), 
  Group(()), Group(()) ]
gap> List(GsHNPfalse44C2xC2,x->Intersection(x,ZG));
[ Group([ (1,5)(2,6)(3,7)(4,8) ]), Group([ (1,5)(2,6)(3,7)(4,8) ]), 
  Group([ (1,5)(2,6)(3,7)(4,8) ]), Group([ (1,5)(2,6)(3,7)(4,8) ]), 
  Group([ (1,5)(2,6)(3,7)(4,8) ]), Group([ (1,5)(2,6)(3,7)(4,8) ]) ]
gap> Syl2G:=SylowSubgroup(G,2);
Group([ (4,8), (1,2)(3,8)(4,7)(5,6), (1,8)(2,3)(4,5)(6,7), (1,5)(2,6), (1,5)
(4,8), (1,5)(2,6)(3,7)(4,8) ])
gap> IsNormal(G,Syl2G);
true
gap> DSyl2G:=DerivedSubgroup(Syl2G);
Group([ (2,6)(3,7), (1,5)(4,8), (2,6)(4,8) ])
gap> StructureDescription(DSyl2G);
"C2 x C2 x C2"
gap> Collected(List(GsHNPfalseC2xC2,x->Order(Intersection(DSyl2G,x))));
[ [ 2, 46 ], [ 4, 7 ] ]
gap> Collected(List(GsHNPtrueC2xC2,x->Order(Intersection(DSyl2G,x))));
[ [ 1, 8 ] ]
gap> GsHNPfalseC4xC2:=Filtered(GsHNPfalse,x->IdSmallGroup(x)=[8,2]);
[ Group([ (1,2)(3,8,7,4)(5,6), (), (3,7)(4,8), (1,5)(2,6) ]), 
  Group([ (1,6,5,2)(3,4)(7,8), (), (1,5)(2,6), (3,7)(4,8) ]), 
  Group([ (1,6,5,2)(3,8,7,4), (), (3,7)(4,8), (1,5)(2,6) ]), Group([ (1,3)
  (2,8,6,4)(5,7), (2,6)(4,8), (1,5)(2,6)(3,7)(4,8) ]), Group([ (1,7,5,3)(2,8)
  (4,6), (1,5)(3,7), (1,5)(2,6)(3,7)(4,8) ]), Group([ (1,3,5,7)
  (2,8,6,4), (), (2,6)(4,8), (1,5)(3,7) ]), Group([ (1,8,5,4)(2,3)
  (6,7), (), (1,5)(4,8), (2,6)(3,7) ]), Group([ (1,4)(2,3,6,7)(5,8), (2,6)
  (3,7), (1,5)(2,6)(3,7)(4,8) ]), Group([ (1,8,5,4)(2,7,6,3), (), (1,5)
  (4,8), (2,6)(3,7) ]) ]
gap> Length(GsHNPfalseC4xC2);
9
gap> GsHNPtrueC4xC2:=Filtered(GsHNPtrue,x->IdSmallGroup(x)=[8,2]);
[ Group([ (1,2)(3,4)(5,6)(7,8), (1,7,5,3)(2,8,6,4), (), (1,5)(2,6)(3,7)
  (4,8) ]), Group([ (1,2)(3,8)(4,7)(5,6), (1,7,5,3)(2,4,6,8), (), (1,5)(2,6)
  (3,7)(4,8) ]), Group([ (1,3)(2,4)(5,7)(6,8), (1,8,5,4)(2,3,6,7), (), (1,5)
  (2,6)(3,7)(4,8) ]), Group([ (1,8)(2,3)(4,5)(6,7), (1,2,5,6)
  (3,4,7,8), (), (1,5)(2,6)(3,7)(4,8) ]), Group([ (1,3)(2,8)(4,6)
  (5,7), (1,4,5,8)(2,3,6,7), (), (1,5)(2,6)(3,7)(4,8) ]), Group([ (1,4)(2,3)
  (5,8)(6,7), (1,2,5,6)(3,8,7,4), (), (1,5)(2,6)(3,7)(4,8) ]) ]
gap> Length(GsHNPtrueC4xC2);
6
gap> Collected(List(GsHNPfalseC4xC2,x->List(Orbits(x),Length)));
[ [ [ 4, 4 ], 9 ] ]
gap> Collected(List(GsHNPtrueC4xC2,x->List(Orbits(x),Length)));
[ [ [ 8 ], 6 ] ]
gap> Collected(List(GsHNPfalseC4xC2,x->Order(Intersection(DSyl2G,x))));
[ [ 4, 9 ] ]
gap> Collected(List(GsHNPtrueC4xC2,x->Order(Intersection(DSyl2G,x))));
[ [ 2, 6 ] ]
gap> Syl2G=TransitiveGroup(8,31);
true
\end{verbatim}
}
\end{example}

\smallskip
\begin{example}[$G=8Tm$ $(m=9,11,15,19,22,32)$]\label{ex8-2}
~{}\vspace*{-2mm}\\

(2-1) $G=8T9\simeq D_4\times C_2$. 
{\small 
\begin{verbatim}
gap> Read("HNP.gap");
gap> G:=TransitiveGroup(8,9); # G=8T9=D4xC2
E(8):2=D(4)[x]2
gap> GeneratorsOfGroup(G); # H=C2
[ (1,8)(2,3)(4,5)(6,7), (1,3)(2,8)(4,6)(5,7), (1,5)(2,6)(3,7)(4,8), 
  (4,5)(6,7) ]
gap> H:=Stabilizer(G,1); # H=C2
Group([ (4,5)(6,7) ])
gap> FirstObstructionN(G).ker; # Obs1N=1
[ [  ], [ [ 2 ], [  ] ] ]
gap> SchurMultPcpGroup(G); # M(G)=C2xC2xC2: Schur multiplicer of G
[ 2, 2, 2 ]
gap> cGs:=MinimalStemExtensions(G);; # 7 minimal stem extensions
gap> for cG in cGs do
> bG:=cG.MinimalStemExtension;
> bH:=PreImage(cG.epi,H);
> Print(KerResH3Z(bG,bH));
> Print("\n");
> od;
[ [ 2 ], [ [ 2, 2, 2 ], [ [ 0, 1, 1 ] ] ] ]
[ [  ], [ [ 2, 2, 2, 2 ], [  ] ] ]
[ [ 2 ], [ [ 2, 2 ], [ [ 0, 1 ] ] ] ]
[ [ 2 ], [ [ 2, 2, 2, 2 ], [ [ 0, 0, 0, 1 ] ] ] ]
[ [ 2 ], [ [ 2, 2 ], [ [ 1, 1 ] ] ] ]
[ [ 2 ], [ [ 2, 2, 2 ], [ [ 0, 0, 1 ] ] ] ]
[ [ 2 ], [ [ 2, 2 ], [ [ 1, 0 ] ] ] ]
gap> for cG in cGs do
> bG:=cG.MinimalStemExtension;
> bH:=PreImage(cG.epi,H);
> Print(FirstObstructionN(bG,bH).ker[1]);
> Print(FirstObstructionDnr(bG,bH).Dnr[1]);
> Print("\n");
> od;
[ 2 ][ 2 ]
[ 2 ][  ]
[ 2 ][ 2 ]
[ 2 ][ 2 ]
[ 2 ][ 2 ]
[ 2 ][ 2 ]
[ 2 ][ 2 ]
gap> cG:=cGs[2];
rec( MinimalStemExtension := Group([ (2,3)(4,6), (1,2)(3,5)(4,7)(6,8), (2,4)
  (3,6) ]), Tid := [ 8, 18 ], 
  epi := [ (2,3)(4,6), (1,2)(3,5)(4,7)(6,8), (2,4)(3,6) ] -> 
    [ (4,5)(6,7), (1,5)(2,6)(3,7)(4,8), (1,3)(2,8)(4,6)(5,7) ] )
gap> bG:=cG.MinimalStemExtension; # bG=G- is a minimal stem extension of G
Group([ (2,3)(4,6), (1,2)(3,5)(4,7)(6,8), (2,4)(3,6) ])
gap> bH:=PreImage(cG.epi,H); # bH=H-
Group([ (2,3)(4,6), (1,7)(2,4)(3,6)(5,8) ])
gap> FirstObstructionN(bG,bH).ker; # Obs1N-=C2
[ [ 2 ], [ [ 2, 2 ], [ [ 0, 1 ] ] ] ]
gap> FirstObstructionDnr(bG,bH).Dnr; # Obs1Dnr-=1
[ [  ], [ [ 2, 2 ], [  ] ] ]
gap> bGs:=AllSubgroups(bG);;
gap> Length(bGs);
106
gap> bGsHNPfalse:=Filtered(bGs,x->FirstObstructionDr(bG,x,bH).Dr[1]=[]);;
gap> Length(bGsHNPfalse);
99
gap> bGsHNPtrue:=Filtered(bGs,x->FirstObstructionDr(bG,x,bH).Dr[1]=[2]);;
gap> Length(bGsHNPtrue);
7
gap> Collected(List(bGsHNPfalse,x->StructureDescription(Image(cG.epi,x))));
[ [ "1", 2 ], [ "C2", 29 ], [ "C2 x C2", 41 ], [ "C2 x C2 x C2", 9 ], 
  [ "C4", 6 ], [ "D8", 12 ] ]
gap> Collected(List(bGsHNPtrue,x->StructureDescription(Image(cG.epi,x))));
[ [ "C2 x C2", 4 ], [ "C2 x C2 x C2", 1 ], [ "C2 x D8", 1 ], [ "C4 x C2", 1 ] 
 ]
gap> GsHNPfalse:=Set(bGsHNPfalse,x->Image(cG.epi,x));;
gap> Length(GsHNPfalse);
28
gap> GsHNPtrue:=Set(bGsHNPtrue,x->Image(cG.epi,x));;
gap> Length(GsHNPtrue);
7
gap> Intersection(GsHNPfalse,GsHNPtrue);
[  ]
gap> GsHNPtrueMin:=Filtered(GsHNPtrue,x->Length(Filtered(GsHNPtrue,
> y->IsSubgroup(x,y)))=1);
[ Group([ (), (1,8)(2,3)(4,5)(6,7), (1,3)(2,8)(4,6)(5,7), (1,4,8,5)
  (2,7,3,6) ]), Group([ (1,2)(3,8)(4,7)(5,6), (), (1,4)(2,7)(3,6)(5,8) ]), 
  Group([ (1,2)(3,8)(4,7)(5,6), (), (1,5)(2,6)(3,7)(4,8) ]), Group([ (1,3)
  (2,8)(4,6)(5,7), (), (1,4)(2,7)(3,6)(5,8) ]), Group([ (1,3)(2,8)(4,6)
  (5,7), (), (1,5)(2,6)(3,7)(4,8) ]) ]
gap> List(GsHNPtrueMin,IdSmallGroup);
[ [ 8, 2 ], [ 4, 2 ], [ 4, 2 ], [ 4, 2 ], [ 4, 2 ] ]
gap> Length(GsHNPtrueMin);
5
gap> Collected(List(GsHNPfalse,x->Filtered(GsHNPtrueMin,y->IsSubgroup(x,y))));
[ [ [  ], 28 ] ]
gap> Gs:=AllSubgroups(G);;
gap> Length(Gs);
35
gap> GsC2xC2:=Filtered(Gs,x->IdSmallGroup(x)=[4,2]);;
gap> Length(GsC2xC2);
13
gap> GsC4xC2:=Filtered(Gs,x->IdSmallGroup(x)=[8,2]);;
gap> Length(GsC4xC2);
1
gap> GsHNPfalseC2xC2:=Filtered(GsHNPfalse,x->IdSmallGroup(x)=[4,2]);
[ Group([ (1,2)(3,8)(4,7)(5,6), (4,5)(6,7) ]), Group([ (1,3)(2,8)(4,6)
  (5,7), (4,5)(6,7) ]), Group([ (), (1,8)(2,3)(4,5)(6,7), (4,5)(6,7) ]), 
  Group([ (1,3)(2,8)(4,6)(5,7), (1,8)(2,3) ]), Group([ (1,2)(3,8)(4,6)
  (5,7), (1,8)(2,3)(4,5)(6,7) ]), Group([ (1,3)(2,8)(4,6)(5,7), (1,8)(2,3)
  (4,5)(6,7) ]), Group([ (1,2)(3,8)(4,7)(5,6), (1,8)(2,3) ]), Group([ (1,5)
  (2,6)(3,7)(4,8), (1,8)(2,3)(4,5)(6,7) ]), Group([ (), (1,8)(2,3)(4,5)
  (6,7), (1,7)(2,4)(3,5)(6,8) ]) ]
gap> Length(GsHNPfalseC2xC2);
9
gap> GsHNPtrueC2xC2:=Filtered(GsHNPtrue,x->IdSmallGroup(x)=[4,2]);
[ Group([ (1,2)(3,8)(4,7)(5,6), (), (1,4)(2,7)(3,6)(5,8) ]), Group([ (1,2)
  (3,8)(4,7)(5,6), (), (1,5)(2,6)(3,7)(4,8) ]), Group([ (1,3)(2,8)(4,6)
  (5,7), (), (1,4)(2,7)(3,6)(5,8) ]), Group([ (1,3)(2,8)(4,6)(5,7), (), (1,5)
  (2,6)(3,7)(4,8) ]) ]
gap> Length(GsHNPtrueC2xC2);
4
gap> Collected(List(GsHNPfalseC2xC2,x->List(Orbits(x),Length)));
[ [ [ 2, 2, 2, 2 ], 1 ], [ [ 2, 2, 4 ], 2 ], [ [ 4, 2, 2 ], 2 ], 
  [ [ 4, 4 ], 4 ] ]
gap> Collected(List(GsHNPtrueC2xC2,x->List(Orbits(x),Length)));
[ [ [ 4, 4 ], 4 ] ]
gap> GsHNPfalse44C2xC2:=Filtered(GsHNPfalseC2xC2,x->List(Orbits(x,[1..8]),
> Length)=[4,4]);
[ Group([ (1,2)(3,8)(4,6)(5,7), (1,8)(2,3)(4,5)(6,7) ]), Group([ (1,3)(2,8)
  (4,6)(5,7), (1,8)(2,3)(4,5)(6,7) ]), Group([ (1,5)(2,6)(3,7)(4,8), (1,8)
  (2,3)(4,5)(6,7) ]), Group([ (), (1,8)(2,3)(4,5)(6,7), (1,7)(2,4)(3,5)
  (6,8) ]) ]
gap> List(GsHNPtrueC2xC2,Elements);
[ [ (), (1,2)(3,8)(4,7)(5,6), (1,4)(2,7)(3,6)(5,8), (1,7)(2,4)(3,5)(6,8) ], 
  [ (), (1,2)(3,8)(4,7)(5,6), (1,5)(2,6)(3,7)(4,8), (1,6)(2,5)(3,4)(7,8) ], 
  [ (), (1,3)(2,8)(4,6)(5,7), (1,4)(2,7)(3,6)(5,8), (1,6)(2,5)(3,4)(7,8) ], 
  [ (), (1,3)(2,8)(4,6)(5,7), (1,5)(2,6)(3,7)(4,8), (1,7)(2,4)(3,5)(6,8) ] ]
gap> List(GsHNPfalseC2xC2,Elements);
[ [ (), (4,5)(6,7), (1,2)(3,8)(4,6)(5,7), (1,2)(3,8)(4,7)(5,6) ], 
  [ (), (4,5)(6,7), (1,3)(2,8)(4,6)(5,7), (1,3)(2,8)(4,7)(5,6) ], 
  [ (), (4,5)(6,7), (1,8)(2,3), (1,8)(2,3)(4,5)(6,7) ], 
  [ (), (1,2)(3,8)(4,6)(5,7), (1,3)(2,8)(4,6)(5,7), (1,8)(2,3) ], 
  [ (), (1,2)(3,8)(4,6)(5,7), (1,3)(2,8)(4,7)(5,6), (1,8)(2,3)(4,5)(6,7) ], 
  [ (), (1,2)(3,8)(4,7)(5,6), (1,3)(2,8)(4,6)(5,7), (1,8)(2,3)(4,5)(6,7) ], 
  [ (), (1,2)(3,8)(4,7)(5,6), (1,3)(2,8)(4,7)(5,6), (1,8)(2,3) ], 
  [ (), (1,4)(2,7)(3,6)(5,8), (1,5)(2,6)(3,7)(4,8), (1,8)(2,3)(4,5)(6,7) ], 
  [ (), (1,6)(2,5)(3,4)(7,8), (1,7)(2,4)(3,5)(6,8), (1,8)(2,3)(4,5)(6,7) ] ]
gap> List(GsHNPfalse44C2xC2,Elements);
[ [ (), (1,2)(3,8)(4,6)(5,7), (1,3)(2,8)(4,7)(5,6), (1,8)(2,3)(4,5)(6,7) ], 
  [ (), (1,2)(3,8)(4,7)(5,6), (1,3)(2,8)(4,6)(5,7), (1,8)(2,3)(4,5)(6,7) ], 
  [ (), (1,4)(2,7)(3,6)(5,8), (1,5)(2,6)(3,7)(4,8), (1,8)(2,3)(4,5)(6,7) ], 
  [ (), (1,6)(2,5)(3,4)(7,8), (1,7)(2,4)(3,5)(6,8), (1,8)(2,3)(4,5)(6,7) ] ]
gap> DG:=DerivedSubgroup(G);
Group([ (1,8)(2,3)(4,5)(6,7) ])
gap> List(GsHNPtrueC2xC2,x->Intersection(x,DG));
[ Group(()), Group(()), Group(()), Group(()) ]
gap> List(GsHNPfalse44C2xC2,x->Intersection(x,DG));
[ Group([ (1,8)(2,3)(4,5)(6,7) ]), Group([ (1,8)(2,3)(4,5)(6,7) ]), 
  Group([ (1,8)(2,3)(4,5)(6,7) ]), Group([ (1,8)(2,3)(4,5)(6,7) ]) ]
\end{verbatim}
}~\\\vspace*{-4mm}

(2-2) $G=8T11\simeq (C_4\times C_2)\rtimes C_2$. 
{\small 
\begin{verbatim}
gap> Read("HNP.gap");
gap> G:=TransitiveGroup(8,11); # G=8T11=(C4xC2):C2
1/2[2^3]E(4)=Q_8:2
gap> GeneratorsOfGroup(G);
[ (1,5)(3,7), (1,3,5,7)(2,4,6,8), (1,4,5,8)(2,3,6,7) ]
gap> H:=Stabilizer(G,1); # H=C2
Group([ (2,6)(4,8) ])
gap> FirstObstructionN(G).ker; # Obs1N=1
[ [  ], [ [ 2 ], [  ] ] ]
gap> SchurMultPcpGroup(G); # M(G)=C2xC2: Schur multiplicer of G
[ 2, 2 ]
gap> cGs:=MinimalStemExtensions(G);; # 3 minimal stem extensions
gap> for cG in cGs do
> bG:=cG.MinimalStemExtension;
> bH:=PreImage(cG.epi,H);
> Print(KerResH3Z(bG,bH));
> Print("\n");
> od;
[ [ 2 ], [ [ 2, 2, 2 ], [ [ 0, 0, 1 ] ] ] ]
[ [  ], [ [ 2, 2 ], [  ] ] ]
[ [ 2 ], [ [ 2, 2 ], [ [ 0, 1 ] ] ] ]
gap> for cG in cGs do
> bG:=cG.MinimalStemExtension;
> bH:=PreImage(cG.epi,H);
> Print(FirstObstructionN(bG,bH).ker[1]);
> Print(FirstObstructionDnr(bG,bH).Dnr[1]);
> Print("\n");
> od;
[ 2 ][ 2 ]
[ 2 ][  ]
[ 2 ][ 2 ]
gap> cG:=cGs[2];
rec( MinimalStemExtension := <permutation group of size 32 with 5 generators>,
  epi := [ (1,5,6,16)(2,9,10,22)(3,12,13,25)(4,14,15,26)(7,18,19,29)(8,20,21,
        30)(11,23,24,31)(17,27,28,32), (1,3,6,13)(2,7,10,19)(4,12,15,25)(5,11,
        16,24)(8,18,21,29)(9,17,22,28)(14,23,26,31)(20,27,30,32), 
      (1,2)(3,7)(4,21)(5,22)(6,10)(8,15)(9,16)(11,28)(12,29)(13,19)(14,20)(17,
        24)(18,25)(23,27)(26,30)(31,32) ] -> 
    [ (1,4,5,8)(2,3,6,7), (1,3,5,7)(2,4,6,8), (2,6)(4,8) ] )
gap> bG:=cG.MinimalStemExtension; # bG=G- is a minimal stem extension of G
<permutation group of size 32 with 5 generators>
gap> bH:=PreImage(cG.epi,H); # bH=H-
Group([ (1,2)(3,7)(4,21)(5,22)(6,10)(8,15)(9,16)(11,28)(12,29)(13,19)(14,20)
(17,24)(18,25)(23,27)(26,30)(31,32), (1,26)(2,30)(3,31)(4,5)(6,14)(7,32)(8,9)
(10,20)(11,12)(13,23)(15,16)(17,18)(19,27)(21,22)(24,25)(28,29) ])
gap> FirstObstructionN(bG,bH).ker; # Obs1N-=C2
[ [ 2 ], [ [ 2, 2 ], [ [ 0, 1 ] ] ] ]
gap> FirstObstructionDnr(bG,bH).Dnr; # Obs1Dnr-=1
[ [  ], [ [ 2, 2 ], [  ] ] ]
gap> bGs:=AllSubgroups(bG);;
gap> Length(bGs);
58
gap> bGsHNPfalse:=Filtered(bGs,x->FirstObstructionDr(bG,x,bH).Dr[1]=[]);;
gap> Length(bGsHNPfalse);
55
gap> bGsHNPtrue:=Filtered(bGs,x->FirstObstructionDr(bG,x,bH).Dr[1]=[2]);;
gap> Length(bGsHNPtrue);
3
gap> Collected(List(bGsHNPfalse,x->StructureDescription(Image(cG.epi,x))));
[ [ "1", 2 ], [ "C2", 17 ], [ "C2 x C2", 11 ], [ "C4", 12 ], [ "C4 x C2", 5 ],
  [ "D8", 7 ], [ "Q8", 1 ] ]
gap> Collected(List(bGsHNPtrue,x->StructureDescription(Image(cG.epi,x))));
[ [ "(C4 x C2) : C2", 1 ], [ "C4 x C2", 2 ] ]
gap> GsHNPfalse:=Set(bGsHNPfalse,x->Image(cG.epi,x));;
gap> Length(GsHNPfalse);
20
gap> GsHNPtrue:=Set(bGsHNPtrue,x->Image(cG.epi,x));;
gap> Length(GsHNPtrue);
3
gap> Intersection(GsHNPfalse,GsHNPtrue);
[  ]
gap> GsHNPtrueMin:=Filtered(GsHNPtrue,x->Length(Filtered(GsHNPtrue,
> y->IsSubgroup(x,y)))=1);
[ Group([ (), (1,5)(2,6)(3,7)(4,8), (1,4,5,8)(2,3,6,7), (1,3,5,7)
  (2,4,6,8) ]), Group([ (), (1,5)(2,6)(3,7)(4,8), (1,3,5,7)(2,4,6,8), (1,8)
  (2,3)(4,5)(6,7) ]) ]
gap> Length(GsHNPtrueMin);
2
gap> List(GsHNPtrueMin,IdSmallGroup);
[ [ 8, 2 ], [ 8, 2 ] ]
gap> Collected(List(GsHNPfalse,x->Filtered(GsHNPtrueMin,y->IsSubgroup(x,y))));
[ [ [  ], 20 ] ]
gap> Gs:=AllSubgroups(G);;
gap> Length(Gs);
23
gap> GsC4xC2:=Filtered(Gs,x->IdSmallGroup(x)=[8,2]);;
gap> Length(GsC4xC2);
3
gap> GsHNPfalseC4xC2:=Filtered(GsHNPfalse,x->IdSmallGroup(x)=[8,2]);
[ Group([ (), (1,5)(2,6)(3,7)(4,8), (1,3,5,7)(2,4,6,8), (2,6)(4,8) ]) ]
gap> Length(GsHNPfalseC4xC2);
1
gap> GsHNPtrueC4xC2:=Filtered(GsHNPtrue,x->IdSmallGroup(x)=[8,2]);
[ Group([ (), (1,5)(2,6)(3,7)(4,8), (1,4,5,8)(2,3,6,7), (1,3,5,7)
  (2,4,6,8) ]), Group([ (), (1,5)(2,6)(3,7)(4,8), (1,3,5,7)(2,4,6,8), (1,8)
  (2,3)(4,5)(6,7) ]) ]
gap> Length(GsHNPtrueC4xC2);
2
gap> Collected(List(GsHNPfalseC4xC2,x->List(Orbits(x),Length)));
[ [ [ 4, 4 ], 1 ] ]
gap> Collected(List(GsHNPtrueC4xC2,x->List(Orbits(x),Length)));
[ [ [ 8 ], 2 ] ]
gap> List(GsHNPfalseC4xC2,Elements);
[ [ (), (2,6)(4,8), (1,3,5,7)(2,4,6,8), (1,3,5,7)(2,8,6,4), (1,5)(3,7), 
      (1,5)(2,6)(3,7)(4,8), (1,7,5,3)(2,4,6,8), (1,7,5,3)(2,8,6,4) ] ]
gap> List(GsHNPtrueC4xC2,Elements);
[ [ (), (1,2)(3,4)(5,6)(7,8), (1,3,5,7)(2,4,6,8), (1,4,5,8)(2,3,6,7), 
      (1,5)(2,6)(3,7)(4,8), (1,6)(2,5)(3,8)(4,7), (1,7,5,3)(2,8,6,4), 
      (1,8,5,4)(2,7,6,3) ], 
  [ (), (1,2,5,6)(3,4,7,8), (1,3,5,7)(2,4,6,8), (1,4)(2,7)(3,6)(5,8), 
      (1,5)(2,6)(3,7)(4,8), (1,6,5,2)(3,8,7,4), (1,7,5,3)(2,8,6,4), 
      (1,8)(2,3)(4,5)(6,7) ] ]
\end{verbatim}
}~\\\vspace*{-4mm}

(2-3) $G=8T15\simeq C_8\rtimes V_4$. 
{\small 
\begin{verbatim}
gap> Read("HNP.gap");
gap> G:=TransitiveGroup(8,15); # G=8T15=C8:V4
[1/4.cD(4)^2]2
gap> GeneratorsOfGroup(G);
[ (1,2,3,4,5,6,7,8), (1,5)(3,7), (1,6)(2,5)(3,4)(7,8) ]
gap> H:=Stabilizer(G,1); # H=V4
Group([ (2,8)(3,7)(4,6), (2,4)(3,7)(6,8) ])
gap> FirstObstructionN(G).ker; # Obs1N=1
[ [  ], [ [ 2, 2 ], [  ] ] ]
gap> SchurMultPcpGroup(G); # M(G)=C2xC2: Schur multiplicer of G
[ 2, 2 ]
gap> cGs:=MinimalStemExtensions(G);; # 3 minimal stem extensions
gap> for cG in cGs do
> bG:=cG.MinimalStemExtension;
> bH:=PreImage(cG.epi,H);
> Print(KerResH3Z(bG,bH));
> Print("\n");
> od;
[ [  ], [ [ 2, 2, 2, 2 ], [  ] ] ]
[ [ 2 ], [ [ 2, 4 ], [ [ 0, 2 ] ] ] ]
[ [ 2 ], [ [ 2, 2, 2 ], [ [ 0, 1, 0 ] ] ] ]
gap> for cG in cGs do
> bG:=cG.MinimalStemExtension;
> bH:=PreImage(cG.epi,H);
> Print(FirstObstructionN(bG,bH).ker[1]);
> Print(FirstObstructionDnr(bG,bH).Dnr[1]);
> Print("\n");
> od;
[ 2 ][  ]
[  ][  ]
[  ][  ]
gap> cG:=cGs[1];
rec( MinimalStemExtension := <permutation group of size 64 with 3 generators>,
  epi := [ (3,4)(5,6)(7,9)(8,10)(11,13)(12,14)(15,18)(17,19)(20,23)(22,24)(25,
        27)(28,30), (1,2)(3,7)(4,9)(5,11)(6,13)(8,12)(10,14)(15,20)(16,21)(17,
        25)(18,23)(19,27)(22,28)(24,30)(26,29)(31,32), 
      (1,3,8,17,26,19,10,4)(2,5,12,22,29,24,14,6)(7,15,25,31,27,18,9,16)(11,
        20,28,32,30,23,13,21) ] -> 
    [ (2,8)(3,7)(4,6), (2,6)(4,8), (1,2,3,4,5,6,7,8) ] )
gap> bG:=cG.MinimalStemExtension; # bG=G- is a minimal stem extension of G
<permutation group of size 64 with 3 generators>
gap> bH:=PreImage(cG.epi,H); # bH=H-
<permutation group of size 8 with 3 generators>
gap> FirstObstructionN(bG,bH).ker; # Obs1N-=C2
[ [ 2 ], [ [ 2, 2, 2 ], [ [ 0, 0, 1 ] ] ] ]
gap> FirstObstructionDnr(bG,bH).Dnr; # Obs1Dnr-=1
[ [  ], [ [ 2, 2, 2 ], [  ] ] ]
gap> bGs:=AllSubgroups(bG);;
gap> Length(bGs);
225
gap> bGsHNPfalse:=Filtered(bGs,x->FirstObstructionDr(bG,x,bH).Dr[1]=[]);;
gap> Length(bGsHNPfalse);
174
gap> bGsHNPtrue:=Filtered(bGs,x->FirstObstructionDr(bG,x,bH).Dr[1]=[2]);;
gap> Length(bGsHNPtrue);
51
gap> Collected(List(bGsHNPfalse,x->StructureDescription(Image(cG.epi,x))));
[ [ "1", 2 ], [ "C2", 45 ], [ "C2 x C2", 55 ], [ "C4", 12 ], [ "C4 x C2", 5 ],
  [ "C8", 6 ], [ "C8 : C2", 1 ], [ "D16", 10 ], [ "D8", 35 ], [ "Q8", 1 ], 
  [ "QD16", 2 ] ]
gap> Collected(List(bGsHNPtrue,x->StructureDescription(Image(cG.epi,x))));
[ [ "(C4 x C2) : C2", 1 ], [ "C2 x C2", 20 ], [ "C2 x C2 x C2", 18 ], 
  [ "C2 x D8", 9 ], [ "C4 x C2", 2 ], [ "C8 : (C2 x C2)", 1 ] ]
gap> GsHNPfalse:=Set(bGsHNPfalse,x->Image(cG.epi,x));;
gap> Length(GsHNPfalse);
47
gap> GsHNPtrue:=Set(bGsHNPtrue,x->Image(cG.epi,x));;
gap> Length(GsHNPtrue);
11
gap> Intersection(GsHNPfalse,GsHNPtrue);
[  ]
gap> GsHNPtrueMin:=Filtered(GsHNPtrue,x->Length(Filtered(GsHNPtrue,
> y->IsSubgroup(x,y)))=1);
[ Group([ (1,5)(3,7), (1,5)(2,4)(6,8) ]), Group([ (2,6)(4,8), (2,6)
  (4,8), (1,3)(4,8)(5,7) ]), Group([ (2,6)(4,8), (2,6)(4,8), (1,7)(2,6)
  (3,5) ]), Group([ (1,5)(3,7), (2,8)(3,7)(4,6) ]), Group([ (1,5)(2,6)(3,7)
  (4,8), (1,2)(3,8)(4,7)(5,6), (1,3,5,7)(2,8,6,4), (1,7,5,3)(2,4,6,8) ]), 
  Group([ (1,5)(2,6)(3,7)(4,8), (1,8)(2,7)(3,6)(4,5), (1,7,5,3)
  (2,4,6,8), (1,3,5,7)(2,8,6,4) ]) ]
gap> Length(GsHNPtrueMin);
6
gap> List(GsHNPtrueMin,IdSmallGroup);
[ [ 4, 2 ], [ 4, 2 ], [ 4, 2 ], [ 4, 2 ], [ 8, 2 ], [ 8, 2 ] ]
gap> Collected(List(GsHNPfalse,x->Filtered(GsHNPtrueMin,y->IsSubgroup(x,y))));
[ [ [  ], 47 ] ]
gap> Gs:=AllSubgroups(G);;
gap> Length(Gs);
58
gap> GsC2xC2:=Filtered(Gs,x->IdSmallGroup(x)=[4,2]);;
gap> Length(GsC2xC2);
15
gap> GsC4xC2:=Filtered(Gs,x->IdSmallGroup(x)=[8,2]);;
gap> Length(GsC4xC2);
3
gap> GsHNPfalseC2xC2:=Filtered(GsHNPfalse,x->IdSmallGroup(x)=[4,2]);
[ Group([ (2,6)(4,8), (2,4)(3,7)(6,8) ]), Group([ (2,4)(3,7)(6,8), (1,5)(2,8)
  (4,6) ]), Group([ (1,5)(2,6)(3,7)(4,8), (1,5)(3,7), (2,6)(4,8) ]), 
  Group([ (2,6)(4,8), (2,6)(4,8), (1,5)(2,4)(6,8) ]), Group([ (1,5)(2,6)(3,7)
  (4,8), (1,5)(2,6)(3,7)(4,8), (2,8)(3,7)(4,6) ]), Group([ (1,5)(2,6)(3,7)
  (4,8), (1,2)(3,8)(4,7)(5,6) ]), Group([ (1,5)(3,7), (1,5)(3,7), (1,3)(4,8)
  (5,7) ]), Group([ (1,5)(2,6)(3,7)(4,8), (1,5)(2,6)(3,7)(4,8), (1,3)(4,8)
  (5,7) ]), Group([ (1,5)(3,7), (1,5)(3,7), (1,7)(2,6)(3,5) ]), Group([ (1,5)
  (2,6)(3,7)(4,8), (1,3)(2,6)(5,7), (1,7)(3,5)(4,8) ]), Group([ (1,5)(2,6)
  (3,7)(4,8), (1,8)(2,7)(3,6)(4,5), (1,5)(2,6)(3,7)(4,8) ]) ]
gap> Length(GsHNPfalseC2xC2);
11
gap> GsHNPtrueC2xC2:=Filtered(GsHNPtrue,x->IdSmallGroup(x)=[4,2]);
[ Group([ (1,5)(3,7), (1,5)(2,4)(6,8) ]), Group([ (2,6)(4,8), (2,6)
  (4,8), (1,3)(4,8)(5,7) ]), Group([ (2,6)(4,8), (2,6)(4,8), (1,7)(2,6)
  (3,5) ]), Group([ (1,5)(3,7), (2,8)(3,7)(4,6) ]) ]
gap> Length(GsHNPtrueC2xC2);
4
gap> Collected(List(GsHNPfalseC2xC2,x->List(Orbits(x),Length)));
[ [ [ 2, 2, 2, 2 ], 1 ], [ [ 2, 4 ], 1 ], [ [ 2, 4, 2 ], 2 ], [ [ 4, 2 ], 3 ],
  [ [ 4, 2, 2 ], 2 ], [ [ 4, 4 ], 2 ] ]
gap> Collected(List(GsHNPtrueC2xC2,x->List(Orbits(x),Length)));
[ [ [ 2, 2, 2, 2 ], 4 ] ]
gap> GsHNPfalse2222C2xC2:=Filtered(GsHNPfalseC2xC2,
> x->List(Orbits(x,[1..8]),Length)=[2,2,2,2]);
[ Group([ (1,5)(2,6)(3,7)(4,8), (1,5)(3,7), (2,6)(4,8) ]) ]
gap> List(GsHNPtrueC2xC2,Elements);
[ [ (), (2,4)(3,7)(6,8), (1,5)(3,7), (1,5)(2,4)(6,8) ], 
  [ (), (2,6)(4,8), (1,3)(4,8)(5,7), (1,3)(2,6)(5,7) ], 
  [ (), (2,6)(4,8), (1,7)(3,5)(4,8), (1,7)(2,6)(3,5) ], 
  [ (), (2,8)(3,7)(4,6), (1,5)(3,7), (1,5)(2,8)(4,6) ] ]
gap> List(GsHNPfalseC2xC2,Elements);
[ [ (), (2,4)(3,7)(6,8), (2,6)(4,8), (2,8)(3,7)(4,6) ], 
  [ (), (2,4)(3,7)(6,8), (1,5)(2,6)(3,7)(4,8), (1,5)(2,8)(4,6) ], 
  [ (), (2,6)(4,8), (1,5)(3,7), (1,5)(2,6)(3,7)(4,8) ], 
  [ (), (2,6)(4,8), (1,5)(2,4)(6,8), (1,5)(2,8)(4,6) ], 
  [ (), (2,8)(3,7)(4,6), (1,5)(2,4)(6,8), (1,5)(2,6)(3,7)(4,8) ], 
  [ (), (1,2)(3,8)(4,7)(5,6), (1,5)(2,6)(3,7)(4,8), (1,6)(2,5)(3,4)(7,8) ], 
  [ (), (1,3)(4,8)(5,7), (1,5)(3,7), (1,7)(3,5)(4,8) ], 
  [ (), (1,3)(4,8)(5,7), (1,5)(2,6)(3,7)(4,8), (1,7)(2,6)(3,5) ], 
  [ (), (1,3)(2,6)(5,7), (1,5)(3,7), (1,7)(2,6)(3,5) ], 
  [ (), (1,3)(2,6)(5,7), (1,5)(2,6)(3,7)(4,8), (1,7)(3,5)(4,8) ], 
  [ (), (1,4)(2,3)(5,8)(6,7), (1,5)(2,6)(3,7)(4,8), (1,8)(2,7)(3,6)(4,5) ] ]
gap> List(GsHNPfalse2222C2xC2,Elements);
[ [ (), (2,6)(4,8), (1,5)(3,7), (1,5)(2,6)(3,7)(4,8) ] ]
gap> DG:=DerivedSubgroup(G);
Group([ (1,5)(2,6)(3,7)(4,8), (1,3,5,7)(2,4,6,8) ])
gap> List(GsHNPtrueC2xC2,x->Intersection(x,DG));
[ Group(()), Group(()), Group(()), Group(()) ]
gap> List(GsHNPfalse2222C2xC2,x->Intersection(x,DG));
[ Group([ (1,5)(2,6)(3,7)(4,8) ]) ]
gap> A8:=AlternatingGroup(8);
Alt( [ 1 .. 8 ] )
gap> List(GsHNPtrueC2xC2,x->IsSubgroup(A8,x));
[ false, false, false, false ]
gap> List(GsHNPfalse2222C2xC2,x->IsSubgroup(A8,x));
[ true ]
gap> GsHNPfalseC4xC2:=Filtered(GsHNPfalse,x->IdSmallGroup(x)=[8,2]);
[ Group([ (1,5)(2,6)(3,7)(4,8), (2,6)(4,8), (1,7,5,3)(2,8,6,4) ]) ]
gap> Length(GsHNPfalseC4xC2);
1
gap> GsHNPtrueC4xC2:=Filtered(GsHNPtrue,x->IdSmallGroup(x)=[8,2]);
[ Group([ (1,5)(2,6)(3,7)(4,8), (1,2)(3,8)(4,7)(5,6), (1,3,5,7)
  (2,8,6,4), (1,7,5,3)(2,4,6,8) ]), Group([ (1,5)(2,6)(3,7)(4,8), (1,8)(2,7)
  (3,6)(4,5), (1,7,5,3)(2,4,6,8), (1,3,5,7)(2,8,6,4) ]) ]
gap> Length(GsHNPtrueC4xC2);
2
gap> Collected(List(GsHNPfalseC4xC2,x->List(Orbits(x),Length)));
[ [ [ 4, 4 ], 1 ] ]
gap> Collected(List(GsHNPtrueC4xC2,x->List(Orbits(x),Length)));
[ [ [ 8 ], 2 ] ]
gap> Collected(List(GsHNPfalseC4xC2,x->StructureDescription(Intersection(DG,x))));
[ [ "C4", 1 ] ]
gap> Collected(List(GsHNPtrueC4xC2,x->StructureDescription(Intersection(DG,x))));
[ [ "C2", 2 ] ]
\end{verbatim}
}~\\\vspace*{-4mm}

(2-4) $G=8T19\simeq (C_2)^3\rtimes C_4$. 
{\small 
\begin{verbatim}
gap> Read("HNP.gap");
gap> G:=TransitiveGroup(8,19); # G=8T19=(C2xC2xC2):C4
E(8):4=[1/4.eD(4)^2]2
gap> GeneratorsOfGroup(G);
[ (1,8)(2,3)(4,5)(6,7), (1,3)(2,8)(4,6)(5,7), (1,5)(2,6)(3,7)(4,8), 
  (1,3)(4,5,6,7) ]
gap> H:=Stabilizer(G,1); # H=C4
Group([ (2,8)(4,5,6,7) ])
gap> FirstObstructionN(G).ker; # Obs1N=1
[ [  ], [ [ 4 ], [  ] ] ]
gap> SchurMultPcpGroup(G); # M(G)=C2xC2: Schur multiplicer of G
[ 2, 2 ]
gap> cGs:=MinimalStemExtensions(G);; # 3 minimal stem extensions
gap> for cG in cGs do
> bG:=cG.MinimalStemExtension;
> bH:=PreImage(cG.epi,H);
> Print(KerResH3Z(bG,bH));
> Print("\n");
> od;
[ [ 2 ], [ [ 2, 4 ], [ [ 0, 2 ] ] ] ]
[ [ 2 ], [ [ 2, 2 ], [ [ 1, 0 ] ] ] ]
[ [  ], [ [ 2, 2 ], [  ] ] ]
gap> for cG in cGs do
> bG:=cG.MinimalStemExtension;
> bH:=PreImage(cG.epi,H);
> Print(FirstObstructionN(bG,bH).ker[1]);
> Print(FirstObstructionDnr(bG,bH).Dnr[1]);
> Print("\n");
> od;
[ 2 ][ 2 ]
[ 2 ][ 2 ]
[ 2 ][  ]
gap> cG:=cGs[3];
rec( MinimalStemExtension := Group([ (2,3,5,4)(6,9)(7,10)(11,13,12,14), (1,2)
  (3,6)(4,7)(5,8)(9,11)(10,12)(13,15)(14,16) ]), Tid := [ 16, 163 ], 
  epi := [ (2,3,5,4)(6,9)(7,10)(11,13,12,14), 
      (1,2)(3,6)(4,7)(5,8)(9,11)(10,12)(13,15)(14,16) ] -> 
    [ (2,8)(4,7,6,5), (1,5)(2,6)(3,7)(4,8) ] )
gap> bG:=cG.MinimalStemExtension; # bG=G- is a minimal stem extension of G
Group([ (2,3,5,4)(6,9)(7,10)(11,13,12,14), (1,2)(3,6)(4,7)(5,8)(9,11)(10,12)
(13,15)(14,16) ])
gap> bH:=PreImage(cG.epi,H); # bH=H-
Group([ (2,4,5,3)(6,9)(7,10)(11,14,12,13), (1,15)(2,13)(3,12)(4,11)(5,14)
(6,10)(7,9)(8,16) ])
gap> FirstObstructionN(bG,bH).ker; # Obs1N-=C2
[ [ 2 ], [ [ 2, 4 ], [ [ 1, 2 ] ] ] ]
gap> FirstObstructionDnr(bG,bH).Dnr; # Obs1Dnr-=1
[ [  ], [ [ 2, 4 ], [  ] ] ]
gap> bGs:=AllSubgroups(bG);;
gap> Length(bGs);
105
gap> bGsHNPfalse:=Filtered(bGs,x->FirstObstructionDr(bG,x,bH).Dr[1]=[]);;
gap> Length(bGsHNPfalse);
86
gap> bGsHNPtrue:=Filtered(bGs,x->FirstObstructionDr(bG,x,bH).Dr[1]=[2]);;
gap> Length(bGsHNPtrue);
19
gap> Collected(List(bGsHNPfalse,x->StructureDescription(Image(cG.epi,x))));
[ [ "(C4 x C2) : C2", 1 ], [ "1", 2 ], [ "C2", 25 ], [ "C2 x C2", 21 ], 
  [ "C2 x C2 x C2", 1 ], [ "C4", 22 ], [ "C4 x C2", 2 ], [ "D8", 12 ] ]
gap> Collected(List(bGsHNPtrue,x->StructureDescription(Image(cG.epi,x))));
[ [ "(C2 x C2 x C2) : C4", 1 ], [ "(C4 x C2) : C2", 1 ], [ "C2 x C2", 4 ], 
  [ "C2 x C2 x C2", 1 ], [ "C2 x D8", 1 ], [ "C4 x C2", 11 ] ]
gap> GsHNPfalse:=Set(bGsHNPfalse,x->Image(cG.epi,x));;
gap> Length(GsHNPfalse);
39
gap> GsHNPtrue:=Set(bGsHNPtrue,x->Image(cG.epi,x));;
gap> Length(GsHNPtrue);
11
gap> Intersection(GsHNPfalse,GsHNPtrue);
[  ]
gap> GsHNPtrueMin:=Filtered(GsHNPtrue,x->Length(Filtered(GsHNPtrue,
> y->IsSubgroup(x,y)))=1);
[ Group([ (1,3)(4,7,6,5), (1,3)(2,8)(4,6)(5,7), (4,6)(5,7) ]), 
  Group([ (1,8,3,2)(5,7), (1,3)(2,8)(4,6)(5,7), (1,3)(2,8) ]), 
  Group([ (1,7,3,5)(2,4,8,6), (1,2)(3,8)(4,7)(5,6), (), (1,3)(2,8)(4,6)
  (5,7) ]), Group([ (1,4)(2,7)(3,6)(5,8), (1,2)(3,8)(4,7)(5,6), () ]), 
  Group([ (1,6)(2,5)(3,4)(7,8), (1,2)(3,8)(4,7)(5,6), () ]), Group([ (1,5)
  (2,6)(3,7)(4,8), (1,8)(2,3)(4,5)(6,7), () ]), Group([ (1,7)(2,4)(3,5)
  (6,8), (1,8)(2,3)(4,5)(6,7), () ]) ]
gap> Length(GsHNPtrueMin);
7
gap> List(GsHNPtrueMin,IdSmallGroup);
[ [ 8, 2 ], [ 8, 2 ], [ 8, 2 ], [ 4, 2 ], [ 4, 2 ], [ 4, 2 ], [ 4, 2 ] ]
gap> Gs:=AllSubgroups(G);;
gap> Length(Gs);
50
gap> GsC2xC2:=Filtered(Gs,x->IdSmallGroup(x)=[4,2]);
[ Group([ (1,2)(3,8)(4,7)(5,6), (1,8)(2,3)(4,5)(6,7) ]), Group([ (1,4)(2,7)
  (3,6)(5,8), (1,2)(3,8)(4,7)(5,6) ]), Group([ (1,6)(2,5)(3,4)(7,8), (1,2)
  (3,8)(4,7)(5,6) ]), Group([ (1,5)(2,6)(3,7)(4,8), (1,8)(2,3)(4,5)(6,7) ]), 
  Group([ (1,7)(2,4)(3,5)(6,8), (1,8)(2,3)(4,5)(6,7) ]), Group([ (1,4)(2,7)
  (3,6)(5,8), (1,3)(2,8)(4,6)(5,7) ]), Group([ (1,5)(2,6)(3,7)(4,8), (1,3)
  (2,8)(4,6)(5,7) ]), Group([ (4,6)(5,7), (1,2)(3,8)(4,7)(5,6) ]), 
  Group([ (4,6)(5,7), (1,8)(2,3)(4,5)(6,7) ]), Group([ (1,3)(2,8), (1,8)(2,3)
  (4,5)(6,7) ]), Group([ (1,3)(2,8), (1,2)(3,8)(4,7)(5,6) ]), Group([ (4,6)
  (5,7), (1,3)(2,8)(4,6)(5,7) ]), Group([ (1,8)(2,3)(4,7)(5,6), (1,3)(2,8)
  (4,6)(5,7) ]) ]
gap> Length(GsC2xC2);
13
gap> GsC4xC2:=Filtered(Gs,x->IdSmallGroup(x)=[8,2]);
[ Group([ (1,2)(3,8)(4,7)(5,6), (1,8)(2,3)(4,5)(6,7), (1,5,3,7)(2,6,8,4) ]), 
  Group([ (2,8)(4,7,6,5), (4,6)(5,7), (1,3)(2,8)(4,6)(5,7) ]), 
  Group([ (1,8,3,2)(4,6), (1,3)(2,8), (1,3)(2,8)(4,6)(5,7) ]), 
  Group([ (1,4,2,5)(3,6,8,7), (1,2)(3,8)(4,5)(6,7), (1,3)(2,8)(4,6)(5,7) ]), 
  Group([ (1,5,8,6)(2,4,3,7), (1,8)(2,3)(4,7)(5,6), (1,3)(2,8)(4,6)(5,7) ]) ]
gap> Length(GsC4xC2);
5
gap> GsHNPfalseC2xC2:=Filtered(GsHNPfalse,x->IdSmallGroup(x)=[4,2]);
[ Group([ (1,2)(3,8)(4,5)(6,7), (4,6)(5,7), (4,6)(5,7) ]), Group([ (1,3)(2,8)
  (4,6)(5,7), (1,3)(2,8), (4,6)(5,7) ]), Group([ (1,8)(2,3)(4,7)(5,6), (4,6)
  (5,7), (4,6)(5,7) ]), Group([ (1,2)(3,8)(4,5)(6,7), (1,3)(2,8), (1,3)
  (2,8) ]), Group([ (1,2)(3,8)(4,5)(6,7), (1,3)(2,8)(4,6)(5,7), (1,3)(2,8)
  (4,6)(5,7) ]), Group([ (1,8)(2,3)(4,7)(5,6), (1,3)(2,8), (1,3)(2,8) ]), 
  Group([ (1,2)(3,8)(4,7)(5,6), (1,3)(2,8)(4,6)(5,7), (1,3)(2,8)(4,6)
  (5,7) ]), Group([ (1,4)(2,7)(3,6)(5,8), (1,3)(2,8)(4,6)(5,7) ]), 
  Group([ (1,7)(2,4)(3,5)(6,8), (1,3)(2,8)(4,6)(5,7) ]) ]
gap> Length(GsHNPfalseC2xC2);
9
gap> GsHNPtrueC2xC2:=Filtered(GsHNPtrue,x->IdSmallGroup(x)=[4,2]);
[ Group([ (1,4)(2,7)(3,6)(5,8), (1,2)(3,8)(4,7)(5,6), () ]), Group([ (1,6)
  (2,5)(3,4)(7,8), (1,2)(3,8)(4,7)(5,6), () ]), Group([ (1,5)(2,6)(3,7)
  (4,8), (1,8)(2,3)(4,5)(6,7), () ]), Group([ (1,7)(2,4)(3,5)(6,8), (1,8)(2,3)
  (4,5)(6,7), () ]) ]
gap> Length(GsHNPtrueC2xC2);
4
gap> Collected(List(GsHNPfalseC2xC2,x->List(Orbits(x),Length)));
[ [ [ 2, 2, 2, 2 ], 1 ], [ [ 2, 2, 4 ], 2 ], [ [ 4, 2, 2 ], 2 ], 
  [ [ 4, 4 ], 4 ] ]
gap> Collected(List(GsHNPtrueC2xC2,x->List(Orbits(x),Length)));
[ [ [ 4, 4 ], 4 ] ]
gap> GsHNPfalse44C2xC2:=Filtered(GsHNPfalseC2xC2,
> x->List(Orbits(x,[1..8]),Length)=[4,4]);
[ Group([ (1,2)(3,8)(4,5)(6,7), (1,3)(2,8)(4,6)(5,7), (1,3)(2,8)(4,6)
  (5,7) ]), Group([ (1,2)(3,8)(4,7)(5,6), (1,3)(2,8)(4,6)(5,7), (1,3)(2,8)
  (4,6)(5,7) ]), Group([ (1,4)(2,7)(3,6)(5,8), (1,3)(2,8)(4,6)(5,7) ]), 
  Group([ (1,7)(2,4)(3,5)(6,8), (1,3)(2,8)(4,6)(5,7) ]) ]
gap> List(GsHNPtrueC2xC2,Elements);
[ [ (), (1,2)(3,8)(4,7)(5,6), (1,4)(2,7)(3,6)(5,8), (1,7)(2,4)(3,5)(6,8) ], 
  [ (), (1,2)(3,8)(4,7)(5,6), (1,5)(2,6)(3,7)(4,8), (1,6)(2,5)(3,4)(7,8) ], 
  [ (), (1,4)(2,7)(3,6)(5,8), (1,5)(2,6)(3,7)(4,8), (1,8)(2,3)(4,5)(6,7) ], 
  [ (), (1,6)(2,5)(3,4)(7,8), (1,7)(2,4)(3,5)(6,8), (1,8)(2,3)(4,5)(6,7) ] ]
gap> List(GsHNPfalseC2xC2,Elements);
[ [ (), (4,6)(5,7), (1,2)(3,8)(4,5)(6,7), (1,2)(3,8)(4,7)(5,6) ], 
  [ (), (4,6)(5,7), (1,3)(2,8), (1,3)(2,8)(4,6)(5,7) ], 
  [ (), (4,6)(5,7), (1,8)(2,3)(4,5)(6,7), (1,8)(2,3)(4,7)(5,6) ], 
  [ (), (1,2)(3,8)(4,5)(6,7), (1,3)(2,8), (1,8)(2,3)(4,5)(6,7) ], 
  [ (), (1,2)(3,8)(4,5)(6,7), (1,3)(2,8)(4,6)(5,7), (1,8)(2,3)(4,7)(5,6) ], 
  [ (), (1,2)(3,8)(4,7)(5,6), (1,3)(2,8), (1,8)(2,3)(4,7)(5,6) ], 
  [ (), (1,2)(3,8)(4,7)(5,6), (1,3)(2,8)(4,6)(5,7), (1,8)(2,3)(4,5)(6,7) ], 
  [ (), (1,3)(2,8)(4,6)(5,7), (1,4)(2,7)(3,6)(5,8), (1,6)(2,5)(3,4)(7,8) ], 
  [ (), (1,3)(2,8)(4,6)(5,7), (1,5)(2,6)(3,7)(4,8), (1,7)(2,4)(3,5)(6,8) ] ]
gap> List(GsHNPfalse44C2xC2,Elements);
[ [ (), (1,2)(3,8)(4,5)(6,7), (1,3)(2,8)(4,6)(5,7), (1,8)(2,3)(4,7)(5,6) ], 
  [ (), (1,2)(3,8)(4,7)(5,6), (1,3)(2,8)(4,6)(5,7), (1,8)(2,3)(4,5)(6,7) ], 
  [ (), (1,3)(2,8)(4,6)(5,7), (1,4)(2,7)(3,6)(5,8), (1,6)(2,5)(3,4)(7,8) ], 
  [ (), (1,3)(2,8)(4,6)(5,7), (1,5)(2,6)(3,7)(4,8), (1,7)(2,4)(3,5)(6,8) ] ]
gap> ZG:=Centre(G);
Group([ (1,3)(2,8)(4,6)(5,7) ])
gap> List(GsHNPtrueC2xC2,x->Intersection(x,ZG));
[ Group(()), Group(()), Group(()), Group(()) ]
gap> List(GsHNPfalse44C2xC2,x->Intersection(x,ZG));
[ Group([ (1,3)(2,8)(4,6)(5,7) ]), Group([ (1,3)(2,8)(4,6)(5,7) ]), 
  Group([ (1,3)(2,8)(4,6)(5,7) ]), Group([ (1,3)(2,8)(4,6)(5,7) ]) ]
gap> UcsG:=UpperCentralSeries(G);
[ Group([ (1,3)(2,8)(4,6)(5,7), (4,6)(5,7), (1,8)(2,3)(4,5)(6,7), (2,8)
  (4,7,6,5), (1,5)(2,6)(3,7)(4,8) ]), Group([ (1,3)(2,8)(4,6)(5,7), (4,6)
  (5,7), (1,8)(2,3)(4,5)(6,7) ]), Group([ (1,3)(2,8)(4,6)(5,7) ]), Group(()) ]
gap> Collected(List(GsHNPfalseC2xC2,x->List(UcsG,y->Order(Intersection(y,x)))));
[ [ [ 4, 2, 2, 1 ], 2 ], [ [ 4, 4, 1, 1 ], 4 ], [ [ 4, 4, 2, 1 ], 3 ] ]
gap> Collected(List(GsHNPtrueC2xC2,x->List(UcsG,y->Order(Intersection(y,x)))));
[ [ [ 4, 2, 1, 1 ], 4 ] ]
gap> GsHNPfalseC4xC2:=Filtered(GsHNPfalse,x->IdSmallGroup(x)=[8,2]);
[ Group([ (1,4,2,5)(3,6,8,7), (1,2)(3,8)(4,5)(6,7), (), (1,3)(2,8)(4,6)
  (5,7) ]), Group([ (1,7,8,4)(2,6,3,5), (1,8)(2,3)(4,7)(5,6), (), (1,3)(2,8)
  (4,6)(5,7) ]) ]
gap> Length(GsHNPfalseC4xC2);
2
gap> GsHNPtrueC4xC2:=Filtered(GsHNPtrue,x->IdSmallGroup(x)=[8,2]);
[ Group([ (1,3)(4,7,6,5), (1,3)(2,8)(4,6)(5,7), (4,6)(5,7) ]), 
  Group([ (1,8,3,2)(5,7), (1,3)(2,8)(4,6)(5,7), (1,3)(2,8) ]), 
  Group([ (1,7,3,5)(2,4,8,6), (1,2)(3,8)(4,7)(5,6), (), (1,3)(2,8)(4,6)
  (5,7) ]) ]
gap> Length(GsHNPtrueC4xC2);
3
gap> Collected(List(GsHNPfalseC4xC2,x->List(Orbits(x),Length)));
[ [ [ 8 ], 2 ] ]
gap> Collected(List(GsHNPtrueC4xC2,x->List(Orbits(x),Length)));
[ [ [ 2, 2, 4 ], 1 ], [ [ 4, 2, 2 ], 1 ], [ [ 8 ], 1 ] ]
gap> GsHNPtrue8C4xC2:=Filtered(GsHNPtrueC4xC2,x->List(Orbits(x,[1..8]),Length)=[8]);
[ Group([ (1,7,3,5)(2,4,8,6), (1,2)(3,8)(4,7)(5,6), (), (1,3)(2,8)(4,6)
  (5,7) ]) ]
gap> List(GsHNPfalseC4xC2,Elements);
[ [ (), (1,2)(3,8)(4,5)(6,7), (1,3)(2,8)(4,6)(5,7), (1,4,2,5)(3,6,8,7), 
      (1,5,2,4)(3,7,8,6), (1,6,2,7)(3,4,8,5), (1,7,2,6)(3,5,8,4), 
      (1,8)(2,3)(4,7)(5,6) ], 
  [ (), (1,2)(3,8)(4,5)(6,7), (1,3)(2,8)(4,6)(5,7), (1,4,8,7)(2,5,3,6), 
      (1,5,8,6)(2,4,3,7), (1,6,8,5)(2,7,3,4), (1,7,8,4)(2,6,3,5), 
      (1,8)(2,3)(4,7)(5,6) ] ]
gap> List(GsHNPtrueC4xC2,Elements);
[ [ (), (4,6)(5,7), (2,8)(4,5,6,7), (2,8)(4,7,6,5), (1,3)(4,5,6,7), 
      (1,3)(4,7,6,5), (1,3)(2,8), (1,3)(2,8)(4,6)(5,7) ], 
  [ (), (4,6)(5,7), (1,2,3,8)(5,7), (1,2,3,8)(4,6), (1,3)(2,8), 
      (1,3)(2,8)(4,6)(5,7), (1,8,3,2)(5,7), (1,8,3,2)(4,6) ], 
  [ (), (1,2)(3,8)(4,7)(5,6), (1,3)(2,8)(4,6)(5,7), (1,4,3,6)(2,7,8,5), 
      (1,5,3,7)(2,6,8,4), (1,6,3,4)(2,5,8,7), (1,7,3,5)(2,4,8,6), 
      (1,8)(2,3)(4,5)(6,7) ] ]
gap> DG:=DerivedSubgroup(G);
Group([ (1,3)(2,8)(4,6)(5,7), (1,8)(2,3)(4,5)(6,7) ])
gap> List(GsHNPfalseC4xC2,x->Intersection(x,DG));
[ Group([ (1,3)(2,8)(4,6)(5,7) ]), Group([ (1,3)(2,8)(4,6)(5,7) ]) ]
gap> List(GsHNPtrue8C4xC2,x->Intersection(x,DG));
[ Group([ (1,3)(2,8)(4,6)(5,7), (1,8)(2,3)(4,5)(6,7) ]) ]
\end{verbatim}
}~\\\vspace*{-4mm}

(2-5) $G=8T22\simeq (C_2)^3\rtimes V_4$. 
{\small 
\begin{verbatim}
gap> Read("HNP.gap");
gap> G:=TransitiveGroup(8,22); # G=8T22=(C2xC2xC2):V4
E(8):D_4=[2^3]2^2
gap> GeneratorsOfGroup(G);
[ (1,8)(2,3)(4,5)(6,7), (1,3)(2,8)(4,6)(5,7), (1,5)(2,6)(3,7)(4,8), 
  (2,3)(4,5), (2,3)(6,7) ]
gap> H:=Stabilizer(G,1); # H=V4
Group([ (2,3)(4,5), (2,3)(6,7) ])
gap> FirstObstructionN(G).ker; # Obs1N=1
[ [  ], [ [ 2, 2 ], [  ] ] ]
gap> SchurMultPcpGroup(G); # M(G)=C2^5: Schur multiplicer of G
[ 2, 2, 2, 2, 2 ]
gap> cGs:=MinimalStemExtensions(G);; # 31 minimal stem extensions
gap> for cG in cGs do
> bG:=cG.MinimalStemExtension;
> bH:=PreImage(cG.epi,H);
> Print(KerResH3Z(bG,bH));
> Print("\n");
> od;
[ [ 2 ], [ [ 2, 2, 2, 2, 2, 2 ], [ [ 0, 0, 0, 1, 0, 0 ] ] ] ]
[ [ 2 ], [ [ 2, 2, 2, 2, 2 ], [ [ 0, 1, 0, 0, 0 ] ] ] ]
[ [ 2 ], [ [ 2, 2, 2, 2, 2 ], [ [ 0, 0, 0, 1, 1 ] ] ] ]
[ [ 2, 2 ], [ [ 2, 2, 2, 2, 4 ], [ [ 0, 0, 0, 1, 0 ], [ 0, 0, 0, 0, 2 ] ] ] ]
[ [ 2, 2, 2 ], 
  [ [ 2, 2, 2, 2 ], [ [ 0, 1, 0, 0 ], [ 0, 0, 1, 0 ], [ 0, 0, 0, 1 ] ] ] ]
[ [ 2 ], [ [ 2, 2, 2, 2, 4 ], [ [ 1, 1, 0, 1, 0 ] ] ] ]
[ [ 2 ], [ [ 2, 2, 2, 2, 2 ], [ [ 1, 1, 0, 1, 0 ] ] ] ]
[ [ 2 ], [ [ 2, 2, 2, 2, 2, 2 ], [ [ 0, 0, 0, 1, 0, 0 ] ] ] ]
[ [ 2, 2 ], [ [ 2, 2, 2, 2, 2 ], [ [ 0, 1, 0, 0, 0 ], [ 0, 0, 1, 0, 0 ] ] ] ]
[ [ 2 ], [ [ 2, 2, 2, 2, 2, 2 ], [ [ 0, 0, 1, 0, 0, 1 ] ] ] ]
[ [ 2 ], [ [ 2, 2, 2, 2, 4 ], [ [ 0, 1, 0, 1, 2 ] ] ] ]
[ [ 2 ], [ [ 2, 2, 2, 2, 2, 2 ], [ [ 0, 0, 1, 1, 0, 1 ] ] ] ]
[ [ 2, 2 ], [ [ 2, 2, 2, 2, 4 ], [ [ 0, 0, 1, 1, 0 ], [ 0, 0, 0, 0, 2 ] ] ] ]
[ [ 2 ], [ [ 2, 2, 2, 2 ], [ [ 0, 0, 1, 1 ] ] ] ]
[ [ 2 ], [ [ 2, 2, 2, 2 ], [ [ 1, 0, 1, 1 ] ] ] ]
[ [  ], [ [ 2, 2, 2, 2, 2 ], [  ] ] ]
[ [ 2 ], [ [ 2, 2, 2, 2, 2 ], [ [ 0, 0, 0, 0, 1 ] ] ] ]
[ [ 2 ], [ [ 2, 2, 2, 2 ], [ [ 0, 1, 0, 0 ] ] ] ]
[ [ 2 ], [ [ 2, 2, 2, 2 ], [ [ 0, 0, 0, 1 ] ] ] ]
[ [ 2, 2 ], [ [ 2, 2, 2, 2, 2 ], [ [ 0, 1, 0, 0, 0 ], [ 0, 0, 0, 1, 0 ] ] ] ]
[ [ 2, 2, 2 ], 
  [ [ 2, 2, 4, 4 ], [ [ 1, 0, 0, 0 ], [ 0, 0, 2, 0 ], [ 0, 0, 0, 2 ] ] ] ]
[ [ 2 ], [ [ 2, 2, 2, 2, 4 ], [ [ 0, 0, 1, 1, 0 ] ] ] ]
[ [ 2 ], [ [ 2, 2, 2, 2 ], [ [ 1, 0, 1, 0 ] ] ] ]
[ [ 2 ], [ [ 2, 2, 2, 2, 2 ], [ [ 0, 1, 0, 0, 0 ] ] ] ]
[ [ 2, 2 ], [ [ 2, 2, 2, 2 ], [ [ 1, 0, 0, 0 ], [ 0, 0, 0, 1 ] ] ] ]
[ [ 2 ], [ [ 2, 2, 2, 2 ], [ [ 1, 0, 1, 0 ] ] ] ]
[ [ 2 ], [ [ 2, 2, 2, 2 ], [ [ 1, 0, 0, 1 ] ] ] ]
[ [ 2 ], [ [ 2, 2, 2, 2, 2, 2 ], [ [ 0, 0, 0, 1, 0, 1 ] ] ] ]
[ [ 2, 2 ], [ [ 2, 2, 2, 2, 2 ], [ [ 0, 0, 1, 0, 1 ], [ 0, 0, 0, 1, 1 ] ] ] ]
[ [ 2 ], [ [ 2, 2, 2, 2, 2 ], [ [ 0, 0, 1, 0, 0 ] ] ] ]
[ [ 2 ], [ [ 2, 2, 2, 2 ], [ [ 0, 1, 1, 0 ] ] ] ]
gap> for cG in cGs do
> bG:=cG.MinimalStemExtension;
> bH:=PreImage(cG.epi,H);
> Print(FirstObstructionN(bG,bH).ker[1]);
> Print(FirstObstructionDnr(bG,bH).Dnr[1]);
> Print("\n");
> od;
[ 2 ][ 2 ]
[  ][  ]
[  ][  ]
[ 2 ][ 2 ]
[ 2 ][ 2 ]
[  ][  ]
[  ][  ]
[ 2 ][ 2 ]
[ 2 ][ 2 ]
[  ][  ]
[  ][  ]
[ 2 ][ 2 ]
[ 2 ][ 2 ]
[  ][  ]
[  ][  ]
[ 2 ][  ]
[ 2 ][ 2 ]
[  ][  ]
[  ][  ]
[ 2 ][ 2 ]
[ 2 ][ 2 ]
[  ][  ]
[  ][  ]
[ 2 ][ 2 ]
[ 2 ][ 2 ]
[  ][  ]
[  ][  ]
[ 2 ][ 2 ]
[ 2 ][ 2 ]
[  ][  ]
[  ][  ]
gap> cG:=cGs[16];;
gap> bG:=cG.MinimalStemExtension; # bG=G- is a minimal stem extension of G
<permutation group of size 64 with 6 generators>
gap> bH:=PreImage(cG.epi,H); # bH=H-
<permutation group of size 8 with 3 generators>
gap> FirstObstructionN(bG,bH).ker; # Obs1N-=C2
[ [ 2 ], [ [ 2, 2, 2 ], [ [ 0, 0, 1 ] ] ] ]
gap> FirstObstructionDnr(bG,bH).Dnr; # Obs1Dnr-=1
[ [  ], [ [ 2, 2, 2 ], [  ] ] ]
gap> bGs:=AllSubgroups(bG);;
gap> Length(bGs);
321
gap> bGsHNPfalse:=Filtered(bGs,x->FirstObstructionDr(bG,x,bH).Dr[1]=[]);;
gap> Length(bGsHNPfalse);
292
gap> bGsHNPtrue:=Filtered(bGs,x->FirstObstructionDr(bG,x,bH).Dr[1]=[2]);;
gap> Length(bGsHNPtrue);
29
gap> Collected(List(bGsHNPfalse,x->StructureDescription(Image(cG.epi,x))));
[ [ "1", 2 ], [ "C2", 49 ], [ "C2 x C2", 101 ], [ "C2 x C2 x C2", 28 ], 
  [ "C2 x D8", 19 ], [ "C4", 18 ], [ "C4 x C2", 15 ], [ "D8", 58 ], 
  [ "Q8", 2 ] ]
gap> Collected(List(bGsHNPtrue,x->StructureDescription(Image(cG.epi,x))));
[ [ "(C2 x C2 x C2) : (C2 x C2)", 1 ], [ "(C4 x C2) : C2", 6 ], 
  [ "C2 x C2", 8 ], [ "C2 x C2 x C2", 2 ], [ "C2 x D8", 6 ], 
  [ "C4 x C2", 6 ] ]
gap> GsHNPfalse:=Set(bGsHNPfalse,x->Image(cG.epi,x));;
gap> Length(GsHNPfalse);
81
gap> GsHNPtrue:=Set(bGsHNPtrue,x->Image(cG.epi,x));;
gap> Length(GsHNPtrue);
29
gap> Intersection(GsHNPfalse,GsHNPtrue);
[  ]
gap> GsHNPtrueMin:=Filtered(GsHNPtrue,x->Length(Filtered(GsHNPtrue,
> y->IsSubgroup(x,y)))=1);
[ Group([ (), (1,8)(2,3)(4,5)(6,7), (1,2)(3,8)(4,6)(5,7), (1,4,8,5)
  (2,6,3,7) ]), Group([ (1,4)(2,6)(3,7)(5,8), (1,6)(2,4)(3,5)(7,8), () ]), 
  Group([ (1,5)(2,7)(3,6)(4,8), (1,7)(2,5)(3,4)(6,8), () ]), 
  Group([ (), (1,8)(2,3)(4,5)(6,7), (1,6,8,7)(2,5,3,4), (1,2)(3,8)(4,7)
  (5,6) ]), Group([ (1,2)(3,8)(4,7)(5,6), (1,7)(2,4)(3,5)(6,8), () ]), 
  Group([ (1,2)(3,8)(4,7)(5,6), (1,6)(2,5)(3,4)(7,8), () ]), 
  Group([ (), (1,8)(2,3)(4,5)(6,7), (1,4)(2,6)(3,7)(5,8), (1,2,8,3)
  (4,6,5,7) ]), Group([ (), (1,8)(2,3)(4,5)(6,7), (1,4,8,5)(2,6,3,7), (1,2,8,
   3)(4,6,5,7) ]), Group([ (), (1,8)(2,3)(4,5)(6,7), (1,2,8,3)
  (4,7,5,6), (1,4)(2,7)(3,6)(5,8) ]), Group([ (), (1,8)(2,3)(4,5)
  (6,7), (1,2,8,3)(4,7,5,6), (1,6)(2,4)(3,5)(7,8) ]), Group([ (1,3)(2,8)(4,6)
  (5,7), (1,6)(2,5)(3,4)(7,8), () ]), Group([ (1,3)(2,8)(4,6)(5,7), (1,7)
  (2,4)(3,5)(6,8), () ]), Group([ (1,4)(2,6)(3,7)(5,8), (1,7)(2,5)(3,4)
  (6,8), () ]), Group([ (1,5)(2,7)(3,6)(4,8), (1,6)(2,4)(3,5)(7,8), () ]) ]
gap> Length(GsHNPtrueMin);
14
gap> List(GsHNPtrueMin,IdSmallGroup);
[ [ 8, 2 ], [ 4, 2 ], [ 4, 2 ], [ 8, 2 ], [ 4, 2 ], [ 4, 2 ], [ 8, 2 ], 
  [ 8, 2 ], [ 8, 2 ], [ 8, 2 ], [ 4, 2 ], [ 4, 2 ], [ 4, 2 ], [ 4, 2 ] ]
gap> Collected(List(GsHNPfalse,x->Filtered(GsHNPtrueMin,y->IsSubgroup(x,y))));
[ [ [  ], 81 ] ]
gap> Gs:=AllSubgroups(G);;
gap> Length(Gs);
110
gap> GsC2xC2:=Filtered(Gs,x->IdSmallGroup(x)=[4,2]);;
gap> Length(GsC2xC2);
33
gap> GsC4xC2:=Filtered(Gs,x->IdSmallGroup(x)=[8,2]);;
gap> Length(GsC4xC2);
9
gap> GsHNPfalseC2xC2:=Filtered(GsHNPfalse,x->IdSmallGroup(x)=[4,2]);
[ Group([ (4,5)(6,7), (2,3)(4,5) ]), Group([ (4,5)(6,7), (1,2)(3,8)(4,7)
  (5,6) ]), Group([ (4,5)(6,7), (1,3)(2,8)(4,6)(5,7), () ]), Group([ (4,5)
  (6,7), (1,8)(6,7), () ]), Group([ (), (1,8)(2,3)(4,5)(6,7), (4,5)(6,7) ]), 
  Group([ (1,4)(2,6)(3,7)(5,8), (1,4)(2,7)(3,6)(5,8) ]), Group([ (1,5)(2,7)
  (3,6)(4,8), (1,5)(2,6)(3,7)(4,8), () ]), Group([ (1,8)(2,3), (1,8)
  (6,7) ]), Group([ (2,3)(6,7), (1,8)(2,3)(4,5)(6,7) ]), Group([ (2,3)
  (4,5), (1,6)(2,5)(3,4)(7,8), () ]), Group([ (2,3)(4,5), (1,7)(2,4)(3,5)
  (6,8), () ]), Group([ (2,3)(4,5), (1,8)(2,3)(4,5)(6,7) ]), Group([ (1,8)
  (2,3), (2,3)(4,5) ]), Group([ (1,8)(2,3), (1,3)(2,8)(4,6)(5,7), () ]), 
  Group([ (), (1,8)(2,3)(4,5)(6,7), (1,2)(3,8)(4,6)(5,7) ]), 
  Group([ (), (1,8)(2,3)(4,5)(6,7), (1,2)(3,8)(4,7)(5,6) ]), Group([ (1,8)
  (2,3), (1,2)(3,8)(4,7)(5,6), () ]), Group([ (1,4)(2,6)(3,7)(5,8), (1,5)
  (2,6)(3,7)(4,8) ]), Group([ (), (1,8)(2,3)(4,5)(6,7), (1,4)(2,6)(3,7)
  (5,8) ]), Group([ (1,4)(2,7)(3,6)(5,8), (1,8)(2,3)(4,5)(6,7) ]), 
  Group([ (1,5)(2,7)(3,6)(4,8), (1,4)(2,7)(3,6)(5,8) ]), Group([ (1,8)
  (6,7), (1,7)(2,4)(3,5)(6,8), () ]), Group([ (), (1,8)(2,3)(4,5)(6,7), (1,6)
  (2,4)(3,5)(7,8) ]), Group([ (), (1,8)(2,3)(4,5)(6,7), (1,6)(2,5)(3,4)
  (7,8) ]), Group([ (1,8)(6,7), (1,6)(2,5)(3,4)(7,8), () ]) ]
gap> Length(GsHNPfalseC2xC2);
25
gap> GsHNPtrueC2xC2:=Filtered(GsHNPtrue,x->IdSmallGroup(x)=[4,2]);
[ Group([ (1,4)(2,6)(3,7)(5,8), (1,6)(2,4)(3,5)(7,8), () ]), Group([ (1,5)
  (2,7)(3,6)(4,8), (1,7)(2,5)(3,4)(6,8), () ]), Group([ (1,2)(3,8)(4,7)
  (5,6), (1,7)(2,4)(3,5)(6,8), () ]), Group([ (1,2)(3,8)(4,7)(5,6), (1,6)
  (2,5)(3,4)(7,8), () ]), Group([ (1,3)(2,8)(4,6)(5,7), (1,6)(2,5)(3,4)
  (7,8), () ]), Group([ (1,3)(2,8)(4,6)(5,7), (1,7)(2,4)(3,5)(6,8), () ]), 
  Group([ (1,4)(2,6)(3,7)(5,8), (1,7)(2,5)(3,4)(6,8), () ]), Group([ (1,5)
  (2,7)(3,6)(4,8), (1,6)(2,4)(3,5)(7,8), () ]) ]
gap> Length(GsHNPtrueC2xC2);
8
gap> Collected(List(GsHNPfalseC2xC2,x->List(Orbits(x),Length)));
[ [ [ 2, 2, 2 ], 4 ], [ [ 2, 2, 2, 2 ], 3 ], [ [ 2, 2, 4 ], 2 ], 
  [ [ 2, 4, 2 ], 4 ], [ [ 4, 2, 2 ], 6 ], [ [ 4, 4 ], 6 ] ]
gap> Collected(List(GsHNPtrueC2xC2,x->List(Orbits(x),Length)));
[ [ [ 4, 4 ], 8 ] ]
gap> GsHNPfalse44C2xC2:=Filtered(GsHNPfalseC2xC2,
> x->List(Orbits(x,[1..8]),Length)=[4,4]);
[ Group([ (), (1,8)(2,3)(4,5)(6,7), (1,2)(3,8)(4,6)(5,7) ]), 
  Group([ (), (1,8)(2,3)(4,5)(6,7), (1,2)(3,8)(4,7)(5,6) ]), 
  Group([ (), (1,8)(2,3)(4,5)(6,7), (1,4)(2,6)(3,7)(5,8) ]), Group([ (1,4)
  (2,7)(3,6)(5,8), (1,8)(2,3)(4,5)(6,7) ]), Group([ (), (1,8)(2,3)(4,5)
  (6,7), (1,6)(2,4)(3,5)(7,8) ]), Group([ (), (1,8)(2,3)(4,5)(6,7), (1,6)
  (2,5)(3,4)(7,8) ]) ]
gap> List(GsHNPtrueC2xC2,Elements);
[ [ (), (1,2)(3,8)(4,6)(5,7), (1,4)(2,6)(3,7)(5,8), (1,6)(2,4)(3,5)(7,8) ], 
  [ (), (1,2)(3,8)(4,6)(5,7), (1,5)(2,7)(3,6)(4,8), (1,7)(2,5)(3,4)(6,8) ], 
  [ (), (1,2)(3,8)(4,7)(5,6), (1,4)(2,7)(3,6)(5,8), (1,7)(2,4)(3,5)(6,8) ], 
  [ (), (1,2)(3,8)(4,7)(5,6), (1,5)(2,6)(3,7)(4,8), (1,6)(2,5)(3,4)(7,8) ], 
  [ (), (1,3)(2,8)(4,6)(5,7), (1,4)(2,7)(3,6)(5,8), (1,6)(2,5)(3,4)(7,8) ], 
  [ (), (1,3)(2,8)(4,6)(5,7), (1,5)(2,6)(3,7)(4,8), (1,7)(2,4)(3,5)(6,8) ], 
  [ (), (1,3)(2,8)(4,7)(5,6), (1,4)(2,6)(3,7)(5,8), (1,7)(2,5)(3,4)(6,8) ], 
  [ (), (1,3)(2,8)(4,7)(5,6), (1,5)(2,7)(3,6)(4,8), (1,6)(2,4)(3,5)(7,8) ] ]
gap> List(GsHNPfalseC2xC2,Elements);
[ [ (), (4,5)(6,7), (2,3)(6,7), (2,3)(4,5) ], 
  [ (), (4,5)(6,7), (1,2)(3,8)(4,6)(5,7), (1,2)(3,8)(4,7)(5,6) ], 
  [ (), (4,5)(6,7), (1,3)(2,8)(4,6)(5,7), (1,3)(2,8)(4,7)(5,6) ], 
  [ (), (4,5)(6,7), (1,8)(6,7), (1,8)(4,5) ], 
  [ (), (4,5)(6,7), (1,8)(2,3), (1,8)(2,3)(4,5)(6,7) ], 
  [ (), (2,3)(6,7), (1,4)(2,6)(3,7)(5,8), (1,4)(2,7)(3,6)(5,8) ], 
  [ (), (2,3)(6,7), (1,5)(2,6)(3,7)(4,8), (1,5)(2,7)(3,6)(4,8) ], 
  [ (), (2,3)(6,7), (1,8)(6,7), (1,8)(2,3) ], 
  [ (), (2,3)(6,7), (1,8)(4,5), (1,8)(2,3)(4,5)(6,7) ], 
  [ (), (2,3)(4,5), (1,6)(2,4)(3,5)(7,8), (1,6)(2,5)(3,4)(7,8) ], 
  [ (), (2,3)(4,5), (1,7)(2,4)(3,5)(6,8), (1,7)(2,5)(3,4)(6,8) ], 
  [ (), (2,3)(4,5), (1,8)(6,7), (1,8)(2,3)(4,5)(6,7) ], 
  [ (), (2,3)(4,5), (1,8)(4,5), (1,8)(2,3) ], 
  [ (), (1,2)(3,8)(4,6)(5,7), (1,3)(2,8)(4,6)(5,7), (1,8)(2,3) ], 
  [ (), (1,2)(3,8)(4,6)(5,7), (1,3)(2,8)(4,7)(5,6), (1,8)(2,3)(4,5)(6,7) ], 
  [ (), (1,2)(3,8)(4,7)(5,6), (1,3)(2,8)(4,6)(5,7), (1,8)(2,3)(4,5)(6,7) ], 
  [ (), (1,2)(3,8)(4,7)(5,6), (1,3)(2,8)(4,7)(5,6), (1,8)(2,3) ], 
  [ (), (1,4)(2,6)(3,7)(5,8), (1,5)(2,6)(3,7)(4,8), (1,8)(4,5) ], 
  [ (), (1,4)(2,6)(3,7)(5,8), (1,5)(2,7)(3,6)(4,8), (1,8)(2,3)(4,5)(6,7) ], 
  [ (), (1,4)(2,7)(3,6)(5,8), (1,5)(2,6)(3,7)(4,8), (1,8)(2,3)(4,5)(6,7) ], 
  [ (), (1,4)(2,7)(3,6)(5,8), (1,5)(2,7)(3,6)(4,8), (1,8)(4,5) ], 
  [ (), (1,6)(2,4)(3,5)(7,8), (1,7)(2,4)(3,5)(6,8), (1,8)(6,7) ], 
  [ (), (1,6)(2,4)(3,5)(7,8), (1,7)(2,5)(3,4)(6,8), (1,8)(2,3)(4,5)(6,7) ], 
  [ (), (1,6)(2,5)(3,4)(7,8), (1,7)(2,4)(3,5)(6,8), (1,8)(2,3)(4,5)(6,7) ], 
  [ (), (1,6)(2,5)(3,4)(7,8), (1,7)(2,5)(3,4)(6,8), (1,8)(6,7) ] ]
gap> List(GsHNPfalse44C2xC2,Elements);
[ [ (), (1,2)(3,8)(4,6)(5,7), (1,3)(2,8)(4,7)(5,6), (1,8)(2,3)(4,5)(6,7) ], 
  [ (), (1,2)(3,8)(4,7)(5,6), (1,3)(2,8)(4,6)(5,7), (1,8)(2,3)(4,5)(6,7) ], 
  [ (), (1,4)(2,6)(3,7)(5,8), (1,5)(2,7)(3,6)(4,8), (1,8)(2,3)(4,5)(6,7) ], 
  [ (), (1,4)(2,7)(3,6)(5,8), (1,5)(2,6)(3,7)(4,8), (1,8)(2,3)(4,5)(6,7) ], 
  [ (), (1,6)(2,4)(3,5)(7,8), (1,7)(2,5)(3,4)(6,8), (1,8)(2,3)(4,5)(6,7) ], 
  [ (), (1,6)(2,5)(3,4)(7,8), (1,7)(2,4)(3,5)(6,8), (1,8)(2,3)(4,5)(6,7) ] ]
gap> ZG:=Centre(G);
Group([ (1,8)(2,3)(4,5)(6,7) ])
gap> List(GsHNPtrueC2xC2,x->Intersection(x,ZG));
[ Group(()), Group(()), Group(()), Group(()), Group(()), Group(()), 
  Group(()), Group(()) ]
gap> List(GsHNPfalse44C2xC2,x->Intersection(x,ZG));
[ Group([ (1,8)(2,3)(4,5)(6,7) ]), Group([ (1,8)(2,3)(4,5)(6,7) ]), 
  Group([ (1,8)(2,3)(4,5)(6,7) ]), Group([ (1,8)(2,3)(4,5)(6,7) ]), 
  Group([ (1,8)(2,3)(4,5)(6,7) ]), Group([ (1,8)(2,3)(4,5)(6,7) ]) ]
gap> GsHNPfalseC4xC2:=Filtered(GsHNPfalse,x->IdSmallGroup(x)=[8,2]);
[ Group([ (), (1,8)(2,3)(4,5)(6,7), (4,5)(6,7), (1,2,8,3)(4,6,5,7) ]), 
  Group([ (2,3)(6,7), (1,4,8,5)(2,6,3,7), (1,8)(2,3)(4,5)(6,7) ]), 
  Group([ (), (1,8)(2,3)(4,5)(6,7), (1,6,8,7)(2,5,3,4), (2,3)(4,5) ]) ]
gap> Length(GsHNPfalseC4xC2);
3
gap> GsHNPtrueC4xC2:=Filtered(GsHNPtrue,x->IdSmallGroup(x)=[8,2]);
[ Group([ (), (1,8)(2,3)(4,5)(6,7), (1,2)(3,8)(4,6)(5,7), (1,4,8,5)
  (2,6,3,7) ]), Group([ (), (1,8)(2,3)(4,5)(6,7), (1,6,8,7)(2,5,3,4), (1,2)
  (3,8)(4,7)(5,6) ]), Group([ (), (1,8)(2,3)(4,5)(6,7), (1,4)(2,6)(3,7)
  (5,8), (1,2,8,3)(4,6,5,7) ]), Group([ (), (1,8)(2,3)(4,5)(6,7), (1,4,8,5)
  (2,6,3,7), (1,2,8,3)(4,6,5,7) ]), Group([ (), (1,8)(2,3)(4,5)
  (6,7), (1,2,8,3)(4,7,5,6), (1,4)(2,7)(3,6)(5,8) ]), Group([ (), (1,8)(2,3)
  (4,5)(6,7), (1,2,8,3)(4,7,5,6), (1,6)(2,4)(3,5)(7,8) ]) ]
gap> Length(GsHNPtrueC4xC2);
6
gap> Collected(List(GsHNPfalseC4xC2,x->List(Orbits(x),Length)));
[ [ [ 4, 4 ], 3 ] ]
gap> Collected(List(GsHNPtrueC4xC2,x->List(Orbits(x),Length)));
[ [ [ 8 ], 6 ] ]
\end{verbatim}
}~\\\vspace*{-4mm}

(2-6) $G^\prime=8T32\simeq ((C_2)^3\rtimes V_4)\rtimes C_3$.
{\small 
\begin{verbatim}
gap> Read("HNP.gap");
gap> G:=TransitiveGroup(8,32); # G=8T32=(C2^3:V4):C3
[2^3]A(4)
gap> GeneratorsOfGroup(G);
[ (1,8)(2,3)(4,5)(6,7), (1,3)(2,8)(4,6)(5,7), (1,5)(2,6)(3,7)(4,8), 
  (1,2,3)(4,6,5), (2,5)(3,4) ]
gap> H:=Stabilizer(G,1); # H=A4
Group([ (2,5)(3,4), (2,3,8)(4,7,5) ])
gap> FirstObstructionN(G).ker; # Obs1N=1
[ [  ], [ [ 3 ], [  ] ] ]
gap> SchurMultPcpGroup(G); # M(G)=C2xC2xC2: Schur multiplier of G
[ 2, 2, 2 ]
gap> cGs:=MinimalStemExtensions(G);; # 7 minimal stem extensions
gap> for cG in cGs do
> bG:=cG.MinimalStemExtension;
> bH:=PreImage(cG.epi,H);
> Print(KerResH3Z(bG,bH));
> Print("\n");
> od;
[ [  ], [ [ 2, 2 ], [  ] ] ]
[ [ 2 ], [ [ 2, 2 ], [ [ 0, 1 ] ] ] ]
[ [ 2 ], [ [ 2, 2 ], [ [ 0, 1 ] ] ] ]
[ [ 2 ], [ [ 2, 2 ], [ [ 1, 1 ] ] ] ]
[ [ 2 ], [ [ 2, 2 ], [ [ 1, 0 ] ] ] ]
[ [ 2 ], [ [ 2, 2 ], [ [ 1, 1 ] ] ] ]
[ [ 2 ], [ [ 2, 2 ], [ [ 1, 1 ] ] ] ]
gap> for cG in cGs do
> bG:=cG.MinimalStemExtension;
> bH:=PreImage(cG.epi,H);
> Print(FirstObstructionN(bG,bH).ker[1]);
> Print(FirstObstructionDnr(bG,bH).Dnr[1]);
> Print("\n");
> od;
[ 2 ][  ]
[ 2 ][ 2 ]
[ 2 ][ 2 ]
[  ][  ]
[  ][  ]
[  ][  ]
[  ][  ]
gap> cG:=cGs[1];;
gap> bG:=cG.MinimalStemExtension; # bG=G- is a minimal stem extension of G
<permutation group of size 192 with 7 generators>
gap> bH:=PreImage(cG.epi,H); # bH=H-
<permutation group of size 24 with 3 generators>
gap> FirstObstructionN(bG,bH).ker; # Obs1N-=C2
[ [ 2 ], [ [ 6 ], [ [ 3 ] ] ] ]
gap> FirstObstructionDnr(bG,bH).Dnr; # Obs1Dnr-=1
[ [  ], [ [ 6 ], [  ] ] ]
gap> bGs:=AllSubgroups(bG);;
gap> Length(bGs);
326
gap> bGsHNPfalse:=Filtered(bGs,x->FirstObstructionDr(bG,x,bH).Dr[1]=[]);;
gap> Length(bGsHNPfalse);
280
gap> bGsHNPtrue:=Filtered(bGs,x->FirstObstructionDr(bG,x,bH).Dr[1]=[2]);;
gap> Length(bGsHNPtrue);
46
gap> Collected(List(bGsHNPfalse,x->StructureDescription(Image(cG.epi,x))));
[ [ "1", 2 ], [ "A4", 8 ], [ "C2", 33 ], [ "C2 x A4", 12 ], [ "C2 x C2", 53 ],
  [ "C2 x C2 x C2", 12 ], [ "C2 x D8", 3 ], [ "C3", 32 ], [ "C4", 18 ], 
  [ "C4 x C2", 15 ], [ "C6", 48 ], [ "D8", 18 ], [ "Q8", 10 ], 
  [ "SL(2,3)", 16 ] ]
gap> Collected(List(bGsHNPtrue,x->StructureDescription(Image(cG.epi,x))));
[ [ "((C2 x C2 x C2) : (C2 x C2)) : C3", 1 ], 
  [ "(C2 x C2 x C2) : (C2 x C2)", 1 ], [ "(C4 x C2) : C2", 6 ], [ "A4", 8 ], 
  [ "C2 x A4", 8 ], [ "C2 x C2", 8 ], [ "C2 x C2 x C2", 2 ], [ "C2 x D8", 6 ],
  [ "C4 x C2", 6 ] ]
gap> GsHNPfalse:=Set(bGsHNPfalse,x->Image(cG.epi,x));;
gap> Length(GsHNPfalse);
129
gap> GsHNPtrue:=Set(bGsHNPtrue,x->Image(cG.epi,x));;
gap> Length(GsHNPtrue);
46
gap> Intersection(GsHNPfalse,GsHNPtrue);
[  ]
gap> GsHNPtrueMin:=Filtered(GsHNPtrue,x->Length(Filtered(GsHNPtrue,
> y->IsSubgroup(x,y)))=1);
[ Group([ (1,7)(2,3)(4,5)(6,8), (1,2)(3,7)(4,8)(5,6), () ]), Group([ (1,3,6,4)
  (2,7,5,8), (1,8,6,7)(2,4,5,3), (1,6)(2,5)(3,4)(7,8), (1,6)(2,5)(3,4)
  (7,8) ]), Group([ (1,4)(2,8)(3,6)(5,7), (1,8)(2,4)(3,5)(6,7), () ]), 
  Group([ (1,8)(2,3)(4,5)(6,7), (1,2)(3,8)(4,7)(5,6), () ]), Group([ (1,7,6,8)
  (2,4,5,3), (1,2)(3,8)(4,7)(5,6), (1,6)(2,5)(3,4)(7,8), (1,6)(2,5)(3,4)
  (7,8) ]), Group([ (1,4)(2,7)(3,6)(5,8), (1,7)(2,4)(3,5)(6,8), () ]), 
  Group([ (1,7,6,8)(2,4,5,3), (1,2,6,5)(3,7,4,8), (1,6)(2,5)(3,4)(7,8), (1,6)
  (2,5)(3,4)(7,8) ]), Group([ (1,3,6,4)(2,7,5,8), (1,7)(2,4)(3,5)(6,8), (1,6)
  (2,5)(3,4)(7,8), (1,6)(2,5)(3,4)(7,8) ]), Group([ (1,5,6,2)(3,7,4,8), (1,3)
  (2,8)(4,6)(5,7), (1,6)(2,5)(3,4)(7,8), (1,6)(2,5)(3,4)(7,8) ]), 
  Group([ (1,5,6,2)(3,7,4,8), (1,3,6,4)(2,8,5,7), (1,6)(2,5)(3,4)(7,8), (1,6)
  (2,5)(3,4)(7,8) ]), Group([ (1,8)(2,4)(3,5)(6,7), (1,5)(2,6)(3,8)
  (4,7), () ]), Group([ (1,7)(2,4)(3,5)(6,8), (1,5)(2,6)(3,7)(4,8), () ]), 
  Group([ (1,8)(2,3)(4,5)(6,7), (1,5)(2,6)(3,7)(4,8), () ]), Group([ (1,7)
  (2,3)(4,5)(6,8), (1,5)(2,6)(3,8)(4,7), () ]) ]
gap> Length(GsHNPtrueMin);
14
gap> List(GsHNPtrueMin,IdSmallGroup);
[ [ 4, 2 ], [ 8, 2 ], [ 4, 2 ], [ 4, 2 ], [ 8, 2 ], [ 4, 2 ], [ 8, 2 ], 
  [ 8, 2 ], [ 8, 2 ], [ 8, 2 ], [ 4, 2 ], [ 4, 2 ], [ 4, 2 ], [ 4, 2 ] ]
gap> Collected(List(GsHNPfalse,x->Filtered(GsHNPtrueMin,y->IsSubgroup(x,y))));
[ [ [  ], 129 ] ]
gap> Gs:=AllSubgroups(G);;
gap> Length(Gs);
175
gap> GsC2xC2:=Filtered(Gs,x->IdSmallGroup(x)=[4,2]);;
gap> Length(GsC2xC2);
33
gap> GsC4xC2:=Filtered(Gs,x->IdSmallGroup(x)=[8,2]);;
gap> Length(GsC4xC2);
9
gap> GsHNPfalseC2xC2:=Filtered(GsHNPfalse,x->IdSmallGroup(x)=[4,2]);
[ Group([ (3,4)(7,8), (2,5)(7,8) ]), Group([ (1,2)(3,8)(4,7)(5,6), (3,4)
  (7,8), (3,4)(7,8) ]), Group([ (1,5)(2,6)(3,7)(4,8), (3,4)(7,8), (3,4)
  (7,8) ]), Group([ (1,6)(3,4), (3,4)(7,8), (1,6)(7,8) ]), Group([ (3,4)
  (7,8), (1,6)(2,5)(3,4)(7,8) ]), Group([ (1,3)(2,8)(4,6)(5,7), (2,5)
  (7,8), (2,5)(7,8) ]), Group([ (1,4)(2,7)(3,6)(5,8), (2,5)(7,8), (2,5)
  (7,8) ]), Group([ (2,5)(7,8), (1,6)(2,5), (1,6)(7,8) ]), Group([ (1,6)
  (3,4), (1,6)(2,5)(3,4)(7,8) ]), Group([ (1,6)(2,5)(3,4)(7,8), (2,5)
  (3,4), (1,6)(7,8) ]), Group([ (2,5)(3,4), (1,6)(2,5) ]), Group([ (1,7)(2,4)
  (3,5)(6,8), (2,5)(3,4), (2,5)(3,4) ]), Group([ (1,8)(2,3)(4,5)(6,7), (2,5)
  (3,4), (2,5)(3,4) ]), Group([ (1,5)(2,6)(3,7)(4,8), (1,6)(2,5), (1,6)
  (2,5) ]), Group([ (1,5)(2,6)(3,8)(4,7), (1,6)(2,5)(3,4)(7,8), (1,6)(2,5)
  (3,4)(7,8) ]), Group([ (1,2)(3,8)(4,7)(5,6), (1,6)(2,5)(3,4)(7,8), (1,6)
  (2,5)(3,4)(7,8) ]), Group([ (1,2)(3,8)(4,7)(5,6), (1,6)(2,5), (1,6)
  (2,5) ]), Group([ (1,4)(2,7)(3,6)(5,8), (1,6)(3,4), (1,6)(3,4) ]), 
  Group([ (1,3)(2,7)(4,6)(5,8), (1,6)(2,5)(3,4)(7,8), (1,6)(2,5)(3,4)
  (7,8) ]), Group([ (1,3)(2,8)(4,6)(5,7), (1,6)(2,5)(3,4)(7,8), (1,6)(2,5)
  (3,4)(7,8) ]), Group([ (1,3)(2,8)(4,6)(5,7), (1,6)(3,4), (1,6)(3,4) ]), 
  Group([ (1,8)(2,3)(4,5)(6,7), (1,6)(7,8), (1,6)(7,8) ]), Group([ (1,7)(2,4)
  (3,5)(6,8), (1,6)(7,8), (1,6)(7,8) ]), Group([ (1,8)(2,4)(3,5)(6,7), (1,6)
  (2,5)(3,4)(7,8), (1,6)(2,5)(3,4)(7,8) ]), Group([ (1,7)(2,4)(3,5)
  (6,8), (1,6)(2,5)(3,4)(7,8), (1,6)(2,5)(3,4)(7,8) ]) ]
gap> Length(GsHNPfalseC2xC2);
25
gap> GsHNPtrueC2xC2:=Filtered(GsHNPtrue,x->IdSmallGroup(x)=[4,2]);
[ Group([ (1,7)(2,3)(4,5)(6,8), (1,2)(3,7)(4,8)(5,6), () ]), Group([ (1,4)
  (2,8)(3,6)(5,7), (1,8)(2,4)(3,5)(6,7), () ]), Group([ (1,8)(2,3)(4,5)
  (6,7), (1,2)(3,8)(4,7)(5,6), () ]), Group([ (1,4)(2,7)(3,6)(5,8), (1,7)(2,4)
  (3,5)(6,8), () ]), Group([ (1,8)(2,4)(3,5)(6,7), (1,5)(2,6)(3,8)
  (4,7), () ]), Group([ (1,7)(2,4)(3,5)(6,8), (1,5)(2,6)(3,7)(4,8), () ]), 
  Group([ (1,8)(2,3)(4,5)(6,7), (1,5)(2,6)(3,7)(4,8), () ]), Group([ (1,7)
  (2,3)(4,5)(6,8), (1,5)(2,6)(3,8)(4,7), () ]) ]
gap> Length(GsHNPtrueC2xC2);
8
gap> Collected(List(GsHNPfalseC2xC2,x->List(Orbits(x),Length)));
[ [ [ 2, 2, 2 ], 4 ], [ [ 2, 2, 2, 2 ], 3 ], [ [ 2, 2, 4 ], 1 ], 
  [ [ 2, 4, 2 ], 5 ], [ [ 4, 2, 2 ], 6 ], [ [ 4, 4 ], 6 ] ]
gap> Collected(List(GsHNPtrueC2xC2,x->List(Orbits(x),Length)));
[ [ [ 4, 4 ], 8 ] ]
gap> GsHNPfalse44C2xC2:=Filtered(GsHNPfalseC2xC2,
> x->List(Orbits(x,[1..8]),Length)=[4,4]);
[ Group([ (1,5)(2,6)(3,8)(4,7), (1,6)(2,5)(3,4)(7,8), (1,6)(2,5)(3,4)
  (7,8) ]), Group([ (1,2)(3,8)(4,7)(5,6), (1,6)(2,5)(3,4)(7,8), (1,6)(2,5)
  (3,4)(7,8) ]), Group([ (1,3)(2,7)(4,6)(5,8), (1,6)(2,5)(3,4)(7,8), (1,6)
  (2,5)(3,4)(7,8) ]), Group([ (1,3)(2,8)(4,6)(5,7), (1,6)(2,5)(3,4)
  (7,8), (1,6)(2,5)(3,4)(7,8) ]), Group([ (1,8)(2,4)(3,5)(6,7), (1,6)(2,5)
  (3,4)(7,8), (1,6)(2,5)(3,4)(7,8) ]), Group([ (1,7)(2,4)(3,5)(6,8), (1,6)
  (2,5)(3,4)(7,8), (1,6)(2,5)(3,4)(7,8) ]) ]
gap> List(GsHNPtrueC2xC2,Elements);
[ [ (), (1,2)(3,7)(4,8)(5,6), (1,3)(2,7)(4,6)(5,8), (1,7)(2,3)(4,5)(6,8) ], 
  [ (), (1,2)(3,7)(4,8)(5,6), (1,4)(2,8)(3,6)(5,7), (1,8)(2,4)(3,5)(6,7) ], 
  [ (), (1,2)(3,8)(4,7)(5,6), (1,3)(2,8)(4,6)(5,7), (1,8)(2,3)(4,5)(6,7) ], 
  [ (), (1,2)(3,8)(4,7)(5,6), (1,4)(2,7)(3,6)(5,8), (1,7)(2,4)(3,5)(6,8) ], 
  [ (), (1,3)(2,7)(4,6)(5,8), (1,5)(2,6)(3,8)(4,7), (1,8)(2,4)(3,5)(6,7) ], 
  [ (), (1,3)(2,8)(4,6)(5,7), (1,5)(2,6)(3,7)(4,8), (1,7)(2,4)(3,5)(6,8) ], 
  [ (), (1,4)(2,7)(3,6)(5,8), (1,5)(2,6)(3,7)(4,8), (1,8)(2,3)(4,5)(6,7) ], 
  [ (), (1,4)(2,8)(3,6)(5,7), (1,5)(2,6)(3,8)(4,7), (1,7)(2,3)(4,5)(6,8) ] ]
gap> List(GsHNPfalseC2xC2,Elements);
[ [ (), (3,4)(7,8), (2,5)(7,8), (2,5)(3,4) ], 
  [ (), (3,4)(7,8), (1,2)(3,7)(4,8)(5,6), (1,2)(3,8)(4,7)(5,6) ], 
  [ (), (3,4)(7,8), (1,5)(2,6)(3,7)(4,8), (1,5)(2,6)(3,8)(4,7) ], 
  [ (), (3,4)(7,8), (1,6)(7,8), (1,6)(3,4) ], 
  [ (), (3,4)(7,8), (1,6)(2,5), (1,6)(2,5)(3,4)(7,8) ], 
  [ (), (2,5)(7,8), (1,3)(2,7)(4,6)(5,8), (1,3)(2,8)(4,6)(5,7) ], 
  [ (), (2,5)(7,8), (1,4)(2,7)(3,6)(5,8), (1,4)(2,8)(3,6)(5,7) ], 
  [ (), (2,5)(7,8), (1,6)(7,8), (1,6)(2,5) ], 
  [ (), (2,5)(7,8), (1,6)(3,4), (1,6)(2,5)(3,4)(7,8) ], 
  [ (), (2,5)(3,4), (1,6)(7,8), (1,6)(2,5)(3,4)(7,8) ], 
  [ (), (2,5)(3,4), (1,6)(3,4), (1,6)(2,5) ], 
  [ (), (2,5)(3,4), (1,7)(2,3)(4,5)(6,8), (1,7)(2,4)(3,5)(6,8) ], 
  [ (), (2,5)(3,4), (1,8)(2,3)(4,5)(6,7), (1,8)(2,4)(3,5)(6,7) ], 
  [ (), (1,2)(3,7)(4,8)(5,6), (1,5)(2,6)(3,7)(4,8), (1,6)(2,5) ], 
  [ (), (1,2)(3,7)(4,8)(5,6), (1,5)(2,6)(3,8)(4,7), (1,6)(2,5)(3,4)(7,8) ], 
  [ (), (1,2)(3,8)(4,7)(5,6), (1,5)(2,6)(3,7)(4,8), (1,6)(2,5)(3,4)(7,8) ], 
  [ (), (1,2)(3,8)(4,7)(5,6), (1,5)(2,6)(3,8)(4,7), (1,6)(2,5) ], 
  [ (), (1,3)(2,7)(4,6)(5,8), (1,4)(2,7)(3,6)(5,8), (1,6)(3,4) ], 
  [ (), (1,3)(2,7)(4,6)(5,8), (1,4)(2,8)(3,6)(5,7), (1,6)(2,5)(3,4)(7,8) ], 
  [ (), (1,3)(2,8)(4,6)(5,7), (1,4)(2,7)(3,6)(5,8), (1,6)(2,5)(3,4)(7,8) ], 
  [ (), (1,3)(2,8)(4,6)(5,7), (1,4)(2,8)(3,6)(5,7), (1,6)(3,4) ], 
  [ (), (1,6)(7,8), (1,7)(2,3)(4,5)(6,8), (1,8)(2,3)(4,5)(6,7) ], 
  [ (), (1,6)(7,8), (1,7)(2,4)(3,5)(6,8), (1,8)(2,4)(3,5)(6,7) ], 
  [ (), (1,6)(2,5)(3,4)(7,8), (1,7)(2,3)(4,5)(6,8), (1,8)(2,4)(3,5)(6,7) ], 
  [ (), (1,6)(2,5)(3,4)(7,8), (1,7)(2,4)(3,5)(6,8), (1,8)(2,3)(4,5)(6,7) ] ]
gap> List(GsHNPfalse44C2xC2,Elements);
[ [ (), (1,2)(3,7)(4,8)(5,6), (1,5)(2,6)(3,8)(4,7), (1,6)(2,5)(3,4)(7,8) ], 
  [ (), (1,2)(3,8)(4,7)(5,6), (1,5)(2,6)(3,7)(4,8), (1,6)(2,5)(3,4)(7,8) ], 
  [ (), (1,3)(2,7)(4,6)(5,8), (1,4)(2,8)(3,6)(5,7), (1,6)(2,5)(3,4)(7,8) ], 
  [ (), (1,3)(2,8)(4,6)(5,7), (1,4)(2,7)(3,6)(5,8), (1,6)(2,5)(3,4)(7,8) ], 
  [ (), (1,6)(2,5)(3,4)(7,8), (1,7)(2,3)(4,5)(6,8), (1,8)(2,4)(3,5)(6,7) ], 
  [ (), (1,6)(2,5)(3,4)(7,8), (1,7)(2,4)(3,5)(6,8), (1,8)(2,3)(4,5)(6,7) ] ]
gap> ZG:=Centre(G);
Group([ (1,6)(2,5)(3,4)(7,8) ])
gap> List(GsHNPtrueC2xC2,x->Intersection(x,ZG));
[ Group(()), Group(()), Group(()), Group(()), Group(()), Group(()), 
  Group(()), Group(()) ]
gap> List(GsHNPfalse44C2xC2,x->Intersection(x,ZG));
[ Group([ (1,6)(2,5)(3,4)(7,8) ]), Group([ (1,6)(2,5)(3,4)(7,8) ]), 
  Group([ (1,6)(2,5)(3,4)(7,8) ]), Group([ (1,6)(2,5)(3,4)(7,8) ]), 
  Group([ (1,6)(2,5)(3,4)(7,8) ]), Group([ (1,6)(2,5)(3,4)(7,8) ]) ]
gap> GsHNPfalseC4xC2:=Filtered(GsHNPfalse,x->IdSmallGroup(x)=[8,2]);
[ Group([ (1,2,6,5)(3,8,4,7), (1,6)(2,5)(3,4)(7,8), (3,4)(7,8) ]), 
  Group([ (1,3,6,4)(2,7,5,8), (1,6)(2,5)(3,4)(7,8), (1,6)(3,4), (2,5)
  (7,8) ]), Group([ (1,7,6,8)(2,4,5,3), (1,6)(2,5)(3,4)(7,8), (2,5)(3,4) ]) ]
gap> Length(GsHNPfalseC4xC2);
3
gap> GsHNPtrueC4xC2:=Filtered(GsHNPtrue,x->IdSmallGroup(x)=[8,2]);
[ Group([ (1,3,6,4)(2,7,5,8), (1,8,6,7)(2,4,5,3), (1,6)(2,5)(3,4)(7,8), (1,6)
  (2,5)(3,4)(7,8) ]), Group([ (1,7,6,8)(2,4,5,3), (1,2)(3,8)(4,7)(5,6), (1,6)
  (2,5)(3,4)(7,8), (1,6)(2,5)(3,4)(7,8) ]), Group([ (1,7,6,8)(2,4,5,3), (1,2,
   6,5)(3,7,4,8), (1,6)(2,5)(3,4)(7,8), (1,6)(2,5)(3,4)(7,8) ]), 
  Group([ (1,3,6,4)(2,7,5,8), (1,7)(2,4)(3,5)(6,8), (1,6)(2,5)(3,4)
  (7,8), (1,6)(2,5)(3,4)(7,8) ]), Group([ (1,5,6,2)(3,7,4,8), (1,3)(2,8)(4,6)
  (5,7), (1,6)(2,5)(3,4)(7,8), (1,6)(2,5)(3,4)(7,8) ]), Group([ (1,5,6,2)
  (3,7,4,8), (1,3,6,4)(2,8,5,7), (1,6)(2,5)(3,4)(7,8), (1,6)(2,5)(3,4)
  (7,8) ]) ]
gap> Length(GsHNPtrueC4xC2);
6
gap> Collected(List(GsHNPfalseC4xC2,x->List(Orbits(x),Length)));
[ [ [ 4, 4 ], 3 ] ]
gap> Collected(List(GsHNPtrueC4xC2,x->List(Orbits(x),Length)));
[ [ [ 8 ], 6 ] ]
gap> Syl2G:=SylowSubgroup(G,2);
Group([ (2,5)(3,4), (2,5)(7,8), (1,2)(3,8)(4,7)(5,6), (1,8)(2,3)(4,5)
(6,7), (1,6)(2,5)(3,4)(7,8) ])
gap> IsNormal(G,Syl2G);
true
gap> IsConjugate(SymmetricGroup(8),Syl2G,TransitiveGroup(8,22));
true
\end{verbatim}
}
\end{example}

\smallskip
\begin{example}[$G=9Tm$ $(m=2,5,7,9,11,14,23)$]\label{ex9}
~{}\vspace*{-2mm}\\

(3-1) $G=9T2\simeq (C_3)^2$. 
{\small 
\begin{verbatim}
gap> Read("HNP.gap");
gap> G:=TransitiveGroup(9,2); # G=9T2=C3xC3
E(9)=3[x]3
gap> H:=Stabilizer(G,1); # H=1
Group(())
gap> FirstObstructionN(G).ker; # Obs1N=1
[ [  ], [ [  ], [  ] ] ]
gap> SchurMultPcpGroup(G); # M(G)=C3: Schur multiplier of G
[ 3 ]
gap> ScG:=SchurCoverG(G);
rec( SchurCover := Group([ (2,4,5)(3,6,7), (1,2,3)(4,7,8)(5,6,9) ]), 
  epi := [ (2,4,5)(3,6,7), (1,2,3)(4,7,8)(5,6,9) ] -> 
    [ (1,4,7)(2,5,8)(3,6,9), (1,2,9)(3,4,5)(6,7,8) ], Tid := [ 9, 7 ] )
gap> StructureDescription(TransitiveGroup(9,7));
"(C3 x C3) : C3"
gap> tG:=ScG.SchurCover; # tG=G~=(C3xC3):C3 is a Schur cover of G 
Group([ (2,4,5)(3,6,7), (1,2,3)(4,7,8)(5,6,9) ])
gap> tH:=PreImage(ScG.epi,H); # tH=H~=C3
Group([ (1,9,8)(2,5,4)(3,6,7) ])
gap> FirstObstructionN(tG,tH).ker; # Obs1N~=C3
[ [ 3 ], [ [ 3 ], [ [ 1 ] ] ] ]
gap> FirstObstructionDnr(tG,tH).Dnr; # Obs1Dnr~=1
[ [  ], [ [ 3 ], [  ] ] ]
gap> tGs:=AllSubgroups(tG);;
gap> Length(tGs);
19
gap> tGsHNPfalse:=Filtered(tGs,x->FirstObstructionDr(tG,x,tH).Dr[1]=[]);;
gap> tGsHNPtrue:=Filtered(tGs,x->FirstObstructionDr(tG,x,tH).Dr[1]=[3]);;
gap> List([tGsHNPfalse,tGsHNPtrue],Length);
[ 18, 1 ]
gap> Collected(List(tGsHNPfalse,StructureDescription));
[ [ "1", 1 ], [ "C3", 13 ], [ "C3 x C3", 4 ] ]
gap> Collected(List(tGsHNPtrue,StructureDescription));
[ [ "(C3 x C3) : C3", 1 ] ]
gap> Collected(List(tGsHNPfalse,x->StructureDescription(Image(ScG.epi,x))));
[ [ "1", 2 ], [ "C3", 16 ] ]
gap> Collected(List(tGsHNPtrue,x->StructureDescription(Image(ScG.epi,x))));
[ [ "C3 x C3", 1 ] ]
\end{verbatim}
}~\\\vspace*{-4mm}

(3-2) $G=9T5\simeq (C_3)^2\rtimes C_2$. 
{\small 
\begin{verbatim}
gap> Read("HNP.gap");
gap> G:=TransitiveGroup(9,5); # G=9T5=(C3xC3):C2
S(3)[1/2]S(3)=3^2:2
gap> H:=Stabilizer(G,1); # H=C2
Group([ (2,9)(3,8)(4,7)(5,6) ])
gap> FirstObstructionN(G).ker; # Obs1N=1
[ [  ], [ [ 2 ], [  ] ] ]
gap> SchurMultPcpGroup(G); # M(G)=C3: Schur multiplier of G
[ 3 ]
gap> ScG:=SchurCoverG(G);
rec( SchurCover := Group([ (2,3)(4,7)(5,6), (1,2,3)(4,7,8)(5,6,9), (2,4,5)
  (3,6,7) ]), Tid := [ 9, 12 ], 
  epi := [ (2,3)(4,7)(5,6), (1,2,3)(4,7,8)(5,6,9), (2,4,5)(3,6,7) ] -> 
    [ (2,9)(3,8)(4,7)(5,6), (1,4,7)(2,5,8)(3,6,9), (1,2,9)(3,4,5)(6,7,8) ] )
gap> tG:=ScG.SchurCover; # tG=G~ is a Schur cover of G 
Group([ (2,3)(4,7)(5,6), (1,2,3)(4,7,8)(5,6,9), (2,4,5)(3,6,7) ])
gap> StructureDescription(tG);
"((C3 x C3) : C3) : C2"
gap> tH:=PreImage(ScG.epi,H); # tH=H~=C6
Group([ (2,3)(4,7)(5,6), (1,8,9)(2,4,5)(3,7,6) ])
gap> FirstObstructionN(tG,tH).ker; # Obs1N~=C3
[ [ 3 ], [ [ 6 ], [ 2 ] ] ]
gap> FirstObstructionDnr(tG,tH).Dnr; # Obs1Dnr~=1
[ [  ], [ [ 6 ], [  ] ] ]
gap> tGs:=AllSubgroups(tG);;
gap> Length(tGs);
62
gap> tGsHNPfalse:=Filtered(tGs,x->FirstObstructionDr(tG,x,tH).Dr[1]=[]);;
gap> tGsHNPtrue:=Filtered(tGs,x->FirstObstructionDr(tG,x,tH).Dr[1]=[3]);;
gap> List([tGsHNPfalse,tGsHNPtrue],Length);
[ 60, 2 ]
gap> Collected(List(tGsHNPfalse,StructureDescription));
[ [ "1", 1 ], [ "C2", 9 ], [ "C3", 13 ], [ "C3 x C3", 4 ], [ "C3 x S3", 12 ], 
  [ "C6", 9 ], [ "S3", 12 ] ]
gap> Collected(List(tGsHNPtrue,StructureDescription));
[ [ "((C3 x C3) : C3) : C2", 1 ], [ "(C3 x C3) : C3", 1 ] ]
gap> Collected(List(tGsHNPfalse,x->StructureDescription(Image(ScG.epi,x))));
[ [ "1", 2 ], [ "C2", 18 ], [ "C3", 16 ], [ "S3", 24 ] ]
gap> Collected(List(tGsHNPtrue,x->StructureDescription(Image(ScG.epi,x))));
[ [ "(C3 x C3) : C2", 1 ], [ "C3 x C3", 1 ] ]
\end{verbatim}
}~\\\vspace*{-4mm}

(3-3) $G=9T7\simeq (C_3)^2\rtimes C_3$. 
{\small 
\begin{verbatim}
gap> Read("HNP.gap");
gap> G:=TransitiveGroup(9,7); # G=9T7=(C3xC3):C3
E(9):3=[3^2]3
gap> H:=Stabilizer(G,1); # H=C3
Group([ (3,4,5)(6,8,7) ])
gap> FirstObstructionN(G).ker; # Obs1N=1
[ [  ], [ [ 3 ], [  ] ] ]
gap> SchurMultPcpGroup(G); # M(G)=C3xC3: Schur multiplier of G
[ 3, 3 ]
gap> cGs:=MinimalStemExtensions(G);;
gap> for cG in cGs do
> tG:=cG.MinimalStemExtension;
> tH:=PreImage(cG.epi,H);
> Print(KerResH3Z(tG,tH));
> Print("\n");
> od;
[ [  ], [ [ 3 ], [  ] ] ]
[ [ 3 ], [ [ 3 ], [ [ 1 ] ] ] ]
[ [ 3 ], [ [ 3, 3 ], [ [ 0, 1 ] ] ] ]
[ [ 3 ], [ [ 3, 3 ], [ [ 0, 1 ] ] ] ]
gap> for cG in cGs do
> bG:=cG.MinimalStemExtension;
> bH:=PreImage(cG.epi,H);
> Print(FirstObstructionN(bG,bH).ker[1]);
> Print(FirstObstructionDnr(bG,bH).Dnr[1]);
> Print("\n");
> od;
[ 3 ][  ]
[ 3 ][ 3 ]
[ 3 ][ 3 ]
[ 3 ][ 3 ]
gap> cG:=cGs[1];
rec( MinimalStemExtension := <permutation group of size 81 with 4 generators>,
  epi := [ (1,5,15)(2,9,24)(3,12,29)(4,14,31)(6,18,37)(7,21,42)(8,23,44)(10,
        26,47)(11,28,49)(13,30,50)(16,34,55)(17,36,57)(19,39,60)(20,41,62)(22,
        43,63)(25,46,65)(27,48,66)(32,52,69)(33,54,71)(35,56,72)(38,59,74)(40,
        61,75)(45,64,76)(51,68,78)(53,70,79)(58,73,80)(67,77,81), 
      (1,2,6)(3,22,51)(4,40,32)(5,23,34)(7,35,25)(8,53,10)(9,36,12)(11,20,
        33)(13,38,16)(14,21,18)(15,63,69)(17,27,19)(24,72,47)(26,43,77)(28,61,
        68)(29,44,71)(30,73,52)(31,62,55)(37,50,60)(39,56,64)(41,70,46)(42,57,
        49)(45,58,67)(48,59,54)(65,75,81)(66,80,78)(74,79,76), 
      (1,3,10)(2,7,19)(4,11,25)(5,12,26)(6,16,32)(8,20,38)(9,21,39)(13,27,
        45)(14,28,46)(15,29,47)(17,33,51)(18,34,52)(22,40,58)(23,41,59)(24,42,
        60)(30,48,64)(31,49,65)(35,53,67)(36,54,68)(37,55,69)(43,61,73)(44,62,
        74)(50,66,76)(56,70,77)(57,71,78)(63,75,80)(72,79,81) ] -> 
    [ (3,4,5)(6,8,7), (1,4,7)(2,5,8)(3,6,9), (1,2,9)(3,4,5)(6,7,8) ] )
gap> bG:=cG.MinimalStemExtension;
<permutation group of size 81 with 4 generators>
gap> bH:=PreImage(cG.epi,H); # bH=H-=C3xC3
<permutation group of size 9 with 2 generators>
gap> KerResH3Z(bG,bH);
[ [  ], [ [ 3 ], [  ] ] ]
gap> FirstObstructionN(bG,bH).ker; # Obs1N-=C3
[ [ 3 ], [ [ 3, 3 ], [ [ 0, 1 ] ] ] ]
gap> FirstObstructionDnr(bG,bH).Dnr; # Obs1Dnr-=1
[ [  ], [ [ 3, 3 ], [  ] ] ]
gap> bGs:=AllSubgroups(bG);;
gap> Length(bGs);
50
gap> bGsHNPfalse:=Filtered(bGs,x->FirstObstructionDr(bG,x,bH).Dr[1]=[]);;
gap> Length(bGsHNPfalse);
36
gap> bGsHNPtrue:=Filtered(bGs,x->FirstObstructionDr(bG,x,bH).Dr[1]=[3]);;
gap> Length(bGsHNPtrue);
14
gap> Collected(List(bGsHNPfalse,StructureDescription));
[ [ "1", 1 ], [ "C3", 22 ], [ "C3 x C3", 7 ], [ "C9", 6 ] ]
gap> Collected(List(bGsHNPtrue,StructureDescription));
[ [ "(C3 x C3 x C3) : C3", 1 ], [ "(C3 x C3) : C3", 1 ], [ "C3 x C3", 9 ],
  [ "C3 x C3 x C3", 1 ], [ "C9 : C3", 2 ] ]
gap> Collected(List(bGsHNPfalse,x->StructureDescription(Image(cG.epi,x))));
[ [ "1", 2 ], [ "C3", 34 ] ]
gap> Collected(List(bGsHNPtrue,x->StructureDescription(Image(cG.epi,x))));
[ [ "(C3 x C3) : C3", 1 ], [ "C3 x C3", 13 ] ]
\end{verbatim}
}~\\\vspace*{-4mm}

(3-4) $G=9T9\simeq (C_3)^2\rtimes C_4$. 
{\small 
\begin{verbatim}
gap> Read("HNP.gap");
gap> G:=TransitiveGroup(9,9); # G=9T9=(C3xC3):C4
E(9):4
gap> H:=Stabilizer(G,1); # H=C4
Group([ (2,5,9,6)(3,4,8,7) ])
gap> FirstObstructionN(G).ker; # Obs1N=1
[ [  ], [ [ 4 ], [  ] ] ]
gap> SchurMultPcpGroup(G); # M(G)=C3: Schur multiplier of G
[ 3 ]
gap> ScG:=SchurCoverG(G);
rec( SchurCover := Group([ (1,2)(3,5,4,6)(7,11,10,12)(8,13,9,14)(15,17)
  (16,18), (2,3,4)(5,7,8)(6,9,10)(11,15,12)(13,16,14) ]), Tid := [ 18, 49 ], 
  epi := [ (1,2)(3,5,4,6)(7,11,10,12)(8,13,9,14)(15,17)(16,18), 
      (2,3,4)(5,7,8)(6,9,10)(11,15,12)(13,16,14) ] -> 
    [ (2,6,9,5)(3,7,8,4), (1,6,5)(2,7,3)(4,9,8) ] )
gap> tG:=ScG.SchurCover; # tG=G~ is a Schur cover of G 
Group([ (1,2)(3,5,4,6)(7,11,10,12)(8,13,9,14)(15,17)(16,18), (2,3,4)(5,7,8)
(6,9,10)(11,15,12)(13,16,14) ])
gap> StructureDescription(tG);
"((C3 x C3) : C3) : C4"
gap> tH:=PreImage(ScG.epi,H); # tH=H~=C12 
Group([ (1,2)(3,6,4,5)(7,12,10,11)(8,14,9,13)(15,17)(16,18), (1,17,18)
(2,15,16)(3,12,14)(4,11,13)(5,7,8)(6,10,9) ])
gap> StructureDescription(tH);
"C12"
gap> FirstObstructionN(tG,tH).ker; # Obs1N~=C3
[ [ 3 ], [ [ 12 ], [ 4 ] ] ]
gap> FirstObstructionDnr(tG,tH).Dnr; # Obs1Dnr~=1
[ [  ], [ [ 12 ], [  ] ] ]
gap> tGs:=AllSubgroups(tG);;
gap> Length(tGs);
81
gap> tGsHNPfalse:=Filtered(tGs,x->FirstObstructionDr(tG,x,tH).Dr[1]=[]);;
gap> tGsHNPtrue:=Filtered(tGs,x->FirstObstructionDr(tG,x,tH).Dr[1]=[3]);;
gap> List([tGsHNPfalse,tGsHNPtrue],Length);
[ 78, 3 ]
gap> Collected(List(tGsHNPfalse,StructureDescription));
[ [ "1", 1 ], [ "C12", 9 ], [ "C2", 9 ], [ "C3", 13 ], [ "C3 x C3", 4 ], 
  [ "C3 x S3", 12 ], [ "C4", 9 ], [ "C6", 9 ], [ "S3", 12 ] ]
gap> Collected(List(tGsHNPtrue,StructureDescription));
[ [ "((C3 x C3) : C3) : C2", 1 ], [ "((C3 x C3) : C3) : C4", 1 ], 
  [ "(C3 x C3) : C3", 1 ] ]
gap> Collected(List(tGsHNPfalse,x->StructureDescription(Image(ScG.epi,x))));
[ [ "1", 2 ], [ "C2", 18 ], [ "C3", 16 ], [ "C4", 18 ], [ "S3", 24 ] ]
gap> Collected(List(tGsHNPtrue,x->StructureDescription(Image(ScG.epi,x))));
[ [ "(C3 x C3) : C2", 1 ], [ "(C3 x C3) : C4", 1 ], [ "C3 x C3", 1 ] ]
\end{verbatim}
}~\\\vspace*{-4mm}

(3-5) $G=9T11\simeq (C_3)^2\rtimes C_6$. 
{\small 
\begin{verbatim}
gap> Read("HNP.gap");
gap> G:=TransitiveGroup(9,11); # G=9T11=(C3xC3):C6
E(9):6=1/2[3^2:2]S(3)
gap> H:=Stabilizer(G,1); # H=C6
Group([ (3,4,5)(6,8,7), (2,9)(3,8)(4,7)(5,6) ])
gap> FirstObstructionN(G).ker; # Obs1N=1
[ [  ], [ [ 6 ], [  ] ] ]
gap> SchurMultPcpGroup(G); # M(G)=C3: Schur multiplier of G
[ 3 ]
gap> ScG:=SchurCoverG(G);
rec( SchurCover := Group([ (1,2)(3,5)(4,6)(7,8)(9,11)(10,12)(13,14)(15,17)
  (16,18), (1,7,13)(2,8,14)(3,9,15)(4,10,16)(5,11,17)(6,12,18), (1,4,5)(2,3,6)
  (7,10,17)(8,15,12)(9,18,14)(11,13,16) ]), Tid := [ 18, 86 ], 
  epi := [ (1,2)(3,5)(4,6)(7,8)(9,11)(10,12)(13,14)(15,17)(16,18), 
      (1,7,13)(2,8,14)(3,9,15)(4,10,16)(5,11,17)(6,12,18), 
      (1,4,5)(2,3,6)(7,10,17)(8,15,12)(9,18,14)(11,13,16) ] -> 
    [ (2,9)(3,8)(4,7)(5,6), (3,4,5)(6,8,7), (1,4,7)(2,5,8)(3,6,9) ] )
gap> tG:=ScG.SchurCover; # tG=G~ is a Schur cover of G 
Group([ (1,2)(3,5)(4,6)(7,8)(9,11)(10,12)(13,14)(15,17)(16,18), (1,7,13)
(2,8,14)(3,9,15)(4,10,16)(5,11,17)(6,12,18), (1,4,5)(2,3,6)(7,10,17)(8,15,12)
(9,18,14)(11,13,16) ])
gap> StructureDescription(tG);
"((C3 x C3 x C3) : C3) : C2"
gap> tH:=PreImage(ScG.epi,H); # tH=H~=C6xC3
Group([ (1,2)(3,5)(4,6)(7,8)(9,11)(10,12)(13,14)(15,17)(16,18), (3,15,9)
(5,17,11), (1,7,13)(2,8,14)(3,15,9)(4,10,16)(5,17,11)(6,12,18) ])
gap> StructureDescription(tH);
"C6 x C3"
gap> FirstObstructionN(tG,tH).ker; # Obs1N~=C3
[ [ 3 ], [ [ 3, 6 ], [ [ 1, 4 ] ] ] ]
gap> FirstObstructionDnr(tG,tH).Dnr; # Obs1Dnr~=1
[ [  ], [ [ 3, 6 ], [  ] ] ]
gap> tGs:=AllSubgroups(tG);;
gap> Length(tGs);
142
gap> tGsHNPfalse:=Filtered(tGs,x->FirstObstructionDr(tG,x,tH).Dr[1]=[]);;
gap> tGsHNPtrue:=Filtered(tGs,x->FirstObstructionDr(tG,x,tH).Dr[1]=[3]);;
gap> List([tGsHNPfalse,tGsHNPtrue],Length);
[ 114, 28 ]
gap> Collected(List(tGsHNPfalse,StructureDescription));
[ [ "1", 1 ], [ "C2", 9 ], [ "C3", 22 ], [ "C3 x C3", 7 ], [ "C3 x S3", 12 ], 
  [ "C6", 36 ], [ "C6 x C3", 9 ], [ "C9", 6 ], [ "S3", 12 ] ]
gap> Collected(List(tGsHNPtrue,StructureDescription));
[ [ "((C3 x C3 x C3) : C3) : C2", 1 ], [ "((C3 x C3) : C3) : C2", 1 ], 
  [ "(C3 x C3 x C3) : C3", 1 ], [ "(C3 x C3) : C3", 1 ], [ "C3 x C3", 9 ], 
  [ "C3 x C3 x C3", 1 ], [ "C3 x C3 x S3", 3 ], [ "C3 x S3", 9 ], 
  [ "C9 : C3", 2 ] ]
gap> Collected(List(tGsHNPfalse,x->StructureDescription(Image(ScG.epi,x))));
[ [ "1", 2 ], [ "C2", 18 ], [ "C3", 34 ], [ "C6", 36 ], [ "S3", 24 ] ]
gap> Collected(List(tGsHNPtrue,x->StructureDescription(Image(ScG.epi,x))));
[ [ "(C3 x C3) : C2", 1 ], [ "(C3 x C3) : C3", 1 ], [ "(C3 x C3) : C6", 1 ], 
  [ "C3 x C3", 13 ], [ "C3 x S3", 12 ] ]
\end{verbatim}
}~\\\vspace*{-4mm}

(3-6) $G=9T14\simeq (C_3)^2\rtimes Q_8$. 
{\small 
\begin{verbatim}
gap> Read("HNP.gap");
gap> G:=TransitiveGroup(9,14); # G=9T14=(C3xC3):Q8
M(9)=E(9):Q_8
gap> H:=Stabilizer(G,1); # H=Q8
Group([ (2,8,9,3)(4,6,7,5), (2,5,9,6)(3,4,8,7) ])
gap> StructureDescription(H);
"Q8"
gap> FirstObstructionN(G).ker; # Obs1N=1
[ [  ], [ [ 2, 2 ], [  ] ] ]
gap> SchurMultPcpGroup(G); # M(G)=C3: Schur multiplier of G
[ 3 ]
gap> ScG:=SchurCoverG(G);
rec( SchurCover := Group([ (2,4,3,5)(6,9,7,8)(10,20,13,19)(11,14,12,17)
  (15,24,16,25)(18,22,21,23), (2,6,3,7)(4,8,5,9)(10,16,13,15)(11,22,12,23)
  (14,18,17,21)(19,25,20,24), (1,2,3)(4,10,11)(5,12,13)(6,14,15)(7,16,17)
  (8,18,19)(9,20,21)(22,23,26)(24,25,27) ]), Tid := [ 27, 83 ], 
  epi := [ (2,4,3,5)(6,9,7,8)(10,20,13,19)(11,14,12,17)(15,24,16,25)(18,22,21,
        23), (2,6,3,7)(4,8,5,9)(10,16,13,15)(11,22,12,23)(14,18,17,21)(19,25,
        20,24), (1,2,3)(4,10,11)(5,12,13)(6,14,15)(7,16,17)(8,18,19)(9,20,
        21)(22,23,26)(24,25,27) ] -> 
    [ (2,8,9,3)(4,6,7,5), (2,6,9,5)(3,7,8,4), (1,6,5)(2,7,3)(4,9,8) ] )
gap> StructureDescription(TransitiveGroup(27,83));
"((C3 x C3) : C3) : Q8"
gap> tG:=ScG.SchurCover; # tG=G~ is a Schur cover of G 
Group([ (2,4,3,5)(6,9,7,8)(10,20,13,19)(11,14,12,17)(15,24,16,25)
(18,22,21,23), (2,6,3,7)(4,8,5,9)(10,16,13,15)(11,22,12,23)(14,18,17,21)
(19,25,20,24), (1,2,3)(4,10,11)(5,12,13)(6,14,15)(7,16,17)(8,18,19)(9,20,21)
(22,23,26)(24,25,27) ])
gap> tH:=PreImage(ScG.epi,H); # tH=H~=C3xQ8
Group([ (2,4,3,5)(6,9,7,8)(10,20,13,19)(11,14,12,17)(15,24,16,25)
(18,22,21,23), (2,7,3,6)(4,9,5,8)(10,15,13,16)(11,23,12,22)(14,21,17,18)
(19,24,20,25), (1,27,26)(2,24,22)(3,25,23)(4,16,21)(5,15,18)(6,19,12)(7,20,11)
(8,13,14)(9,10,17) ])
gap> StructureDescription(tH);
"C3 x Q8"
gap> FirstObstructionN(tG,tH).ker; # Obs1N~=C3
[ [ 3 ], [ [ 2, 6 ], [ [ 0, 2 ] ] ] ]
gap> FirstObstructionDnr(tG,tH).Dnr; # Obs1Dnr~=1
[ [  ], [ [ 2, 6 ], [  ] ] ]
gap> tGs:=AllSubgroups(tG);;
gap> Length(tGs);
138
gap> tGsHNPfalse:=Filtered(tGs,x->FirstObstructionDr(tG,x,tH).Dr[1]=[]);;
gap> tGsHNPtrue:=Filtered(tGs,x->FirstObstructionDr(tG,x,tH).Dr[1]=[3]);;
gap> List([tGsHNPfalse,tGsHNPtrue],Length);
[ 132, 6 ]
gap> Collected(List(tGsHNPfalse,StructureDescription));
[ [ "1", 1 ], [ "C12", 27 ], [ "C2", 9 ], [ "C3", 13 ], [ "C3 x C3", 4 ], 
  [ "C3 x Q8", 9 ], [ "C3 x S3", 12 ], [ "C4", 27 ], [ "C6", 9 ], 
  [ "Q8", 9 ], [ "S3", 12 ] ]
gap> Collected(List(tGsHNPtrue,StructureDescription));
[ [ "((C3 x C3) : C3) : C2", 1 ], [ "((C3 x C3) : C3) : C4", 3 ], 
  [ "((C3 x C3) : C3) : Q8", 1 ], [ "(C3 x C3) : C3", 1 ] ]
gap> Collected(List(tGsHNPfalse,x->StructureDescription(Image(ScG.epi,x))));
[ [ "1", 2 ], [ "C2", 18 ], [ "C3", 16 ], [ "C4", 54 ], [ "Q8", 18 ], 
  [ "S3", 24 ] ]
gap> Collected(List(tGsHNPtrue,x->StructureDescription(Image(ScG.epi,x))));
[ [ "(C3 x C3) : C2", 1 ], [ "(C3 x C3) : C4", 3 ], [ "(C3 x C3) : Q8", 1 ], 
  [ "C3 x C3", 1 ] ]
\end{verbatim}
}~\\\vspace*{-4mm}

(3-7) $G=9T23\simeq ((C_3)^2\rtimes Q_8)\rtimes C_3$. 
{\small 
\begin{verbatim}
gap> Read("HNP.gap");
gap> G:=TransitiveGroup(9,23); # G=9T23=((C3xC3):Q8):C3
E(9):2A_4
gap> H:=Stabilizer(G,1); # H=SL(2,3)
Group([ (3,4,5)(6,8,7), (2,4,6)(5,9,7) ]) 
gap> StructureDescription(H);
"SL(2,3)"
gap> FirstObstructionN(G).ker; # Obs1N=1
[ [  ], [ [ 3 ], [  ] ] ]
gap> SchurMultPcpGroup(G); # M(G)=C3: Schur multiplier of G
[ 3 ]
gap> ScG:=SchurCoverG(G);
rec( SchurCover := Group([ (2,4,5)(3,7,6)(10,22,21)(12,19,24)(15,20,25)
  (17,23,18), (2,5,3,6)(4,8,7,9)(10,23,17,22)(11,12,16,15)(13,18,14,21)
  (19,25,20,24), (1,2,3)(4,10,11)(5,12,13)(6,14,15)(7,16,17)(8,18,19)(9,20,21)
  (22,26,23)(24,27,25) ]), Tid := [ 27, 212 ], 
  epi := [ (2,4,5)(3,7,6)(10,22,21)(12,19,24)(15,20,25)(17,23,18), 
      (2,5,3,6)(4,8,7,9)(10,23,17,22)(11,12,16,15)(13,18,14,21)(19,25,20,24), 
      (1,2,3)(4,10,11)(5,12,13)(6,14,15)(7,16,17)(8,18,19)(9,20,21)(22,26,
        23)(24,27,25) ] -> [ (3,4,5)(6,8,7), (2,8,9,3)(4,6,7,5), 
      (1,6,5)(2,7,3)(4,9,8) ] )
gap> StructureDescription(TransitiveGroup(27,212));
"(((C3 x C3) : C3) : Q8) : C3"
gap> tG:=ScG.SchurCover; # tG=G~ is a Schur cover of G 
Group([ (2,4,5)(3,7,6)(10,22,21)(12,19,24)(15,20,25)(17,23,18), (2,5,3,6)
(4,8,7,9)(10,23,17,22)(11,12,16,15)(13,18,14,21)(19,25,20,24), (1,2,3)
(4,10,11)(5,12,13)(6,14,15)(7,16,17)(8,18,19)(9,20,21)(22,26,23)(24,27,25) ])
gap> tH:=PreImage(ScG.epi,H); # tH=H~=C3xSL(2,3)
Group([ (1,27,26)(2,15,17)(3,12,10)(4,20,23)(5,25,18)(6,24,21)(7,19,22)
(8,11,14)(9,16,13), (1,26,27)(2,13,19)(3,14,20)(4,18,15)(5,22,11)(6,23,16)
(7,21,12)(8,17,24)(9,10,25), (1,26,27)(2,23,25)(3,22,24)(4,18,15)(5,17,20)
(6,10,19)(7,21,12)(8,14,11)(9,13,16) ])
gap> StructureDescription(tH);
"C3 x SL(2,3)"
gap> FirstObstructionN(tG,tH).ker; # Obs1N~=C3
[ [ 3 ], [ [ 3, 3 ], [ [ 1, 2 ] ] ] ]
gap> FirstObstructionDnr(tG,tH).Dnr; # Obs1Dnr~=1
[ [  ], [ [ 3, 3 ], [  ] ] ]
gap> tGs:=AllSubgroups(tG);;
gap> Length(tGs);
495
gap> tGsHNPfalse:=Filtered(tGs,x->FirstObstructionDr(tG,x,tH).Dr[1]=[]);;
gap> tGsHNPtrue:=Filtered(tGs,x->FirstObstructionDr(tG,x,tH).Dr[1]=[3]);;
gap> List([tGsHNPfalse,tGsHNPtrue],Length);
[ 384, 111 ]
gap> Collected(List(tGsHNPfalse,x->StructureDescription(Image(ScG.epi,x))));
[ [ "1", 2 ], [ "C2", 18 ], [ "C3", 88 ], [ "C4", 54 ], [ "C6", 144 ], 
  [ "Q8", 18 ], [ "S3", 24 ], [ "SL(2,3)", 36 ] ]
gap> Collected(List(tGsHNPtrue,x->StructureDescription(Image(ScG.epi,x))));
[ [ "((C3 x C3) : Q8) : C3", 1 ], [ "(C3 x C3) : C2", 1 ], 
  [ "(C3 x C3) : C3", 4 ], [ "(C3 x C3) : C4", 3 ], [ "(C3 x C3) : C6", 4 ], 
  [ "(C3 x C3) : Q8", 1 ], [ "C3 x C3", 49 ], [ "C3 x S3", 48 ] ]
\end{verbatim}
}
\end{example}


\smallskip
\begin{example}[$G=10T7\simeq A_5$, $G=10T26\simeq \PSL_2(\bF_9)\simeq A_6$ and $G=10T32\simeq S_6$]\label{ex10}
~{}\vspace*{-2mm}\\

(4-1) $G=10T7\simeq A_5$. 
{\small 
\begin{verbatim}
gap> Read("HNP.gap");
gap> G:=TransitiveGroup(10,7); # G=10T7=A5
A_5(10)
gap> H:=Stabilizer(G,1); # H=S3
Group([ (2,8)(3,6)(4,7)(5,10), (2,10)(3,9)(4,8)(5,7) ])
gap> StructureDescription(H);
"S3"
gap> FirstObstructionN(G).ker; # Obs1N=C2
[ [ 2 ], [ [ 2 ], [ [ 1 ] ] ] ]
gap> FirstObstructionDnr(G).Dnr; # Obs1Dnr=1
[ [  ], [ [ 2 ], [  ] ] ]
gap> Gs:=AllSubgroups(G);;
gap> Length(Gs);
59
gap> GsHNPfalse:=Filtered(Gs,x->FirstObstructionDr(G,x).Dr[1]=[]);;
gap> GsHNPtrue:=Filtered(Gs,x->FirstObstructionDr(G,x).Dr[1]=[2]);;
gap> List([GsHNPfalse,GsHNPtrue],Length);
[ 48 , 11 ]
gap> Collected(List(GsHNPfalse,StructureDescription));
[ [ "1", 1 ], [ "C2", 15 ], [ "C3", 10 ], [ "C5", 6 ], [ "D10", 6 ], 
  [ "S3", 10 ] ]
gap> Collected(List(GsHNPtrue,StructureDescription));
[ [ "A4", 5 ], [ "A5", 1 ], [ "C2 x C2", 5 ] ]
\end{verbatim}
}~\\\vspace*{-4mm}

(4-2) $G=10T26\simeq \PSL_2(\bF_9)\simeq A_6$. 
{\small 
\begin{verbatim}
gap> Read("HNP.gap");
gap> G:=TransitiveGroup(10,26); # G=10T26=SPL(2,9)=A6
L(10)=PSL(2,9)
gap> H:=Stabilizer(G,1); # H=(C3xC3):C4
Group([ (3,9,6,10)(4,8,5,7), (2,4)(3,7)(6,9)(8,10) ])
gap> StructureDescription(H);
"(C3 x C3) : C4"
gap> FirstObstructionN(G).ker; # Obs1N=C4
[ [ 4 ], [ [ 4 ], [ [ 1 ] ] ] ]
gap> FirstObstructionDnr(G).Dnr; # ObsDnr=C2
[ [ 2 ], [ [ 4 ], [ [ 2 ] ] ] ]
gap> Gs:=AllSubgroups(G);;
gap> Length(Gs);
501
gap> GsHNPfalse:=Filtered(Gs,x->FirstObstructionDr(G,x).Dr[1]<>[4]);;
gap> GsHNPtrue:=Filtered(Gs,x->FirstObstructionDr(G,x).Dr[1]=[4]);;
gap> List([GsHNPfalse,GsHNPtrue],Length);
[ 425, 76 ]
gap> Collected(List(GsHNPfalse,StructureDescription));
[ [ "(C3 x C3) : C2", 10 ], [ "(C3 x C3) : C4", 10 ], [ "1", 1 ], 
  [ "A4", 30 ], [ "A5", 12 ], [ "C2", 45 ], [ "C2 x C2", 30 ], 
  [ "C3", 40 ], [ "C3 x C3", 10 ], [ "C4", 45 ], [ "C5", 36 ], 
  [ "D10", 36 ], [ "S3", 120 ] ]
gap> Collected(List(GsHNPtrue,StructureDescription));
[ [ "A6", 1 ], [ "D8", 45 ], [ "S4", 30 ] ]
\end{verbatim}
}~\\\vspace*{-4mm}

(4-3) $G=10T32\simeq S_6$. 
{\small 
\begin{verbatim}
gap> Read("HNP.gap");
gap> G:=TransitiveGroup(10,32); # G=10T32=S6
S_6(10)=L(10):2
gap> GeneratorsOfGroup(G);
[ (1,2,10)(3,4,5)(6,7,8), (1,3,2,6)(4,5,8,7), (1,2)(4,7)(5,8)(9,10), 
  (3,6)(4,7)(5,8) ]
gap> H:=Stabilizer(G,1); # H=(S3xS3):C2
Group([ (3,6)(4,7)(5,8), (3,10)(6,9)(7,8), (2,4)(3,7)(6,9)(8,10) ])
gap> FirstObstructionN(G).ker; # Obs1N=C2
[ [ 2 ], [ [ 2, 2 ], [ [ 1, 1 ] ] ] ]
gap> FirstObstructionDnr(G).Dnr; # Obs1Dnr=1
[ [  ], [ [ 2, 2 ], [  ] ] ]
gap> Gs:=AllSubgroups(G);;
gap> Length(Gs);
1455
gap> GsHNPfalse:=Filtered(Gs,x->FirstObstructionDr(G,x).Dr[1]=[]);;
gap> Length(GsHNPfalse);
1153
gap> GsHNPtrue:=Filtered(Gs,x->FirstObstructionDr(G,x).Dr[1]=[2]);;
gap> Length(GsHNPtrue);
302
gap> Collected(List(GsHNPfalse,StructureDescription));
[ [ "(C3 x C3) : C2", 10 ], [ "(C3 x C3) : C4", 10 ], 
  [ "(S3 x S3) : C2", 10 ], [ "1", 1 ], [ "A4", 30 ], [ "A5", 12 ], 
  [ "C2", 75 ], [ "C2 x C2", 120 ], [ "C3", 40 ], [ "C3 x C3", 10 ], 
  [ "C3 x S3", 40 ], [ "C4", 90 ], [ "C5", 36 ], [ "C5 : C4", 36 ], 
  [ "C6", 120 ], [ "D10", 36 ], [ "D12", 120 ], [ "D8", 135 ], [ "S3", 160 ], 
  [ "S3 x S3", 20 ], [ "S4", 30 ], [ "S5", 12 ] ]
gap> Collected(List(GsHNPtrue,StructureDescription));
[ [ "A6", 1 ], [ "C2 x A4", 30 ], [ "C2 x C2", 45 ], [ "C2 x C2 x C2", 30 ], 
  [ "C2 x D8", 45 ], [ "C2 x S4", 30 ], [ "C4 x C2", 45 ], [ "D8", 45 ], 
  [ "S4", 30 ], [ "S6", 1 ] ]
gap> GsHNPtrueMin:=Filtered(GsHNPtrue,x->Length(Filtered(GsHNPtrue,
> y->IsSubgroup(x,y)))=1);;
gap> Collected(List(GsHNPtrueMin,StructureDescription));
[ [ "C2 x C2", 45 ], [ "D8", 45 ] ]
gap> GsHNPfalseC2xC2:=Filtered(GsHNPfalse,x->IdSmallGroup(x)=[4,2]);;
gap> Length(GsHNPfalseC2xC2);
120
gap> GsHNPtrueC2xC2:=Filtered(GsHNPtrue,x->IdSmallGroup(x)=[4,2]);;
gap> Length(GsHNPtrueC2xC2); # there exist 45 minimal true cases
45
gap> Collected(List(GsHNPfalseC2xC2,x->List(Orbits(x),Length)));
[ [ [ 2, 2, 2, 4 ], 3 ], [ [ 2, 2, 4 ], 15 ], [ [ 2, 2, 4, 2 ], 6 ], 
  [ [ 2, 4, 2 ], 30 ], [ [ 2, 4, 2, 2 ], 9 ], [ [ 4, 2, 2 ], 45 ], 
  [ [ 4, 2, 2, 2 ], 12 ] ]
gap> Collected(List(GsHNPtrueC2xC2,x->List(Orbits(x),Length)));
[ [ [ 2, 2, 2, 2, 2 ], 45 ] ]
gap> Collected(List(GsHNPfalseC2xC2,x->Collected(List(x,
> y->List(Orbits(Group(y)),Length)))));
[ [ [ [ [  ], 1 ], [ [ 2, 2, 2 ], 2 ], [ [ 2, 2, 2, 2 ], 1 ] ], 90 ], 
  [ [ [ [  ], 1 ], [ [ 2, 2, 2, 2 ], 3 ] ], 30 ] ]
gap> Collected(List(GsHNPtrueC2xC2,x->Collected(List(x,
> y->List(Orbits(Group(y)),Length)))));
[ [ [ [ [  ], 1 ], [ [ 2, 2, 2 ], 2 ], [ [ 2, 2, 2, 2 ], 1 ] ], 45 ] ]
gap> S10:=SymmetricGroup(10);
Sym( [ 1 .. 10 ] )
gap> NS10G:=Normalizer(S10,G);
Group([ (1,8,4)(2,7,5)(3,9,10), (1,5,8,10)(2,7,9,3), (1,8)(2,3)(4,6)
(7,9), (2,3)(5,10)(7,9), (2,10,7,5)(3,4,9,6) ])
gap> StructureDescription(NS10G);
"(A6 : C2) : C2"
gap> CS10G:=Centralizer(S10,G);
Group(())
gap> StructureDescription(NS10G/CS10G); # Aut(G)=NS10G/CS10G<=S10
"(A6 : C2) : C2"
gap> Collected(List(GsHNPfalseC2xC2,
> x->StructureDescription(Normalizer(NS10G,x))));
[ [ "C2 x D8", 90 ], [ "C2 x S4", 30 ] ]
gap> Collected(List(GsHNPtrueC2xC2,
> x->StructureDescription(Normalizer(NS10G,x))));
[ [ "C8 : (C2 x C2)", 45 ] ]
gap> ChG:=CharacteristicSubgroups(G);
[ Group(()), Group([ (1,8,3)(2,6,4)(5,10,7), (1,10)(2,9)(3,6)(4,7), (1,2)(3,6)
  (4,8)(5,7) ]), S_6(10)=L(10):2 ]
gap> List(ChG,StructureDescription);
[ "1", "A6", "S6" ]
gap> GsHNPfalseD4:=Filtered(GsHNPfalse,x->IdSmallGroup(x)=[8,3]);;
gap> Length(GsHNPfalseD4);
135
gap> GsHNPtrueD4:=Filtered(GsHNPtrue,x->IdSmallGroup(x)=[8,3]);;
gap> Length(GsHNPtrueD4);
45
gap> A6:=DerivedSubgroup(G);
Group([ (1,8,3)(2,6,4)(5,10,7), (1,10)(2,9)(3,6)(4,7), (1,2)(3,6)(4,8)(5,7) ])
gap> Collected(List(GsHNPfalseD4,x->StructureDescription(Intersection(A6,x))));
[ [ "C2 x C2", 90 ], [ "C4", 45 ] ]
gap> Collected(List(GsHNPtrueD4,x->StructureDescription(Intersection(A6,x))));
[ [ "D8", 45 ] ]
\end{verbatim}
}
\end{example}

\smallskip
\begin{example}[$G=14T30\simeq \PSL_2(\bF_{13})$]\label{ex14}
~{}\vspace*{-2mm}\\

(5) $G=14T30\simeq \PSL_2(\bF_{13})$. 
{\small 
\begin{verbatim}
gap> Read("HNP.gap");
gap> G:=TransitiveGroup(14,30); # G=14T30=PSL(2,13)
L(14)=PSL(2,13)
gap> H:=Stabilizer(G,1); # H=C13:C6
Group([ (2,12,11,5,9,3)(4,6,7,8,10,13), (2,6)(3,8)(4,13)(5,14)(9,11)(10,12) ])
gap> StructureDescription(H);
"C13 : C6"
gap> FirstObstructionN(G).ker; # Obs1N=C6
[ [ 6 ], [ [ 6 ], [ [ 1 ] ] ] ]
gap> FirstObstructionDnr(G).Dnr; # ObsDnr=C3
[ [ 3 ], [ [ 6 ], [ [ 2 ] ] ] ]
gap> Gs:=AllSubgroups(G);;
gap> Length(Gs);
942
gap> GsHNPfalse1:=Filtered(Gs,x->FirstObstructionDr(G,x).Dr[1]=[]);;
gap> GsHNPfalse2:=Filtered(Gs,x->FirstObstructionDr(G,x).Dr[1]=[3]);;
gap> GsHNPtrue1:=Filtered(Gs,x->FirstObstructionDr(G,x).Dr[1]=[2]);;
gap> GsHNPtrue2:=Filtered(Gs,x->FirstObstructionDr(G,x).Dr[1]=[6]);;
gap> List([GsHNPfalse1,GsHNPfalse2,GsHNPtrue1,GsHNPtrue2],Length);
[ 276, 392, 91, 183 ]
gap> Sum(last);
942
gap> Collected(List(GsHNPfalse1,StructureDescription));
[ [ "1", 1 ], [ "C13", 14 ], [ "C2", 91 ], [ "C7", 78 ], [ "D14", 78 ], 
  [ "D26", 14 ] ]
gap> Collected(List(GsHNPfalse2,StructureDescription));
[ [ "C13 : C3", 14 ], [ "C13 : C6", 14 ], [ "C3", 91 ], [ "C6", 91 ], 
  [ "S3", 182 ] ]
gap> Collected(List(GsHNPtrue1,StructureDescription));
[ [ "C2 x C2", 91 ] ]
gap> Collected(List(GsHNPtrue2,StructureDescription));
[ [ "A4", 91 ], [ "D12", 91 ], [ "PSL(2,13)", 1 ] ]
\end{verbatim}
}
\end{example}

\smallskip
\begin{example}[$G=15T9\simeq (C_5)^2\rtimes C_3$ and $G=15T14\simeq (C_5)^2\rtimes S_3$]\label{ex15}
~{}\vspace*{-2mm}\\

(6-1) $G=15T9\simeq (C_5)^2\rtimes C_3$. 
{\small 
\begin{verbatim}
gap> Read("HNP.gap");
gap> G:=TransitiveGroup(15,9); # G=15T9=(C5xC5):C3
[5^2]3
gap> H:=Stabilizer(G,1); # H=C5
Group([ (2,5,8,11,14)(3,15,12,9,6) ])
gap> StructureDescription(H);
"C5"
gap> FirstObstructionN(G).ker; # Obs1N=C5
[ [ 5 ], [ [ 5 ], [ [ 1 ] ] ] ]
gap> FirstObstructionDnr(G).Dnr; # ObsDnr=1
[ [  ], [ [ 5 ], [  ] ] ]
gap> Gs:=AllSubgroups(G);;
gap> Length(Gs);
34
gap> GsHNPfalse:=Filtered(Gs,x->FirstObstructionDr(G,x).Dr[1]=[]);;
gap> GsHNPtrue:=Filtered(Gs,x->FirstObstructionDr(G,x).Dr[1]=[5]);;
gap> List([GsHNPfalse,GsHNPtrue],Length);
[ 32, 2 ]
gap> Collected(List(GsHNPfalse,StructureDescription));
[ [ "1", 1 ], [ "C3", 25 ], [ "C5", 6 ] ]
gap> Collected(List(GsHNPtrue,StructureDescription));
[ [ "(C5 x C5) : C3", 1 ], [ "C5 x C5", 1 ] ]
\end{verbatim}
}~\\\vspace*{-4mm}

(6-2) $G=15T14\simeq (C_5)^2\rtimes S_3$. 
{\small 
\begin{verbatim}
gap> Read("HNP.gap");
gap> G:=TransitiveGroup(15,14); # G=15T14=(C5xC5):S3
5^2:2[1/2]S(3)
gap> H:=Stabilizer(G,1); # H=C10
Group([ (2,5,8,11,14)(3,15,12,9,6), (2,12)(3,11)(4,13)(5,9)(6,8)(7,10)(14,15) ])
gap> StructureDescription(H);
"C10"
gap> FirstObstructionN(G).ker; # Obs1N=C5
[ [ 5 ], [ [ 10 ], [ 2 ] ] ]
gap> FirstObstructionDnr(G).Dnr; # ObsDnr=1
[ [  ], [ [ 10 ], [  ] ] ]
gap> Gs:=AllSubgroups(G);;
gap> Length(Gs);
96
gap> GsHNPfalse:=Filtered(Gs,x->FirstObstructionDr(G,x).Dr[1]=[]);;
gap> GsHNPtrue:=Filtered(Gs,x->FirstObstructionDr(G,x).Dr[1]=[5]);;
gap> List([GsHNPfalse,GsHNPtrue],Length);
[ 90, 6 ]
gap> Collected(List(GsHNPfalse,StructureDescription));
[ [ "1", 1 ], [ "C10", 15 ], [ "C2", 15 ], [ "C3", 25 ], [ "C5", 6 ], 
  [ "D10", 3 ], [ "S3", 25 ] ]
gap> Collected(List(GsHNPtrue,StructureDescription));
[ [ "(C5 x C5) : C3", 1 ], [ "(C5 x C5) : S3", 1 ], [ "C5 x C5", 1 ], 
  [ "C5 x D10", 3 ] ]
\end{verbatim}
}
\end{example}

\section{Application $1$: $R$-equivalence in algebraic $k$-tori}\label{S7}

\begin{definition}
Let $k$ be a field and 
$T$ be an algebraic $k$-torus. 
A exact sequence of algebraic $k$-tori 
\begin{align*}
1\to S\to Q\to T\to 1
\end{align*}
is called {\it flabby resolution} of $T$ if 
\begin{align*}
0\to\widehat{T}\to\widehat{Q}\to\widehat{S}\to 0
\end{align*}
is a flabby resolution of $G$-lattice $\widehat{T}$. 
\end{definition}

\begin{definition}[{Manin \cite[II. \S 14]{Man74}}]
We say that a rational map of $k$-varieties $f:Z\to X$ 
{\it covers a point} $x\in X(k)$ if there exists a point $z\in Z(k)$ 
such that $f$ is defined at $z$ and $f(z)=x$. 
Two points $x,y\in X(k)$ are called {\it $R$-equivalent} if 
there exist a finite sequence of points $x=x_1,\ldots,x_r=y$ and 
rational maps $f_i:\bP^1\to X$ $(1\leq i\leq r-1)$ 
such that $f_i$ covers $x_i$, $x_{i+1}$. 
\end{definition}

\begin{theorem}[{Colliot-Th\'{e}l\`{e}ne and Sansuc \cite[Theorem 2, page 199]{CTS77}, see also \cite[Section 17.1]{Vos98}}]
Let $k$ be a field, 
$T$ be an algebraic $k$-torus and 
$1\to S\to Q\to T\to 1$ be a flabby resolution of $T$. 
Then the connecting homomorphism 
\begin{align*}
T(k)\to H^1(k,S)
\end{align*}
induces an isomorphism 
\begin{align*}
T(k)/R\simeq H^1(k,S).
\end{align*}
\end{theorem}
\begin{theorem}[{Colliot-Th\'{e}l\`{e}ne and Sansuc \cite[Corollary 5, page 201]{CTS77}, see also \cite[Section 17.2]{Vos98}}]\label{thCTS77-5}
Let $k$ be a field and 
$T$ be an algebraic $k$-torus which splits over finite Galois extension 
$K$ of $k$ with $G={\rm Gal}(K/k)$. 
Let $1\to S\to Q\to T\to 1$ be a flabby resolution of $T$. 
Then\\ 
{\rm (i)} If $k=\bF_q$ or a field of cohomological dimension 
${\rm cd}(k)\leq 1$, then 
\begin{align*}
T(k)/R=0;
\end{align*}
{\rm (ii)} If $k$ is a local field, 
then 
\begin{align*}
T(k)/R\simeq H^1(G,\widehat{S})^{\vee};
\end{align*}
{\rm (iii)} If $k$ is a global field, then there exists an exact sequence 
\begin{align*}
0\to\Sha^2(G,\widehat{S})^{\vee}\to T(k)/R\to \Ch^1(G,\widehat{S})^{\vee}\to 0
\end{align*}
where 
\begin{align*}
\Sha^2(G,\widehat{S})&={\rm Ker}\{
H^2(G,\widehat{S})\xrightarrow{\rm res}\bigoplus_{v\in V_k} 
H^2(G_v,\widehat{S})
\},\\
\Ch^1(G,\widehat{S})&={\rm Coker}\{
H^1(G,\widehat{S})\xrightarrow{\rm res}\bigoplus_{v\in V_k} 
H^1(G_v,\widehat{S})
\}.
\end{align*}
\end{theorem}

When $T=R^{(1)}_{K/k}(\bG_m)$ and $K/k$ is a finite Galois extension, 
we have $\widehat{T}=J_G$ and $H^1(k,\widehat{S})\simeq H^3(G,\bZ)$. 
Hence Theorem \ref{thCTS77-5} can be stated as follows:
\begin{theorem}[{Colliot-Th\'{e}l\`{e}ne and Sansuc \cite[Corollary 1, page 207]{CTS77}, see also \cite[Section 17.2]{Vos98}}]\label{thCTS77-1}
Let $k$ be a field and 
$K/k$ be a finite Galois extension with Galois group $G={\rm Gal}(K/k)$. 
Let $T=R^{(1)}_{K/k}(\bG_m)$ be the norm one torus defined by $K/k$. 
Then\\ 
{\rm (i)} If $k=\bF_q$ or a field of cohomological dimension 
${\rm cd}(k)\leq 1$, then 
\begin{align*}
T(k)/R=0;
\end{align*}
{\rm (ii)} If $k$ is a local field, then 
\begin{align*}
T(k)/R\simeq H^3(G,\bZ)^{\vee};
\end{align*}
{\rm (iii)} If $k$ is a global field, then there exists an exact sequence 
\begin{align*}
0\to\Sha^4(G,\bZ)^{\vee}\to T(k)/R\to \Ch^3(G,\bZ)^{\vee}\to 0.
\end{align*}
\end{theorem}
When $k$ is a local field, 
Voskresenskii's theorem (\cite{Vos67}), 
Kunyavskii's theorem (Theorem \ref{thKun1}), 
Theorem \ref{thmain1-4p} and Theorem \ref{thmain1-5p} 
enable us to get $T(k)/R$ 
for algebraic $k$-tori $T$ of dimension $\leq 5$. 
We also refer to Merkurjev \cite{Mer08} for algebraic 
$k$-tori $T$ of dimension $3$. 
\begin{theorem}
Let $k$ be a local field and 
$T$ be an algebraic $k$-torus of dimension $n\leq 5$. 
Then 
\begin{align*}
T(k)/R\leq 
\begin{cases}
0 & (n=1,2),\\
\bZ/2\bZ & (n=3),\\
(\bZ/2\bZ)^{\oplus 2} & (n=4,5)\\
\end{cases}
\end{align*}
and 
$T(k)/R\simeq H^1(G,[\widehat{T}]^{fl})$ 
is given as in Theorem \ref{thKun1} $(n=3)$, 
Theorem \ref{thmain1-4p} $(n=4)$ 
and Theorem \ref{thmain1-5p} $(n=5)$. 
\end{theorem}
Also, Theorem \ref{thmain2} enables us 
to obtain $T(k)/R\simeq H^1(k,{\rm Pic}\,\overline{X})\simeq 
H^1(G,[J_{G/H}]^{fl})$ for norm one tori $T=R^{(1)}_{K/k}(\bG_m)$ 
where $k$ is a local field and $[K:k]=n\leq 15$ and $n\neq 12$.

\begin{theorem}
Let $2\leq n\leq 15$ be an integer with $n\neq 12$. 
Let $k$ be a local field, 
$K/k$ be a separable field extension of degree $n$ 
and $L/k$ be the Galois closure of $K/k$. 
Assume that $G={\rm Gal}(L/k)=nTm$ is a transitive subgroup of $S_n$ 
and $H={\rm Gal}(L/K)$ with $[G:H]=n$. 
Let $T=R^{(1)}_{K/k}(\bG_m)$ be the norm one torus of $K/k$ of 
dimension $n-1$. 
Then 
\begin{align*}
T(k)/R\simeq H^1(G,[J_{G/H}]^{fl})\leq 
\begin{cases}
0 & (n=2,3,5,7,11,13)\\
\bZ/2\bZ & (n=4,6,10,14),\\
(\bZ/2\bZ)^{\oplus 3} & (n=8),\\
\bZ/3\bZ & (n=9),\\
\bZ/5\bZ & (n=15)
\end{cases}
\end{align*}
and $T(k)/R\neq 0$
if and only if $G$ is given as in Table $1$. 
\end{theorem}
%

\section{Application $2$: Tamagawa number $\tau(T)$}\label{S8}

By Theorem \ref{thmain3}, we obtain the 
Tamagawa number $\tau(T)$ of algebraic $k$-tori $T$ 
(see Ono \cite{Ono63}, \cite{Ono65} and Voskresenskii \cite[Chapter 5]{Vos98}). 

\begin{theorem}[{Ono \cite[Main theorem, page 68]{Ono63}, see also \cite[Theorem 2, page 146]{Vos98}}]
Let $k$ be a global field, $T$ be an algebraic $k$-torus 
and $\tau(T)$ be the Tamagawa number of $T$.  
Then 
\begin{align*}
\tau(T)=\frac{|H^1(k,\widehat{T})|}{|\Sha(T)|}.
\end{align*}
In particular, 
if $T$ is retract $k$-rational, then $\tau(T)=|H^1(k,\widehat{T})|$.
\end{theorem}


For the last assertion, see Theorem \ref{thEM73}. 
As a consequence of Theorem \ref{thVos69}, 
Theorem \ref{thmain1-4p} and Theorem \ref{thmain1-5p}
(Theorem \ref{thmain1-4} and Theorem \ref{thmain1-5}), we have: 
\begin{theorem}
Let $k$ be a global field and 
$T$ be an algebraic $k$-torus of dimension $4$ $($resp. $5$$)$. 
Among $710$ $($reps. $6079$$)$ cases of algebraic $k$-tori $T$, 
if $T$ is one of the $688$ $($resp. $5805$$)$ cases 
with $H^1(k,{\rm Pic}\,\overline{X})=0$, 
then $\tau(T)=|H^1(k,\widehat{T})|$. 
\end{theorem}
When $T=R^{(1)}_{K/k}(\bG_m)$ and $K/k$ is a finite Galois extension, 
i.e. $\widehat{T}=J_G$, 
it follows from Endo and Miyata \cite[Theorem 1.5]{EM75} that 
if all the Sylow subgroups of $G={\rm Gal}(K/k)$ are cyclic, 
then $|\Sha(T)|=1$ and hence 
$\tau(T)=|H^1(G,J_G)|=|H^2(G,\bZ)|=|H^1(G,\bQ/\bZ)|=|G^{ab}|$. 
For norm one tori $T=R^{(1)}_{K/k}(\bG_m)$ with 
$[K:k]=n\leq 15$ and $n\neq 12$, 
Kunyavskii's theorem (Theorem \ref{thKun2}), 
Drakokhrust and Platonov's theorem (Theorem \ref{thDP}) 
and Theorem \ref{thmain3} enable us to 
get the Tamagawa number $\tau(T)$: 

\begin{theorem}
Let $2\leq n\leq 15$ be an integer with $n\neq 12$. 
Let $k$ be a number field, 
$K/k$ be a field extension of degree $n$ 
and $L/k$ be the Galois closure of $K/k$. 
Assume that $G={\rm Gal}(L/k)=nTm$ 
is a transitive subgroup of $S_n$ and 
$H={\rm Gal}(L/K)$ with $[G:H]=n$. 
Let $T=R^{(1)}_{K/k}(\bG_m)$ be the norm one torus of $K/k$ 
of dimension $n-1$. 
Then $\tau(T)=|H^1(G,J_{G/H})|$ 
except for the cases in Table $1$. 
For the cases in Table $1$, 
we have $\tau(T)=|H^1(G,J_{G/H})|/|\Sha(T)|$ 
where $H^1(G,J_{G/H})$ is given as in Section \ref{S9} and 
$\Sha(T)$ is given as in Theorem \ref{thmain3}. 
\end{theorem}

We give GAP computations 
of $H^1(G,J_{G/H})$ 
for $G={\rm Gal}(L/k)=nTm$ $(n\leq 15)$ 
in Section \ref{S9} as the appendix of this paper. 

\section{Appendix: Computation of $H^1(G,J_{G/H})$ for $G={\rm Gal}(L/k)=nTm$ $(n\leq 15)$ }\label{S9}
~{}\vspace*{-4mm}\\
{\small 
\begin{verbatim}
gap> Read("FlabbyResolutionFromBase.gap");
gap> for n in [2..15] do for m in [1..NrTransitiveGroups(n)] do
> Print([[n,m],Filtered(H1(Norm1TorusJ(n,m)),x->x>1)],"\n");od;Print("\n");od;
[ [ 2, 1 ], [ 2 ] ]

[ [ 3, 1 ], [ 3 ] ]
[ [ 3, 2 ], [  ] ]

[ [ 4, 1 ], [ 4 ] ]
[ [ 4, 2 ], [ 2, 2 ] ]
[ [ 4, 3 ], [ 2 ] ]
[ [ 4, 4 ], [  ] ]
[ [ 4, 5 ], [  ] ]

[ [ 5, 1 ], [ 5 ] ]
[ [ 5, 2 ], [  ] ]
[ [ 5, 3 ], [  ] ]
[ [ 5, 4 ], [  ] ]
[ [ 5, 5 ], [  ] ]

[ [ 6, 1 ], [ 6 ] ]
[ [ 6, 2 ], [ 2 ] ]
[ [ 6, 3 ], [ 2 ] ]
[ [ 6, 4 ], [ 3 ] ]
[ [ 6, 5 ], [ 2 ] ]
[ [ 6, 6 ], [ 3 ] ]
[ [ 6, 7 ], [  ] ]
[ [ 6, 8 ], [  ] ]
[ [ 6, 9 ], [ 2 ] ]
[ [ 6, 10 ], [ 2 ] ]
[ [ 6, 11 ], [  ] ]
[ [ 6, 12 ], [  ] ]
[ [ 6, 13 ], [ 2 ] ]
[ [ 6, 14 ], [  ] ]
[ [ 6, 15 ], [  ] ]
[ [ 6, 16 ], [  ] ]

[ [ 7, 1 ], [ 7 ] ]
[ [ 7, 2 ], [  ] ]
[ [ 7, 3 ], [  ] ]
[ [ 7, 4 ], [  ] ]
[ [ 7, 5 ], [  ] ]
[ [ 7, 6 ], [  ] ]
[ [ 7, 7 ], [  ] ]

[ [ 8, 1 ], [ 8 ] ]
[ [ 8, 2 ], [ 2, 4 ] ]
[ [ 8, 3 ], [ 2, 2, 2 ] ]
[ [ 8, 4 ], [ 2, 2 ] ]
[ [ 8, 5 ], [ 2, 2 ] ]
[ [ 8, 6 ], [ 2 ] ]
[ [ 8, 7 ], [ 4 ] ]
[ [ 8, 8 ], [ 2 ] ]
[ [ 8, 9 ], [ 2, 2 ] ]
[ [ 8, 10 ], [ 4 ] ]
[ [ 8, 11 ], [ 2, 2 ] ]
[ [ 8, 12 ], [  ] ]
[ [ 8, 13 ], [ 2 ] ]
[ [ 8, 14 ], [ 2 ] ]
[ [ 8, 15 ], [ 2 ] ]
[ [ 8, 16 ], [ 4 ] ]
[ [ 8, 17 ], [ 2 ] ]
[ [ 8, 18 ], [ 2 ] ]
[ [ 8, 19 ], [ 2 ] ]
[ [ 8, 20 ], [ 4 ] ]
[ [ 8, 21 ], [ 2, 2 ] ]
[ [ 8, 22 ], [ 2, 2 ] ]
[ [ 8, 23 ], [  ] ]
[ [ 8, 24 ], [ 2 ] ]
[ [ 8, 25 ], [  ] ]
[ [ 8, 26 ], [ 2 ] ]
[ [ 8, 27 ], [ 4 ] ]
[ [ 8, 28 ], [ 2 ] ]
[ [ 8, 29 ], [ 2 ] ]
[ [ 8, 30 ], [ 2 ] ]
[ [ 8, 31 ], [ 2, 2 ] ]
[ [ 8, 32 ], [  ] ]
[ [ 8, 33 ], [ 2 ] ]
[ [ 8, 34 ], [ 2 ] ]
[ [ 8, 35 ], [ 2 ] ]
[ [ 8, 36 ], [  ] ]
[ [ 8, 37 ], [  ] ]
[ [ 8, 38 ], [  ] ]
[ [ 8, 39 ], [  ] ]
[ [ 8, 40 ], [  ] ]
[ [ 8, 41 ], [ 2 ] ]
[ [ 8, 42 ], [ 2 ] ]
[ [ 8, 43 ], [  ] ]
[ [ 8, 44 ], [  ] ]
[ [ 8, 45 ], [ 2 ] ]
[ [ 8, 46 ], [ 2 ] ]
[ [ 8, 47 ], [ 2 ] ]
[ [ 8, 48 ], [  ] ]
[ [ 8, 49 ], [  ] ]
[ [ 8, 50 ], [  ] ]

[ [ 9, 1 ], [ 9 ] ]
[ [ 9, 2 ], [ 3, 3 ] ]
[ [ 9, 3 ], [  ] ]
[ [ 9, 4 ], [ 3 ] ]
[ [ 9, 5 ], [  ] ]
[ [ 9, 6 ], [ 3 ] ]
[ [ 9, 7 ], [ 3 ] ]
[ [ 9, 8 ], [  ] ]
[ [ 9, 9 ], [  ] ]
[ [ 9, 10 ], [  ] ]
[ [ 9, 11 ], [  ] ]
[ [ 9, 12 ], [  ] ]
[ [ 9, 13 ], [ 3 ] ]
[ [ 9, 14 ], [  ] ]
[ [ 9, 15 ], [  ] ]
[ [ 9, 16 ], [  ] ]
[ [ 9, 17 ], [ 3 ] ]
[ [ 9, 18 ], [  ] ]
[ [ 9, 19 ], [  ] ]
[ [ 9, 20 ], [  ] ]
[ [ 9, 21 ], [  ] ]
[ [ 9, 22 ], [ 3 ] ]
[ [ 9, 23 ], [  ] ]
[ [ 9, 24 ], [  ] ]
[ [ 9, 25 ], [ 3 ] ]
[ [ 9, 26 ], [  ] ]
[ [ 9, 27 ], [  ] ]
[ [ 9, 28 ], [ 3 ] ]
[ [ 9, 29 ], [  ] ]
[ [ 9, 30 ], [  ] ]
[ [ 9, 31 ], [  ] ]
[ [ 9, 32 ], [  ] ]
[ [ 9, 33 ], [  ] ]
[ [ 9, 34 ], [  ] ]

[ [ 10, 1 ], [ 10 ] ]
[ [ 10, 2 ], [ 2 ] ]
[ [ 10, 3 ], [ 2 ] ]
[ [ 10, 4 ], [ 2 ] ]
[ [ 10, 5 ], [ 2 ] ]
[ [ 10, 6 ], [ 2 ] ]
[ [ 10, 7 ], [  ] ]
[ [ 10, 8 ], [ 5 ] ]
[ [ 10, 9 ], [ 2 ] ]
[ [ 10, 10 ], [ 2 ] ]
[ [ 10, 11 ], [ 2 ] ]
[ [ 10, 12 ], [ 2 ] ]
[ [ 10, 13 ], [  ] ]
[ [ 10, 14 ], [ 5 ] ]
[ [ 10, 15 ], [  ] ]
[ [ 10, 16 ], [  ] ]
[ [ 10, 17 ], [ 2 ] ]
[ [ 10, 18 ], [ 2 ] ]
[ [ 10, 19 ], [ 2 ] ]
[ [ 10, 20 ], [ 2 ] ]
[ [ 10, 21 ], [ 2 ] ]
[ [ 10, 22 ], [ 2 ] ]
[ [ 10, 23 ], [  ] ]
[ [ 10, 24 ], [  ] ]
[ [ 10, 25 ], [  ] ]
[ [ 10, 26 ], [  ] ]
[ [ 10, 27 ], [ 2 ] ]
[ [ 10, 28 ], [ 2 ] ]
[ [ 10, 29 ], [  ] ]
[ [ 10, 30 ], [  ] ]
[ [ 10, 31 ], [  ] ]
[ [ 10, 32 ], [  ] ]
[ [ 10, 33 ], [ 2 ] ]
[ [ 10, 34 ], [  ] ]
[ [ 10, 35 ], [  ] ]
[ [ 10, 36 ], [  ] ]
[ [ 10, 37 ], [  ] ]
[ [ 10, 38 ], [  ] ]
[ [ 10, 39 ], [  ] ]
[ [ 10, 40 ], [ 2 ] ]
[ [ 10, 41 ], [ 2 ] ]
[ [ 10, 42 ], [ 2 ] ]
[ [ 10, 43 ], [ 2 ] ]
[ [ 10, 44 ], [  ] ]
[ [ 10, 45 ], [  ] ]

[ [ 11, 1 ], [ 11 ] ]
[ [ 11, 2 ], [  ] ]
[ [ 11, 3 ], [  ] ]
[ [ 11, 4 ], [  ] ]
[ [ 11, 5 ], [  ] ]
[ [ 11, 6 ], [  ] ]
[ [ 11, 7 ], [  ] ]
[ [ 11, 8 ], [  ] ]

[ [ 12, 1 ], [ 12 ] ]
[ [ 12, 2 ], [ 2, 6 ] ]
[ [ 12, 3 ], [ 2, 2 ] ]
[ [ 12, 4 ], [ 3 ] ]
[ [ 12, 5 ], [ 4 ] ]
[ [ 12, 6 ], [ 3 ] ]
[ [ 12, 7 ], [ 6 ] ]
[ [ 12, 8 ], [  ] ]
[ [ 12, 9 ], [ 2 ] ]
[ [ 12, 10 ], [ 2, 2 ] ]
[ [ 12, 11 ], [ 4 ] ]
[ [ 12, 12 ], [ 2 ] ]
[ [ 12, 13 ], [ 2 ] ]
[ [ 12, 14 ], [ 6 ] ]
[ [ 12, 15 ], [ 2 ] ]
[ [ 12, 16 ], [ 2, 2 ] ]
[ [ 12, 17 ], [ 4 ] ]
[ [ 12, 18 ], [ 2, 2 ] ]
[ [ 12, 19 ], [ 4 ] ]
[ [ 12, 20 ], [ 3 ] ]
[ [ 12, 21 ], [ 2 ] ]
[ [ 12, 22 ], [  ] ]
[ [ 12, 23 ], [ 2 ] ]
[ [ 12, 24 ], [ 2 ] ]
[ [ 12, 25 ], [ 6 ] ]
[ [ 12, 26 ], [ 3 ] ]
[ [ 12, 27 ], [  ] ]
[ [ 12, 28 ], [ 2 ] ]
[ [ 12, 29 ], [ 6 ] ]
[ [ 12, 30 ], [ 2 ] ]
[ [ 12, 31 ], [ 3 ] ]
[ [ 12, 32 ], [ 3 ] ]
[ [ 12, 33 ], [  ] ]
[ [ 12, 34 ], [ 2, 2 ] ]
[ [ 12, 35 ], [ 2 ] ]
[ [ 12, 36 ], [ 2 ] ]
[ [ 12, 37 ], [ 2, 2 ] ]
[ [ 12, 38 ], [ 2 ] ]
[ [ 12, 39 ], [ 4 ] ]
[ [ 12, 40 ], [ 2, 2 ] ]
[ [ 12, 41 ], [ 4 ] ]
[ [ 12, 42 ], [ 2 ] ]
[ [ 12, 43 ], [  ] ]
[ [ 12, 44 ], [  ] ]
[ [ 12, 45 ], [ 3 ] ]
[ [ 12, 46 ], [ 4 ] ]
[ [ 12, 47 ], [ 2, 2 ] ]
[ [ 12, 48 ], [ 2 ] ]
[ [ 12, 49 ], [  ] ]
[ [ 12, 50 ], [ 2 ] ]
[ [ 12, 51 ], [ 6 ] ]
[ [ 12, 52 ], [ 2 ] ]
[ [ 12, 53 ], [ 2 ] ]
[ [ 12, 54 ], [ 2 ] ]
[ [ 12, 55 ], [ 3 ] ]
[ [ 12, 56 ], [ 3 ] ]
[ [ 12, 57 ], [ 3 ] ]
[ [ 12, 58 ], [ 6 ] ]
[ [ 12, 59 ], [ 3 ] ]
[ [ 12, 60 ], [ 3 ] ]
[ [ 12, 61 ], [ 3 ] ]
[ [ 12, 62 ], [  ] ]
[ [ 12, 63 ], [  ] ]
[ [ 12, 64 ], [  ] ]
[ [ 12, 65 ], [  ] ]
[ [ 12, 66 ], [  ] ]
[ [ 12, 67 ], [  ] ]
[ [ 12, 68 ], [  ] ]
[ [ 12, 69 ], [ 2 ] ]
[ [ 12, 70 ], [ 2, 2 ] ]
[ [ 12, 71 ], [ 2, 2 ] ]
[ [ 12, 72 ], [ 4 ] ]
[ [ 12, 73 ], [ 4 ] ]
[ [ 12, 74 ], [ 2 ] ]
[ [ 12, 75 ], [ 2 ] ]
[ [ 12, 76 ], [  ] ]
[ [ 12, 77 ], [ 2, 2 ] ]
[ [ 12, 78 ], [ 2 ] ]
[ [ 12, 79 ], [ 4 ] ]
[ [ 12, 80 ], [ 2 ] ]
[ [ 12, 81 ], [ 2 ] ]
[ [ 12, 82 ], [ 2 ] ]
[ [ 12, 83 ], [  ] ]
[ [ 12, 84 ], [ 2 ] ]
[ [ 12, 85 ], [ 3 ] ]
[ [ 12, 86 ], [ 2 ] ]
[ [ 12, 87 ], [ 6 ] ]
[ [ 12, 88 ], [ 3 ] ]
[ [ 12, 89 ], [ 3 ] ]
[ [ 12, 90 ], [ 3 ] ]
[ [ 12, 91 ], [ 3 ] ]
[ [ 12, 92 ], [ 3 ] ]
[ [ 12, 93 ], [ 3 ] ]
[ [ 12, 94 ], [ 3 ] ]
[ [ 12, 95 ], [  ] ]
[ [ 12, 96 ], [  ] ]
[ [ 12, 97 ], [  ] ]
[ [ 12, 98 ], [  ] ]
[ [ 12, 99 ], [ 3 ] ]
[ [ 12, 100 ], [  ] ]
[ [ 12, 101 ], [  ] ]
[ [ 12, 102 ], [  ] ]
[ [ 12, 103 ], [  ] ]
[ [ 12, 104 ], [ 3 ] ]
[ [ 12, 105 ], [ 6 ] ]
[ [ 12, 106 ], [ 2 ] ]
[ [ 12, 107 ], [ 2 ] ]
[ [ 12, 108 ], [ 2 ] ]
[ [ 12, 109 ], [ 2 ] ]
[ [ 12, 110 ], [  ] ]
[ [ 12, 111 ], [  ] ]
[ [ 12, 112 ], [  ] ]
[ [ 12, 113 ], [  ] ]
[ [ 12, 114 ], [  ] ]
[ [ 12, 115 ], [  ] ]
[ [ 12, 116 ], [ 2 ] ]
[ [ 12, 117 ], [ 2, 2 ] ]
[ [ 12, 118 ], [ 2 ] ]
[ [ 12, 119 ], [ 4 ] ]
[ [ 12, 120 ], [ 2 ] ]
[ [ 12, 121 ], [ 2 ] ]
[ [ 12, 122 ], [  ] ]
[ [ 12, 123 ], [ 2 ] ]
[ [ 12, 124 ], [  ] ]
[ [ 12, 125 ], [ 2 ] ]
[ [ 12, 126 ], [ 2 ] ]
[ [ 12, 127 ], [  ] ]
[ [ 12, 128 ], [  ] ]
[ [ 12, 129 ], [ 3 ] ]
[ [ 12, 130 ], [ 2, 2 ] ]
[ [ 12, 131 ], [ 4 ] ]
[ [ 12, 132 ], [  ] ]
[ [ 12, 133 ], [  ] ]
[ [ 12, 134 ], [ 6 ] ]
[ [ 12, 135 ], [ 2 ] ]
[ [ 12, 136 ], [ 2 ] ]
[ [ 12, 137 ], [  ] ]
[ [ 12, 138 ], [  ] ]
[ [ 12, 139 ], [  ] ]
[ [ 12, 140 ], [  ] ]
[ [ 12, 141 ], [ 3 ] ]
[ [ 12, 142 ], [ 3 ] ]
[ [ 12, 143 ], [ 3 ] ]
[ [ 12, 144 ], [ 3 ] ]
[ [ 12, 145 ], [ 2 ] ]
[ [ 12, 146 ], [  ] ]
[ [ 12, 147 ], [  ] ]
[ [ 12, 148 ], [  ] ]
[ [ 12, 149 ], [  ] ]
[ [ 12, 150 ], [  ] ]
[ [ 12, 151 ], [  ] ]
[ [ 12, 152 ], [  ] ]
[ [ 12, 153 ], [  ] ]
[ [ 12, 154 ], [ 2 ] ]
[ [ 12, 155 ], [ 2 ] ]
[ [ 12, 156 ], [ 2 ] ]
[ [ 12, 157 ], [  ] ]
[ [ 12, 158 ], [ 2 ] ]
[ [ 12, 159 ], [ 2 ] ]
[ [ 12, 160 ], [ 2 ] ]
[ [ 12, 161 ], [ 2 ] ]
[ [ 12, 162 ], [ 2 ] ]
[ [ 12, 163 ], [ 2 ] ]
[ [ 12, 164 ], [ 3 ] ]
[ [ 12, 165 ], [  ] ]
[ [ 12, 166 ], [ 3 ] ]
[ [ 12, 167 ], [ 2 ] ]
[ [ 12, 168 ], [ 2, 2 ] ]
[ [ 12, 169 ], [ 2 ] ]
[ [ 12, 170 ], [ 4 ] ]
[ [ 12, 171 ], [ 2, 2 ] ]
[ [ 12, 172 ], [ 2, 2 ] ]
[ [ 12, 173 ], [ 4 ] ]
[ [ 12, 174 ], [ 2, 2 ] ]
[ [ 12, 175 ], [  ] ]
[ [ 12, 176 ], [  ] ]
[ [ 12, 177 ], [  ] ]
[ [ 12, 178 ], [  ] ]
[ [ 12, 179 ], [  ] ]
[ [ 12, 180 ], [ 2 ] ]
[ [ 12, 181 ], [ 2 ] ]
[ [ 12, 182 ], [ 2 ] ]
[ [ 12, 183 ], [ 2 ] ]
[ [ 12, 184 ], [  ] ]
[ [ 12, 185 ], [  ] ]
[ [ 12, 186 ], [  ] ]
[ [ 12, 187 ], [ 3 ] ]
[ [ 12, 188 ], [ 3 ] ]
[ [ 12, 189 ], [ 3 ] ]
[ [ 12, 190 ], [  ] ]
[ [ 12, 191 ], [  ] ]
[ [ 12, 192 ], [  ] ]
[ [ 12, 193 ], [ 2 ] ]
[ [ 12, 194 ], [  ] ]
[ [ 12, 195 ], [ 2 ] ]
[ [ 12, 196 ], [ 2 ] ]
[ [ 12, 197 ], [ 2 ] ]
[ [ 12, 198 ], [ 2 ] ]
[ [ 12, 199 ], [ 2 ] ]
[ [ 12, 200 ], [ 2 ] ]
[ [ 12, 201 ], [ 2 ] ]
[ [ 12, 202 ], [ 2 ] ]
[ [ 12, 203 ], [ 2 ] ]
[ [ 12, 204 ], [  ] ]
[ [ 12, 205 ], [ 3 ] ]
[ [ 12, 206 ], [  ] ]
[ [ 12, 207 ], [  ] ]
[ [ 12, 208 ], [ 2 ] ]
[ [ 12, 209 ], [ 2 ] ]
[ [ 12, 210 ], [ 2, 2 ] ]
[ [ 12, 211 ], [ 4 ] ]
[ [ 12, 212 ], [ 2 ] ]
[ [ 12, 213 ], [  ] ]
[ [ 12, 214 ], [ 2, 2 ] ]
[ [ 12, 215 ], [ 4 ] ]
[ [ 12, 216 ], [ 2 ] ]
[ [ 12, 217 ], [ 2 ] ]
[ [ 12, 218 ], [  ] ]
[ [ 12, 219 ], [ 2 ] ]
[ [ 12, 220 ], [ 2 ] ]
[ [ 12, 221 ], [  ] ]
[ [ 12, 222 ], [ 3 ] ]
[ [ 12, 223 ], [  ] ]
[ [ 12, 224 ], [  ] ]
[ [ 12, 225 ], [  ] ]
[ [ 12, 226 ], [  ] ]
[ [ 12, 227 ], [  ] ]
[ [ 12, 228 ], [ 3 ] ]
[ [ 12, 229 ], [ 3 ] ]
[ [ 12, 230 ], [  ] ]
[ [ 12, 231 ], [  ] ]
[ [ 12, 232 ], [  ] ]
[ [ 12, 233 ], [  ] ]
[ [ 12, 234 ], [  ] ]
[ [ 12, 235 ], [ 2 ] ]
[ [ 12, 236 ], [ 2 ] ]
[ [ 12, 237 ], [ 2 ] ]
[ [ 12, 238 ], [ 2 ] ]
[ [ 12, 239 ], [  ] ]
[ [ 12, 240 ], [ 2 ] ]
[ [ 12, 241 ], [ 2 ] ]
[ [ 12, 242 ], [ 2, 2 ] ]
[ [ 12, 243 ], [ 2 ] ]
[ [ 12, 244 ], [ 4 ] ]
[ [ 12, 245 ], [ 4 ] ]
[ [ 12, 246 ], [ 2, 2 ] ]
[ [ 12, 247 ], [ 2 ] ]
[ [ 12, 248 ], [ 2 ] ]
[ [ 12, 249 ], [ 2 ] ]
[ [ 12, 250 ], [  ] ]
[ [ 12, 251 ], [  ] ]
[ [ 12, 252 ], [  ] ]
[ [ 12, 253 ], [ 3 ] ]
[ [ 12, 254 ], [  ] ]
[ [ 12, 255 ], [  ] ]
[ [ 12, 256 ], [  ] ]
[ [ 12, 257 ], [  ] ]
[ [ 12, 258 ], [  ] ]
[ [ 12, 259 ], [  ] ]
[ [ 12, 260 ], [ 2 ] ]
[ [ 12, 261 ], [ 2, 2 ] ]
[ [ 12, 262 ], [ 2 ] ]
[ [ 12, 263 ], [ 2 ] ]
[ [ 12, 264 ], [ 4 ] ]
[ [ 12, 265 ], [ 3 ] ]
[ [ 12, 266 ], [ 2 ] ]
[ [ 12, 267 ], [ 2 ] ]
[ [ 12, 268 ], [  ] ]
[ [ 12, 269 ], [ 2 ] ]
[ [ 12, 270 ], [  ] ]
[ [ 12, 271 ], [  ] ]
[ [ 12, 272 ], [  ] ]
[ [ 12, 273 ], [ 3 ] ]
[ [ 12, 274 ], [ 2 ] ]
[ [ 12, 275 ], [  ] ]
[ [ 12, 276 ], [  ] ]
[ [ 12, 277 ], [  ] ]
[ [ 12, 278 ], [ 2 ] ]
[ [ 12, 279 ], [ 2 ] ]
[ [ 12, 280 ], [  ] ]
[ [ 12, 281 ], [  ] ]
[ [ 12, 282 ], [  ] ]
[ [ 12, 283 ], [  ] ]
[ [ 12, 284 ], [ 3 ] ]
[ [ 12, 285 ], [  ] ]
[ [ 12, 286 ], [  ] ]
[ [ 12, 287 ], [  ] ]
[ [ 12, 288 ], [ 2 ] ]
[ [ 12, 289 ], [  ] ]
[ [ 12, 290 ], [  ] ]
[ [ 12, 291 ], [  ] ]
[ [ 12, 292 ], [ 3 ] ]
[ [ 12, 293 ], [  ] ]
[ [ 12, 294 ], [  ] ]
[ [ 12, 295 ], [  ] ]
[ [ 12, 296 ], [ 2 ] ]
[ [ 12, 297 ], [ 2 ] ]
[ [ 12, 298 ], [ 2 ] ]
[ [ 12, 299 ], [ 2 ] ]
[ [ 12, 300 ], [  ] ]
[ [ 12, 301 ], [  ] ]

[ [ 13, 1 ], [ 13 ] ]
[ [ 13, 2 ], [  ] ]
[ [ 13, 3 ], [  ] ]
[ [ 13, 4 ], [  ] ]
[ [ 13, 5 ], [  ] ]
[ [ 13, 6 ], [  ] ]
[ [ 13, 7 ], [  ] ]
[ [ 13, 8 ], [  ] ]
[ [ 13, 9 ], [  ] ]

[ [ 14, 1 ], [ 14 ] ]
[ [ 14, 2 ], [ 2 ] ]
[ [ 14, 3 ], [ 2 ] ]
[ [ 14, 4 ], [ 2 ] ]
[ [ 14, 5 ], [ 2 ] ]
[ [ 14, 6 ], [ 7 ] ]
[ [ 14, 7 ], [ 2 ] ]
[ [ 14, 8 ], [ 2 ] ]
[ [ 14, 9 ], [ 7 ] ]
[ [ 14, 10 ], [  ] ]
[ [ 14, 11 ], [  ] ]
[ [ 14, 12 ], [ 2 ] ]
[ [ 14, 13 ], [ 2 ] ]
[ [ 14, 14 ], [ 2 ] ]
[ [ 14, 15 ], [ 2 ] ]
[ [ 14, 16 ], [ 2 ] ]
[ [ 14, 17 ], [  ] ]
[ [ 14, 18 ], [  ] ]
[ [ 14, 19 ], [ 2 ] ]
[ [ 14, 20 ], [ 2 ] ]
[ [ 14, 21 ], [ 7 ] ]
[ [ 14, 22 ], [ 2 ] ]
[ [ 14, 23 ], [ 2 ] ]
[ [ 14, 24 ], [ 2 ] ]
[ [ 14, 25 ], [ 2 ] ]
[ [ 14, 26 ], [ 2 ] ]
[ [ 14, 27 ], [  ] ]
[ [ 14, 28 ], [  ] ]
[ [ 14, 29 ], [ 7 ] ]
[ [ 14, 30 ], [  ] ]
[ [ 14, 31 ], [ 2 ] ]
[ [ 14, 32 ], [ 2 ] ]
[ [ 14, 33 ], [  ] ]
[ [ 14, 34 ], [  ] ]
[ [ 14, 35 ], [  ] ]
[ [ 14, 36 ], [ 2 ] ]
[ [ 14, 37 ], [ 2 ] ]
[ [ 14, 38 ], [  ] ]
[ [ 14, 39 ], [  ] ]
[ [ 14, 40 ], [  ] ]
[ [ 14, 41 ], [  ] ]
[ [ 14, 42 ], [  ] ]
[ [ 14, 43 ], [  ] ]
[ [ 14, 44 ], [  ] ]
[ [ 14, 45 ], [ 2 ] ]
[ [ 14, 46 ], [ 2 ] ]
[ [ 14, 47 ], [ 2 ] ]
[ [ 14, 48 ], [  ] ]
[ [ 14, 49 ], [ 2 ] ]
[ [ 14, 50 ], [  ] ]
[ [ 14, 51 ], [  ] ]
[ [ 14, 52 ], [ 2 ] ]
[ [ 14, 53 ], [  ] ]
[ [ 14, 54 ], [  ] ]
[ [ 14, 55 ], [  ] ]
[ [ 14, 56 ], [  ] ]
[ [ 14, 57 ], [  ] ]
[ [ 14, 58 ], [ 2 ] ]
[ [ 14, 59 ], [ 2 ] ]
[ [ 14, 60 ], [ 2 ] ]
[ [ 14, 61 ], [ 2 ] ]
[ [ 14, 62 ], [  ] ]
[ [ 14, 63 ], [  ] ]

[ [ 15, 1 ], [ 15 ] ]
[ [ 15, 2 ], [  ] ]
[ [ 15, 3 ], [ 3 ] ]
[ [ 15, 4 ], [ 5 ] ]
[ [ 15, 5 ], [  ] ]
[ [ 15, 6 ], [  ] ]
[ [ 15, 7 ], [  ] ]
[ [ 15, 8 ], [ 3 ] ]
[ [ 15, 9 ], [ 3 ] ]
[ [ 15, 10 ], [  ] ]
[ [ 15, 11 ], [  ] ]
[ [ 15, 12 ], [ 3 ] ]
[ [ 15, 13 ], [  ] ]
[ [ 15, 14 ], [  ] ]
[ [ 15, 15 ], [  ] ]
[ [ 15, 16 ], [ 3 ] ]
[ [ 15, 17 ], [  ] ]
[ [ 15, 18 ], [  ] ]
[ [ 15, 19 ], [ 3 ] ]
[ [ 15, 20 ], [  ] ]
[ [ 15, 21 ], [  ] ]
[ [ 15, 22 ], [  ] ]
[ [ 15, 23 ], [  ] ]
[ [ 15, 24 ], [ 3 ] ]
[ [ 15, 25 ], [ 3 ] ]
[ [ 15, 26 ], [ 5 ] ]
[ [ 15, 27 ], [  ] ]
[ [ 15, 28 ], [  ] ]
[ [ 15, 29 ], [  ] ]
[ [ 15, 30 ], [ 3 ] ]
[ [ 15, 31 ], [  ] ]
[ [ 15, 32 ], [  ] ]
[ [ 15, 33 ], [ 5 ] ]
[ [ 15, 34 ], [  ] ]
[ [ 15, 35 ], [  ] ]
[ [ 15, 36 ], [ 5 ] ]
[ [ 15, 37 ], [  ] ]
[ [ 15, 38 ], [ 3 ] ]
[ [ 15, 39 ], [ 3 ] ]
[ [ 15, 40 ], [  ] ]
[ [ 15, 41 ], [  ] ]
[ [ 15, 42 ], [  ] ]
[ [ 15, 43 ], [  ] ]
[ [ 15, 44 ], [ 5 ] ]
[ [ 15, 45 ], [  ] ]
[ [ 15, 46 ], [  ] ]
[ [ 15, 47 ], [  ] ]
[ [ 15, 48 ], [  ] ]
[ [ 15, 49 ], [  ] ]
[ [ 15, 50 ], [ 3 ] ]
[ [ 15, 51 ], [  ] ]
[ [ 15, 52 ], [  ] ]
[ [ 15, 53 ], [  ] ]
[ [ 15, 54 ], [  ] ]
[ [ 15, 55 ], [  ] ]
[ [ 15, 56 ], [  ] ]
[ [ 15, 57 ], [ 3 ] ]
[ [ 15, 58 ], [  ] ]
[ [ 15, 59 ], [ 3 ] ]
[ [ 15, 60 ], [  ] ]
[ [ 15, 61 ], [  ] ]
[ [ 15, 62 ], [  ] ]
[ [ 15, 63 ], [  ] ]
[ [ 15, 64 ], [  ] ]
[ [ 15, 65 ], [  ] ]
[ [ 15, 66 ], [  ] ]
[ [ 15, 67 ], [ 3 ] ]
[ [ 15, 68 ], [  ] ]
[ [ 15, 69 ], [  ] ]
[ [ 15, 70 ], [  ] ]
[ [ 15, 71 ], [ 5 ] ]
[ [ 15, 72 ], [  ] ]
[ [ 15, 73 ], [  ] ]
[ [ 15, 74 ], [  ] ]
[ [ 15, 75 ], [ 3 ] ]
[ [ 15, 76 ], [  ] ]
[ [ 15, 77 ], [  ] ]
[ [ 15, 78 ], [  ] ]
[ [ 15, 79 ], [  ] ]
[ [ 15, 80 ], [  ] ]
[ [ 15, 81 ], [ 5 ] ]
[ [ 15, 82 ], [  ] ]
[ [ 15, 83 ], [  ] ]
[ [ 15, 84 ], [  ] ]
[ [ 15, 85 ], [  ] ]
[ [ 15, 86 ], [  ] ]
[ [ 15, 87 ], [  ] ]
[ [ 15, 88 ], [  ] ]
[ [ 15, 89 ], [  ] ]
[ [ 15, 90 ], [  ] ]
[ [ 15, 91 ], [  ] ]
[ [ 15, 92 ], [ 3 ] ]
[ [ 15, 93 ], [  ] ]
[ [ 15, 94 ], [  ] ]
[ [ 15, 95 ], [ 3 ] ]
[ [ 15, 96 ], [  ] ]
[ [ 15, 97 ], [  ] ]
[ [ 15, 98 ], [ 3 ] ]
[ [ 15, 99 ], [  ] ]
[ [ 15, 100 ], [  ] ]
[ [ 15, 101 ], [ 3 ] ]
[ [ 15, 102 ], [  ] ]
[ [ 15, 103 ], [  ] ]
[ [ 15, 104 ], [  ] ]
\end{verbatim}
}

\section{GAP algorithms}\label{S10}
We give GAP algorithms for computing the total obstruction 
${\rm Obs}(K/k)$ and 
the first obstruction ${\rm Obs}_1(L/K/k)$ 
as in Section \ref{S6}. 
The functions which are provided in this section are available from\\
{\tt https://www.math.kyoto-u.ac.jp/\~{}yamasaki/Algorithm/Norm1ToriHNP/}.\\
~{}\vspace*{-4mm}\\
{\small 
\begin{verbatim}LoadPackage("HAP");

Norm1TorusJ :=function(d,n)
    local I,M1,M2,M,f,Sn,T;
    I:=IdentityMat(d-1);
    Sn:=SymmetricGroup(d);
    T:=TransitiveGroup(d,n);
    M1:=Concatenation(List([2..d-1],x->I[x]),[-List([1..d-1],One)]);
    if d=2 then
        M:=[M1];
    else
        M2:=Concatenation([I[2],I[1]],List([3..d-1],x->I[x]));
        M:=[M1,M2];
    fi;
    f:=GroupHomomorphismByImages(Sn,Group(M),GeneratorsOfGroup(Sn),M);
    return Image(f,T);
end;

AbelianInvariantsSNF := function(G)
  local n,m,s,l;
  if Order(G)=1 then
    return [];
  fi;
  n:=AbelianInvariants(G);
  m:=DiagonalMat(n);
  s:=SmithNormalFormIntegerMat(m);
  return Filtered(DiagonalOfMat(s),x -> x>1);
end;

AbelianizationGen:= function(G)
    local Gab,pi,inv,A,iso,gen,genrep;
    Reset(GlobalMersenneTwister);
    Reset(GlobalRandomSource);
    pi:=NaturalHomomorphismByNormalSubgroup(G,DerivedSubgroup(G));
    Gab:=Image(pi);
    inv:=AbelianInvariantsSNF(Gab);
    A:=AbelianGroup(inv);
    iso:=IsomorphismGroups(A,Gab);
    gen:=List(GeneratorsOfGroup(A),x->Image(iso,x));
    genrep:=List(gen,x->PreImagesRepresentative(pi,x));
    return rec(Gab:=Gab, gen:=gen, genrep:=genrep, inv:=inv, pi:=pi);
end;

FindGenFiniteAbelian:= function(g)
    local e,a,ga,iso;
    e:=AbelianInvariants(g);
    if Length(e)>1 then
        e:=SmithNormalFormIntegerMat(DiagonalMat(e));
        e:=List([1..Length(e)],x->e[x][x]);
        e:=Filtered(e,x->x>1);
    fi;
    a:=AbelianGroup(e);
    ga:=GeneratorsOfGroup(a);
    iso:=IsomorphismGroups(a,g);
    return List(ga,x->Image(iso,x));
end;

EltFiniteAbelian:= function(arg)
    local g,c,gg,F,gF,hom,cF,e;
    g:=arg[1];
    c:=arg[2];
    if Length(arg)=3 then
        gg:=arg[3];
    else
        gg:=GeneratorsOfGroup(g);
    fi;
    F:=FreeGroup(Length(gg));
    gF:=GeneratorsOfGroup(F);
    hom:=GroupHomomorphismByImages(F,g,gF,gg);
    cF:=PreImagesRepresentative(hom,c);
    e:=List(gF,x->ExponentSumWord(cF,x));
    return e;
end;

FirstObstructionN:= function(arg)
    local G,H,Gab,Hab,K,Kinv,mat,v,Habbase,ker1;
    G:=arg[1];
    if Length(arg)=1 then
        H:=Stabilizer(G,1);
    else
        H:=arg[2];
    fi;
    Gab:=AbelianizationGen(G);
    Hab:=AbelianizationGen(H);
    Hab.Hab:=Hab.Gab;
    Unbind(Hab.Gab);
    if DerivedSubgroup(H)=H then
        return rec(ker:=[[],[[],[]]], Hab:=Hab, Gab:=Gab, psi:=[]);
    fi;
    if DerivedSubgroup(G)=G then
        return rec(ker:=[Hab.inv,[Hab.inv,IdentityMat(Length(Hab.inv))]],
            Hab:=Hab, Gab:=Gab, psi:=List(Hab.inv,x->[]));
    fi;
    K:=Image(Hab.pi,Intersection(H,DerivedSubgroup(G)));
    Kinv:=AbelianInvariantsSNF(K);
    mat:=[];
    for v in Hab.genrep do
        Add(mat,EltFiniteAbelian(Gab.Gab,Image(Gab.pi,v),Gab.gen));
    od;
    Habbase:=DiagonalMat(Hab.inv);
    ker1:=List(GeneratorsOfGroup(K),x->EltFiniteAbelian(Hab.Hab,x,Hab.gen));
    ker1:=LatticeBasis(Concatenation(Habbase,ker1));
    ker1:=LatticeBasis(Difference(ker1,Habbase));
    return rec(ker:=[Kinv,[Hab.inv,ker1]],
        Hab:=Hab, Gab:=Gab, psi:=mat);;
end;

FirstObstructionDnr:= function(arg)
    local G,H,Gab,Hab,HG,HGrep,Dnrgen,h,x,Dnr,Dnrinv,Habbase,Dnrmat;
    G:=arg[1];
    if Length(arg)=1 then
        H:=Stabilizer(G,1);
    else
        H:=arg[2];
    fi;
    Gab:=AbelianizationGen(G);
    Hab:=AbelianizationGen(H);
    Hab.Hab:=Hab.Gab;
    Unbind(Hab.Gab);
    if DerivedSubgroup(H)=H then
        return rec(Dnr:=[[],[[],[]]], Hab:=Hab, Gab:=Gab);
    fi;
    Reset(GlobalMersenneTwister);
    Reset(GlobalRandomSource);
    HG:=RightCosets(G,H);
    HGrep:=List(HG,Representative);
    Dnrgen:=[];
    for x in HGrep do
        for h in GeneratorsOfGroup(Intersection(H,H^x)) do
            Add(Dnrgen,Image(Hab.pi,Comm(h,x^-1)));
        od;
    od;
    Dnr:=Group(Dnrgen,Identity(Hab.Hab));
    Dnrinv:=AbelianInvariantsSNF(Dnr);
    Habbase:=DiagonalMat(Hab.inv);
    Dnrmat:=List(Dnrgen,x->EltFiniteAbelian(Hab.Hab,x,Hab.gen));
    Dnrmat:=LatticeBasis(Concatenation(Habbase,Dnrmat));
    Dnrmat:=LatticeBasis(Difference(Dnrmat,Habbase));
    return rec(Dnr:=[Dnrinv,[Hab.inv,Dnrmat]],
        Hab:=Hab, Gab:=Gab);
end;

FirstObstructionDr:= function(arg)
    local G,Gv,H,Gab,Hab,HGGv,HGGvrep,Hwi,Hwiab,Gvab,psi2i,i,psi2iimage,Hw,
          psi2,ker,phi1i,phi1iimage,phi1,Dr,Drinv,Habbase,Drmat;
    G:=arg[1];
    Gv:=arg[2];
    if Length(arg)=2 then
        H:=Stabilizer(G,1);
    else
        H:=arg[3];
    fi;
    Gab:=AbelianizationGen(G);
    Hab:=AbelianizationGen(H);
    Hab.Hab:=Hab.Gab;
    Unbind(Hab.Gab);
    if DerivedSubgroup(H)=H then
        return rec(Dr:=[[],[[],[]]], Hab:=Hab, Gab:=Gab);
    fi;
    HGGv:=DoubleCosets(G,H,Gv);
    HGGvrep:=List(HGGv,Representative);
    Hwi:=List(HGGvrep,x->Intersection(Gv^(x^(-1)),H));
    Hwiab:=List(Hwi,AbelianizationGen);
    Gvab:=AbelianizationGen(Gv);
    psi2i:=[];
    for i in [1..Length(HGGv)] do
        psi2iimage:=List(Hwiab[i].genrep,x->x^HGGvrep[i]);
        psi2iimage:=List(psi2iimage,x->Image(Gvab.pi,x));
        Add(psi2i,GroupHomomorphismByImages(Hwiab[i].Gab,Gvab.Gab,Hwiab[i].gen,
            psi2iimage));
    od;
    Hw:=DirectProduct(List(Hwiab,x->x.Gab));
    psi2:=GroupHomomorphismByFunction(Hw,Gvab.Gab,x->
     Product([1..Length(HGGv)],i->Image(psi2i[i],Image(Projection(Hw,i),x))));
    ker:=Kernel(psi2);
    phi1i:=[];
    for i in [1..Length(HGGv)] do
        phi1iimage:=List(Hwiab[i].genrep,x->Image(Hab.pi,x));
        Add(phi1i,GroupHomomorphismByImages(Hwiab[i].Gab,Hab.Hab,Hwiab[i].gen,
            phi1iimage));
    od;
    phi1:=GroupHomomorphismByFunction(Hw,Hab.Hab,x->
     Product([1..Length(HGGv)],i->Image(phi1i[i],Image(Projection(Hw,i),x))));
    Dr:=Image(phi1,ker);
    Drinv:=AbelianInvariantsSNF(Dr);
    Habbase:=DiagonalMat(Hab.inv);
    Drmat:=List(GeneratorsOfGroup(Dr),x->EltFiniteAbelian(Hab.Hab,x,Hab.gen));
    Drmat:=LatticeBasis(Concatenation(Habbase,Drmat));
    Drmat:=LatticeBasis(Difference(Drmat,Habbase));
    return rec(Dr:=[Drinv,[Hab.inv,Drmat]],
        Hab:=Hab, Gab:=Gab);
end;

MaximalSubgroups2:= function(G)
    Reset(GlobalMersenneTwister);
    Reset(GlobalRandomSource);
    return SortedList(MaximalSubgroups(G));
end;

SchurCoverG:= function(G)
    local epi,iso,ScG,ScGg,GG,GGg,Gg,n,i,id;
    Reset(GlobalMersenneTwister);
    Reset(GlobalRandomSource);
    epi:=EpimorphismSchurCover(G);
    iso:=IsomorphismPermGroup(Source(epi));
    ScG:=Source(epi);
    ScGg:=GeneratorsOfGroup(ScG);
    GG:=Range(iso);
    GGg:=List(ScGg,x->Image(iso,x));
    Gg:=List(ScGg,x->Image(epi,x));
    epi:=GroupHomomorphismByImages(GG,G,GGg,Gg);
    n:=NrMovedPoints(Source(epi));
    if n>=2 and n<=30 and IsTransitive(Source(epi),[1..n]) then
        for i in [1..NrTransitiveGroups(n)] do
            if Order(TransitiveGroup(n,i))=Order(Source(epi)) and
             IsConjugate(SymmetricGroup(n),
              TransitiveGroup(n,i),Source(epi)) then
                id:=[n,i];
                break;
            fi;
        od;
        return rec(SchurCover:=Source(epi), epi:=epi, Tid:=id);
    else
        return rec(SchurCover:=Source(epi), epi:=epi);
    fi;
end;

MinimalStemExtensions:= function(G)
    local ScG,ScGg,K,MK,ans,m,pi,cG,cGg,iso,GG,GGg,Gg,epi,n,i,id;
    ScG:=SchurCoverG(G);
    ScGg:=GeneratorsOfGroup(ScG.SchurCover);
    K:=Kernel(ScG.epi);
    MK:=MaximalSubgroups2(K);
    ans:=[];
    for m in MK do
        pi:=NaturalHomomorphismByNormalSubgroup(ScG.SchurCover,m);
        cG:=Range(pi);
        cGg:=List(ScGg,x->Image(pi,x));
        iso:=IsomorphismPermGroup(Range(pi));
        GG:=Range(iso);
        GGg:=List(cGg,x->Image(iso,x));
        Gg:=List(ScGg,x->Image(ScG.epi,x));
        epi:=GroupHomomorphismByImages(GG,G,GGg,Gg);
        n:=NrMovedPoints(Source(epi));
        if n>=2 and n<=30 and IsTransitive(Source(epi),[1..n]) then
            for i in [1..NrTransitiveGroups(n)] do
                if Order(TransitiveGroup(n,i))=Order(Source(epi)) and
                 IsConjugate(SymmetricGroup(n),
                  TransitiveGroup(n,i),Source(epi)) then
                    id:=[n,i];
                    break;
                fi;
            od;
            Add(ans,rec(MinimalStemExtension:=Source(epi), epi:=epi, Tid:=id));
        else
            Add(ans,rec(MinimalStemExtension:=Source(epi), epi:=epi));
        fi;
    od;
    return ans;
end;

ResHnZ:= function(arg)
    local RG,RH,n,G,H,inj,map,mapZ,CRGn,CRHn,HnG,HnH,m,res,null,ker,Hng,Hnggen,
          Hnh,Hnhgen,resHnggen,torbase,im,coker,hom,cokergen,cokergen1;
    RG:=arg[1];
    RH:=arg[2];
    n:=arg[3];
    G:=RG!.group;
    H:=RH!.group;
    inj:=GroupHomomorphismByFunction(H,G,x->x);
    map:=EquivariantChainMap(RH,RG,inj);
    mapZ:=HomToIntegers(map);
    if Length(arg)>=4 then
        CRGn:=arg[4];
    else
        CRGn:=CR_CocyclesAndCoboundaries(RG,n,true);
    fi;
    if Length(arg)=5 then
        CRHn:=arg[5];
    else
        CRHn:=CR_CocyclesAndCoboundaries(RH,n,true);
    fi;
    HnG:=CRGn.torsionCoefficients;
    HnH:=CRHn.torsionCoefficients;
    if HnG=[] then
        if HnH=[] then
            return rec(HnGZ:=[],HnHZ:=HnH,Res:=[],Ker:=[[],[[],[]]],
                       Coker:=[[],[[],[]]]);
        else
            return rec(HnGZ:=[],HnHZ:=HnH,Res:=[],Ker:=[[],[[],[]]],
                       Coker:=[HnH,[HnH,IdentityMat(Length(HnH))]]);
        fi;
    fi;
    if HnH=[] then
        return rec(HnGZ:=HnG,HnHZ:=[],
            Res:=List(HnG,x->[]),Ker:=[HnG,[HnG,IdentityMat(Length(HnG))]],
                      Coker:=[[],[[],[]]]);
    fi;
    m:=List(IdentityMat(Length(HnG)),x->
            CRHn.cocycleToClass(mapZ!.mapping(CRGn.classToCocycle(x),n)));
    null:=NullspaceIntMat(m);
    Hng:=AbelianGroup(HnG);
    Hnggen:=GeneratorsOfGroup(Hng);
    Hnh:=AbelianGroup(HnH);
    Hnhgen:=GeneratorsOfGroup(Hnh);
    resHnggen:=List(m,x->Product([1..Length(Hnhgen)],y->Hnhgen[y]^x[y]));
    res:=GroupHomomorphismByImages(Hng,Hnh,Hnggen,resHnggen);
    ker:=Kernel(res);
    im:=Image(res);
    null:=List(GeneratorsOfGroup(ker),x->EltFiniteAbelian(Hng,x,Hnggen));
    torbase:=DiagonalMat(HnG);
    null:=LatticeBasis(Concatenation(torbase,null));
    null:=LatticeBasis(Difference(null,torbase));
    hom:=NaturalHomomorphismByNormalSubgroup(Hnh,im);
    coker:=Image(hom);
    if Order(coker)=1 then
        return rec(HnGZ:=HnG,HnHZ:=HnH,Res:=m,
            Ker:=[AbelianInvariantsSNF(ker),[HnG,null]],Coker:=[[],[HnH,[]]]);
    fi;
    cokergen:=FindGenFiniteAbelian(coker);
    cokergen1:=List(cokergen,x->Representative(PreImages(hom,x)));
    cokergen1:=List(cokergen1,x->EltFiniteAbelian(Hnh,x,Hnhgen));
    return rec(HnGZ:=HnG,HnHZ:=HnH,Res:=m,
        Ker:=[AbelianInvariantsSNF(ker),[HnG,null]],
              Coker:=[AbelianInvariants(coker),[HnH,cokergen1]]);
end;

CosetRepresentationTid:= function(G,H)
    local Gg,HG,HGg,HGgr,n,i,id;
    Gg:=GeneratorsOfGroup(G);
    HG:=RightCosets(G,H);
    HGg:=List(Gg,x->Permutation(x,HG,OnRight));
    HGgr:=Group(HGg,());
    n:=Index(G,H);
    if n=1 then
        id:=[1,1];
    elif n<=30 then
        for i in [1..NrTransitiveGroups(n)] do
            if Order(TransitiveGroup(n,i))=Order(HGgr) and
              IsConjugate(SymmetricGroup(n),TransitiveGroup(n,i),HGgr) then
                id:=[n,i];
                break;
            fi;
        od;
    else
        id:=fail;
    fi;
    return id;
end;

AlwaysHNPholds:= function(Tid)
    local n,i,tbl,tbl4,tbl6,tbl8,tbl9,tbl10,tbl14,tbl15;
    tbl4:=[2,4];
    tbl6:=[4,12];
    tbl8:=[2,3,4,9,11,13,14,15,19,21,22,31,32,37,38];
    tbl9:=[2,5,7,9,11,14,23];
    tbl10:=[7,26,32];
    tbl14:=[30];
    tbl15:=[9,14];
    tbl:=[[],[],[],tbl4,[],tbl6,[],tbl8,tbl9,tbl10,[],[],[],tbl14,tbl15];
    if Tid=fail then
        return fail;
    fi;
    n:=Tid[1];
    i:=Tid[2];
    if IsPrime(n) or n=1 then
        return true;
    elif n=12 or n>15 then
        return fail;
    elif i in tbl[n] then
        return false;
    else
        return true;
    fi;
end;

IsMetacyclic:= function(G)
    local p;
    if Order(G)=1 then
        return true;
    fi;
    for p in Set(Factors(Order(G))) do
        if not IsCyclic(SylowSubgroup(G,p)) then
            return false;
        fi;
    od;
    return true;
end;

ChooseGi:= function(bG,bH)
    local bGs,Gicandidates,Gis,cGi,Gi,His,Hi,flag;
    bGs:=ConjugacyClassesSubgroups(bG);
    Gicandidates:=Filtered(bGs,x->not IsMetacyclic(Representative(x)));
    Gis:=[];
    for cGi in Gicandidates do
        for Gi in Elements(cGi) do
            His:=Reversed(List(ConjugacyClassesSubgroups(Intersection(Gi,bH)),
             Representative));
            flag:=false;
            for Hi in His do
                if AlwaysHNPholds(CosetRepresentationTid(Gi,Hi))=true then
                    Add(Gis,Gi);
                    flag:=true;
                    break;
                fi;
            od;
            if flag=true then
                break;
            fi;
        od;
    od;
    return Gis;
end;

KerResH3Z:= function(G,H)
    local RG,CRG3,H3Z,torbase,kerbase,Gis,Gi,RGi,ker,H3,H3g,K;
    if IsNilpotent(G) then
        RG:=ResolutionNormalSeries(LowerCentralSeries(G),4);
    elif IsSolvable(G) then
        RG:=ResolutionNormalSeries(DerivedSeries(G),4);
    else
        RG:=ResolutionFiniteGroup(G,4);
    fi;
    CRG3:=CR_CocyclesAndCoboundaries(RG,3,true);
    H3Z:=CRG3.torsionCoefficients;
    if H3Z=[] then
        return [[],[[],[]]];
    fi;
    torbase:=DiagonalMat(H3Z);
    kerbase:=IdentityMat(Length(H3Z));
    Gis:=ChooseGi(G,H);
    for Gi in Gis do
        if IsNilpotent(Gi) then
            RGi:=ResolutionNormalSeries(LowerCentralSeries(Gi),4);
        elif IsSolvable(Gi) then
            RGi:=ResolutionNormalSeries(DerivedSeries(Gi),4);
        else
            RGi:=ResolutionFiniteGroup(Gi,4);
        fi;
        ker:=ResHnZ(RG,RGi,3,CRG3).Ker;
        kerbase:=LatticeIntersection(kerbase,Union(ker[2][2],torbase));
        kerbase:=LatticeBasis(kerbase);
    od;
    kerbase:=LatticeBasis(Difference(kerbase,torbase));
    H3:=AbelianGroup(H3Z);
    H3g:=GeneratorsOfGroup(H3);
    K:=Group(List(kerbase,x->Product([1..Length(x)],y->H3g[y]^x[y])),Identity(H3));
    return [AbelianInvariantsSNF(K),[H3Z,kerbase]];
end;
\end{verbatim}
}


\end{document}